\documentclass[10pt,reqno]{amsart}
\RequirePackage[OT1]{fontenc}
\RequirePackage{amsthm,amsmath}
\usepackage{mathtools}
\RequirePackage[numbers]{natbib}
\RequirePackage[colorlinks,linkcolor=blue,citecolor=blue,urlcolor=blue]{hyperref}
\usepackage{amssymb}
\usepackage{anysize}
\usepackage{hyperref}
\usepackage{extarrows}
\usepackage{multirow}
\usepackage{natbib}
\usepackage{stmaryrd}
\usepackage{color}
\usepackage{amsmath}
\usepackage{enumitem}
\usepackage{mathtools} 
\usepackage{dsfont} 
\usepackage{comment}
\usepackage{soul}
\usepackage{amsmath,amsthm,amssymb, bbm, array, graphicx, tikz, tikz-cd}
\usepackage{pgfplots}
\usepackage{bm}
\usepackage{extarrows}

\usepackage{bm}

\numberwithin{equation}{section}
\theoremstyle{plain} 
\newtheorem{theorem}{Theorem}[section]
\newtheorem{lemma}[theorem]{Lemma}
\newtheorem{corollary}[theorem]{Corollary}
\newtheorem{proposition}[theorem]{Proposition}

\newtheorem{assumption}[theorem]{Assumption}

\newtheorem{remark}[theorem]{Remark}

\renewcommand{\Re}{\mathrm{Re}\,}
\renewcommand{\Im}{\mathrm{Im}\,}

\newcommand{\E}{{\mathbf E }}

\newcommand{\R}{{\mathbb R }}
\newcommand{\N}{{\mathbb N}}
\newcommand{\Z}{{\mathbb Z}}
\renewcommand{\P}{{\mathbf P}}
\newcommand{\C}{{\mathbb C}}

\newcommand{\X}{{\mathcal X}}

\newcommand{\s}{{\mathcal S}}
\newcommand{\LL}{{\mathcal L}}

\newcommand{\II}{{\mathcal I}}
\newcommand{\F}{{\mathfrak F}}

\newcommand{\ov}{\overline}
\DeclareMathOperator{\Tr}{Tr}

\DeclareMathOperator{\sgn}{sgn}

\DeclarePairedDelimiter{\abs}{\lvert}{\rvert}
\newcommand{\dif}{\operatorname{d}\!{}}
\DeclarePairedDelimiter{\norm}{\lVert}{\rVert}%

\newcommand{\ii}{\mathrm{i}}
\newcommand{\ee}{\mathrm{e}}

\newcommand{\dd}{\mathrm{d}}
\newcommand{\ie}{\emph{i.e., }}
\newcommand{\eg}{\emph{e.g., }}
\newcommand{\cf}{\emph{c.f., }}
\newcommand{\wt}{\widetilde}

\newcommand{\wh}{\widehat}
\newcommand{\gz}{G^{z}}

\newcommand{\uu}{\mathbf{u}}
\newcommand{\vv}{\mathbf{v}}
\newcommand{\ww}{\mathbf{w}}

\newcommand{\nc}{\normalcolor}

\newcommand{\bs}{\boldsymbol}

\def\ga{G^{z_1}}
\def\gb{G^{z_2}}

\def\Tr{\mathrm{Tr}}

\def\F{\mathcal{F}}

\def\one{\mathds{1}}
\def\Dim{\big[\mathfrak{D}_{E_0,\eta_0} \widetilde{\mathrm{Im}}\big]\,}

\def\<{\langle}
\def\>{\rangle}

\renewcommand{\mathbf}[1]{\bs{#1}}
 
\marginsize{25mm}{25mm}{25mm}{26mm}

\allowdisplaybreaks

\begin{document}

\begin{minipage}{0.85\textwidth}
		\vspace{2.5cm}
	\end{minipage}
	\begin{center}
		\large\bf Universality of extremal eigenvalues of large random  matrices. 
		
	\end{center}

	\renewcommand*{\thefootnote}{\fnsymbol{footnote}}	
	\vspace{0.5cm}
	
	\begin{center}
		\begin{minipage}{1.0\textwidth}
			\begin{minipage}{0.33\textwidth}
				\begin{center}
					Giorgio Cipolloni\footnotemark[1]\\
					\footnotesize 
					{University of Rome Tor Vergata}\\
					{\it cipolloni@axp.mat.uniroma2.it}
				\end{center}
			\end{minipage}
			\begin{minipage}{0.33\textwidth}
				\begin{center}
					L\'aszl\'o Erd\H{o}s\footnotemark[2]\\
					\footnotesize 
					{IST Austria}\\
					{\it lerdos@ist.ac.at}
				\end{center}
			\end{minipage}
			\begin{minipage}{0.33\textwidth}
				\begin{center}
					Yuanyuan Xu\footnotemark[2]\\
					\footnotesize 
					{AMSS,~CAS}\\
					{\it yyxu2023@amss.ac.cn}
				\end{center}
			\end{minipage}
		\end{minipage}
	\end{center}
	
	\bigskip
	
	\footnotetext[1]{\footnotesize{Partially supported by the MUR Excellence Department Project MatMod@TOV awarded to the Department of Mathematics, University of Rome Tor Vergata, CUP E83C18000100006.}}
	\footnotetext[2]{\footnotesize{Partially supported by ERC Advanced Grant "RMTBeyond" No.~101020331.}}

\renewcommand*{\thefootnote}{\arabic{footnote}}
	
	\vspace{5mm}

	\begin{center}
		\begin{minipage}{0.91\textwidth}\small{
				{\bf Abstract.}}
				We prove that the spectral radius of 
a large  random matrix $X$ with independent, identically distributed  complex
entries follows the Gumbel law irrespective 
of the distribution of the matrix elements. This solves a long-standing conjecture
of Bordenave and Chafa{\"\i}   and it  establishes
the first universality result for one of the most prominent extremal spectral statistics
in random matrix theory. 
Furthermore, we also prove that the argument of the largest eigenvalue is uniform on the unit circle and that the extremal eigenvalues of $X$ form a Poisson point process. \nc

		\end{minipage}
	\end{center}

	\vspace{5mm}
	
	{\footnotesize
		{\noindent\textit{Keywords}: Extremal statistics, Gumbel distribution, Ginibre ensemble, Dyson Brownian motion}\\
		{\noindent\textit{MSC number}:  60B20, 60G55, 60G70}\\
		{\noindent\textit{Date}:  \today \\
	}
	
	\vspace{2mm}

	\thispagestyle{headings}

\bigskip

\normalsize

\section{Introduction}

\subsection{The main result}
Let $X$ be a large $n\times n$ matrix with independent, identically distributed (i.i.d.) complex entries that have zero 
expectation and variance $1/n$. It is well known by the {\it Circular Law}
\cite{Girko1984,Bai1997,GT10,PZ10,TV10} 
that most eigenvalues lie in the unit disc, in fact their density becomes uniform as $n$ tends to infinity.
Denote by $\rho(X)$ 
the spectral radius of $X$,  i.e. the modulus of its largest eigenvalue.
Our main result is that the fluctuation of $\rho(X)$, after appropriate rescaling, follows the
{\it Gumbel Law}, in particular it is universal. More precisely,  we show that
\begin{equation}\label{informal}
  \rho(X) = 1+\sqrt{\frac{\gamma_n }{4n}}+\frac{\mathcal{G}_n}{\sqrt{4n \gamma_n}}, 
  \qquad \gamma_n:= \log n-2\log \log n-\log 2 \pi,
\end{equation}
where $\mathcal{G}_n$ is an asymptotically Gumbel random variable, i.e. $\P (\mathcal{G}_n\le r) \to \exp(-\ee^{-r})$
for any $r\in \R$
as $n\to\infty$.  This solves a long-standing conjecture by Bordenave and Chafa{\"\i} \cite{BC12}, \cite{Chblog}, see also \cite{RS14}
and \cite{CP14} for the log-gas counterpart of this question.
 We also show that  the argument of the largest eigenvalue of $X$ is uniform on $[0,2\pi]$,
  and that its  few largest (in modulus) eigenvalues form a {\it Poisson point process}.

\subsection{Related universality results}
We now place  this result in the broader context of random matrix theory.
Since the revolutionary discovery of E. Wigner \cite{wigner_annals}  on the universal behaviour of local
eigenvalue statistics for large random 
matrices, rigorous proofs aiming at verifying Wigner's vision 
have been in the forefront of mathematical research.
Motivated by quantum mechanical Hamiltonians, Wigner considered Hermitian random matrices
and predicted  that the local correlation functions depend only on the basic symmetry type 
but otherwise are insensitive to the details of the matrix ensemble.
In particular, real symmetric and complex Hermitian matrices exhibit somewhat
different but universal eigenvalue statistics. Mathematically this was formalised as the Wigner-Dyson-Gaudin-Mehta universality
conjecture in the 1960's   \cite{Metha}.  After important partial results 
 on the most prominent {\it Wigner
matrices}  \cite{Johansson} 
 and {\it sample covariance matrices} \cite{BP_CPAM},
  the conjecture was proven in full generality   in a series of papers by L. Erd{\H o}s, H.-T. Yau, and their collaborators,
 after they developed the theory of {\it Dyson Brownian motion (DBM)} for spectral universality in the {\it  bulk }
 of the spectrum, see  \cite{bull_american_math_society,book}  for a summary.  In a parallel development,
 around the same time, T. Tao and V. Vu  \cite{TV11} 
 also proved bulk universality for the complex symmetry class by a different method based upon  the {\it four
 moment matching} idea. The eigenvalue distribution near the {\it spectral edges} has been identified 
 in the early 1990's by C. Tracy and H. Widom \cite{TW1,TW2},
 and its universality was proven shortly afterwards by A. Soshnikov \cite{Soshnikov}. 
 The third, and last type of local spectral universality for Hermitian matrices, described by the {\it Pearcey process},
 occurring at the possible cusp singularities
 of the limiting eigenvalue density, was proven recently in  \cite{Cusp1,Cusp2}. This 
 largely completes the theory of Hermitian universality for the simplest mean field ensembles. However, 
much more work has been done and yet to be done on more general Hermitian ensembles,
 especially random band matrices \cite{Shcherbina1,Shcherbina2, CS22, CPSS, BE17, BEYY17,BYYY19, BYY18, EKYY, G22, S09, Sodin10,XYYY,YYY,YYY2,YY, DYYY25a, DYYY25b, ER25, YY25}, on adjacency matrices of sparse random graphs  \cite{sparse_bulk,sparse_edge,regular_bulk,regular_edge,ADK19,ADK20,ADK21,ADK23,HM23, HMY24}, and on L\'evy matrices \cite{ABL22, ALM21, ALM21b,  ABS09, BG14, BG13, BG17}, that we will not discuss here.

Non-Hermitian matrices exhibit similar universality patterns; although their mathematical theory is much 
less developed, the progress spectacularly accelerated  in the last  years. The local spectral 
universality for matrices with
i.i.d. entries at the 
spectral edge (the non-Hermitian analogue of Tracy-Widom regime) was proven in \cite{CES21}
both for real and complex matrices (see also \cite{CCEJ25, LZ25} for more general deformed i.i.d. matrices).
The universality in the bulk (analogue of  the Hermitian  Wigner-Dyson statistics) for the complex symmetry class
was solved very recently both in the complex \cite{bulk_complex} (see also \cite{ASS24, Z24}) and in the real case \cite{DY25, Osmanreal}.
The existence of the critical edge
(analogue of the Hermitian cusp regime) was proven in \cite{HongChang2023} for both symmetry classes, the
local statistics around these critical points were computed in  \cite{Liu2023}  for the complex case, and the critical edge universality was recently proven in \cite{CEJ25}.
 
In contrast to local spectral statistics, 
our main result \eqref{informal} 
 is the prime example of a universal {\it extremal} statistics arising in 
non-integrable random matrix theory. 
It has no natural Hermitian counterpart: the largest eigenvalue of a Hermitian matrix 
is always found around the endpoint of the support of the density and its distribution 
has already been identified by the Tracy-Widom edge universality. In contrast, the spectrum of
a non-Hermitian i.i.d matrix is essentially supported on the unit disc by the Circular Law, 
there are about $\sqrt{n}$  eigenvalues outside  but very close,
and around each reference point on the unit circle the eigenvalues locally form a strongly correlated
determinantal process \cite{Ginibre,BS09,CES21}. 
 Thus  the spectral radius is  realised as the maximum of $\sqrt{n}$ different eigenvalue moduli
 that themselves are correlated, albeit on  a short scale. Their tails decay fast but they are  not compactly
 supported hence, if they were independent, their extreme value distribution would naturally fall into the Gumbel class
 among the three possible limiting extreme value distributions given by the
 Fisher-Tippett-Gnedenko theorem  \cite{FT1928,Gnedenko} 
 for extrema of i.i.d. random variables. The correlation, however, prohibits us from using standard
 extreme value theory directly. Although our final conclusion coincides with the naive prediction
 based upon independence, in fact we even prove that the few largest eigenvalues form a 
 Poisson point process,   for the proof we need to follow a completely different path than these heuristics.
 
For completeness, 
we mention that another very prominent extremal statistics problem arises in 
 the celebrated Fyodorov-Hiary-Keating (FHK) conjecture \cite{FHK} that connects
 extrema of log-determinants of Haar unitary random matrices with the Riemann $\zeta$-function.
 In particular, the precise distribution of the maximum of the logarithm of the characteristic polynomial
 of the circular $\beta$-ensemble was identified very recently \cite{Zeitouni22}~(see previous papers \cite{ABB17,CMN18,Zeitouni18}), proving the random 
 matrix side of the FHK-conjecture. For the currently best result on  the number theory side of FHK-conjecture,
 see \cite{APM20,APM23} and references therein.  We also refer to \cite{CL25, L19} for recent progress in the computation of the leading order asymptotic in a non-Hermitian analog of the FHK conjecture for i.i.d. matrices. \nc
  Superficially, these results are reminiscent to \eqref{informal}, 
 but the underlying point process is much more correlated than 
 our extreme eigenvalues, in particular the limiting law is not given directly by the Gumbel distribution, and even the order $\log\log N$ term is different due to the log-correlated structure. 
 On the other hand,  for the circular $\beta$-ensemble \nc
 the structure of the underlying randomness is very different, and, as consequence, so are the proof methods. 
 For the circular $\beta$-ensemble the joint density of all eigenvalues is explicitly known, this
 serves as the fundamental starting point for the analysis.

 Returning to i.i.d. matrices,  if the original  random matrix  has independent Gaussian entries (called {\it Ginibre ensemble}; see the recent survey papers \cite{BF22a,BF22b}), then
 the limiting distribution of the spectral radius can be explicitly computed.
 The Gumbel law was obtained  in this way 
 for the complex Ginibre ensemble in \cite{R03} 
 and for the real Ginibre ensemble in  \cite{RS14} (see also \cite{Bender,AkemannP} for 
 similar calculations for {\it elliptic} Ginibre matrices). 
 Therefore our task  is to prove that this distribution for any i.i.d. matrix  is universal,
 hence it coincides with the Gumbel law for the Ginibre ensemble. 
 In this sense our proof is modelled by all previous universality proofs; the 
 core of the work is to show universality, then the 
limiting distribution is identified via explicit calculations performed
 on the distinguished Gaussian ensemble. In this paper we prove universality 
 for the complex ensemble. The core of our method
 also applies to the real case but the technical details are somewhat more tedious and will be left to future work.

 \subsection{New ingredients}
 Our proof relies on two 
 new inputs  compared with any previous work 
 on spectral statistics of non-Hermitian matrices, in particular compared with our previous
 papers \cite{maxRe,SpecRadius} that identified the correct order of magnitude
 of the spectral radius but failed to identify its distribution itself. We briefly outline them here.

Almost all mathematical works on the spectrum of non-Hermitian random matrices  
(with the notable exception of  \cite{BCCT18,BCG22,FG23}) are based upon 
{\it Girko's formula}  \cite{Girko1984},  see e.g.  \cite{Bai1997,TV15,BC12,BYY14_bulk, AEK19, maxRe}.
Modulo an irrelevant cutoff at infinity (see~\eqref{girko0}), this identity,
for any sufficiently regular test function $f:\C\to \R$, asserts that
\begin{equation}\label{girko1}
\sum_{\sigma\in \mbox{Spec}(X)} f(\sigma)=-\frac{1}{4 \pi} \int_{\C} \Delta_z f(z) A(z) \dd^2 z, \qquad
A(z):= \int_0^\infty \Im \Tr G^{z}(\ii \eta) \dd \eta ,
\end{equation}
where $G^z$  is a family of Hermitized resolvents, 
parametrized by a  family of spectral parameters \nc $z\in \C$:
\begin{align}\label{def_G}
	 G^{z}(w):=(H^{z}-w)^{-1}, \qquad H^{z}:=\begin{pmatrix}
		0  &  X-z  \\
		X^*-\overline{z}   & 0
	\end{pmatrix},\qquad w \in \C\setminus \R.
\end{align}
The advantage of \eqref{girko1} is that linear eigenvalue statistics of the  original non-Hermitian problem
are translated into a Hermitian problem, which is technically more tractable since
Hermitian matrices tend to be much more stable against perturbations. 
In particular the resolvent $G^z$ tends to concentrate around its deterministic approximation. 
There are two main disadvantages of this formula: i) it requires the understanding of several $H^z$'s for different $z$ simultaneously,
 ii) the Laplacian $\Delta f$ is very large if $f$ is chosen strongly localized to capture few 
 eigenvalues.
 Moreover,  the spectral parameter
$\ii\eta$ must be integrated down to the real axis which makes the $\eta$-integral logarithmically divergent
and it requires an extra probabilistic regularisation via an independently obtained lower tail bound 
on the lowest non-negative eigenvalue $\lambda_1^z$ of $H^z,$ noticing that $\| G^z(\ii\eta)\|\lesssim
 |\lambda_1^z+\eta|^{-1}$.
 Very crudely speaking, the behaviour of a single resolvent $G^z(\ii\eta)$ is governed by two different mechanisms.
 If $\eta$ is much smaller than the lowest eigenvalue $\lambda_1^z$, then $G^z(\ii\eta)$ is strongly fluctuating,
 mainly depending on the fluctuation on $\lambda_1^z$. If $\eta$ is much bigger than the typical size of $\lambda_1^z$,
 then $G^z(\ii\eta)$ tends to be deterministic -- such concentration results
are commonly called {\it local laws}. In the intermediate regime both effects are present and contribute to the answer.
 A meticulous analysis of \eqref{girko1} 
along these ideas were behind all previous results on eigenvalue statistics of $X$ on fine scales 
 \cite{BYY14_bulk,BYY14,TV15,AEK18, CES19,CES22,maxRe,SpecRadius}, with the notable exceptions \cite{DY25, bulk_complex, Osmanreal} (which are restricted to the bulk of the spectrum). \nc
 
Catching the spectral radius via Girko's formula is especially difficult. In fact, the universality of size and fluctuation of the spectral radius was not known even under the additional four moment matching assumption \cite{TV15}, showing the difficulty of grasping this extremal spectral statistics.  The main difficulty lies in the fact that the test function $f$ needs to be a very 
{\it anisotropic} function; it is essentially a smoothed characteristic function of a narrow annulus of
radius $1+\sqrt{\gamma_n/4n}\approx 1$ and width of order $1/\sqrt{n\log n}$ -- a domain
where the eigenvalue with largest modulus is typically located.
 For such function $\int |\Delta f|$
is  especially large compared with its integral $\int f$ thus the effect of  any error term  in  estimating $A(z)$
  in~\eqref{girko1} will be drastically magnified when integrated against the absolute value of $\Delta f$. 
  As our previous work \cite{SpecRadius} on the size of $\rho(X)$ demonstrated, 
  all additional technical refinements 
along these ideas were barely not sufficient to identify the distribution of $\rho(X)$, so we need to change 
the strategy. 

The \emph{first new input} is to 
exploit the underlying \emph{chiral symmetry} of $H^z$ after an integration by parts in \eqref{girko1}.
As already noticed in~\cite{He23}, one integration by parts leads to
\begin{equation}\label{girko4}
\sum_{\sigma\in \mbox{Spec}(X)} f(\sigma)
 = -\frac{1}{ \pi}  \int_{\C} \partial_{\bar z} f(z) 
 \int_0^\infty \Im \Tr \big[ G^{z}FG^z\big] (\ii \eta) \dd \eta  \dd^2 z  \qquad \mbox{with} \quad F:=
\begin{pmatrix}
	0  &  1 \\
	0   & 0
\end{pmatrix}.
\end{equation}
As it was done in~\cite[Eq (1.9)]{He23}, the $\eta$-integration can then be performed
\[ 
\int_{\eta_0}^\infty
\Tr [G^z(\ii\eta)]^2  F \dd \eta= \Tr G^z(\ii\eta_0)  F,
\]
if an $\eta_0$ cutoff is introduced, typically a bit below the relevant microscopic scale
(i.e. below the typical spacing of eigenvalues; in the edge  regime it is  $n^{-3/4}$).
 In the  small $\eta$-regime the quadratic singularity 
of $G^2$ may be regularized by the smallest singular value,
$\int \dd \eta/[\lambda_1^z+\eta]^2\sim 1/\lambda_1^z$, but the available tail estimates
on $\lambda_1^z$ are not strong enough to control this singularity. Therefore,  in all previous works (see \eg \cite{CES21,He23}),
the $\eta\le \eta_0$ regime was handled without integration by parts. The advantage is that the 
singularity from a single resolvent, $\int \dd \eta/[\lambda_1^z+\eta]\sim \log|\lambda_1^z|$, is much milder and manageable with available tail estimates. The disadvantage is that in the small $\eta$ regime the
bad bound on $\int |\Delta f|$ has to be used, which prevents us from identifying the Gumbel distribution of extremal eigenvalues in \cite{maxRe,SpecRadius}.

The new observation is that in many aspects $G^zFG^z$ behaves better than its trivial $1/[\lambda_1^z+\eta]^2$
upper bound indicates, thanks to the chiral symmetry of $H^z$; for example its off-diagonal elements 
exactly vanish at $\eta=0$ if $\lambda_1^z\ne 0$. This  cancellation effect becomes stronger as $\eta$
gets smaller! We thus find that the right hand side of~\eqref{girko4} 
is controllable for much smaller $\eta$’s than before, as long as $\lambda_1^z\ne 0$. 
This changes the main strategy of the proof: instead of  initially subdividing  the 
 $\eta$-integration regimes, we first introduce a cutoff $\one_{\lambda^z_1 > E_0}$
 with an appropriately chosen $E_0$ (we will pick $E_0$ slightly below the typical spacing $n^{-3/4}$) 
 and use the power of~\eqref{girko4} for the almost whole
 $\eta$-integration regime 
 (essentially for any $\eta\ge n^{-100}$); see \eqref{eq0} below. 
  Thus the bad bound on $\int |\Delta f|$ is less harmful as
 it is used only in the extremely small $\eta$ regime. However, 
 the (regularized) cutoff $\one_{\lambda^z_1 > E_0}$
 and its $z$-derivatives  generate several additional terms that need new estimates, that will be handled by our second new input below.

We remark that one can push this idea further and perform two integrations by parts giving
rise to
\begin{equation}\label{girko2}
\sum_{\sigma\in \mbox{Spec}(X)} f(\sigma)=-\frac{1}{4 \pi} \int_{\C}  f(z) 
 \int_0^\infty \Im \Tr \Delta_z G^{z}(\ii \eta) \dd \eta  \dd^2 z.
\end{equation}
The $z$-integration in $\int |f(z)|$
now becomes harmless, but 
\[
\Delta_z G^{z} = 4 G^z F G^z F^* G^z+4 G^z F^* G^z F G^z 
\]
looks even more  singular. Chiral symmetry, however, 
 implies
$G^z F G^z F^* =0$ exactly at $\eta=0$ if $\lambda^z_1 \neq 0$, which gives more room for
analysis in the extremely small $\eta$ regime.
The actual proof is much  more subtle due to the delicate condition $\lambda^z_1 \neq 0$.
In fact, if one naively ignores this condition, then 
\[
\int_0^\infty \mathrm{Tr} \Delta_z[G^z(\ii \eta)]\dd\eta=-4 \ii \Tr(G^z F G^z F^*)(\ii\eta)\big|_{\eta=0}^{\eta=\infty}=0, \qquad \mbox{if $\lambda_1^z\neq 0$},
\] 
would yield that \eqref{girko2} is always zero which is clearly wrong.
The correct procedure is to regularise the condition
 $\lambda^z_1 \neq 0$~(actually the cutoff function $\one_{\lambda^z_1 > E_0}$, see Section~\ref{sec:strategy_gft} for details) and 
consider the $z$-derivatives of this regularisation as well in the integration by parts~\eqref{girko2}; see (\ref{two_int_by_part})-(\ref{term_1}) later.
While  two integrations by parts is not used  for our main result, 
we demonstrate its mechanism when we prove that the (rescaled) spectral radius 
converges not only in distribution sense but also in all finite moment sense 
to the Gumbel distribution.

\medskip

Our  \emph{second new input} is to  understand the decorrelation mechanism of the resolvents $G^z(\ii \eta)$
and $G^{z’}(\ii \eta)$ at two different distant parameters $z, z’$ in the local law regime, i.e. to prove
a \emph{multi-resolvent local law} with an effective $|z-z’|$-decay. We need it in the regime where $z$ is at the edge
of the circular law, $|z|\approx 1$,  which corresponds to the density of states of $H^z$ having  a cusp singularity at 0. 
 We recall that to 
 understand the distribution of~\eqref{girko1}, we will need to control its higher moments,
 hence we need the joint distribution of the resolvent $G^z(\ii \eta)$ at several different $z$’s
 that are typically far away from each other. Since heuristically the correlation length of the
 non-Hermitian eigenvalues is $n^{-1/2}$, we expect that $G^z$ and $G^{z’}$ are almost independent
 if $|z-z’|\gg n^{-1/2}$. This fact needs to be exploited both in the local law regime $\eta\gg n^{-3/4}$
 and in the opposite regime where the resolvent is strongly fluctuating owing to 
 $\lambda_1^z$. 
 The latter case is more delicate, so we explain it first.

We  prove that the lower
tail distribution of the lowest eigenvalues $\lambda_1^{z_1}, \lambda_1^{z_2}, \ldots  ,\lambda_1^{z_k}$
are essentially independent whenever the pairwise distance of the $z_p$ parameters is bigger than 
the typical eigenvalue spacing, $|z_p-z_{p'}|\gg n^{-1/2}$ for any $p\ne p'$.  This  improves 
our analogous previous result  that proved a similar independence but only when $|z_p-z_{p'}|\ge n^{-\epsilon}$ 
\cite[Proposition 4.3]{SpecRadius}. Here we rely on the strong contraction effect 
of the coupled {\it Dyson Brownian motion (DBM)} 
that allows us to mimic the joint distribution of several $\lambda_1^z$ 
by a collection of {\it a priori} independent eigenvalues.  This idea has been exploited 
earlier \cite{CES19,CES22,maxRe,SpecRadius},  but the current precision down to
 the scale $|z_p-z_{p'}|\gg n^{-1/2}$ has only been done in the bulk and now we need it in
 the more delicate the edge regime. 
 In fact, we use exactly the same  DBM analysis as
  in \cite[Proposition 4.3]{SpecRadius} but its key input  is now stronger: we use
   that the eigenvectors of $H^z$
 belonging to $\lambda_1^{z}$ and $\lambda_1^{z'}$ are essentially orthogonal as long as $|z-z'|\gg n^{-1/2}$.
 This fact is proven by controlling products of resolvents  $G^zG^{z'}$ for spectral parameters
 $\eta$ slightly above the eigenvalue spacing.

In this paper we prove a local law for $\mbox{Tr} G^z(\ii\eta) [G^{z'}(\ii\eta)]^* $
with an effective $|z-z’|$--decay in the edge regime $|z|, |z’|\approx 1$.
In fact, 
 it is the  simplest prototype  of 
 a multi-resolvent local law  to understand the correlation decay 
 when $|z-z'|\gg n^{-1/2}$ since this quantity also governs the
 covariance of $\mbox{Tr} G^z(\ii\eta)$ and $\mbox{Tr}  G^{z'}(\ii\eta)$, directly
 needed in the $\eta\gg n^{-3/4}$ regime in~\eqref{girko1}.
 Strictly speaking, for the proof of our Gumbel result, the 
explicit rate of $|z-z’|$--decay  would  not
 be  necessary, we only need that some decay kicks in already from $|z-z’|\gg n^{-1/2}$,
 but our proof directly gives an effective decay without any additional effort.
 
 When $z=z'$, then 
 a simple resolvent identity (also called the {\it Ward-identity}
 in this context) shows that $G^z(\ii\eta) [G^{z}(\ii\eta)]^* = \eta^{-1}\Im G^z(\ii \eta)$,
 i.e. a product of two resolvents can be reduced to one resolvent. Roughly speaking, a multi-resolvent local law
 for $z=z'$ 
 can be reduced to a single resolvent  local law that is much better understood. However, for  $z\ne z'$ 
 the spectral resolutions of $H^z$ and $H^{z'}$ are different and such reduction is not possible;
 a new idea is needed that we will explain  in Section~\ref{sec:dyn}.
  
 In fact, we will also need another type of multi-resolvent local laws for $G^zF$, $G^zFG^z$ and $G^z F G^z F^*$ 
 to handle the larger $\eta$ regime in~\eqref{girko2} and the $z$-derivative of the regularized cutoff $\one_{\lambda^z_1 > E_0}$ with an improvement owing to 
 the special choice of $F$.
 The key phenomenon that $GAG$ 
 is much smaller if the deterministic $A$  is chosen in a special subspace
 has been first discovered in  \cite{CES21} for the resolvent $G$ of a Wigner matrix.
 Later it has been extended to more general
  situations \cite{Function_CLT,Optimal_local,Rank_uniform,CEHS23,CEH23} 
  by developing the method of (static)
 {\it master inequalities} for the fluctuations in the local law. However, here we 
  need to exploit this phenomenon  in the cusp regime for $G^z$ which would be especially cumbersome
  via the master inequalities.  A new dynamical approach to local laws, gradually developed very recently, also solves this problem in a short and elegant way.

 \subsection{Dynamical approach to multi-resolvent local laws}\label{sec:dyn}
 Multi-resolvent local laws have recently been derived in various situations; they are necessary to prove the 
 central limit theorem for linear eigenvalue statistics
  \cite{KKP96_CLT,LP_CLT,Shcherbina_CLT,HK17,LS20,LLS23,CES19,Function_CLT,CES22, meso34, LSX21, SW13, Vova1, Vova2}
 and the Eigenstate
 Thermalisation Hypothesis   \cite{CES21,CEH23,ADXY23}.
 Their canonical proof is to derive an approximate self-consistent equation, e.g. for   $\mbox{Tr} G^z[G^{z'}]^* $
 for definiteness, then to compute
 the fluctuation error in this equation by a cumulant expansion and invert the equation.  
 In certain cases, including  $\mbox{Tr} G^z[G^{z'}]^* $, the self-consistent equation is very unstable
 which is manifested by the smallness of the explicitly computable smallest eigenvalue of its linear stability operator.
 This instability is offset by a cancellation mechanism on the fluctuation side that 
 is effective exactly in the "worst direction" of the stability operator. Such cancellation is algebraically visible
 when Ward identity is available, but otherwise it remains very mysterious.  One expects an "effective Ward-like identity"
 operating in the background, but it has never been found in a sufficiently robust form.
 The interested reader is encouraged to work out the simplest instance of this phenomenon: prove the local law
 for $G(z)$ and for $G(z)G^*(z)$, where
 $G(z)= (W-z)^{-1}$ is the resolvent of the standard (Hermitian) Wigner matrix $W$ following the cumulant expansion
 strategy, e.g. in  \cite{KKP96_CLT,HK17,Optimal_local}.
  The procedure applied directly for $GG^*$ will give a bound that is off by a factor $1/|\Im z|\gg 1$. 
 If, however, one first uses the Ward identity,  $GG^* = \frac{1}{\Im z}\Im G$, and then applies the procedure for
 $\Im G$, then it gives the optimal result.
 
 In this paper we present a different solution to this {\it "missing Ward-identity"}
 problem  that does not require to invert the stability operator at all
 in its unstable direction:
  we apply the method of {\it characteristic flow}  for a critical multi-resolvent situation.
The basic idea  to estimate resolvents {\it dynamically}
 appeared first in~\cite{LeeSchnelli} in the context of edge universality.
  Here we consider a matrix valued Ornstein-Uhlenbeck process $X_t$
 with initial value $X_0=X$, denote by $G_t$ its Hermitized resolvent after a $z$-shift.
 We also let both spectral parameters $z$ and $\eta$ evolve  in time in such a way
 that the time derivative of $\mbox{Tr} G_t^{z_t}(\ii \eta_t) [G_t^{z_t'}(\ii \eta_t)]^* $ has
 an algebraic cancellation among its leading terms.
 The corresponding equation for the spectral 
 parameters is called the characteristic flow, the name originates from the fact that 
 in the simplest case of the single resolvent of a Wigner matrix this equation describes the characteristics of the underlying
 complex Burgers equation \cite{meso7, meso30, Past72, VS19}. 
 Along this flow $\ii\eta_t$ approaches the real axis, thus
  one can effectively compare the resolvents at small $\eta$ with resolvents at large $\eta$ where local 
  laws are much easier to prove since the stability operator has no unstable direction. The price
  one pays for this convenience is the added Gaussian component to $X_t$, but this can 
  later be removed by fairly standard {\it Green function comparison theorems (GFT)}.

 The method of characteristic flow has  been effectively
 used for  local laws and closely related quantities in various models in 
  \cite{meso30,meso1,LLS23,meso2,meso34,meso9,meso7,LeeSchnelli} 
  but only for a single resolvent situation
 where it eventually leads to a linear equation thanks to the Ward identity.
 At the moment this paper was first posted on arXiv, \nc two or more  resolvents with dynamical method 
 were considered only in  
  \cite[Theorem 3.3]{CES22}, \cite[Proposition 4.5]{meso9}, \cite[Proposition 3.3]{CEH23}, and more recently in
    \cite[Lemma 4.2]{StoneYangYin}. The latter three works focused on a different, less delicate aspects and
    exhibited no correlation decay that would have been the analogue of our $|z-z'|$-decay. 
      In~\cite{CES22}  a local law for
  $\mbox{Tr} G^{z}(\ii \eta)  G^{z'}(\ii \eta') $ was proven in the bulk regime (i.e. when $|z|, |z'|\le 1-\epsilon$)
  with an error term that exhibited  a small improvement when $|z-z'|\gg n^{-1/2}$. 
  Our  new  result handles the  more critical edge regime  (i.e. when $|z|, |z'|\approx 1$)
  and it proves  an effective decay in $|z-z'|$.
    It has been observed earlier \cite{AEK19}  that 
   in the edge regime, $|z|\approx 1$,  the eigenvalue density of the Hermitised matrix $H^z$ 
   has a cusp like singularity near the origin and the stability operator has {\it two} small eigenvalues,
   hence two "bad directions". Therefore our current analysis requires to monitor all four quantities
   of the form $\mbox{Tr} G^{z}(\ii \eta) A_1 G^{z'}(\ii \eta')  A_2$, where the deterministic matrices
   $A_1, A_2$ lie in one or the other bad direction. 
    This is the first systematic analysis of the local law for  general $G^zA_1G^{z’}A_2$ that
   fully exploits the underlying structure of all relevant instabilities. After the first 
   version of the current work appeared on arXiv, this approach inspired the proof of several other
   recent result, see e.g. \cite{CEH23, CEHK24, CEX26, O25, O26}.
   
   There are two key reduction steps in this analysis. First, 
   the evolution equation for a two-resolvent quantity, schematically like
    $\mbox{Tr} GAGA$, naturally involves a three- and four-resolvent quantities
    of the form $\mbox{Tr} GAGAGAGA$ (see~\eqref{eq:fulleqaasimp} and \eqref{eq:quadvarnew} later),
     whose own evolution equation involves eight resolvents etc.
   This would lead to an infinite hierarchy of equations, reminiscent to the well-known BBGKY-hierarchy 
   from statistical physics  for correlation functions of interacting particle systems. Such hierarchy is uncontrollable
   without a {\it closure} step, e.g. $\mbox{Tr} GAGAGAGA$ needs to be estimated by the square of
   $\mbox{Tr} GAGA$. For this purpose we rely on a \emph{reduction inequality} (see \eqref{eq:redinnew} below),  this leads to a closed coupled system of four  {\it nonlinear} 
   differential equations plus an error term (its linearised form  is given in \eqref{eq:4by4ode} later).
      
   Second,   the linearised  part
  of this system blows up exactly when the characteristic flow reaches the real axis and we need to compute the
  precise rate of the blow-up as it gives the optimal  estimate in the local law, which is in fact necessary to obtain the desired $|z-z'|$--gain.   We then need to 
  control the higher order nonlinear terms and show that they remain negligible in  the entire  
  range of  $\eta$'s above the eigenvalue spacing, i.e. where local laws are expected to hold.
  Eventually the nonlinear terms would prevail, so with the terminology of nonlinear evolution equations
  our problem is strictly speaking
   {\it supercritical}, but with an effective tiny cutoff provided by the underlying discreteness of 
  spectrum.

 In summary, the methodological novelty of our work on multi-resolvent local laws is to identify 
   the nonlinear system of equations
 that governs the evolution of the multi-resolvent local laws along the characteristic flow in the critical cusp regime. For brevity, we present
 its analysis in our special situation of two resolvents of the Hermitisation of a complex i.i.d. matrix
 that is needed to prove the Gumbel law, but the method is very general; it  is applicable to prove
 any multi-resolvent local law for any mean-field matrix ensemble.
 Special cases of multi-resolvent local laws have already played a key role in 
 several recent works with broader scope  \cite{bulk_complex, NST23, PR23}
  and we expect their impact will continue to grow as they become available for more general ensembles (see e.g. \cite{DYYY25a, DYYY25b, ER25,YY25} for recent breakthroughs for random band matrices)\nc.
  In particular, the necessary
 precise explicit calculations near the blow-up point can be done systematically based upon the 
 underlying {\it Dyson equation}~\eqref{MDE} and the corresponding characteristic flow~\eqref{eq:matchar}, 
  which hold for very general random matrix ensembles.

\medskip
	
\subsection*{Acknowledgement} We thank Oleksii Kolupaiev for several critical remarks about 
 the draft version of this paper.  We also thank the anonymous referees whose comments substantially improved the readability of this paper. \nc

\subsection*{Conventions and notations}

For integers $k,l\in\N$, with $k<l$, we use the notations $[k,l]:= \{k,\dots, l\}$ and $[k]:=[1,k]$. For any $z\in\C$ we use 
$\dd^2 z:= \frac{1}{2}\ii(\dd z\wedge \dd \overline{z})$ for the two dimensional area form on $\C$ and 
 $\Delta_z:=4\partial_z \partial_{\bar z}$ denotes the Laplace operator on $\C$. 
For positive quantities $f,g$ we write $f\lesssim g$ and $f\sim g$ if $f \le C g$ or $c g\le f\le Cg$, 
respectively, for some constants $c,C>0$ which depend only on the constants appearing 
in~\eqref{eq:hmb}. We set $f\wedge g= \min\{ f, g\}$ and $f \vee  g= \max\{ f, g\}$. 
Furthermore, for $n$--dependent positive
quantities $f_n,g_n$ we use the notation $f_n\ll g_n$ to denote that $\lim_{n\to\infty} (f_n/g_n)=0$. 
Throughout the paper $c,C>0$ denote small and large constants, respectively, which may change from line to line. We denote vectors by bold-faced lower case Roman letters $\bm{x},\bm{y}\in\C^d$, for some $d\in\N$. Vector and matrix norms, $\lVert\bm{x}\rVert$ and $\lVert A\rVert$, indicate the usual Euclidean norm and the corresponding induced matrix norm. For any $d \times d$ matrix $A$ we use $\langle A\rangle:= d^{-1}\mathrm{Tr} A$ to denote the normalized trace of $A$. Moreover, for vectors ${\bm x}, {\bm y}\in\C^d$ we define the usual scalar product 
$\langle \bm{x},\bm{y}\rangle:= \sum_{i=1}^d \overline{x_i} y_i. $
Furthermore,  
$\lVert f\rVert_\infty$ and  $\lVert f\rVert_1$  denote the $L^\infty$ and  $L^1$-norm of a function $f$, respectively. 
Throughout the paper we will use the notations $\P^{\mathrm{Gin}}$ and $\E^{\mathrm{Gin}}$ to denote the corresponding probability and expectation for $X$ being a complex Ginibre matrix. Additionally, for two matrices $A, B\in \C^d$ we use the notation $\Im AB:=\Im(A)B$, i.e. the imaginary part applies only to the first matrix.

\noindent

We define the following special $2n\times 2n$ matrices
\begin{equation}
\label{eq:defF}
F:=
\begin{pmatrix}
	0  &  1 \\
	0   & 0
\end{pmatrix}, 
\qquad\quad
E_1:=
\begin{pmatrix}
	1  &  0 \\
	0   & 0
\end{pmatrix}, \qquad\quad E_2=
\begin{pmatrix}
0  &  0 \\
0   & 1
\end{pmatrix}, \qquad
E_\pm:=E_1\pm E_2=\begin{pmatrix}
1  &  0 \\
0   & \pm 1
\end{pmatrix},
\end{equation}
in particular, $E_+$ is  the $2n$--dimensional identity matrix.

\noindent

We will often use the concept of ``with very high probability'' for an $n$-dependent event,
 meaning that for any fixed $D>0$ the probability of the event is bigger than $1-n^{-D}$ if $n\ge n_0(D)$. Moreover, we use the convention that $\xi>0$ denotes an arbitrary small positive exponent 
  which is independent of $n$. We recall the standard notion of \emph{stochastic domination}:  
     given two families of non-negative random variables
\[
X=\left(X^{(n)}(u) \,:\, n\in\N, u\in U^{(n)}\right)\quad\text{and}\quad Y=\left(Y^{(n)}(u) \,:\, n\in\N, u\in U^{(n)}\right)
\] 
indexed by $n$ (and possibly some parameter $u$  in some parameter space $U^{(n)}$), 
we say that $X$ is {\it stochastically dominated} by $Y$, if for any $\xi, D>0$ we have \begin{equation}
\label{stochdom}
\sup_{u\in U^{(n)}} \P\left[X^{(n)}(u)>n^\xi  Y^{(n)}(u)\right]\le n^{-D}
\end{equation}
for large enough $n\geq n_0(\xi,D)$. In this case we use the notation $X\prec Y$ or $X= O_\prec(Y)$
and we say that $X(u)\prec Y(u)$ {\it holds uniformly} in $u\in U^{(n)}$.

\section{Main results}\label{sec:thm}
 
 We consider \emph{complex i.i.d.  matrices} $X \in \C^{n\times n}$, i.e. $n\times n$ matrices with independent and identically distributed (i.i.d.) entries $x_{ab}\stackrel{\mathrm{d}}{=}n^{-1/2}\chi$, for some complex 
 random variable $\chi$.  If $\chi$ is a standard complex Gaussian random variable, then $X$ is called a {\it (complex) 
 Ginibre matrix}.
 In general, we make the following assumptions on $\chi$:
\begin{assumption}
	\label{ass:mainass}
	 We assume that $\E \chi=0$, $\E |\chi|^2=1$ and $\E \chi^2=0$. Furthermore, we require that higher moments exist\footnote{Here, for simplicity, we assume that all moment exists, but inspecting our proof it is clear that we need to formulate this assumption only for any $p\le p_0$, for a certain fixed $p_0\in\N$.}, i.e. that for any $p\in\N$
	  there exists a constant $C_p>0$ such that
	\begin{equation}
		\label{eq:hmb}
		\E|\chi|^p\le C_p.
	\end{equation}
\end{assumption}

We denote the eigenvalues of $X$ by $\{\sigma_i\}_{i \in [n]}$, and label them so that\footnote{We chose this labelling purely for notational convenience. All our result are insensitive to the labelling;
 in fact, we think of the eigenvalues as an unlabelled point process.}
\begin{align}\label{order}
	|\sigma_1| \geq |\sigma_2| \geq \cdots \geq |\sigma_n|.
\end{align}
\subsection{Spectral radius}
The spectral radius  of $X$ is given by $\rho(X)=|\sigma_1|$. 
 Note that $|\sigma_1|>|\sigma_2|$ with high probability (see Lemma \ref{lemma:unique} later),
 so the spectral radius comes from a unique eigenvalue. If $\chi$ has a density function, then
 $|\sigma_1|>|\sigma_2|$  holds even almost surely.

Our first main result is the universality of the joint fluctuations of $|\sigma_1|$ and its argument: 
\begin{theorem}[Spectral radius]\label{thm:gumbel}
Let $X$ be a complex i.i.d. matrix satisfying Assumption \ref{ass:mainass}. Set
	\begin{align}\label{gamma}
		\gamma_n:= \log n-2\log \log n-\log 2 \pi.
	\end{align} 
Then for any fixed $r\in \R$ and $\theta \in [0,2\pi)$, we have\footnote{We formulated the convergence as a limiting statement, but inspecting our proof we can also obtain 
an effective  speed of convergence $O((\log\log n)^2/(\log n))$ in~\eqref{theta_x_joint} uniformly in $(r,\theta)$  similarly
to    \cite[Theorem 1]{maxRe_Gin}. Optimal speed of convergence $O(\log\log n/(\log n))$ was later achieved by \cite{Ma2025}.} 
\begin{align}\label{theta_x_joint}
	\lim_{n\rightarrow \infty}\P\Big(  |\sigma_1|\leq 1+\sqrt{\frac{\gamma_n }{4n}}+\frac{r}{\sqrt{4n \gamma_n}},\quad 0\leq \arg(\sigma_1) \leq \theta \Big)=\frac{\theta}{2\pi}\ee^{-\ee^{-r}}.
\end{align}
\end{theorem}

Note that \eqref{theta_x_joint} proves the universality of the joint distribution of the largest eigenvalue
 (in absolute value) and its argument.  In particular, by \eqref{theta_x_joint} it follows that  $\rho(X)=|\sigma_1|$ 
 has Gumbel fluctuation:
\begin{align}\label{spectral_radius}
	\lim_{n\rightarrow \infty}\P \Big( |\sigma_1|\leq 1+\sqrt{\frac{\gamma_n }{4n}}+\frac{r}{\sqrt{4n \gamma_n}}\Big)=\ee^{-\ee^{-r}}.
\end{align}	
Furthermore, \eqref{theta_x_joint} also implies that  asymptotically the law of $\arg(\sigma_1)$ 
is uniform and independent of $|\sigma_1|$:
\begin{align}\label{theta_uniform}
	\lim_{n\rightarrow \infty}\P\Big( \arg(\sigma_1) \in [0,\theta]\Big) =\frac{\theta}{2\pi}.
\end{align}

\begin{remark}
\textit{For simplicity we stated \eqref{thm:gumbel} only for the largest eigenvalue in absolute value, however, inspecting our proof in Section~\ref{sec:proof_main}, it is clear that we can also prove universality of the joint distribution of the largest $k$ eigenvalues and their arguments for any finite $k\in\N$.}
\end{remark}

\begin{remark}[Convergence of moments]
\label{remark:mom}
Our method proves not only the convergence in distribution to Gumbel, but also the convergence of all  moments. 
 More precisely, define the rescaled quantity
\[
Y_n:=\sqrt{4n\gamma_n}\left[|\sigma_1|-1-\sqrt{\frac{\gamma_n}{4n}}\right],
\]
 then, for any fixed $p$, we have that $\E |Y_n|^p$ converges to the $p$--th moment of a standard Gumbel random variable. A similar result applies to all the joint moments of the first $k$ largest eigenvalues (in absolute value) and their arguments. More details will be given in Remark~\ref{rem:betterGFT} and in Section~\ref{sec:polyF}.

\end{remark}

Our second main result is to prove that the eigenvalues (in modulus) 
around $1+\sqrt{\gamma_n/4n}$ with the scaling $\sqrt{4n\gamma_n}$
form an inhomogeneous Poisson point process in the complex plane in the $n\rightarrow \infty$
limit,  and we also identify its intensity function. 
\begin{theorem}[Poisson Point Process]\label{thm:poisson}
Under the conditions of Theorem \ref{thm:gumbel}, we rescale the eigenvalues
\begin{equation}\label{scale_eig}
	\sigma_j=\ee^{\ii\theta_j} \Big(1+\sqrt{\frac{\gamma_n}{4n}} + \frac{r_j}{\sqrt{4\gamma_n n}}\Big), \qquad r_j \in \R, \quad \theta_j \in [0,2\pi).
\end{equation}
    Fix any $t\in\R$ and any non-negative
function $g\in C^2_c(\C)$ supported on $\{z=r\ee^{\ii\theta} : r \in [t,\infty), \theta \in [0,2\pi)\}$.  Then, we have 
\begin{align}\label{poisson}
		\lim_{n\rightarrow \infty}\E \Big[ \ee^{-\sum_{j=1}^n  g(r_j \ee^{\ii \theta_j})}\Big]
		= \exp\Bigl(- \int_{0}^{2\pi}\int_{\R} \big(1-\ee^{-g(r\ee^{\ii\theta})}\big) \frac{\ee^{-r}}{2\pi}\dd r\dd \theta\Bigr).
\end{align}
\end{theorem}

\subsection{Rightmost eigenvalue}

Motivated by the long time behavior of the linear system of ODE's $\dot{\bold{u}}=X\bold{u}$, in \cite{maxRe}
 we computed the precise size of the real part of the rightmost eigenvalue of $X$. However, in \cite{maxRe} we were not able to identify the fluctuations of this quantity around its deterministic limit. Similarly to 
 Theorems~\ref{thm:gumbel} and \ref{thm:poisson} now we also prove the following universality result:
\begin{theorem}[Rightmost eigenvalue]\label{thm:maxRe}
	Under the conditions of Theorem \ref{thm:gumbel}, for any fixed $x\in \R$ we~have
\begin{align}\label{maxRe}
	\lim_{n\rightarrow \infty}\P \Big( \max_{j \in [n]} \Re \sigma_j\leq 1+\sqrt{\frac{\gamma'_n }{4n}}+\frac{x}{\sqrt{4n \gamma'_n}}\Big)=\ee^{-\ee^{-x}},
\end{align}
where $\gamma'_n$ is given by
\begin{equation}
\label{gamma_prime}
\gamma'_n:= \frac{\log n-5\log \log n-\log (2 \pi)^4}{2}.
\end{equation}
Furthermore, rescaling the eigenvalues 
\[
\sigma_j= 1+\sqrt{\frac{\gamma_n'}{4n}} + \frac{x_j}{\sqrt{4\gamma'_n n}} + \frac{\ii y_j}{(\gamma'_n n)^{1/4}},
\qquad x_j, y_j\in \R, 
\]
and fixing any $t\in\R$ and any non-negative function $g\in C^2_c(\C)$ supported on $[t,\infty)\times\ii\R$, 
 we have 
\begin{equation}
	\label{maxRe_ppp}
	\lim_{n\rightarrow \infty}\E \Big[ \ee^{-\sum_{j=1}^n  g(x_j+\ii y_j)}\Big]= \exp\left(-\int_\R\int_\R \big(1-\ee^{-g(x+\ii y)}\big) \frac{\ee^{-x-y^2}}{\sqrt{\pi}}\dd y\dd x\right).
\end{equation}
\end{theorem}

We point out that a similar result holds for the largest eigenvalue in any given fixed direction, we stated only the result for the eigenvalue with largest real part because it is the most relevant direction for the application 
to the stability theory of the ODE $\dot{\bold{u}}=X\bold{u}$.

\section{Main technical results: Multi--resolvent local laws at the spectral edge of $X$}
\label{sec:tools}

In this section we present the main technical results which are needed to prove the results in Section~\ref{sec:thm}. 
For any $z\in \C$, recall the definition of the resolvent $G^z$  of the matrix $H^z$ at spectral parameter $w$ 
from~\eqref{def_G}.
The $2n\times 2n$ matrix $H^z$ is known as the \emph{Hermitization} of $X-z$. The $2 \times 2$ block structure (\emph{chiral symmetry}) of $H^{z}$ yields a  spectrum that is symmetric with respect to zero, i.e. the eigenvalues of $H^z$ can be labelled as $\{\lambda^z_{\pm i}\}_{i \in [n]}$ such that  $0\leq \lambda_1^z \leq \lambda^z_2 \leq \cdots \leq \lambda^z_n$ and $\lambda^z_{-i}=-\lambda^z_{i}$. Note that $\{\lambda^z_i\}_{i \in [n]}$ are the singular values of $X-z$. Denote by $\bm{w}_{i}^z$ and $\bm{w}_{-i}^z$
 the orthonormal eigenvector of $H^z$ corresponding to the eigenvalue $\lambda_{i}^z$ and $\lambda_{-i}^z$,
then, as a consequence of the chiral symmetry, 
the eigenvectors naturally split into two $n$-vectors  in the form
\[
\bm{w}_{\pm i}^z=\left(\begin{matrix}
\bm{u}_i^z \\
\pm\bm{v}_i^z\end{matrix}\right), \qquad i\in[n],
\]
where $\bm{u}_i^z, \bm{v}_i^z$ are the left and right  singular vectors of $X-z$, respectively, normalized so 
that $\|\uu^z_i\|^2=\|\vv^z_i\|^2=1/2$. Furthermore, for spectral parameter $w=\ii \eta$, $\eta\in \R\setminus\{0\}$,
 on the imaginary axis, we have
\[
\gz_{vv}(\ii \eta)=\ii \Im \gz_{vv}(\ii \eta), \qquad \eta\Im \gz_{vv}(\ii \eta)>0, \qquad v\in[1,2n].
\]

In this section we prove {\it local laws}, i.e. concentration 
statements asserting that single resolvents and certain products of them
become deterministic in the large $n$ limit even if the spectral parameter $\eta$ is very small. 
All results are formulated in the {\it cusp regime}, i.e. for $|z|\approx 1$, where the limiting
density of states of $H^z$ has a cusp-like singularity at the origin. The cusp is technically
the most complicated regime but this is necessary to analyse the spectral properties of $X$ 
near its spectral edge, the unit circle.
 In Section~\ref{sec:singleG} we recall the {\it averaged} and {\it isotropic} single resolvent laws for $G^z$ from  \cite[Theorem 3.1]{SpecRadius}. In Section~\ref{sec:gaing1g2} 
 we pick $A_1,A_2\in \{ E_+,E_-\}$, with $E_\pm$ being defined in \eqref{eq:defF}, and
 we prove a new local law for $\langle G^{z_1}A_1G^{z_2}A_2\rangle$ 
  with spectral parameters on the imaginary axis,
 exhibiting that  the size of $\langle G^{z_1}A_1G^{z_2}A_2\rangle$ 
 decays as $|z_1-z_2|$ increases. Then in Section~\ref{sec:gaingfgf} we consider another mechanism
 to improve local laws:  
 we show that if $A_1,A_2\in \{F,F^*\}$, with $F$ defined in \eqref{eq:defF}, then
  $\langle G^zA_1\rangle$ and $\langle G^zA_1G^zA_2\rangle$ 
 are considerably  smaller than for  general $A_1, A_2$.
 The fact that the resolvent tested against certain matrices  is smaller than expected was first observed in \cite{CES21} for Wigner matrices and later extended to more general models in the bulk of the spectrum 
 (see e.g. \cite{ADXY23, CEHK23, CEHS23}) and at the spectral edges (see \cite{CEH23}). However, Theorem~\ref{theorem_F} below is the first result when we exhibit this effect at the cusp singularity
 of the spectral density. Finally, in Section~\ref{sec:smallsingval} we first show that the smallest singular values of $X-z_1,X-z_2,\dots, X-z_k$, with a fixed $k\in \N$, are asymptotically independent as long as $|z_i-z_j|\gg n^{-1/2}$, and then we recall a probabilistic tail bound
 for the smallest singular value of $X-z$ when $z$ is outside of the spectrum (i.e. $|z|>1$) from \cite{SpecRadius}.

\subsection{Local law for $G^{z}$}
\label{sec:singleG}

We first  define the deterministic approximation of $G^z(w)$ as follows:
\begin{align}\label{Mmatrix}
	M^{z}(w):=\begin{pmatrix}
		m^{z}(w)  &  -zu^z(w)  \\
		-\overline{z}u^z(w)   & m^{z}(w)
	\end{pmatrix},\quad \mbox{with} \qquad  u^z(w):=\frac{{m}^z(w)}{w+ {m}^z(w)},
\end{align}
where ${m}^z(w)$ is the unique solution of the scalar cubic equation with side condition
\begin{align}
	\label{m_function}
	-\frac{1}{{m}^z(w)}=w+{m}^z(w)-\frac{|z|^2}{w+{m}^z(w)}, \qquad\quad \mathrm{Im}[m^z(w)]\mathrm{Im}w>0.
\end{align}
We consider $M^z(w)$ as a $(2n)\times (2n)$ matrix which is block constant, 
i.e. it consists of four $n\times n$ blocks and each of them
 is a constant multiple of the $n\times n$ identity matrix. 
In fact, $M^z$ is the unique solution of  the {\it Matrix Dyson equation (MDE)} (see e.g. \cite{MDE})
 \begin{equation}\label{MDE}
- [M^z(w)]^{-1} = w + Z+\mathcal{S}[ M^z(w)], \qquad Z:=\left(\begin{matrix}
0 & z \\
\overline{z} & 0
\end{matrix}\right),
\end{equation}
with the side condition $\Im M^z(w) \cdot \Im w >0$. Here 
\begin{equation}\label{cov}
\mathcal{S}[\cdot]:=\langle \cdot E_+\rangle E_+-\langle \cdot E_-\rangle E_-, \qquad \mathcal{S}: \C^{2n\times 2n}
\to \C^{2n\times 2n},
\end{equation}
is the {\it covariance operator} of the Hermitization of $X$. 
Additionally, we define the {\it self-consistent eigenvalue density}  
 of $H^z$ by (see e.g. \cite{AEK18})
\begin{align}\label{rho_0}
	\rho^z(E):=\lim_{\eta\to 0^+}\frac{1}{\pi} \Im \<M^z(E+\ii \eta)\> , \quad E\in \R.
\end{align}
In the following we collect  properties of the density $\rho^z(E)$ as summarized in \cite[Section 3.1]{SpecRadius}
 (based upon results from \cite{Cusp1, AEK19, AEK20, MDE}):
\begin{enumerate}
	\item[1)] if $|z|<1$, then the density $\rho^z(E)$ has a local
minimum at $E=0$ of height $\rho^z(0) \sim (1-|z|^2)^{1/2}$; 

\item[2)] if $|z|=1$, then the density has an exact cubic cusp singularity at $E=0$; 

\item[3)]  if $|z|>1$, then there exists a small gap in the support of the symmetric density $\rho^z(E)$,  denoted by $[-\frac{\Delta}{2},\frac{\Delta}{2}]$ with $\Delta \sim (|z|^2-1)^{3/2}$; see Figure \ref{fig:M1}.  
\end{enumerate}
Moreover, we also recall the following quantitative properties of  the density $ \rho^z(w):=\pi^{-1} |\Im \<M^z(w)\>|$ extended to the complex plane, $w\in \C$.  If $|z|\le 1$, then for any $w=E+\ii\eta$, we have
$$ \rho^z(w)\sim (1-|z|^2)^{1/2} +(|E|+\eta)^{1/3}, \qquad |E|\leq c,\quad 0\leq \eta\leq 1,
$$
for some small $c>0$, 
while in the complementary regime
$|z|> 1$, for any $w=\frac{\Delta}{2}+\kappa+\ii \eta$, we have 
\begin{align}\label{rho_E}
	\rho^z(w) \sim \begin{cases}
		(|\kappa|+\eta)^{1/2} (\Delta+|\kappa|+\eta)^{-1/6}, &\quad \kappa \in [0,c]\\ 
		\frac{\eta}{(\Delta+|\kappa|+\eta)^{1/6}(|\kappa|+\eta)^{1/2}} , &\quad  \kappa \in [-\Delta/2,0]
	\end{cases},
	\quad 0\leq \eta\leq 1.  
\end{align}
Furthermore,  on the imaginary axis $m^{z}(\ii \eta)=\<M^z(\ii \eta)\>$ 
is purely imaginary  and $u^z(\ii \eta)$ defined in (\ref{Mmatrix}) is real. 
We may often drop the superscript $z$ and the argument $w=\ii \eta$ for brevity, \ie  we write $m=m^z(\ii\eta)$, and
\begin{align} \label{rho}
	\rho:=\frac{|\Im m|}{\pi} \sim \begin{cases}
		\eta^{1/3}+|1-|z|^2|^{1/2} , &\qquad |z| \leq 1\\
	\frac{\eta}{|1-|z|^2|+\eta^{2/3}}, &\qquad |z| > 1
	\end{cases},
	\qquad  0\leq \eta\leq 1.
\end{align}

\begin{figure}
	\centering
	\begin{tikzpicture}
		\draw[black] (0, 3.5) node{\large Spec $X$};
		\draw[black] [->] (-3, 0)--(3, 0);
		\draw[black] [->](0, -3)--(0, 3);
		\draw[black] (0, 0) circle (2);
		\path[fill=gray] (-0.58, -1.57) circle (0.07);
		\path[fill=gray] (1.16, 1.46) circle (0.07);
		\path[fill=gray] (-0.86, -0.05 )circle (0.07);
		\path[fill=gray] (-1.08, 1.76) circle (0.07);
		\path[fill=gray] (1.02, -1.38) circle (0.07);
		\path[fill=gray] (-0.18, -1.15) circle (0.07);
		\path[fill=gray] (-0.90, -1.87) circle (0.07);
		\path[fill=gray] (-1.94, 0.11) circle (0.07);
		\path[fill=gray] (-1.43, -1.08) circle (0.07);
		\path[fill=gray] (1.42, 0.64) circle (0.07);
		\path[fill=gray] (-1.49, 0.37) circle (0.07);
		\path[fill=gray] (-1.64, -0.36) circle (0.07);
		\path[fill=gray] (-0.55, 1.57) circle (0.07);
		\path[fill=gray] (-0.98, 1.04) circle (0.07);
		\path[fill=gray] (-0.37, 0.15) circle (0.07);
		\path[fill=gray] (-0.57, 0.59) circle (0.07);
		\path[fill=gray] (1.51, 0.07) circle (0.07);
		\path[fill=gray] (0.64, 1.91) circle (0.07);
		\path[fill=gray] (0.14, 1.40) circle (0.07);
		\path[fill=gray] (0.32, 0.93) circle (0.07);
		\path[fill=gray] (1.62, -0.46) circle (0.07);
		\path[fill=gray] (1.16, -0.44) circle (0.07);
		\path[fill=gray] (0.93, 0.32) circle (0.07);
		\path[fill=gray] (0.04, 0.21) circle (0.07);
		\path[fill=gray] (1.07, -1.01) circle (0.07);
		\path[fill=gray] (1.89, -0.51) circle (0.07);
		\path[fill=gray] (0.58, -0.77) circle (0.07);
		\path[fill=gray] (0.80, -0.24) circle (0.07);
		\path[fill=gray] (0.12, -0.34) circle (0.07);
		\path[fill=gray] (0.47, -1.85) circle (0.07);
		\path[fill=gray] (0.60, -1.34) circle (0.07);
		\path[fill=gray] (-0.82, -0.95) circle (0.07);
		\path[fill=gray] (-0.44, -0.38) circle (0.07);
		\path[fill=gray] (-1.59, 0.98) circle (0.07);
		\path[fill=gray] (, ) circle (0.07);
		\path[fill=gray] (, ) circle (0.07);
		\path[fill=gray] (, ) circle (0.07);
		\path[fill=red] (2, 1) circle (0.09);
		\draw[black](-2, 2) node{\large $|z|=1$};
		\draw(2, 1.05) node[right]{\large $z$};
	\end{tikzpicture}
	\qquad\qquad
	\raisebox{1cm}{
		\begin{tikzpicture}
			\draw[black] (0, 4.5) node{\large Spec $H^{z}$};
			\draw foreach\s in{1,...,3}{(0.5+0.15*\s+0.01*rand, -0.1) to+(0, 0.2)}; 
			\draw foreach\s in{1,...,15}{(1+0.1*\s+0.01*rand, -0.1) to+(0, 0.2)}; 
			\draw[black] (0.8, 0.3) node{$\lambda^z_1$};
			\draw[black][->] (-3,0) -- (3,0);
			\draw[black][->] (0, 0)--(0, 3.7);
			\draw[line width=1pt, red, domain=0:1.3 ,smooth,variable=\t]
			plot ({\t*\t*\t+0.5},{2*\t});
			\draw[black] (0.5, 0)--(0.5, -0.4);
			\draw[black][->] (0, -0.2)--(0.5, -0.2);
			\begin{scope}[xscale=-1]
				\draw[line width=1pt, red, domain=0:1.3 ,smooth,variable=\t]
				plot ({\t*\t*\t+0.5},{2*\t});
				\draw[black] (0.5, 0)--(0.5, -0.4);
				\draw[black][->] (0, -0.2)--(0.5, -0.2);
				\draw[black] (0,-0.7) node{\large $\Delta \sim (|z|^2-1)^{3/2}$};
				\draw foreach\s in{1,...,3}{(0.5+0.15*\s+0.01*rand, -0.1) to+(0, 0.2)}; 
				\draw foreach\s in{1,...,15}{(1+0.1*\s+0.01*rand, -0.1) to+(0, 0.2)}; 
			\end{scope}
			\draw[black] (2, 3.1) node{$\rho^z=\pi^{-1}\mathrm{Im}~m^{z}$};
		\end{tikzpicture}
	}
	\caption{
		For $z$ in the regime $|z|>1$ slightly beyond the unit disk, the picture on the right shows the
		corresponding density $\rho^z$ of $H^z$, which exhibits a small gap around zero of size $\Delta\sim (|z|^2-1)^{3/2}$; see the paragraph below \cite[Eq. (5.2)]{AEK19}.
	}
	\label{fig:M1}
\end{figure}
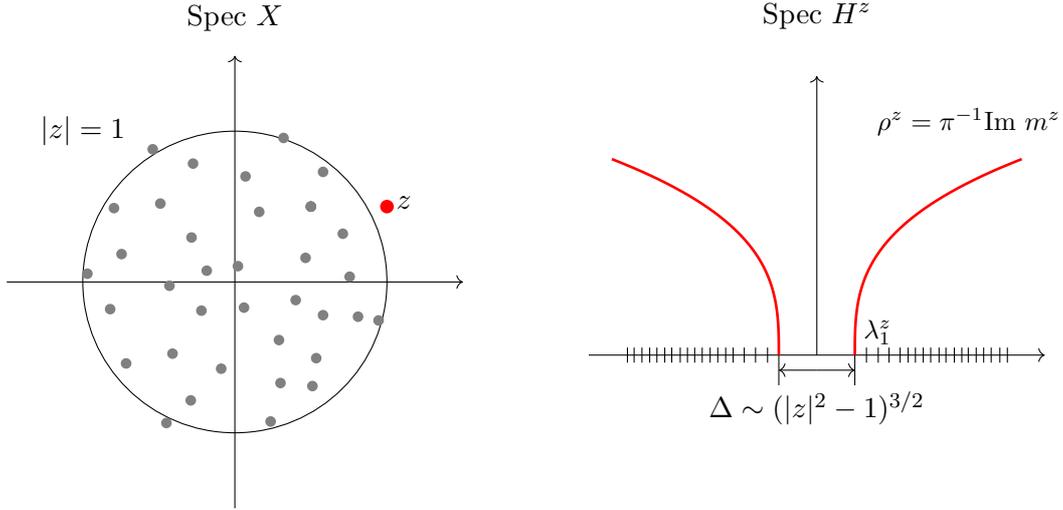

We are now ready to recall the  following optimal local law for the resolvent $G^z$ near the cusp.
\begin{theorem}[Theorem 3.1 in \cite{SpecRadius}]		\label{local_thm}
	There are sufficiently small constants $c,c'>0$ such that for any $z\in\C$ with $\big| |z|-1\big| \leq c$   and 
	for any $w=E+\ii\eta$ with $|E|\le c'$ we have
	\begin{align}
		\big| \langle \mathbf{x},  (G^{z}(w)-M^{z}(w)) \mathbf{y} \rangle\big| \prec \; & \|\mathbf{x}\| \|\mathbf{y}\| \left( \sqrt{\frac{\rho^z(w)}{n|\eta|}}+\frac{1}{n|\eta|}\right),\label{entrywise}\\
		\big|\big\<\big( G^{z}(w)-M^{z}(w)\big)A\big\>\big| \prec \; & \frac{\|A\|}{n|\eta|},
		\label{average}
	\end{align}
	uniformly in spectral parameter $|\eta|>0$, deterministic vectors $\mathbf{x}, \mathbf{y}\in\C^{2n}$, and matrices $A\in \C^{2n \times 2n}$.
\end{theorem}

By standard arguments, see e.g. \cite[Corollary 1.10-1.11]{AEK17}, using the local law
 in \eqref{average} as an input, we readily conclude the following rigidity estimate for the eigenvalues $\lambda_i^z$. For this purpose we define the \emph{quantiles} $\gamma_i^z$ of $\rho^z$ implicitly by the relation
\begin{equation}\label{quantile}
\int_0^{\gamma_i^z}\rho^z(x)\,\dd x=\frac{i}{2n}, \qquad\quad \mathrm{for}\quad i\in [n].
\end{equation}
For negative indices we set $\gamma_{-i}^z:=-\gamma_i^z$, with $i\in [n]$.

	\begin{corollary}[Corollary 3.2 in \cite{SpecRadius}]\label{cor:rigidity}
	Fix a small constant $c>0$ and pick $z$ such that  $0\le |z|-1\le c$. Then there exists a small $c'>0$ such that 	
		\begin{equation}\label{rigidity}
			|\lambda_i^z-\gamma_i^z|\prec \max\left\{\frac{1}{n^{3/4}|i|^{1/4}}, \frac{\Delta^{1/9}}{n^{2/3}|i|^{1/3}}\right\}, \qquad \quad  \mathrm{for}\quad |i|\le c' n,
		\end{equation}
				with $\Delta$ denoting the size of the gap around zero
		in the support of $\rho^z$.
		 In addition there exists a small $c''>0$ such that, for any $E_1<E_2$ with $\max\{|E_1|,|E_2|\} \leq c''$, we have
		\begin{equation}\label{rigidity3}
			\Big| \#\{j: E_1 \leq \lambda^z_j \leq E_2\}-2n \int_{E_1}^{E_2} \rho^z(x) \dd x \Big| \prec 1.
		\end{equation}
	\end{corollary}

\subsection{Local law for $G^{z_1}AG^{z_2}$ near the cusp}
\label{sec:gaing1g2}

We now consider the product of two resolvents $G^{z_1}AG^{z_2}$, with a deterministic matrix  $A\in\C^{2n\times 2n}$
in between, for $z_1\ne z_2$ close to but
outside of the unit disk and
with spectral parameters on the imaginary axis. This is the situation used in Girko's formula. 
Proving a local law for this quantity is quite challenging, 
because $G^{z_1}$ and $G^{z_2}$ belong to two different spectral families,
 i.e. no resolvent identity is available to linearize their product even if $A$ were the identity.
 Before stating the local law for $G^{z_1}AG^{z_2}$, 
 we define its deterministic approximation, 
 which is {\it not} simply the product $M^{z_1}AM^{z_2}$ as one might guess at first. 
 The correct formula is given by
 \begin{align}
\label{eq:defM12}
		M_{12}^A=M^A_{12}(\eta_1,z_1,\eta_2,z_2):= \mathcal{B}^{-1}_{12}\Big[M^{z_1}(\ii \eta_1) A M^{z_2}(\ii \eta_2)\Big],
	\end{align}
 where the {\it two--body stability operator}  $\mathcal{B}_{12}:\C^{2n\times 2n}\to\C^{2n\times 2n}$ is defined by
\begin{equation}
\label{eq:defstabop}
\mathcal{B}_{12}[\cdot]:=1-M^{z_1}(\ii\eta_1)\mathcal{S}[\cdot] M^{z_2}(\ii\eta_2).
\end{equation}
Furthermore, we have the following optimal  bound (proven in
 Appendix~\ref{sec:det})
\begin{align}\label{local_2g_M}
		\Big\|M_{12}^A(\eta_1,z_1,\eta_2,z_2)\Big\| \lesssim& \frac{\|A\|}{|z_1-z_2|+\frac{|\eta_1|}{\rho_1}+\frac{|\eta_2|}{\rho_2}},
\end{align}
using the short--hand notation 
$m_i:=m^{z_i}(\ii\eta_i)$ and $\rho_i:=\pi^{-1}|\Im m_i|$ 
for any $\eta_1,\eta_2\ne 0$ and $z_1,z_2\in\C$. We stress that we work in the regime where $0\le |z_i|-1\ll1$  and thus  the eigenvalue density $\rho^{z_i}$  
has a cusp-like singularity. This fact strongly influences the form of~\eqref{local_2g_M} and its optimality.
With our current method, a very similar (in fact easier) analysis would yield an optimal two-resolvent local law
  in the non-Hermitian bulk regime, $|z_i|\le 1-\epsilon$,
as it was done non-optimally in \cite{CES19, CES22, real_CLT}.  In this paper, however, we focus on
the  regime $|z_i|\ge 1$ needed for the results in Section~\ref{sec:thm}.

We are now ready to formulate our first multi--resolvent local law, which has two parts.
The first part is for random matrices with a substantial Gaussian component. 
The second part is for arbitrary random i.i.d. matrices, but instead of $\ga, \gb$ we consider only their imaginary parts 
and we test $\Im\ga A\Im\gb$ only against $A$. In both parts the test matrices are $A\in \{E_+,E_-\}$. 
This will be sufficient for our purposes and it simplifies the proof.  
In Section~\ref{sec:opt12} we will also comment on the optimality  and  extensions of this theorem.

\begin{theorem}[Local law  with $|z_1-z_2|$-decorrelation]\label{thm:2G}
Fix any small $c,\epsilon>0$.  Then, for any $z_1,z_2$ with $0\leq |z_i|-1\leq c$, the following statements hold uniformly in spectral parameters $\eta_1, \eta_2$ satisfying $n\ell \ge n^\epsilon$, where $\ell:=\min_i(|\eta_i|\rho_i)$.
\begin{enumerate}
\item[\bf Part A:] 
Let $X$ be a random matrix satisfying Assumption \ref{ass:mainass}, and let $X^{\mathrm{Gin}}$ be a complex Ginibre 
matrix independent of $X$. Fix any small $\mathfrak{s}>0$, consider a matrix of the form $X_{\mathfrak{s}}:=\sqrt{1-\mathfrak{s}^2}X+\mathfrak{s} X^{\mathrm{Gin}}$, and denote by $G_{\mathfrak{s}}^{z}(w)$ the resolvent of the Hermitization of $X_{\mathfrak{s}}-z$ defined as in \eqref{def_G}. 
Then we have		
\begin{align}
	\label{local_2g}
		\Big|\Big\< \big(\ga_{\mathfrak{s}}( \ii \eta_1 ) A_1 \gb_{\mathfrak{s}}(\ii \eta_2)-M_{12}^{A_1} \big)A_2\Big\>\Big| \prec& \left(\frac{1}{(n\ell)^{1/2}}\frac{1}{|z_1-z_2|+\frac{|\eta_1|}{\rho_1}+\frac{|\eta_2|}{\rho_2}}\right) 
		\wedge \left(\frac{1}{n|\eta_1\eta_2|}\right) \nc,
	\end{align}
for deterministic matrices $A_1,A_2\in \{E_+,E_-\}$, with $M_{12}^{A_1}=M^{A_1}_{12}(\eta_1,z_1,\eta_2,z_2)$ defined in~(\ref{eq:defM12}).

\item[\bf Part B:] Consider any i.i.d. random matrix $X$ satisfying Assumption \ref{ass:mainass}. Then we have	
\begin{align}
	\label{local_2g_im}
	\Big|\Big\< \big( \Im \ga( \ii \eta_1 ) A\Im \gb(\ii \eta_2)-\wh M_{12}^A \big)A\Big\>\Big| \prec& \left(\frac{1}{(n\ell)^{1/2}}\frac{1}{|z_1-z_2|+\frac{|\eta_1|}{\rho_1}+\frac{|\eta_2|}{\rho_2}}\right) 
	\wedge \left(\frac{1}{n|\eta_1\eta_2|}\right) \nc,
\end{align}
uniformly in comparable spectral parameters $|\eta_1|\sim |\eta_2|$, $\rho_1\sim \rho_2$, 
and deterministic matrix $A\in \{E_+,E_-\}$. Here the deterministic approximation 
$\wh{M}_{12}^A$  is naturally defined as a fourfold linear combination of $M^A_{12}(\pm \eta_1, z_1, \pm \eta_2, z_2)$, 
using the identity $\Im G=\frac{1}{2\ii}(G-G^*)$ twice.
\end{enumerate}	
\end{theorem}

In particular,  the error term in the local law \eqref{local_2g}  for $G_{\mathfrak{s}}^{z_1}A_1G_{\mathfrak{s}}^{z_2}$ is always smaller than the bound \eqref{local_2g_M} for the deterministic term $M_{12}^{A_1}$ since $n\ell\gg 1$. 
Note that $n|\eta_i|\rho_i$ is roughly the number of eigenvalues of $H^{z_i}$ in a window of size $\eta_i$ 
about the origin, thus the condition $n\ell\gg 1$ guarantees that we are in the local law regime for both resolvents
$G_i= G^{z_i}(\ii\eta_i)$, i.e. when
concentration of $G_1AG_2$ around a deterministic quantity  is expected.
Note that both estimates~\eqref{local_2g}--\eqref{local_2g_im} involve the minimum of two 
terms; the first one exhibits a $|z_1-z_2|$-decay, the second one is independent of $z_1, z_2$.

As a consequence of Theorem \ref{thm:2G}, we obtain the following upper bound on the overlaps between the eigenvectors $\ww^z_{i}=(\uu^z_i,\vv^z_i)$ for different $z_1,z_2$; for 
simplicity we state this result only for $z_1,z_2$ in the regime $0\le  |z_i|-1 \leq n^{-1/2+\tau}$.

\begin{corollary}\label{eigenvector_overlap}
Fix any small $\tau>0$. Then there exists a small $0<\omega_c<\tau/10$  such that for any $z_1,z_2$ with 
 $0\le  |z_i|-1 \leq n^{-1/2+\tau}$\nc we have
 \begin{equation}\label{overlap}
  |\<\uu^{z_1}_i,\uu_j^{z_2}\>|^2+|\<\vv^{z_1}_i,\vv_j^{z_2}\>|^2 \prec  \frac{n^{2\tau}}{\sqrt{n}[|z_1-z_2|+n^{-1/2}]}, \qquad  |i|, |j| \leq n^{\omega_c}.
 \end{equation}
\end{corollary}

\begin{proof}
	 By the rigidity estimate of the eigenvalues in \eqref{rigidity}
	and $\Delta \sim (|z|-1)^{3/2}\le n^{-3/4+3\tau/2}$, we have
	\begin{equation}\label{lg}
	|\lambda_i^z-\gamma_i^z| \prec \max\left\{\frac{1}{n^{3/4}|i|^{1/4}}, \frac{\Delta^{1/9}}{n^{2/3}|i|^{1/3}}\right\}\lesssim \frac{n^{\tau/6}}{n^{3/4}|i|^{1/4}}.
	\end{equation}
		Choose $\eta=n^{-3/4+\tau}$, then
	there exists $0<\omega_c<\tau/10$ such that  $|\lambda^{z}_i|\lesssim \eta$ for any $|i|\leq n^{\omega_c}$ using \eqref{lg} and the elementary bound  $|\gamma_i^z| \le \eta $ from \eqref{rho_E} and \eqref{quantile}.
Additionally, by the spectral decomposition, we have
	\begin{align}
		\<\Im G^{z_1}(\ii \eta) \Im G^{z_2}(\ii \eta)\>=\frac{1}{n}\sum_{i,j} \frac{\eta^2 \big(|\<\uu^{z_1}_i,\uu_j^{z_2}\>|^2+|\<\vv^{z_1}_i,\vv_j^{z_2}\>|^2\big)}{\big((\lambda^{z_1}_i)^2+\eta^2\big)((\lambda^{z_2}_j)^2+\eta^2\big)},
	\end{align}
	where we used that $\Im G^z$ is diagonal\footnote{The fact that $\Im G^z$ is diagonal follows by spectral decomposition and the symmetry of the spectrum around zero, $\lambda^z_{-i} =-\lambda^z_i$, (see the beginning of Section~\ref{sec:tools}):
	\[
	\Im G^z=\sum_{|i|\le N} \frac{\eta}{(\lambda_i^z)^2+\eta^2}\left(\begin{matrix}
	\uu_i^z(\uu_i^z)^* & \uu_i^z(\vv_i^z)^* \\
	\vv_i^z(\uu_i^z)^* & \vv_i^z(\vv_i^z)^*
	\end{matrix}\right)=2\sum_{i=1}^N \frac{\eta}{(\lambda_i^z)^2+\eta^2}\left(\begin{matrix}
	\uu_i^z(\uu_i^z)^* & 0 \\
	0 & \vv_i^z(\vv_i^z)^*
	\end{matrix}\right),
	\]
	where we also used that for $i\in [N]$ we have $\uu_{-i}^z=\uu_i^z$ and $\vv_{-i}^z=-\vv_i^z$.}. 
	Thus for any $|i|, |j| \leq n^{\omega_c}$, we have
	\begin{align}
		  |\<\uu^{z_1}_i,\uu_j^{z_2}\>|^2+|\<\vv^{z_1}_i,\vv_j^{z_2}\>|^2 \lesssim&  \; n\eta^2 \<\Im G^{z_1}(\ii \eta) \Im G^{z_2}(\ii \eta)\> \prec  \frac{n^{-1/2+2\tau}}{|z_1-z_2|}, 
	\end{align}
where we used the local law in (\ref{local_2g_im}) with $A=E_+$ and the deterministic bound in (\ref{local_2g_M}). This proves (\ref{overlap}) for $|z_1-z_2|\ge n^{-1/2}$.
For the complementary regime $|z_1-z_2|\le n^{-1/2}$, we use the trivial bound $|\<\uu^{z_1}_i,\uu_j^{z_2}\>|^2+|\<\vv^{z_1}_i,\vv_j^{z_2}\>|^2\le 1$, yielding~\eqref{overlap}.
\end{proof}

We conclude Section~\ref{sec:gaing1g2} by commenting on the optimality and restrictions in Theorem~\ref{thm:2G}.

\subsubsection{Optimality of the $\langle G^{z_1}A_1G^{z_2}A_2\rangle$ local law}
\label{sec:opt12}
First, in Theorem~\ref{thm:2G} we stated the local law only for $A_1,A_2\in \{E_+,E_-\}$ as we need only this case for our application and its proof  is shorter  but  with a similar method we can prove exactly the same bound for 
any deterministic $A_1,A_2\in \C^{2n \times 2n}$. Actually, the proof for $A_1,A_2\in \{E_-,E_+\}$ is the most delicate
 since the stability operator $\mathcal{B}_{12}$ has two small eigenvalues in $\mathrm{Span}\{E_-,E_+\}$, 
 while $\mathcal{B}_{12}$ has eigenvalues equal to $1$ on the complement of $\mathrm{Span}\{ E_-,E_+\}$.

 Second,  in the regime where we gain from the $|z_1-z_2|$--factor, \nc we only 
prove that the error term in the local laws \eqref{local_2g}--\eqref{local_2g_im} is smaller than the deterministic approximation by a factor $(n\ell)^{-1/2}$, however the optimal factor
would be $(n\ell)^{-1}$. We do not pursue this improvement here since we will use the
estimate in the regime $n\ell\sim n^\epsilon$, for some very small $\epsilon>0$, 
 where  practically there is no difference between the two estimates.  
 More importantly, however, the $|z_1-z_2|^{-1}$ decay in \eqref{local_2g} is optimal.

 In particular, \nc we expect that the optimal form of~\eqref{local_2g} is
\begin{equation}\label{optimallaw}
\Big|\Big\< \big(\ga( \ii \eta_1 ) A_1 \gb(\ii \eta_2)-M_{12}^{A_1} \big)A_2\Big\>\Big| \prec  \| A_1\|\|A_2\|\left(\frac{1}{n\ell}\frac{1}{|z_1-z_2|+\frac{|\eta_1|}{\rho_1}+\frac{|\eta_2|}{\rho_2}} \right)
\end{equation}
in the  regime  $0\le |z_i|-1\le n^{-1/2+\tau}$, 
and it holds for general $A_1, A_2$ and for any i.i.d. matrix $X$.  Note that once we replace $(n\ell)^{-1/2}$ with $(n\ell)^{-1}$ then the factor $1/(n|\eta_1\eta_2|)$ in \eqref{local_2g} becomes obsolete. The same bound also holds 
even if $0\le |z_i|-1 \le c$, but it is optimal only in the most relevant regime $|z_i|\approx 1$. 
If $|z_i|-1>n^{-1/2+\tau}$ then the error  
becomes smaller since in this regime there are typically no eigenvalues close to zero \cite[Theorems 2.1--2.2]{AEK19}
as the limiting eigenvalue density has a gap $\Delta$ of size $\Delta\sim(|z|^2-1)^{3/2}$.
 In fact, the basic estimate
\eqref{local_2g_M} on the deterministic term also improves for larger $\Delta$, but here we will not investigate the 
 regime far outside of the spectrum as it is irrelevant for our main results in Section~\ref{sec:thm}.
 
For technical convenience we stated \eqref{local_2g_im} only for the special case $|\eta_1|\sim |\eta_2|$ and $\rho_1\sim \rho_2$, because this regime is enough for our main application  in \eqref{overlap} and 
it  allows for a simpler {\it Green function comparison (GFT)} argument in Appendix \ref{app:lindGFT} to remove the Gaussian component in \eqref{local_2g}. This restriction can easily be relaxed using the stronger 
 Lindeberg replacement approach for the GFT by finding a matrix, with a large Gaussian component, which matches the third moment and (almost) matches the fourth moment  of $X$.  The same strengthening of GFT is necessary
 to achieve \eqref{optimallaw} for a general i.i.d. $X$ without Gaussian component. 

Furthermore, we mention that even relaxing the assumption $|\eta_1|\sim |\eta_2|$ and $\rho_1\sim \rho_2$ in \eqref{local_2g_im} would not give an optimal bound (modulo the $(n\ell)^{-1}$ improvement) as the optimal 
 local law for the product of two $\Im G$'s carries an additional smallness 
 involving the local density $\rho$ that appears because of considering $\Im G$ instead of $G$
 (see \cite[Theorem 2.4]{CEH23}  for a similar effect).
 We expect that
	\begin{align}
		\label{local_2g_optimal}
		\Big|\Big\< \big(\Im\ga( \ii \eta_1 ) A_1\Im\gb(\ii \eta_2)-\wh M_{12}^{A_1} \big)A_2\Big\>\Big| \prec& \;\frac{\rho_1\rho_2}{n\ell}\frac{\|A_1\|\| A_2\|}{\big[|z_1-z_2|+\frac{|\eta_1|}{\rho_1}+\frac{|\eta_2|}{\rho_2}\big]^2}
	\end{align}
	holds
	 in the  regime $0\le |z_i|-1\le c$ for general $A_1, A_2$ and for any i.i.d. matrix $X$,
	 and we expect this bound to be optimal when
	  $0\le |z_i|-1\le n^{-1/2+\tau}$. In fact, in \cite{CEX26} we already  proved \eqref{local_2g_optimal}, with the minor caveat\footnote{We did not obtain the optimal $(n\ell)^{-1}$-bound since, instead of analyzing it separately, 
	  we just trivially estimated chains of four resolvents in terms of the square of chains of two resolvents.}
	   that $n\ell$ was replaced with $\sqrt{n\ell}$.
In particular the local law \eqref{local_2g_optimal} has a $|z_1-z_2|^{-2}$ decay, this also reflects the fact that for  the 
deterministic approximation we have
\begin{equation}\label{MA}
\|\wh M_{12}^A\|\lesssim \frac{\|A\|\rho_1\rho_2}{\big[|z_1-z_2|+\frac{|\eta_1|}{\rho_1}+\frac{|\eta_2|}{\rho_2}\big]^2}.
\end{equation}

We stress that our method to prove multi--resolvent local laws dynamically
would enable us to prove \eqref{optimallaw}  and 
\eqref{local_2g_optimal} as well. For this purpose we would need to study the flow not only for traces
 $\langle G^{z_1}A_1G^{z_2}A_2\rangle$ as we do now, 
but also for {\it isotropic} quantities $\langle \bm{x},G^{z_1}A_1G^{z_2}\bm{y}\rangle$. For conciseness
we do not pursue these improvements here since we do not need them for our main results in Section~\ref{sec:thm}.

Lastly, we note that the optimal local law \eqref{local_2g_optimal} would imply  the optimal overlap bound
(cf.~\eqref{overlap})
\begin{equation}\label{optimaloverlap}
	|\<\uu^{z_1}_i,\uu_j^{z_2}\>|^2+|\<\vv^{z_1}_i,\vv_j^{z_2}\>|^2 \prec  \frac{n^{2\tau}}{n[|z_1-z_2|+n^{-1/2}]^2}, 
	\qquad  |i|, |j| \leq n^{\omega_c}.
\end{equation}

\subsection{Local law for $G^zF$ and $G^zFG^zF^*$ near the cusp}
\label{sec:gaingfgf}

We now explore a different mechanism to improve the single and multi--resolvent local laws,  \eqref{average}
and \eqref{local_2g}.  We  show that the size of  the resolvent $G^z(\ii\eta)$ 
 (in tracial sense) is smaller by $\rho$-factors when tested against the special matrices $F, F^*$ from~\eqref{eq:defF}.
\begin{theorem}[Local laws with $F$-improvement]\label{theorem_F}
	Fix any small $c,\epsilon>0$. For any $z\in \C$ with $0\le  |z|-1 \leq c$ \nc, we have
	\begin{align}
		\Big|\big\<\big(\gz(\ii \eta)-M^z(\ii \eta)\big)F^{(*)}\big\>\Big| \prec& \frac{\rho}{n|\eta|},\label{local_F}\\
		\Big|\big\<\gz(\ii \eta)F\gz(\ii \eta)-M^{F}_{12}(\ii \eta)\big\>\Big| \prec& 
		\frac{\rho}{n\eta^2},\label{local_GFG}\\
		 	\Big|\Big\< \big(\gz(\ii \eta) F \gz(\ii \eta) -M^F_{12} (\ii \eta)\big) F^*\Big\>\Big| \prec& \frac{\rho^2}{n\eta^2}\sqrt{n|\eta|\rho},\label{local_2F}
	\end{align}
	uniformly in spectral parameters $n|\eta|\rho\ge n^\epsilon$. Here $F^{(*)}$ denotes that both options $F$ and $F^*$ are allowed, $\rho=\rho^z(\ii \eta)$ is given in (\ref{rho}), the deterministic matrices $M^z$ and $M^F_{12}$ are defined in \eqref{Mmatrix} and \eqref{eq:defM12} with $z_1=z_2=z$, $\eta_1=\eta_2=\eta$, and $A=F$, respectively.
	\end{theorem}

	On the deterministic approximations $\big\<M^{F}_{12}(\ii \eta)\big\>$, $\big\<M^{F}_{12}(\ii \eta) F^* \big\>$ in \eqref{local_GFG} and \eqref{local_2F}, respectively, 
	we have the bounds (see 
	Lemma~\ref{lem:Mbounds}) 
	\begin{align}
\label{MFMF}
 \big|\big\<M^{F}_{12}(\ii \eta)\big\> \big|  \lesssim \frac{\rho^2}{|\eta|} \nc \qquad\qquad\quad 	\big|\big\<M^{F}_{12}(\ii \eta) F^* \big\> \big|  \lesssim \frac{\rho^3}{|\eta|},
\end{align}
i.e. the error terms in \eqref{local_GFG}--\eqref{local_2F} are   smaller than the bounds \eqref{MFMF} by a factor $(n|\eta|\rho)^{-1}$ and $(n|\eta|\rho)^{-1/2}$, respectively.  
 We point out that \eqref{local_F}--\eqref{local_GFG} are optimal, while \eqref{local_2F} is off by a factor $\sqrt{n|\eta|\rho}$. Our proof method in Section~
\ref{sec:G1G2llaw} can clearly improve \eqref{local_2F} to be optimal, but we omit this for brevity as we will not need it.

As a corollary of Theorem \ref{theorem_F}, we have the following optimal upper bound on the overlaps between the left and right singular vectors of $X-z$.
\begin{corollary}\label{coro_overlap}
	Fix any small $\tau>0$. Then there exists a small $0<\omega_c<\tau/10$ such that, for any 
	 $0\le  |z|-1 \leq n^{-1/2+\tau}$\nc, we have\footnote{Similarly to \cite[Theorem 2.4]{CEHS23},
	  the optimal bound \eqref{uv_bound} for singular vector overlap can be turned into an optimal lower bound on the overlap of non--Hermitian eigenvectors. More precisely, let $\{L_i,R_i\}_{i\in [n]}$ be the left and right bi--orthogonal eigenvectors of $X$ with corresponding eigenvalue $\sigma_i$. Then, for the diagonal overlap $O_{ii}:=\|L_i\|^2\|R_i\|^2$ we have
	\[
	O_{ii}\ge n^{1/2-\xi},
	\]
	with very high probability for any small $\xi>0$, where the index $i\in[n]$ is such that $|\sigma_i|\ge 1$.
	}	\begin{align}\label{uv_bound}
		|\<\uu_i^z, \vv_j^z\>|^2 \prec \frac{n^{2\tau}}{\sqrt{n}}, \qquad\quad |i|, |j| \leq n^{\omega_c}.
	\end{align}
\end{corollary}
We mention that this bound for $i=j$ 
 is optimal for $z$ very close to the edge. In the bulk regime, $|z|\le 1-\epsilon$,
the optimal bound is $\prec 1/n$, see \cite[Theorem 2.7]{CEHS23}. In other words, in the bulk regime
 the overlap of the left and right singular vectors behaves as if they were independent random vectors, 
 while at the edge  $|\< \uu_i^z, \vv_i^z\>|^2 \prec n^{-1/2}$ indicates a "half-way" between the full independence
 and the substantial  alignment   $|\< \uu_i^z, \vv_i^z\>|^2 \sim 1$  in the regime $|z|\ge 1+\epsilon$.

\begin{proof} Following  the proof of Corollary~\ref{eigenvector_overlap},
	choose $\eta=n^{-3/4+\tau}$ and recall that there exists $\omega_c>0$ such that $|\lambda^z_i|,|\lambda^z_j| \lesssim \eta$, for any $ |i|,|j|\leq n^{\omega_c}$.  Using the spectral decomposition of $\Im \gz$, we have
	\begin{align}
		|\<\uu^{z}_i,\vv_j^{z}\>|^2 \lesssim  n\eta^2 \<\gz(\ii \eta) F \gz(\ii \eta) F^*\> \prec  n\eta^2 \frac{\rho^3}{\eta} \lesssim n^{-1/2+2\tau},
	\end{align}
	where we used the local law in (\ref{local_2F}) and the deterministic bound in (\ref{MFMF}) and \eqref{rho} in the last step.
\end{proof}

We conclude this section with several useful identities which will often be used throughout the paper. 
They follow from the chiral symmetry of $H^z$. Their  elementary proof  is presented in Appendix~\ref{sec:det}.
 \begin{lemma}
 \label{lem:chirid}
Let $H^z$ be defined as in \eqref{def_G}, let $G^z=(H^z-\ii\eta)^{-1}$, and let $M^z$, $M^F_{12}$ be defined as in  \eqref{Mmatrix} and \eqref{eq:defM12} with $z=z_1=z_2=z$, $\eta_1=\eta_2=\eta$, and $A=F$, respectively. Then we have
\begin{equation}
 \label{eq:identities}
  \langle G^zFG^zE_-\rangle=0, \quad \langle M^F_{12}E_-\rangle=0, \quad \langle G^zE_-\rangle=0, 
  \quad \langle M^zE_-\rangle=0, \quad FG^zF^*=\ii F\Im G^z F^*.
 \end{equation}
 \end{lemma}

\subsection{Smallest singular value}
\label{sec:smallsingval}

In this section, we show that the lower tails of the smallest singular values of $X-z_1$, $X-z_2$, $\dots$, $X-z_k$ are asymptotical independent for any finite collection $\{z_j\}_{j=1}^k$ with mesoscopically distant values. This will follow by the almost orthogonality of the overlaps in \eqref{overlap} for $|z_i-z_{j}|\gg n^{-1/2}$ together with the \emph{Dyson Brownian Motion} argument from \cite[Section 7]{SpecRadius}. A similar result was obtained in \cite[Proposition 4.3]{SpecRadius} 
with $|z_i-z_{j}|\ge n^{-\gamma}$ for a small $\gamma>0$ (that result was 
stated for $k=2$ only, 
the proof can be easily extended to a general $k\geq 2$). The main difference is that there
we required a sufficiently large, almost macroscopic distance between $z_i$ and $z_j$. 
Now the improved eigenvector overlap estimates in Corollary \ref{eigenvector_overlap} allow us to extend the statement to the optimal regime 
$|z_i-z_{j}|\geq n^{-1/2+\omega}$ for any small $\omega>0$. This is formulated
in the following proposition whose proof is postponed to Appendix~\ref{app:tail_gft}.

\begin{proposition}
	\label{prop_zz}
	Let $X$ be a complex i.i.d. random matrix satisfying Assumption \ref{ass:mainass}, and let $\lambda_1^z$ be the smallest singular value of $X-z$. Fix any $k\in \N$ and  any sufficiently small $ \tau,\omega,\gamma>0$ such that
	 $\tau\leq \gamma\omega/10$.
		 Then for
	 any $\{E_j\}_{j\in [k]}$ with $0< E_j\lesssim n^{-3/4}$, for
	  any $\{z_j\}_{j\in [k]}$ with $ 0\leq |z_j|-1 \leq n^{-1/2+\tau}$ and any $|z_{i}-z_{j}|\ge n^{-1/2+\omega}~(\forall i,j\in[k])$, we have
		\begin{align}\label{lambdatail}
		\prod_{j=1}^k \P \Big( \lambda^{z_j}_1 \leq E_j-4 n^{-3/4-\gamma\omega+\tau}\Big)-&O(n^{-D}) \lesssim	\P\Big( \lambda_1^{z_j} \leq E_j,~\forall j \in [k]\Big) \nonumber\\
		&   \lesssim \prod_{j=1}^k \P \Big( \lambda^{z_j}_1 \leq E_j+4 n^{-3/4-\gamma\omega+\tau}\Big)+ O(n^{-D}),
	\end{align}	
	for any fixed (large) $D>0$ if $n$ is sufficiently large.
\end{proposition}

 Note that the lower bound above is effective only when $E_j> n^{-3/4-\gamma \omega +\tau}$. To estimate the upper and lower bounds in (\ref{lambdatail}) precisely, \nc we recall from~\cite{maxRe,SpecRadius} the following tail bound for the smallest singular value of $X-z$ for $z$ being outside of the unit disk.  
\begin{proposition}[Proposition 5.2 in \cite{SpecRadius}]\label{prop1} 
	Let $X$ be a complex i.i.d. random matrix satisfying Assumption~\ref{ass:mainass}, and let $\lambda_1^z$ be the smallest singular value of $X-z$. Fix any sufficiently small $\epsilon, \epsilon',\tau>0$ with $0<\tau/2<\epsilon'$. For any $E$ with $n^{-1+\epsilon}\leq E
	\leq n^{-3/4-\epsilon'}$, and any $z\in \C$ with $n^{-1/2} \ll |z|-1 \leq n^{-1/2+\tau}$\nc, 
		we have\footnote{The  restriction $|z|-1 \gg n^{-1/2}$  can be relaxed to $|z|-1 \gtrsim n^{-1/2}$; see  \cite[Remark A.3]{maxRe}. }
	\begin{align}\label{tail_bound}
		\P\big(\lambda_1^z\leq E \big) \lesssim n^{3/2} E^2 \ee^{-n \delta^2/2}+n^{-D}, \qquad \delta:=|z|^2-1,
	\end{align}
	for any fixed (large) $D>0$ if $n$ is sufficiently large.
\end{proposition}

\bigskip

\section{Proofs of Theorems~\ref{thm:gumbel}, \ref{thm:poisson} and \ref{thm:maxRe} }
\label{sec:proof_main}

For the purpose of proving our main universality results, Theorems~\ref{thm:gumbel}, \ref{thm:poisson}
and   \ref{thm:maxRe},
 we will compare the statistics of the spectral radius and of the eigenvalue with the largest real part 
of general i.i.d. matrices directly with the ones of the complex Ginibre ensemble, which are
explicitly known. The key technical input  is the following Green function comparison theorem.
 Its proof  will be given in Sections \ref{sec:proof_gft}-\ref{sec:L_0}.

\begin{theorem}[Green function comparison]\label{main_thm}
	Fix any integer $k\geq 1$ and any sufficiently small $\epsilon,\tau,\nu>0$ with $\max\{\tau,\nu\}<\epsilon/100$. 	For 
	each $1\leq p\leq k$, let $f_p(z):\C \rightarrow \R$ be a  
 $C^2$-function such that 
	\begin{align}\label{f_norm}
		\|f_p\|_\infty \lesssim 1, \qquad \|\partial_z f_p(z)\|_\infty+\|\partial_{\bar z} f_p(z)\|_\infty \lesssim n^{1/2+\nu}, \qquad	\|\Delta_z f_p(z)\|_\infty \lesssim n^{1+2\nu},
	\end{align}
and assume that $f_p$ is compactly supported on the annulus 
	\begin{align}\label{f_cond}
		&\sqrt{\frac{\gamma_n}{4n}}-\frac{C_n}{\sqrt{4n \gamma_n}} \leq |z|-1 \leq \frac{n^{\tau}}{\sqrt{n}},
			\end{align}
	 with $\gamma_n$ given in (\ref{gamma}) and $C_n/\sqrt{\log n} \rightarrow 0$ as $n \rightarrow \infty$.   \nc
	Furthermore, let $\F(w_1,\cdots,w_k): \R^k \rightarrow \R$ be any smooth (but not necessarily bounded) function with uniformly bounded partial derivatives, \ie for some constants $C_{\bm{\alpha}}>0$, 
	\begin{align}\label{F_cond}
		\big|\partial^{\bm{\alpha}} \F(w_1,\cdots, w_k)\big| \leq C_{\bm{\alpha}}, \qquad \bm{\alpha}:=(\alpha_1,\cdots,\alpha_k), \qquad 1\leq |\bm{\alpha}|\leq 10.\nc
	\end{align}
	Then, we have
	\begin{align}\label{error}
		\left|\E\Big[\F\Big(\sum_{i=1}^n f_1(\sigma_i), \cdots, \sum_{i=1}^n f_k(\sigma_i)\Big)\Big] -\E^{\mathrm{Gin}}\Big[\F\Big(\sum_{i=1}^n f_1(\sigma_i), \cdots, \sum_{i=1}^n f_k(\sigma_i)\Big)\Big]\right|=O(n^{-\epsilon}).
	\end{align}
\end{theorem}

\begin{remark}
\label{rem:maxre}
\textit{Theorem~\ref{main_thm} also holds true if instead of \eqref{f_cond}, $\{f_p\}_{p=1}^k$ are supported on the set
	\begin{align}\label{f_cond_2}	&\sqrt{\frac{\gamma'_n}{4n}}-\frac{C_n}{\sqrt{4n \gamma'_n}} \leq |z|-1 \leq \frac{n^{\tau}}{\sqrt{n}},  \qquad |\mathrm{arg}(z)| \leq \frac{n^{\tau/2}}{n^{1/4}},
	\end{align}
	 with $\gamma_n'$ given in (\ref{gamma_prime}) and $C_n/\sqrt{\log n} \rightarrow 0$ as $n \rightarrow \infty$. The proof is analogous and will be omitted.
	 The support~\eqref{f_cond_2}
	 is tailored  for proving the universality of the eigenvalue with the largest real part in Theorem~\ref{thm:maxRe}.
	 In both cases, $\gamma_n$ and $\gamma_n'$ are chosen such that 
	 if $C_n= O(1)$ then typically there are finitely ($n$-independent) many eigenvalues
	 in the sets \eqref{f_cond} and \eqref{f_cond_2}, respectively, hence $\sum_{i=1}^n f(\sigma_i)$ in
	 the argument of $\F$ is typically finite. }
\end{remark}

Note that in this Green function comparison theorem we use only that the first and second moments
of the two ensembles match, no higher moment matching is used. Furthermore, we stated Theorem~\ref{main_thm} for the class of functions $\mathcal{F}$ in \eqref{F_cond} to keep the proof simpler and because this is enough for
 proving convergence of the (shifted and rescaled) spectral radius to Gumbel
 in distribution. 
 However, the same result holds  for a larger class of test functions, that, in particular, allows us
to conclude the convergence  in   any finite moment sense 
(Remark~\ref{remark:mom}). 

	\begin{remark}
	\label{rem:betterGFT}
	\textit{
		Theorem \ref{main_thm}  holds true for any  $\F$ with derivatives of polynomial growth, i.e. 
		\begin{align}\label{F_gen}
		\big|\partial^{\bm{\alpha}} \F(w_1,\cdots, w_k)\big| \leq C_{\bm{\alpha}} \prod_{p=1}^k \big(1+|w_p|\big)^{C\alpha_p}, \qquad 1\leq |\bm{\alpha}|\leq 10
	\end{align}
	with some $C>0$.
Later in Section \ref{sec:polyF}  we will  briefly discuss the modifications needed to extend Theorem \ref{main_thm} to this larger class of functions.
}
\end{remark}

Using Theorem~\ref{main_thm}, we now prove the main results stated in Section \ref{sec:thm}, as well as the fact that with high probability there is a unique eigenvalue such that $|\sigma_1|$ is equal to the spectral radius~(\ie Lemma \ref{lemma:unique} below). For pedagogical reasons, we first present the proof of \eqref{spectral_radius} and \eqref{theta_uniform}, and then explain the minor modifications needed to prove Theorem~\ref{thm:gumbel} in full generality. Then we prove Lemma \ref{lemma:unique} and Theorem~\ref{thm:poisson}. 
 In this way we can first present the main ideas of the proofs without obfuscating them with  technical complications. Lastly, at the end of this section, using Remark~\ref{rem:maxre} we prove Theorem~\ref{thm:maxRe}.

\subsection{Proof of \eqref{spectral_radius}}
\label{sec:sp} 
Recall the labelling of the eigenvalues $|\sigma_1| \geq |\sigma_2| \geq \cdots \geq |\sigma_n|$. From 
\cite[Theorem 2.1]{AEK19}, we have a strong concentration estimate for the spectral radius $|\sigma_1|$, \ie for any small $\tau>0$
	\begin{align}\label{prec_sp}
		1-\frac{n^{\tau}}{\sqrt{n}} \leq |\sigma_1|\leq 1+\frac{n^{\tau}}{\sqrt{n}},
	\end{align}
	with a very high probability. For any set $\Omega \subset \C$, we let
	\begin{align}
		\mathcal{N}_{\Omega}:=\# \big\{ \sigma_i: ~ \sigma_i \in \Omega, \quad i\in[n]\big\} =
		\sum_{i=1}^n \one_{\Omega}(\sigma_i)
	\end{align}
 denote the number of eigenvalues within  $\Omega$, where $\one_{\Omega}$ is the indicator function of $\Omega$.  Then for any fixed\footnote{Inspecting the proof, 
  the same result holds for any $|r|\ll\sqrt{\log n}$.}  
  $r \in \R$ , using the upper bound in \eqref{prec_sp}, we have
 	\begin{align}\label{link}
		\P\left( |\sigma_1|\geq 1+\sqrt{\frac{\gamma_n }{4n}}+\frac{r}{\sqrt{4n \gamma_n}}\right)
		=\P\Big( \mathcal{N}_{\Omega_r} \geq 1\Big)+O(n^{-D})
		=\E\Big[\Phi\big(\mathcal{N}_{\Omega_r}\big)\Big]+O(n^{-D}),
	\end{align}
	where $\Omega_{r}\subset \C$ is a centered annulus given by
	\begin{align}\label{omega}
		\Omega_{r}:=\Big\{ z \; : \; 1+\sqrt{\frac{\gamma_n }{4n}}+\frac{r}{\sqrt{4n \gamma_n}}  \leq |z| \leq 1+\frac{n^{\tau}}{n^{1/2}}\Big\},
	\end{align}
	and $\Phi: \R_+ \rightarrow [0,1]$ is a smooth and non-decreasing cut-off function such that
	\begin{equation}\label{F_function}
		\Phi(w)=0, \quad \mbox{if} \quad 0 \leq w \leq 1/9; \qquad \Phi(w)=1, \quad \mbox{if} \quad w \geq 2/9.
	\end{equation}

	Next, we regularize the indicator function of $\Omega_{r}$
		 as follows. For any small $\nu>0$, choose two test functions $f^{\pm}_{r} \in \C^{2}_c(\C)$ such that
	\begin{align}\label{test}
		\one_{\Omega_{r+n^{-\nu}}} \leq f^{-}_{r} \leq \one_{\Omega_{r}} \leq f^{+}_{r} \leq \one_{\Omega_{r-n^{-\nu}}},
	\end{align}
	and they both satisfy the $L^{\infty}$ norm bounds in (\ref{f_norm}) as well as the support condition 
	 (\ref{f_cond}). Thus, using Theorem \ref{main_thm} for $k=1$, $\F(w)=\Phi(w)$ and $f(z)=f_r^\pm(z)$, we have
	\begin{align}\label{gft_estimate}
		\left| \E\Big[\Phi\big(\sum_{i=1}^n f^\pm_{r}(\sigma_i)\big)\Big]-\E^{\mathrm{Gin}}\Big[\Phi\big(\sum_{i=1}^n f^\pm_{r}(\sigma_i)\big)\Big]  \right|=O(n^{-\epsilon}).
	\end{align}
	Since $\Phi$ is a non-decreasing function,  the last two inequalities in (\ref{test}) implies that
	\begin{align}
		&\E\Big[\Phi\Big(\mathcal{N}_{\Omega_{r}}\Big)\Big]\leq \E\Big[\Phi\big(\sum_{i=1}^n f^{+}_{r}(\sigma_i)\big)\Big] \leq \E\Big[\Phi\Big(\mathcal{N}_{\Omega_{r-n^{-\nu}}}\Big)\Big]\label{linear_stat_eq}
	\end{align}
and the same holds for $\E$ replaced with $\E^{\mathrm{Gin}}$.

	 Next, using the first inequality in (\ref{linear_stat_eq}), the GFT estimate in (\ref{gft_estimate}) for $f^+_r$, and the second inequality in (\ref{linear_stat_eq}) for the Ginibre ensemble, we have
	\begin{align}\label{upper}
		\E\Big[\Phi\big(\mathcal{N}_{\Omega_{r}}\big)\Big]\leq \E\Big[\Phi\big(\sum_{i=1}^n f^{+}_{r}(\sigma_i)\big)\Big] \leq& \E^{\mathrm{Gin}}\Big[\Phi\big(\sum_{i=1}^n f^{+}_{r}(\sigma_i)\big)\Big]+O(n^{-\epsilon})\leq \E^{\mathrm{Gin}}\Big[\Phi\Big(\mathcal{N}_{\Omega_{r-n^{-\nu}}}\Big)\Big]+O(n^{-\epsilon}).
	\end{align}
	One can also obtain a similar lower bound using the first two inequalities
	in (\ref{test}) and (\ref{gft_estimate}) for $f^-_r$, \ie
	\begin{align}\label{lower}
		\E\Big[\Phi\big(\mathcal{N}_{\Omega_{r}}\big)\Big] \geq 
		\E^{\mathrm{Gin}}\Big[\Phi\Big(\mathcal{N}_{\Omega_{r+n^{-\nu}}}\Big)\Big]-O(n^{-\epsilon}).
	\end{align}
	Combining (\ref{upper}) and (\ref{lower}) with (\ref{link}), we hence obtain that
	\begin{align}\label{uu}
		\P^{\mathrm{Gin}}\Big(|\sigma_1|\geq 1+\sqrt{\frac{\gamma_n }{4n}}+\frac{r+n^{-\nu}}{\sqrt{4n \gamma_n}}\Big)-Cn^{-\epsilon} &\leq \P\Big(|\sigma_1|\geq 1+\sqrt{\frac{\gamma_n }{4n}}+\frac{r}{\sqrt{4n \gamma_n}}\Big)\nonumber\\
		& \leq \P^{\mathrm{Gin}}\Big( |\sigma_1|\geq 1+\sqrt{\frac{\gamma_n }{4n}}+\frac{r-n^{-\nu}}{\sqrt{4n \gamma_n}}\Big)+Cn^{-\epsilon},
	\end{align}
	for some constant $C>0$. From the  explicit Ginibre calculation in Corollary \ref{cor:Ginibre} (or \cite[Theorem 1]{R03}) for the complex Ginibre ensemble\nc, for any $r\in \R$ we have
     \begin{align}\label{spectral_radius_gin}
     	\lim_{n\rightarrow \infty}\P^{\mathrm{Gin}} \Big( |\sigma_1|\leq 1+\sqrt{\frac{\gamma_n }{4n}}+\frac{r}{\sqrt{4n \gamma_n}}\Big)=\ee^{-\ee^{-r}}.
     \end{align}	
       Since the limiting function in \eqref{spectral_radius_gin} is bounded and Lipschitz continuous on $[-C, \infty)$, we hence finish the proof of~(\ref{spectral_radius}) from~\eqref{uu}.

     \subsection{Proof of \eqref{theta_uniform}}
     \label{sec:arg}
     In this section we prove the uniform law \eqref{theta_uniform} for the argument of $\sigma_1$. 
     For any integer $k\geq 1$, we divide $\C$ into $k$ congruent angular sectors
     $$\Theta_p:=\Big\{ z \in \C:\quad  \frac{2\pi (p-1)}{k} \leq \arg(z) < \frac{2\pi p}{k}\Big\}, \qquad p\in [k].$$
     Then for any  fixed $(r_1,\ldots,r_k)\in \R^{k}$, we have
     \begin{align}\label{link_theta}
     	\P\left( \max_{\sigma_i\in\Theta_p}|\sigma_i|\geq 1+\sqrt{\frac{\gamma_n }{4n}}+\frac{r_p}{\sqrt{4n \gamma_n}},~\forall \,\, p\in [k] \right)
     	=&\E\left[\prod_{p=1}^k \Phi\big(\mathcal{N}_{\Omega_{r_p} \cap \Theta_p}\big)\right]+O(n^{-D}).
      \end{align}
     with $\Omega_{r_p}$ is defined as in (\ref{omega}) and $\Phi$ is given in (\ref{F_function}). 
     We define the regularisation of the indicator function of $\Omega_{r_p} \cap \Theta_p~(p\in [k])$ as in (\ref{test}), denoted by $f_{r_p}^\pm$, so that the conditions in (\ref{f_norm}) and (\ref{f_cond}) are both satisfied. Therefore for any $k$, using Theorem \ref{main_thm} for 
     \begin{align}\label{F_choice}
     	\F(w_1, \ldots, w_k)=\prod_{p=1}^k \Phi(w_p), \qquad f_p=f_{r_p}^\pm, \quad p\in [k],
     \end{align}
 with $\Phi$ defined in (\ref{F_function}),  we conclude that
     \begin{equation}
     \label{eq:limdistrun}
     \lim_{n\to\infty}\P\left( \max_{\sigma_i\in \Theta_p}|\sigma_i|\geq 1+\sqrt{\frac{\gamma_n }{4n}}+\frac{r_p}{\sqrt{4n \gamma_n}},~\forall p\in [k]\right)
     \end{equation}
     coincides with its Ginibre counterpart.  Note that all correlation functions  of the
      Ginibre eigenvalues are rotational invariant (see (\ref{det_pp}) later), which implies that the probability  in (\ref{eq:limdistrun}), even before the limit,  is a symmetric function in $(r_1,r_2,\ldots,r_k)$ under cyclic permutations\nc{\footnote{  In fact, more information is available about (\ref{eq:limdistrun}).      
      First, for each fixed $1\leq p\leq k$, the limiting law of the largest eigenvalue in the $p$-th angular sector $\Theta_p$
       is given by $\exp{\big(-\frac{1}{k}\ee^{-r_p}\big)}$ with  the rescaling in (\ref{eq:limdistrun}),
       see (\ref{k=1} later).
        By further computations using (\ref{det_pp}) and Proposition \ref{lemma_K},  the joint limiting distribution in (\ref{eq:limdistrun}) is indeed a symmetric function given by $\prod_{p=1}^k \exp{\big(-\frac{1}{k}\ee^{-r_p}\big)}$.}}.    
       We also note that
     \[
     |\sigma_1|=\max_{p\in [k]} \left\{\max_{\sigma_i\in \Theta_p}|\sigma_i| \right\}.
     \]
     This implies that $\P(\sigma_1 \in \Theta_p)$ is independent of the index $p$, hence $\P(\sigma_1 \in \Theta_p)=\frac{1}{k}$ for each $p\in [k]$. Repeating this argument for any integer $k\geq 1$, we obtain that the argument of $\sigma_1$ is uniformly distributed on $[0,2\pi)$ in the limit as $k$ increases. This proves~\eqref{theta_uniform}.
    
    \medskip

     Now we are ready to prove Theorem~\ref{thm:gumbel} in full generality. 
     
     \begin{proof}[Proof of Theorem~\ref{thm:gumbel}]
     
       For any fixed $-\infty <r_1<r_2<\infty$ and $0\leq \theta_1<\theta_2\leq 2\pi$, we aim to show that the limit of the joint distribution
     \begin{align}\label{joint}
     	\P\left( \sqrt{4n \gamma_n}\Big( |\sigma_1|-1-\sqrt{\frac{\gamma_n }{4n}} \Big) \in [r_1,r_2),\quad  \arg(\sigma_1) \in [\theta_1,\theta_2) \right)
     \end{align}	
  can be identified similarly to the proof of \eqref{spectral_radius}--\eqref{theta_uniform} using Theorem \ref{main_thm} by choosing an appropriate
   $\F$ function and test functions $\{f_p\}$ as in (\ref{F_choice}), together with
  an  inclusion-exclusion argument. This argument is standard and follows  the same spirit of the above proofs, so we 
  omit the details. In short, Theorem \ref{main_thm} implies that  the limit of the joint distribution in (\ref{joint})  
     coincides with its Ginibre counterpart. This, together with Corollary \ref{cor:Ginibre} in the 
     Appendix, concludes the proof.  
     \end{proof}

 We conclude this section with the following lemma:
 \begin{lemma}\label{lemma:unique}
 	Let $X$ be a complex i.i.d. random matrix satisfying Assumption \ref{ass:mainass}. Then for any (small) $\nu>0$, 
 	$$\lim_{n\rightarrow \infty}\P\big( \sqrt{4 n \gamma_n} \big(|\sigma_1|-|\sigma_2|\big) \leq n^{-\nu}\big)=0.$$
 \end{lemma}
\begin{proof}[Proof of Lemma \ref{lemma:unique}]
	Fix small $\epsilon,\nu>0$ with $\nu \leq \epsilon/100$. Normalize the moduli of eigenvalues as
	\begin{align}\label{norm_x}
		r_i:=\sqrt{4n\gamma_n }\Big(|\sigma_i|-1-\sqrt{\frac{\gamma_n}{4n}}\Big).
	\end{align}
Then from \cite[Theorem 2.2]{SpecRadius}, for any sequence $C_n \rightarrow \infty$, we have 
	$$\lim_{n\rightarrow \infty} \P\big(|r_1| \geq C_n\big)=0.$$ 
\nc
Thus we have
	\begin{align}\label{inqab}
	\P\big( \sqrt{4 n \gamma_n} \big(|\sigma_1|-|\sigma_2|)\leq n^{-\nu}\big)&=\P\big( |r_1-r_2| \leq n^{-\nu}\big)=\P\big( | r_1-r_2| \leq n^{-\nu},|r_1| \leq C_n \big)+o_n(1)\nonumber\\
	\leq &\sum_{k=0}^{2 \lceil C_n \rceil n^{\nu}} \P\Big( r_1 \in[ a_k, a_{k+1}),~   r_2 \in [a_k -n^{-\nu}, a_{k+1}+n^{-\nu}) \Big)+o_n(1),
	\end{align}	
 with $a_k:=-C_n+kn^{-\nu}$, for $0\leq k\leq 2 \lceil C_n \rceil n^{\nu}$, 
 and $[-C_n,C_n] \subset \bigcup_{k} [a_k,a_{k+1}]$. Choose the sequence $C_n \ll \sqrt{\log n}$ sufficiently slowly 
 diverging  
  and choose an appropriate function $\F$ and functions $f^\pm$ as in (\ref{F_choice}) satisfying the conditions in (\ref{f_norm})-(\ref{F_cond}). By Theorem \ref{main_thm}, the right side of (\ref{inqab}) coincides with its Ginibre counterpart up to an error term $n^{-\epsilon}$.  
 Note that for Ginibre matrices
  Kostlan's observation~\cite{Kostlan92}  (see Remark \ref{rem:kostlan} below) implies that, for any $1\leq k\leq n$,
\begin{align} 
 (r_1,r_2, \ldots, r_k)\stackrel{\mathrm{d}}{=} (y_1,y_2,\ldots, y_k), \qquad y_k:=\sqrt{4n\gamma_n }\Big(x_k-1-\sqrt{\frac{\gamma_n}{4n}}\Big),
\end{align}
where $x_1\ge x_2\ge \ldots \ge x_n$  is the ordered statistics
of $n$ independent chi-distributed random variables $\{\chi_{2k}\}_{k\in[n]}$. By a direct computation we have
\begin{align}
	&\P^{\mathrm{Gin}}\Big( r_1 \in[ a_k, a_{k+1}),~   r_2 \in  [a_k -n^{-\nu}, a_{k+1}+n^{-\nu}) \Big)\nonumber\\
	=&\P\Big( y_1+C_n \in \big[ kn^{-\nu}, (k+1)n^{-\nu}\big),~   y_2+C_n \in
	 \big[ (k-1)n^{-\nu} , (k+2)n^{-\nu}\big) \Big) \lesssim  n^{-2 \nu},
\end{align}
with $C_n \ll \sqrt{\log n}$. Summing up the terms with $0\leq k\leq 2 \lceil C_n \rceil n^{\nu}$ in (\ref{inqab}), the RHS of \eqref{inqab} goes to zero, hence
we conclude Lemma~\ref{lemma:unique}. 
\end{proof}

\subsection{Proof of Theorem~\ref{thm:poisson}}\label{sec:ppp}
By the rescaling of the eigenvalues as in \eqref{scale_eig}, we have
\begin{align}\label{ppprove}
	r_i=\sqrt{4\gamma_n n} \Big(|\sigma_i|-1-\sqrt{\frac{\gamma_n}{4n}}\Big), \qquad\quad \theta_i=\arg(\sigma_i).
\end{align}
Then the test function $f(z):=g(h(z))$ with $g\in C^2_c(\C)$ and $h(z)=\sqrt{4 n \gamma_n} \big(|z|-1-\sqrt{\frac{\gamma_n}{4n}} \big) \frac{z}{|z|}$~(which is chosen so that $f(\sigma_i)=g(r_i \ee^{\ii \theta_i})$) satisfies the $L^\infty$ norm bounds in (\ref{f_norm}) and the support condition in (\ref{f_cond}). Hence Theorem \ref{thm:poisson} follows directly by using Theorem \ref{main_thm} for $k=1$, choosing $\F(w)=\ee^{-w}$ for $w\in \R_+$, in combination with the explicit Ginibre computations in Corollary \ref{cor:Ginibre} in the Appendix. 

\qed

\subsection{Proof of Theorem~\ref{thm:maxRe}}\label{sec:rightmost}

In this section we present the minor modifications needed to prove the universality of the rightmost eigenvalue. 
The local laws in  \cite[Theorem 1.2]{BYY14} and \cite[Theorem 2.1]{AEK19}
 imply that, for any small $\tau>0$ there exists an eigenvalue in the rectangle
     \begin{align}\label{rect}
     	\left[ 1-\frac{n^{\tau}}{n^{1/2}},  1+\frac{n^{\tau}}{n^{1/2}}\right] \times \left[ -\frac{n^{\tau/2}}{n^{1/4}},\frac{n^{\tau/2}}{n^{1/4}} \right]
     \end{align}
with very high probability, according to the curvature of the boundary. Combining this with the strong concentration estimate in \eqref{prec_sp}, we know that the rightmost eigenvalue is indeed located in the rectangle given in (\ref{rect}). Thus, instead of the annulus  $\Omega_r$ in \eqref{omega}, we choose a rectangle given by
     \begin{align}\label{omega_rec}
     	\Omega'_x:= \left[ 1+\sqrt{\frac{\gamma'_n }{4n}}+\frac{x}{\sqrt{4n \gamma'_n}},  1+n^{-1/2+\tau}\right] \times \left[ -\frac{n^{\tau/2}}{n^{1/4}},\frac{n^{\tau/2}}{n^{1/4}} \right],
     \end{align}
     with $\gamma'_n$ in (\ref{gamma_prime}). Choosing similar test functions as in (\ref{test}), they both satisfy the $L^\infty$ norm bounds in (\ref{f_norm}) and the support condition in (\ref{f_cond_2}). Thus using Remark \ref{rem:maxre} in combination with \cite[Theorem 1]{maxRe_Gin} for the Ginibre ensemble, one can prove the Gumbel law in (\ref{maxRe}) for the rightmost eigenvalue as in Section \ref{sec:sp}. The proof of the Poisson point process in \eqref{maxRe_ppp} is analogous to Section~\ref{sec:ppp}, but with the different scaling in \eqref{gamma_prime}, in combination with \cite[Theorem 2]{maxRe_Gin} for the Ginibre ensemble, and so omitted.

\section{Proof of the local laws, Theorems~\ref{thm:2G} and  \ref{theorem_F} \nc}

\label{sec:G1G2llaw}

Consider the  following matrix valued stochastic differential equation,
the Ornstein--Uhlenbeck flow,
\begin{equation}\label{OU}
\dd X_t=-\frac{1}{2}X_t\dd t+\frac{\dd B_t}{\sqrt{n}},
\end{equation}
with $B_t$ being an $n\times n$
 matrix whose entries are i.i.d. standard complex Brownian motions. Set
 \begin{equation}
 \label{eq:bigstoch}
   W_t:=\left(\begin{matrix}
0 & X_t \\
X_t^* & 0
\end{matrix}\right)
\end{equation}
 to be the Hermitization of $X_t$. Note that $W_t$ solves
 \begin{equation}
 \label{eq:bigflow}
 \dd W_t=-\frac{1}{2}W_t\dd t+\frac{\dd \mathfrak{B}_t}{\sqrt{n}},
 \end{equation}
 with $\mathfrak{B}_t$ being the Hermitization of $B_t$. Define a 
matrix $\Lambda\in \C^{2n\times 2n}$ containing   the spectral parameters, written in a 
$2\times 2$ block form  as follows
\begin{equation}
\label{def:Lambda}
\Lambda:=Z+\ii\eta, \qquad
Z=\left(\begin{matrix}
0 & z \\
\overline{z} & 0
\end{matrix}\right), \qquad \eta \in  \R, \qquad z \in \C. 
\end{equation}
The main object of interest of this section is the evolution of the resolvent $(W_t-\Lambda_t)^{-1}$ (and of certain products of two such resolvents) along the {\it characteristics  flow} of the spectral parameters (see also \cite[Eq. (5.3)]{CES22}
\begin{equation}
\label{eq:matchar}
\partial_t \Lambda_t=-\mathcal{S}[M(\Lambda_t)]-\frac{\Lambda_t}{2}, \qquad t\ge 0,
\end{equation}
where $\mathcal{S}[\cdot]$ is the covariance operator \eqref{cov} and $M(\Lambda_t):=M^{z_t}(\ii\eta_t)$.
We point out that if the initial condition $\Lambda_{i,0}$ of \eqref{eq:matchar} is constant diagonal in each of its four $n\times n$ blocks, as in our case from \eqref{Mmatrix},  
then this structure persists for $\Lambda_t$ for any $t\ge 0$ and \eqref{eq:matchar}  simplifies
to two coupled scalar equations:
\begin{equation}
\label{eq:scalchar}
\partial_t\eta_t=-\Im m^{z_t}(\ii\eta_t)-\frac{\eta_t}{2}, \qquad\quad \partial_t z_t=-\frac{z_t}{2}.
\end{equation}
Notice that the evolution of $\eta_t$ depends on $z_t$ through $\Im m^{z_t}$, but the evolution of $z_t$ is independent of $\eta_t$. We will often use the fact that $\Lambda_t$ and the pair $(z_t, \eta_t)$ carry the same information.
We note that both the matrix Dyson equation \eqref{MDE} and the characteristic flow~\eqref{eq:matchar} literally
hold true for arbitrary Hermitian {\it deformation matrix} $Z$, but then neither $\Lambda_t$ nor $M(\Lambda_t)$
is constant block diagonal.

Throughout this section we will always consider the characteristics \eqref{eq:matchar} for times smaller than the maximal time
\begin{equation}
T^*(\eta_0,z_0):=\sup\{ t : \mathrm{sgn}(\eta_t)= \mathrm{sgn}(\eta_0)\},
\end{equation}
for which the diagonal part of the characteristics $\ii \eta_{T^*}$ reaches the real axis. 
In  this section we will always consider initial conditions $\eta_0,z_0$ such that $|\eta_0|+|z_0|\lesssim 1$ (even if not stated explicitly); this implies that $\lVert\Lambda_0\rVert\lesssim 1$. \nc
Note that $T^*\lesssim 1$ as a consequence of $\lVert \Lambda_0 \rVert\lesssim 1$.  
Furthermore, since $\Im m^{z_t}(\ii\eta_t)\cdot \eta_t >0$ (see \eqref{m_function}), 
 $t\to |\eta_t|$ is strictly decreasing,
 while $z_t$ moves towards the origin of $\C$ by~\eqref{eq:scalchar}.
 In the following lemma we summarize several trivially checkable properties of the characteristics \eqref{eq:scalchar}:
\begin{lemma}
\label{lem:usinf}
Let $\Lambda_t$ be the solution of \eqref{eq:matchar} with  some
initial condition $\Lambda_0$ and let $M_t:=M^{z_t}(\ii\eta_t)$, $\rho_t:=\pi^{-1}|\Im m_t|$, then for any 
$0\le s\le t\le T^*(\eta_0, z_0)$ we have 
\begin{equation}
\label{eq:relchar}
z_t=e^{-(t-s)/2}z_s, \qquad M_t=e^{(t-s)/2}M_s, 
\qquad  \frac{|\eta_t|}{\rho_t}=e^{s-t}\frac{|\eta_s|}{\rho_s}-\pi(1-e^{s-t});
\end{equation}
in particular  $t\mapsto |\eta_t|/\rho_t$ is a decreasing function. Additionally, for $0\le s\le t\le T^*(\eta_0, z_0)$, we have
\begin{equation}
\label{eq:addetarho}
|\eta_t|\rho_t=|\eta_s|\rho_s-\pi(1-e^{s-t})\rho_t^2.
\end{equation}
 Furthermore, for any $\alpha> 1$ there is a $C_\alpha$ such that 
\begin{equation}
\label{eq:relineta}
\int_s^t \frac{1}{|\eta_r|^\alpha}\,\dd r\le 
 \frac{C_\alpha}{|\eta_t|^{\alpha-1}\rho_t}, \qquad \mbox{and} \qquad\quad \int_s^t \frac{\rho_r}{|\eta_r|}\,\dd r= \frac{1}{\pi}\left[\log\left(\frac{\eta_s}{\eta_t}\right)-\frac{t-s}{2}\right].
\end{equation}
\end{lemma}
\begin{proof}
The first relation in \eqref{eq:relchar} immediately follows by solving the second equation in \eqref{eq:scalchar}. To obtain the second relation in \eqref{eq:relchar} we rewrite the MDE in \eqref{MDE} using (\ref{def:Lambda}) (see also \cite[Lemma A.1]{CES22}):
\[
-\frac{1}{M_t}=\Lambda_t+\mathcal{S}[M_t],
\]
with $\mathcal{S}$ from \eqref{cov}. Differentiating this equation in time and we obtain (here the dot denotes the time derivative, e.g. $\dot{M}_t:=\dd M_t/\dd t$)
\[
M_t^{-1}\dot{M}_tM_t^{-1}=\dot{\Lambda}_t+\mathcal{S}[\dot{M}_t]=-\mathcal{S}[M_t]-\frac{\Lambda_t}{2}+\mathcal{S}[\dot{M}_t]=-\frac{1}{2}\mathcal{S}[M_t]-\frac{M_t^{-1}}{2}+\mathcal{S}[\dot{M}_t].
\]
We point out that in the first equality we used \eqref{eq:matchar}. Multiplying both sides by $M_t$ we obtain $\mathcal{B}[\dot{M}_t]=\frac{1}{2}\mathcal{B}[M_t]$, with $\mathcal{B}=\mathcal{B}_{12}$ from (\ref{eq:defstabop}) with $z_1=z_2, \eta_1=\eta_2$. Inverting the stability operator $\mathcal{B}$ we obtain the desired result. Next, using the second relation in \eqref{eq:relchar} together with the first equation in \eqref{eq:scalchar}, and that $\rho_t=\pi^{-1} \langle \Im M_t\rangle$, we obtain
\[
\partial_t \frac{\eta_t}{\rho_t}=-\pi-\frac{\eta_t}{\rho_t},
\]
which after integration implies the third relation in \eqref{eq:relchar}. The proof of \eqref{eq:addetarho} is completely analogous and so omitted.

We now turn to the proof of \eqref{eq:relineta}. We only present the proof of the first bound, as the second relation can be obtained in a completely analogous way. By \eqref{eq:scalchar} it follows (assume that $\eta_r>0$ for simplicity of notation)
\[
\int_s^t \frac{\pi\rho_r}{\eta_r^\alpha}\,\dd r=-\int_s^t \frac{\partial_r \eta_r}{\eta_r^\alpha}\,\dd r+O\left(\frac{1}{\rho_s\eta_s^{\alpha-2}}+\frac{1}{\rho_t\eta_t^{\alpha-2}}\right)
\le \frac{C_\alpha}{\eta_t^{\alpha-1}},
\]
where in the last inequality we used that $t\mapsto \eta_t$ is decreasing and that $\eta_t\le 1$. This implies the first bound in \eqref{eq:relineta} by noticing that $\rho_r\sim\rho_t$ for any $r\in [0,t]$ by the second relation in \eqref{eq:relchar}.
\end{proof}

 Additionally, by standard ODE theory we have (see \cite[Lemma 5.2]{CES22} for an analogous proof):
\begin{lemma}\label{lem:ode}
Fix $0<T\le 1$, and pick $\Lambda$ of the form~\eqref{def:Lambda} 
 with $|\eta|>0$, $|z|\ge 1$ and $\|\Lambda\|\lesssim 1$. 
 Then there exists an initial condition $\Lambda_0$, with $T^*(\eta_0,z_0)> T$, $\|\Lambda_0\|\lesssim 1$,
 such that the solution $\Lambda_t$ of  \eqref{eq:matchar} with this initial condition at time $T$
  satisfies $\Lambda_T=\Lambda$. 
  Furthermore, there exist constants $c_1,c_2>0$ such that $|z_0|^2-1\ge c_1 T$ and $|\eta_0|/\rho_0\ge c_2 T$.
\end{lemma}

From now on we will consider the evolution of two different characteristics $\Lambda_{1,t}, \Lambda_{2,t}$
with maximal times $T^*_i:= T^*(\eta_{i,0}, z_{i,0})$, $i=1,2$.
Note that the first two relations in \eqref{eq:relchar} imply that
\begin{equation}
\label{eq:z1z2eq}
|z_{1,s}-z_{2,s}|\sim |z_{1,t}-z_{2,t}| \quad\mbox{and} \quad \rho_{i, s}\sim \rho_{i,t} \quad 
\mathrm{for} \quad s,t\in [0, \min_i T_i^*).
\end{equation}
 
 For $i=1,2$, define the "resolvent" of $W_t$ with the generalised spectral parameter $\Lambda_{i,t}$:
 \begin{equation}\label{def:Gt}
 G_{i,t}:=G_t^{z_{i,t}}(\ii\eta_{i,t})=(W_t-\Lambda_{i,t})^{-1}.
\end{equation}
We now describe the time evolution of $ \langle G_{1,t}A_1G_{2,t}A_2\rangle$ for general,
time independent deterministic matrices $A_1,A_2\in\C^{2n\times 2n}$.
 By It\^{o}'s formula and~\eqref{def:Gt}  (see Appendix~\ref{sec:rand} for its derivation), we have
\begin{equation}
\label{eq:flowbefchar}
\begin{split}
\dd \langle G_{1,t}A_1G_{2,t}A_2\rangle&=\sum_{a,b=1}^{2n}\partial_{ab}\langle G_{1,t}A_1G_{2,t}A_2\rangle\frac{\dd \mathfrak{B}_{ab,t}}{\sqrt{n}}+\langle G_{1,t}A_1G_{2,t}A_2\rangle\dd t \\
&\quad+2\sum_{i\ne j} \langle G_{1,t}A_1G_{2,t}E_i\rangle\langle G_{2,t}A_2G_{1,t}E_j\rangle\dd t \\
&\quad+ \langle (G_{1,t}-M_{1,t})\rangle\langle G_{1,t}^2A_1G_{2,t}A_2\rangle\dd t + \langle (G_{2,t}-M_{2,t})\rangle\langle G_{2,t}^2A_2G_{1,t}A_1\rangle\dd t,
\end{split}
\end{equation}
where the sum $\sum_{i\ne j}$ in the second line indicates a summation
over the matrix pairs $(E_i,E_j)\in \{(E_1,E_2), (E_2,E_1)\}$ (this notation will be often used in 
 this section even if not stated explicitly). 
Here $\partial_{ab}$ denotes the directional derivative $\partial_{w_{ab}}$,  where $w_{ab}$ are the entries of the Hermitized matrix $W_t$. 

The deterministic approximation of  $G_{1,t}AG_{2,t}$,
for any deterministic $A\in\C^{2n\times 2n}$, is given by
\begin{equation}
\label{eq:M12}
M_{12,t}^A:=M^{A}(\eta_{1,t}, z_{1,t},\eta_{2,t},z_{2,t}):=\mathcal{B}_{12,t}^{-1}[M_{1,t}AM_{2,t}].
\end{equation}
 Here the two--body stability operator $\mathcal{B}_{12,t}$ is 
 defined exactly as in \eqref{eq:defstabop} with $M^{z_i}(\ii\eta_i)$ replaced with 
 its time--dependent version $M^{z_{i,t}}(\ii\eta_{i,t})$. 
 The following lemma  borrowed from \cite{CES22} shows that $\partial_t \langle M_{12,t}^{A_1}A_2\rangle$ exactly
 cancels the deterministic approximations of the second and third terms in the
  rhs. of \eqref{eq:flowbefchar}: 
\begin{lemma}[Lemma 5.5 \cite{CES22}]
For any matrices $A_1,A_2\in\C^{2n\times 2n}$, we have
\begin{equation}
\label{eq:evolM12}
\partial_t \langle M_{12,t}^{A_1}A_2\rangle=\langle M_{12,t}^{A_1}A_2\rangle+\langle \mathcal{S}[M_{12,t}^{A_1}]M_{21,t}^{A_2}\rangle.
\end{equation}
\end{lemma}
Recall the definition of $\mathcal{S}[\cdot]$ from (\ref{cov}). Combining \eqref{eq:flowbefchar} and \eqref{eq:evolM12}, we just
derived a stochastic differential
 equation which describes the time evolution of $ \langle (G_{1,t}A_1G_{2,t}-M_{12,t}^{A_1})A_2\rangle$:
 \begin{equation}
\label{eq:fulleqaasimp}
\begin{split}
\dd \langle (G_{1,t}A_1G_{2,t}-M_{12,t}^{A_1})A_2\rangle&=\sum_{a,b=1}^{2n}\partial_{ab}\langle G_{1,t}A_1G_{2,t}A_2\rangle\frac{\dd \mathfrak{B}_{ab,t}}{\sqrt{n}}+\langle (G_{1,t}A_1G_{2,t}-M_{12,t}^{A_1})A_2\rangle\dd t \\
&\quad+2\sum_{i\ne j}\langle (G_{1,t}A_1G_{2,t}-M_{12,t}^{A_1})E_i\rangle\langle M_{21,t}^{A_2}E_j\rangle\dd t \\
&\quad+2\sum_{i\ne j}\langle M_{12,t}^{A_1} E_i\rangle\langle (G_{2,t}A_2G_{1,t}-M_{21,t}^{A_2})E_j\rangle\dd t \\
&\quad+2\sum_{i\ne j}\langle (G_{1,t}A_1G_{2,t}-M_{12,t}^{A_1})E_i\rangle\langle (G_{2,t}A_2G_{1,t}-M_{21,t}^{A_2})E_j\rangle\dd t \\
&\quad+ \langle (G_{1,t}-M_{1,t})\rangle\langle G_{1,t}^2A_1G_{2,t}A_2\rangle\dd t + \langle (G_{2,t}-M_{2,t})\rangle\langle G_{2,t}^2A_2G_{1,t}A_1\rangle\dd t.
\end{split}
\end{equation}
This equation is the starting point of our analysis to propagate a two-resolvent local law from 
the global regime, where $|\eta_{i,0}|\sim 1$, to the local regime, where  $\ii\eta_{i,t}$
is very close to the real axis.

In the next two sections we focus on  two different effects that improve the two resolvent local law
and eventually yield Theorems~\ref{thm:2G} and  \ref{theorem_F}.
 In Section~\ref{sec:g1g2} we consider the case $z_1\ne z_2$, i.e. we want exploit the gain in the two resolvent local law when $|z_1-z_2|$ is fairly large, whilst in Section~\ref{sec:gfgf} we consider $z_1=z_2=z$ and $\eta_{1,t}=\eta_{2,t}=\eta_t$ and exploit the additional smallness of $ \langle (G_tA_1G_t-M_t^{A_1})A_2\rangle$, with $M_t^A=M_{12}^A(\eta_t, z_t,\eta_t,z_t,)$, due to the special choice  $A_1=F$, $A_2=F^*$.  We remark that the gain from both effects simultaneously is not present. In particular, for $G_{1,t}FG_{2,t}$ it would not be possible to gain smallness both from the $F$ observable and from the $|z_1-z_2|$-decay.

\subsection{Gain from $|z_1-z_2|$ in the local law for $G^{z_1}AG^{z_2}$ (Theorem~\ref{thm:2G})}
\label{sec:g1g2}

The main result of this section shows that if initially we have a good control of $G_{1,0}AG_{2,0}$ 
in terms of  $|z_1-z_2|$ then this is preserved for $G_{1,t}AG_{2,t}$ along the characteristics \eqref{eq:matchar}.

\begin{proposition}
\label{pro:mainpro}
Fix  sufficiently small $n$--independent
constants $\epsilon,\omega_d,\omega_1,c>0$, and for $i=1,2$ fix  spectral parameters 
\begin{equation}
\label{eq:defincond}
\Lambda_{i,0}:=\left(\begin{matrix}
\ii\eta_{i,0} & z_{i,0} \\
\overline{z_{i,0}} & \ii\eta_{i,0}
\end{matrix}\right),
\end{equation}
with $0\le |z_{i,0}|-1\le c$, $|z_{1,0}-z_{2,0}|\le \omega_d$, and $0<|\eta_{i,0}|\le \omega_1$. 
 For $0\le t<T^*(\eta_{i,0},z_{i,0})$,
 let $\Lambda_{i,t}$ be the solution of \eqref{eq:matchar} with initial condition $\Lambda_{i,0}$. Recall
 the resolvent $G_{i,t}=(W_t-\Lambda_{i,t})^{-1}$ and 
  $M_{12,t}^A$,  the deterministic approximation of $G_{1,t}AG_{2,t}$,  from  \eqref{eq:M12}.  Define
  the following control parameters and their time dependent versions:
 \begin{equation}
 \label{eq:defellbetas}
 \begin{split}
\ell&=\ell(\eta_1,z_1,\eta_2,z_2):=\min_{i=1,2} |\eta_i|\rho^{z_i}(\ii\eta_i), \qquad\qquad\quad \ell_t:=\ell(\eta_{1,t},z_{1,t},\eta_{2,t},z_{2,t}), \\
 \gamma& =\gamma(\eta_1,z_1,\eta_2,z_2):=|z_1-z_2|+\frac{|\eta_1|}{\rho_1}+\frac{|\eta_2|}{\rho_2}, \nc\qquad\quad \gamma_t:=\gamma(\eta_{1,t},z_{1,t},\eta_{2,t},z_{2,t}).
\end{split}
 \end{equation}
Assume that for some small $0<\xi\le \epsilon/10$, with very high probability, it holds
 \begin{equation}
\label{eq:inasspart1}
\big|\langle (G_0^{z_{1,0}}(\ii\eta_1)A_1G_0^{z_{2,0}}(\ii\eta_2)-M_{12}^{A_1}(\eta_1, z_{1,0},\eta_2,z_{2,0}))A_2\rangle\big|\lesssim 
  \frac{n^\xi}{\sqrt{n\ell_0}\gamma_0}\wedge\frac{n^\xi}{n|\eta_{1,0}\eta_{2,0}|}\nc,
\end{equation}
for any choice $A_1,A_2\in  \{E_+,E_-\}$ and for any $\eta_i$ with $|\eta_i|\in [|\eta_{i,0}|,  n^{100}]$, where $\ell_0$ and $\gamma_0$ are evaluated at the spectral parameters $(\eta_1,z_{1,0}, \eta_2, z_{2,0})$.
Then
\begin{equation}
\label{eq:usllaw}
\big|\langle (G_{1,t}A_1G_{2,t}-M_{12,t}^{A_1})A_2\rangle\big|\lesssim  \frac{n^{2\xi}}{\sqrt{n\ell_t}\gamma_{t}}\wedge\frac{n^{2\xi}}{n|\eta_{1,t}\eta_{2,t}|}\nc,
\end{equation}\
with very high probability uniformly in all $t\le \min_i T^*(\eta_{i,0},z_{i,0})$ such that $n\ell_t\ge n^\epsilon$.
\end{proposition}
 Note that by the strict monotonicity of $t\to\eta_{i,t}/\rho_{i,t}$, and by $|z_{1,t}-z_{2,t}|=e^{-(t-s)/2}|z_{1,0}-z_{2,0}|$, it follows that $\gamma_s\ge \gamma_t$ for any $0\le s\le t\le T^*$, i.e. $t\mapsto \gamma_t$ is strictly decreasing.\nc

Before presenting the proof of Proposition~\ref{pro:mainpro} we first conclude the proof of  Part A of
Theorem~\ref{thm:2G}  for matrices with a large Gaussian component. The proof for general i.i.d. matrices (Part B) 
will be given in Appendix~\ref{app:lindGFT}.

\begin{proof}[Proof of Theorem~\ref{thm:2G} (Part A) with Gaussian component]

 For any small $c^*>0$, \nc by the global law outside of the non-Hermitian spectrum 
 \cite[Theorem 5.2]{CES19}, for any  small $\xi>0$ we have 
\begin{equation}
\label{eq:global}
\big|\langle (G^{z_1}(\ii \eta_1)A_1G^{z_2}(\ii \eta_2)-M_{12}^{A_1})A_2\rangle\big|\lesssim \frac{n^\xi}{n},
\end{equation}
with very high probability, uniformly in spectral parameters
such that $ 1+c^*\nc\le |z_i|\le 10$, $|\eta_i|\le n^{100}$, and deterministic matrices $\| A_1\|+\| A_2\|\lesssim 1$. 
Define $t=t(\mathfrak{s})$ such that $\mathfrak{s} = (1-e^{-t/2})^{1/2}$, then by the property of the Ornstein-Uhlenbeck process~\eqref{OU}, $G_t$ from~\eqref{def:Gt}
and $G_{\mathfrak{s}}^z$ defined in Theorem~\ref{thm:2G} have the same distribution.
Fix any small $\omega_1,\omega_d,c$, and  pick any $z_1,z_2$ such that $1\le |z_i|\le 1+c/2$ and
$|z_1-z_2|\le\omega_d/2$, and pick any $\eta_i$ such that $|\eta_i|\le \omega_1/2$ and $n\ell\ge n^\epsilon$.  Then, by Lemma~\ref{lem:ode} and that fact that $\mathfrak{s}$ can be chosen small in Part A of Theorem~\ref{thm:2G},
there exist $z_{i,0}$ and $\eta_{i,0}$ with $|z_{1,0}-z_{2,0}|\le \omega_d$, $|\eta_{i,0}|\le \omega_1$,
such that $z_{i,t}=z_i$, $\eta_{i,t}=\eta_i$, $c_*\le |z_{i,0}|-1\le c$,
 for some small $c_*=c_*(\mathfrak{s})>0$. Choosing these parameters sufficiently small so that \eqref{eq:global} holds, and noticing that $1/n\le 1/(n\ell\gamma)$, this \nc guarantees that 
  the assumption \eqref{eq:inasspart1} is satisfied.
  We thus conclude from \eqref{eq:usllaw},
   using $z_{i,t}= z_i$, $\eta_{i,t}=\eta_i$, hence $\ell=\ell_t$ and $\gamma=\gamma_t$,
   that (here $A_1,A_2\in\{E_+,E_-\}$)
\begin{equation}
\big|\langle (G_t^{z_{1}}(\ii \eta_{1})A_1G_t^{z_{2}}(\ii \eta_{2})-M_{12}^{A_1})A_2\rangle\big|\lesssim \frac{n^{2\xi}}{\sqrt{n\ell}\gamma}+\frac{n^{2\xi}}{n|\eta_1\eta_2|}\nc.
\end{equation}
Since $G_t \stackrel{{\rm d}}{=} G_{\mathfrak{s}}^z$,  
  by the choice of $t$, we concluded
\eqref{local_2g}. 
\end{proof}

\begin{proof}[Proof of Proposition~\ref{pro:mainpro}] By \cite[Theorem 3.1]{SpecRadius} we have $|\langle G_{i,t}-M_{i,t}\rangle|\lesssim n^\xi/(n|\eta_{i,t}|)$ 
with very high probability for any fix $t$, but in fact also implies 
$\sup_{0\le t\le T^*} (n|\eta_{i,t}|)|\langle G_{i,t}-M_{i,t}\rangle|\lesssim n^\xi$, i.e. that the local law holds with very high probability simultaneously for any $0\le t\le T^*$. This follows by a simple procedure
called the {\it grid argument}. It consists of choosing equidistant
intermediate times  $t_1,t_2,\dots, t_{K_n}\in [0,T^*]$, for some large enough 
$K_n$ (say $K_n=n^{10}$), using the fact that for each $t=t_i$ the local law holds and 
then extending the bound to any intermediate 
$t\in [0,T^*]$ by  the $(1/2)^-$-H\"older continuity of $G_{i,t}-M_{i,t}$ (with a H\"older constant which may depend on $n$). 
The probability of the 
exceptional sets where the bound does not hold for some $t_j$ can be estimated by a standard union bound
since there are only polynomially many of them. 
We will use a similar
procedure also for local laws for products of different resolvents. We also point out that throughout this proof, even if not stated explicitly, all the inequalities hold with very high probability. Additionally, to keep the notation simpler we assume that $\eta_{1,t}$ and $\eta_{2,t}$ are both positive, all the other cases are analogous and so omitted.

 Fix any $\epsilon>0$, and define the time,
\begin{equation}
\label{Teps}
T_i^\epsilon=T^\epsilon(\eta_{i,0},z_{i,0}):=\min\{t>0 : n\eta_{i,t}\rho_{i,t}=n^\epsilon \},
\end{equation}
such that $\eta_{i,t}\rho_{i,t}$ reaches the critical threshold $n\eta_{i,T_i^\epsilon}\rho_{i,T_i^\epsilon}=n^\epsilon$. Notice that,  by \eqref{eq:scalchar} and by taking the trace in the second relation in \eqref{eq:relchar}, it follows that \nc
\[
\partial_t(\eta_{i,t}\rho_{i,t})=-\pi\rho_{i,t}^2.
\]
Then, by the uniqueness of the solution of this simple ODE,  together with the strict monotonicity\footnote{This
 monotonicity follows from a direct calculation of   the derivative of $m=m^{z_{i,0}}(\ii \eta)$
\[
\partial_\eta \Im m=\Im m\frac{1-2\Im m(\eta+\Im m)}{2(\Im m)^2(\eta+\Im m)+\eta}>0,
\]
where  we used that $2\Im m(\eta+\Im m)\le C\omega_1^{2/3}$ and that 
 $\omega_1$ is chosen sufficiently small.} 
of the map $\eta\mapsto \eta\rho^{z_{i,0}}(\ii\eta)$,  it follows that for a fixed $z_{i,0}$ the map $\eta_{i,0}\mapsto T^\epsilon(\eta_{i,0}, z_{i,0})$ is increasing.

 We start defining
\begin{equation}\label{defY}
  Y_{\sigma, \sigma' , t} (\eta_1, z_1, \eta_2, z_2): =
   \langle \big(G_t^{z_{1}}(\ii\eta_1)E_{\sigma} G_t^{z_{2}}(\ii {\eta}_2)-
   M_{12}^{E_{\sigma}}({\eta}_{1},z_{1},{\eta}_{2},z_{2})\big)E_{\sigma'}\rangle, \qquad \sigma, \sigma' \in \{+, -\}.
 \end{equation}
 Additionally, to keep the notation short we introduce the control parameter 
 \begin{equation}\label{alphadef}
 \alpha(\eta_1,z_1,\eta_2,z_2):=\frac{1}{\sqrt{n\ell(\eta_1,z_1,\eta_2,z_2)}\gamma(\eta_1,z_1,\eta_2,z_2)}
 \wedge \frac{1}{n|\eta_1\eta_2|}, \qquad \alpha_t := \alpha(\eta_{1,t},z_{1,t},\eta_{2,t},z_{2,t}), \nc
  \end{equation}
 with $\ell, \gamma$ being defined in \eqref{eq:defellbetas}.
 \nc
 Within the whole proof, for any given $\eta_{i,0}, z_{i,0}$, $i=1,2$, we will consider the flow \eqref{eq:fulleqaasimp} up to times $t\le \tau$, with $\tau=
\tau (\eta_{1,0}, z_{1,0}, \eta_{2,0}, z_{2,0})$ 
being the stopping time defined by\footnote{Here the stopping time $\tau$ should not be confused with the small (deterministic) exponent used often starting from Section~\ref{sec:tools} to denote $|z|\le 1+ n^{-1/2+\tau}$.}
\begin{equation}
\label{eq:deftildetau}
\tau:=\inf\left\{t\ge 0: \max_{i=1,2} \;
\sup^*_{ \widetilde{\eta}_{i,0}} 
\max_{\sigma, \sigma'\in\{+,-\}}   \alpha(\widetilde{\eta}_{1,t},z_{1,t},\widetilde{\eta}_{2,t},z_{2,t})^{-1}\nc \big|
Y_{\sigma, \sigma' , t} (\widetilde\eta_{1,t}, z_{1,t}, \widetilde\eta_{2,t}, z_{2,t})
\big|=n^{2\xi} \right\}\wedge  T_1^\epsilon\wedge T_2^\epsilon\nc,
\end{equation}
\nc for some $0< \xi\le\epsilon/10$. The supremum $\sup^*$ is taken over all $\widetilde{\eta}_{i,0}$ satisfying\footnote{Strictly speaking, to avoid problems with taking the supremum $\sup^*$ over 
a continuum, one should discretize $ [\eta_{i,0}, \omega_1]$ over a 
very finite mesh $\mathcal{P}$ of spacing $N^{-100}$ and use $\sup$ over $\widetilde{\eta}_{i,0}\in \mathcal{P}$
in the definition of $\tau$.
To show that if the desired bounds hold for $\widetilde{\eta}_{i,0}\in\mathcal{P}$ then they hold simultaneously for all $\widetilde{\eta}_{i,0}$’s, one then uses the Lipschitz continuity of the resolvent in $\eta$ and the regularity of the characteristics $\eta_{i,t}$ on the initial data (see e.g. the paragraph above \eqref{Teps} for the details of a \emph{grid argument}). 
Practically, the bound on $Y$ for any $t\le\tau$ encoded in the stopping time will hold for every 
$\widetilde{\eta}_{i,0}\in  [\eta_{i,0}, \omega_1]$ with a slightly worse factor than $n^{2\xi}$.
We omit these minor details and with a slight abuse of notation we wrote $\sup^*$ over a continuum set.
\label{disc}}
$\widetilde{\eta}_{i,0}\in [\eta_{i,0}, \omega_1]$, with $\omega_1$ given in Proposition~\ref{pro:mainpro}.
 \nc Here $\widetilde{\eta}_{i,t}$ is the solution to the diagonal part of \eqref{eq:matchar} with initial conditions $z_{i,0}$, $\widetilde{\eta}_{i,0}$. Note that the initial condition $z_{i,0}$, hence $z_{i,t}=e^{-t/2}z_{i,0}$ along the tilde-flow,
 is always the same, irrespective of $\widetilde{\eta}_{i,0}$.  Furthermore, by \eqref{eq:inasspart1},
  together with a standard grid argument, 
   it follows that $\tau>0$ with very high probability.

Additionally, by the definition of the stopping time $\tau$, it follows that for any $t\le \tau$ we have
\begin{equation}
\label{eq:ufhig}
\sup_{\eta_{i,t}\le\widehat{\eta}_{i}\le \omega_1/2} \max_{\sigma, \sigma'\in\{+,-\}}    \alpha(\widehat{\eta}_1,z_{1,t},\widehat{\eta}_2,z_{2,t})^{-1}\nc
\big|Y_{\sigma, \sigma' , t}(\widehat{\eta_1},z_{1,t},\widehat{\eta_2},z_{2,t})\big| \nc\le n^{2\xi},
\end{equation}
with very high probability. 
We remark that to obtain \eqref{eq:ufhig} we also used that for any fixed $t\le T_1^\epsilon\wedge T_2^\epsilon$, 
by Lemma~\ref{lem:ode}, any 
$\widehat{\eta}_{i}\in [\eta_{i,t},\omega_1/2]$ can be achieved as the solution $\widetilde{\eta}_{i,t}$ of \eqref{eq:scalchar} 
with an appropriately chosen initial condition $\widetilde{\eta}_{i,0}\in [\eta_{i,0},\omega_1]$, i.e. 
$\widehat{\eta}_{i} = \widetilde{\eta}_{i,t}$.
 Note that in \eqref{eq:ufhig} the supremum is taken only over $\widehat{\eta}_{i}$'s, whilst the $z_{i,t}$'s are fixed. 
 
 The supremum over a broader set of initial conditions $\widetilde{\eta}_{i,0}$ in \eqref{eq:deftildetau},
  as well as over $\widehat{\eta}_{i}$ in \eqref{eq:ufhig},
   needs to be considered for a minor technical reason:  at some point  we will also need
   to consider quantities of the form $GA|G|$ instead of $GAG$, and we use an integral representation
   of $|G|$ in terms of $G$ that requires integration over an entire segment of spectral parameters 
   (see later in Appendix~\ref{sec:rand}).
   The reader can safely ignore all tildes  and hats at
   first reading and focus on the special case $\widetilde\eta_{i,t}=\eta_{i,t}$ and $\widehat{\eta_i}=\eta_{i,t}$
   in these suprema.

Before proceeding with the estimate of the various terms in \eqref{eq:fulleqaasimp}, we state some bounds for the functions $\langle M_{12,t}^{A_1}A_2\rangle$ in the special case when $A_1,A_2\in \{E_+, E_-\}$. The proof of this lemma is postponed to Appendix~\ref{sec:det}.
\begin{lemma}
\label{lem:propertms}
Let $\gamma(\eta_1,z_1,\eta_2,z_2)$ be defined in \eqref{eq:defellbetas}. There exists an explicit real function $b(\eta_1,z_1,\eta_2,z_2)\in\R$ such that
\begin{equation}
\label{eq:rel12}
\begin{split}
\big|\langle M_{12}^{E_+} E_+\rangle\big| &\lesssim \frac{1}{\gamma(\eta_1,z_1,\eta_2,z_2)}, \qquad\qquad\quad 
\big|\langle M_{12}^{E_-} E_-\rangle\big| \lesssim \frac{1}{\gamma(\eta_1,z_1,\eta_2,z_2)}, \\
\langle M_{12}^{E_+} E_-\rangle=-\langle M_{12}^{E_-} E_+\rangle&=\ii b(\eta_1, z_1,\eta_2, z_2), \qquad\qquad\quad 
\big|b(\eta_1,z_1,\eta_2,z_2)\big|\lesssim \frac{1}{\gamma(\eta_1,z_1,\eta_2,z_2)}. 
\end{split}
\end{equation}
Additionally, $\langle M_{12}^{E_+}E_+\rangle$, $\langle M_{12}^{E_-}E_-\rangle$ are real and they satisfy
\begin{equation}
\label{eq:neednow}
\begin{split}
\langle M_{12}^{E_+} E_+\rangle&\le  -\langle M_{12}^{E_-} E_-\rangle \qquad \mathrm{for} \quad \eta_1\eta_2>0, \\
\langle M_{12}^{E_+} E_+\rangle&\ge  -\langle M_{12}^{E_-} E_-\rangle \qquad \mathrm{for} \quad \eta_1\eta_2<0.
\end{split}
\end{equation}\nc
\end{lemma}
   
We now present the estimate of the various terms in the rhs. of \eqref{eq:fulleqaasimp} one by one. 

\underline{\textbf{Estimate of martingale term:}} We now proceed with the estimate of the quadratic variation of the stochastic term in \eqref{eq:fulleqaasimp} (see Appendix~\ref{sec:rand} for the detailed computation):
\begin{equation}
\label{eq:quadvarnew}
\begin{split}
&\frac{1}{n^2} \sum_{i\ne j}
\Bigg[\langle G_{1,t}A_1G_{2,t}A_2G_{1,t}E_iG_{1,t}^*A_2G_{2,t}^*A_1G_{1,t}^*E_j\rangle+\langle G_{2,t}A_2G_{1,t}A_1G_{2,t}E_iG_{2,t}^*A_1G_{1,t}^*A_2G_{2,t}^*E_j\rangle \\
&\qquad\qquad\quad+2\Re \langle G_{1,t}A_1G_{2,t}A_2G_{1,t}E_iG_{2,t}^*A_1G_{1,t}^*A_2G_{2,t}^*E_j
\rangle\Bigg]\,\dd t.
\end{split}
\end{equation}
Here we used that $A_1, A_2$ are Hermitian, since  $A_1, A_2\in\{ E_+,E_-\}$.
We present the estimate only for the first term in \eqref{eq:quadvarnew}, the the second and third terms are
handled completely analogously. To estimate this term we use
\begin{equation}
\label{eq:ward}
\begin{split}
&\langle G_{1,t}A_1G_{2,t}A_2G_{1,t}E_iG_{1,t}^*A_2G_{2,t}^*A_1G_{1,t}^*E_j\rangle \\
&\lesssim \langle G_{1,t}A_1G_{2,t}A_2G_{1,t}E_iG_{1,t}^*A_2G_{2,t}^*A_1G_{1,t}^*\rangle^{1/2} \langle G_{1,t}^*A_2G_{2,t}^*A_1G_{1,t}^*E_jG_{1,t}A_1G_{2,t}A_2G_{1,t}\rangle^{1/2}\\
&\lesssim \lVert E_i\rVert\lVert E_j\rVert \langle G_{1,t}A_1G_{2,t}A_2G_{1,t}G_{1,t}^*A_2G_{2,t}^*A_1G_{1,t}^*\rangle^{1/2} \langle G_{1,t}^*A_2G_{2,t}^*A_1G_{1,t}^*G_{1,t}A_1G_{2,t}A_2G_{1,t}\rangle^{1/2}\\
& \lesssim \frac{1}{\eta_{1,t}^2}\langle \Im G_{1,t}A_1G_{2,t}A_2\Im G_{1,t} A_2G_{2,t}^*A_1\rangle,
\end{split}
\end{equation}
where to go from the first to the second line we used a Schwarz inequality to separate $G_{1,t}A_1G_{2,t}A_2G_{1,t}E_i$ from the rest, in the second inequality we used the matrix inequality $\langle QEQ^*\rangle\le\lVert E\rVert \langle QQ^*\rangle$, and in the last inequality we used $\lVert E_i\rVert+\lVert E_j\rVert\lesssim 1$ and the Ward identity $GG^*=\Im G/\eta$. 
We now face the problem that even if we are studying the evolution of the product of two resolvents the quadratic variation consists of traces of products of four resolvents that needs to be reduced to traces with two resolvents 
in order to have a closed equation.
For this purpose we will give two different estimates for  \eqref{eq:ward}. 
One estimate uses  the deterministic bound $\| G\|\le 1/\eta$ \nc
to reduce the number of resolvents at the price of $1/\eta$  factor, reminiscent to the Ward identity.
 This procedure 
loses the  natural factor $1/\gamma_t$ originating from every instance a factor $G_{1,t}AG_{2,t}$ appears. 
Note that the factor
 $1/\gamma_t$ is much better than $1/\eta_t$,
in particular the former carries the optimal $|z_1-z_2|$ decay.
The second estimate, instead, will preserve the number of $1/\gamma_t$'s
(hence the $|z_1-z_2|$-decay) at the price of losing a factor $\sqrt{n\ell_t}$. In particular, the factor $\sqrt{n\ell_t}$ is lost by the use of the reduction inequality \eqref{eq:redinnew} below (see also \eqref{eq:need0} below).

We start with this latter estimate. We rely on the following elementary \emph{reduction inequality}
\begin{equation}
\label{eq:redinnew}
\langle \Im G_{1,t} A_1 G_{2,t} A_2\Im G_{1,t} A_2 G_{2,t}^* A_1\rangle
\le n \langle\Im G_{1,t}A_1|G_{2,t}|A_1\rangle\langle\Im G_{1,t}A_2|G_{2,t}|A_2\rangle,
\end{equation}
 which easily follows by spectral decomposition 
\begin{equation}
\begin{split}
&\langle \Im G_{1,t} A_1 G_{2,t} A_2\Im G_{1,t} A_2 G_{2,t}^* A_1\rangle \\
&\qquad\quad=\frac{\eta_1^2}{n}\sum_{ijkl} \frac{\langle \bm{u}_i^{z_1}(t),A_1\bm{u}_j^{z_2}(t)\rangle\langle \bm{u}_j^{z_2}(t),A_2\bm{u}_k^{z_1}(t)\rangle\langle \bm{u}_k^{z_1}(t),A_2\bm{u}_l^{z_2}(t)\rangle\langle \bm{u}_l^{z_2}(t),A_1\bm{u}_i^{z_1}(t)\rangle}{[(\lambda_i^{z_1}(t))^2+\eta_{1,t}^2](\lambda_j^{z_2}(t)-\ii\eta_{2,t})[(\lambda_k^{z_1}(t))^2+\eta_{1,t}^2](\lambda_l^{z_2}(t)-\ii\eta_{2,t})} \\
&\qquad\quad\lesssim n \langle\Im G_{1,t}A_1|G_{2,t}|A_1\rangle\langle\Im G_{1,t}A_2|G_{2,t}|A_2\rangle  ,
\end{split}
\end{equation}
where to go from the first to the second line we used the Schwarz inequality in the form (omitting the time dependence of the eigenvectors)
\[
\big|\langle \bm{u}_i^{z_1},A_1\bm{u}_j^{z_2}\rangle\langle \bm{u}_j^{z_2},A_2\bm{u}_k^{z_1}\rangle\langle \bm{u}_k^{z_1},A_2\bm{u}_l^{z_2}\rangle\langle \bm{u}_l^{z_2},A_1\bm{u}_i^{z_1}\rangle\big|\le |\langle \bm{u}_i^{z_1},A_1\bm{u}_j^{z_2}\rangle\langle \bm{u}_k^{z_1},A_2\bm{u}_l^{z_2}\rangle|^2+|\langle \bm{u}_j^{z_2},A_2\bm{u}_k^{z_1}\rangle\langle \bm{u}_l^{z_2},A_1\bm{u}_i^{z_1}\rangle|^2.
\]
\nc
To estimate the rhs. of~\eqref{eq:redinnew}, we claim that
\begin{equation}
\label{eq:absvalbound}
\langle \Im G_{1,t}A_1|G_{2,t}|A_1\rangle\lesssim \frac{\log n}{\gamma_{t}}+\frac{n^{2\xi}\log n}{\sqrt{n\ell_t}\gamma_{t}},
\end{equation}
i.e. that the bound \eqref{eq:ufhig} with the choice $\widehat\eta_t=\widetilde\eta_{i,t}$
still essentially holds when one of the $G$'s is replaced by its absolute value. The proof of 
 \eqref{eq:absvalbound} is presented in Appendix~\ref{sec:rand}. 
 For notational simplicity, in the remainder of the proof we will always omit $(\log n)$--factors. 
 
  For $t\le\tau$, we thus estimate the integral of the quadratic variation of stochastic term in \eqref{eq:fulleqaasimp} as follows
\begin{equation}
\label{eq:need0}
\begin{split}
\int_0^t \frac{1}{n^2\eta_{1,s}^2} \langle \Im G_{1,s} A_1 G_{2,s} A_2\Im G_{1,s} A_2 G_{2,s}^* A_1\rangle\, \dd s&\le \int_0^t \frac{1}{n\eta_{1,s}^2}  \langle\Im G_{1,s}A_1|G_{2,s}|A_1\rangle\langle\Im G_{1,s}A_2|G_{2,s}|A_2\rangle\, \dd s\\
&\lesssim \int_0^t \frac{1}{n\eta_{1,s}^2} \left(\frac{1}{\gamma_s}+\frac{n^{2\xi}}{\sqrt{n\ell_s}\gamma_s}\right)^2\, \dd s \\
&\lesssim \int_0^t \frac{1}{n\eta_{1,s}^2\gamma_s^2}\, \dd s \lesssim \frac{1}{n\ell_t\gamma_t^2}
\end{split}
\end{equation}
where in the first inequality we used \eqref{eq:redinnew}, in the second inequality we used \eqref{eq:absvalbound}, in the third inequality we used that $n^{4\xi}\le n\ell_s$, since $\xi\le \epsilon/10$ (we will often use these facts in the remainder of this section even if not stated explicitly), and in the last inequality we used  the first relation in \eqref{eq:relineta} together with 
 $\gamma_{s}\gtrsim \gamma_{t}$.

We now present the other estimate for the integral of the quadratic variation which loses any $1/\gamma_t$ but it is better in terms of $1/n$--powers:
\begin{equation}
\begin{split}
\label{eq:estquadvar}
\int_0^t \frac{1}{n^2\eta_{1,s}^2}\langle \Im G_{1,s} A_1 G_{2,s} A_2\Im G_{1,s} A_2 G_{2,s}^* A_1\rangle\,\dd s&\le \int_0^t \frac{1}{n^2\eta_{1,s}^2}\lVert  A_1 G_{2,s} A_2\Im G_{1,s} A_2 G_{2,s}^* A_1\rVert \langle \Im G_{1,s}\rangle\,\dd s\\
&\lesssim \int_0^t \frac{\rho_{1,s}}{n^2\eta_{1,s}^3\eta_{2,s}^2}\, \dd s \lesssim \frac{1}{(n\eta_{1,t}\eta_{2,t})^2},
\end{split}
\end{equation}
where in the first inequality we tacitly used the cyclicity of the trace
 and the positivity of $\Im G$ to write $\langle \Im G_{1} A_1 G_{2} A_2\Im G_{1} A_2 G_{2}^* A_1\rangle=
\langle [\Im G_{1}]^{1/2} A_1 G_{2} A_2\Im G_{1} A_2 G_{2}^* A_1[\Im G_{1}]^{1/2}\rangle$ before
 the operator inequalities $\|A_i\|\lesssim 1$, $\|G_i\|\le 1/\eta_i$ were  applied. Furthermore, in the second inequality we also used the single resolvent local law  
 $|\langle G_{i,s}-M_{i,s}\rangle|\lesssim n^\xi/(n\eta_{i,s})$
in the form $\langle \Im G_{i,s}\rangle \lesssim \rho_{i,s} + n^\xi/(n\eta_{i,s}) \lesssim  \rho_{i,s}$
 since $n\ell_s\ge n^\epsilon$. Here we set $\rho_{i,s}:=\pi^{-1}|\Im m^{z_{i,s}}(\ii\eta_{i,s})|$.
 In the last step we used $\eta_{2,t}\le \eta_{2,s}$ by monotonicity of $t\to\eta_t$
 and the first relation in \eqref{eq:relineta} together with the fact that $\rho_{i,s}\sim \rho_{i,t}$, for any $s,t$, by the second relation in \eqref{eq:relchar} (note that $\rho_{i,s}=\pi^{-1}\langle \Im M_{i,s}\rangle$). We will use the fact that $\rho_{i,s}\sim \rho_{i,t}$ throughout this section, even if not stated explicitly.

  Then, combining the two inequalities \eqref{eq:need0} and \eqref{eq:estquadvar},
    we conclude the following  very high probability  bound for the  quadratic variation of the 
    It\^{o} stochastic term in \eqref{eq:fulleqaasimp} (recall that by $[\cdot]_t$ we denote the quadratic variation process):
    \begin{equation}
    \label{stoc}
\left(\left[\int_0^{\cdot}\partial_{ab}\langle G_{1,s}A_1G_{2,s}A_2\rangle\frac{\dd B_{ab,s}}{\sqrt{n}}\right]_{t\wedge \tau}\right)^{1/2}\lesssim  \frac{n^\xi}{\sqrt{n\ell_{t\wedge\tau}}\gamma_{t\wedge \tau}}\wedge \frac{n^\xi}{n\eta_{1,t\wedge \tau}\eta_{2,t\wedge\tau}}=n^\xi\alpha_{t\wedge \tau}.
\end{equation}

\underline{\textbf{Estimate of forcing terms:}} We now continue the estimates of the other terms in~\eqref{eq:fulleqaasimp}.
The second term in the first line of \eqref{eq:fulleqaasimp} can be incorporated in the lhs. of \eqref{eq:fulleqaasimp} by considering the evolution of $e^{-t} \langle (G_{1,t}A_1G_{2,t}-M_{12,t}^{A_1})A_2\rangle$; we thus neglect this term from the analysis.

Next, we consider the terms in the last line of \eqref{eq:fulleqaasimp}. These terms, using the single resolvent local law $|\langle G_{1,t}-M_{1,t}\rangle|\lesssim n^\xi/(n\eta_{1,t})$, for $t\le T_1^\epsilon\wedge T_2^\epsilon$,
 are bounded by (we present the bound for only one of them) 
\begin{equation}
\label{eq:need1}
\begin{split}
&\int_0^t\big|\langle G_{1,s}-M_{1,s}\rangle\big| \big|\langle G_{1,s}^2A_1G_{2,s}A_2\rangle\big| \, \dd s \\
&\qquad\quad \lesssim \int_0^t \big|\langle (G_{1,s}-M_{1,s})\rangle\big| \frac{1}{\eta_{1,s}}\langle \Im G_{1,s}A_1|G_{2,s}|A_1\rangle^{1/2} \langle \Im G_{1,s}A_2|G_{2,s}|A_2\rangle^{1/2} \, \dd s \\
&\qquad\quad\lesssim \int_0^t \frac{n^\xi}{n\eta_{1,s}^2}
 \left(\frac{1}{\gamma_s}+\frac{n^{2\xi}}{\sqrt{n \ell_s}\gamma_s}\right)\,\dd s \lesssim \frac{n^\xi}{n\ell_t\gamma_t}.
\end{split}
\end{equation}
Here in the first inequality we used the bound
\begin{equation}
\label{eq:need2}
\begin{split}
\big|\langle G_1^2A_1G_2A_2\rangle\big|&\lesssim \sum_{i,j}\frac{\big|\langle \ww_i^{z_1}, A_1\ww_j^{z_2}\rangle\langle \ww_j^{z_2}, A_2\ww_j^{z_1}\rangle \big|}{|\lambda_i^{z_1}-\ii\eta_1|^2|\lambda_j^{z_2}-\ii\eta_2\big|}\ \\
&\lesssim \left(\sum_{i,j}\frac{|\langle \ww_i^{z_1}, A_1\ww_j^{z_2}\rangle|^2}{|\lambda_i^{z_1}-\ii\eta_1|^2|\lambda_j^{z_2}-\ii\eta_2\big|}\right)^{1/2}\left(\sum_{i,j}\frac{|\langle \ww_j^{z_2}, A_2\ww_j^{z_1}\rangle|^2}{|\lambda_i^{z_1}-\ii\eta_1|^2|\lambda_j^{z_2}-\ii\eta_2\big|}\right)^{1/2}\ \\
&= \frac{1}{\eta_1}\langle \Im G_1A_1|G_2|A_1\rangle^{1/2} \langle \Im G_1A_2|G_2|A_2\rangle^{1/2},
\end{split}
\end{equation}
while in second inequality we used \eqref{eq:absvalbound}. In the last step, we used $n^{2\xi}\le \sqrt{n\ell_s}$, 
$\gamma_t\le \gamma_s$, and ~\eqref{eq:relineta}. Similarly to the estimate for the martingale term above, we now give another estimate for the term in \eqref{eq:need1}:
\begin{equation}
\begin{split}
&\int_0^t\big|\langle G_{1,s}-M_{1,s}\rangle\big| \big|\langle G_{1,s}^2A_1G_{2,s}A_2\rangle\big| \, \dd s \\
&\qquad\quad \lesssim \int_0^t \big|\langle (G_{1,s}-M_{1,s})\rangle\big| \frac{1}{\eta_{1,s}}\langle \Im G_{1,s}\rangle^{1/2} \langle \Im G_{1,s}A_1G_{2,s}A_2^2G_{2,s}^*A_1\rangle^{1/2} \, \dd s \\
&\qquad\quad\lesssim \int_0^t \frac{n^\xi}{n\eta_{1,s}^2}\cdot \frac{\rho_{1,s}}{\eta_{2,s}}\,\dd s \lesssim \frac{n^\xi}{n\eta_{1,t}\eta_{2,t}},
\end{split}
\end{equation}
where in the first inequality we used a simple Schwarz inequality, and in the second inequality we used a norm bound, as in \eqref{eq:estquadvar}, to estimate the last trace in the second line.\nc

Then, for $t\le\tau$, writing $E_i,E_j$ as a linear combination
of  $E_+,E_-$ (using the definitions in \eqref{eq:defF}) so that we could use ~\eqref{eq:deftildetau},
 we estimate the term in the fourth line of \eqref{eq:fulleqaasimp} by
 \begin{equation}
\label{eq:need34}
\begin{split}
\int_0^t \big|\langle (G_{1,s}A_1G_{2,s}-M_{12,s}^{A_1})E_i\rangle\big|\big|\langle (G_{2,s}A_2G_{1,s}-M_{21,s}^{A_2})E_j\rangle \big|\, \dd s&\lesssim \int_0^t  n^{4\xi}\left(\frac{1}{n\ell_s\gamma_s^2}\wedge\frac{1}{n^2\eta_{1,s}^2\eta_{2,s}^2}\right)\, \dd s \\
&\lesssim n^{4\xi}\left(\frac{1}{n \ell_t\gamma_t}\wedge \frac{1}{(n\ell_t)(n\eta_{1,t}\eta_{2,t})}\right).
\end{split}
\end{equation}
In the last step, to estimate the term $1/(n\ell_s\gamma_s^2)$, we used again $\gamma_s\ge\gamma_t$
and the second formula of~\eqref{eq:relineta} for the bound
\begin{equation}
\label{eq:usefintused}
\frac{1}{\gamma_t}\int_0^t \frac{1}{\ell_s}\,\dd s\le \frac{1}{\gamma_t}\int_0^t\left(\frac{1}{\eta_{1,s}\rho_{1,s}}+\frac{1}{\eta_{2,s}\rho_{2,s}}\right)\,\dd s\lesssim \frac{1}{\gamma_t}\left(\frac{1}{\rho_{1,s}^2}+\frac{1}{\rho_{2,s}^2}\right)\lesssim \frac{1}{\ell_t},
\end{equation}
where in the last inequality we used that $\gamma_t\ge \max_i (\eta_{i,t}/\rho_{i,t})$. 
The estimate for $1/(n\eta_{1,s}\eta_{2,s})^2$ is simpler, and it follows by the monotonicity $\eta_{i,s}\ge \eta_{i,t}$, together with $\int 1/\eta_{i,s}^2\le1/\ell_t$. \nc

\underline{\textbf{Summary of previous bounds:}} Collecting the bounds~\eqref{stoc}, \eqref{eq:need1} and \eqref{eq:need34} and
writing again $E_1,E_2$ in terms of $E_+,E_-$, we are thus left with  
\begin{equation}
\label{eq:critical}
\begin{split}
&\dd \langle (G_{1,t}A_1G_{2,t}-M_{12,t}^{A_1})A_2\rangle \\
&\qquad\qquad\quad= \langle G_{1,t}A_1G_{2,t}-M_{12,t}^{A_1}\rangle\langle M_{21,t}^{A_2}\rangle\dd t -\langle (G_{1,t}A_1G_{2,t}-M_{12,t}^{A_1})E_-\rangle\langle M_{21,t}^{A_2}E_-\rangle\dd t\\
&\qquad\qquad\quad\quad+\langle M_{12,t}^{A_1} \rangle\langle G_{2,t}A_2G_{1,t}-M_{21,t}^{A_2}\rangle\dd t-\langle M_{12,t}^{A_1} E_- \rangle\langle (G_{2,t}A_2G_{1,t}-M_{21,t}^{A_2})E_-\rangle\dd t+ h_t\dif t\nc +\dd e_t,
\end{split}
\end{equation}
with $h_t$ being a forcing term, and  $e_t$ being a martingale \nc term such that
\[
\int_0^{t\wedge \tau} \big|h_s\big|\,\dd s+\left[\int_0^{\cdot}\dif e_s\right]_{t\wedge \tau}^{1/2}\lesssim n^\xi\alpha_{t\wedge\tau},
\]
with very high probability.

\underline{\textbf{Solving \eqref{eq:critical} via Gr\"onwall:}} With the four choices $ A_1= E_\pm$, $A_2=E_\pm$, we view \eqref{eq:critical} a 
coupled system of four ODE's modulo a small error. 
Recalling the definition~\eqref{defY}, we introduce the short--hand notations 
\begin{equation}
Y_{\sigma,\sigma' ,t}:=Y_{\sigma, \sigma' , t}(\eta_{1,t},z_{1,t},\eta_{2,t},z_{2,t}), \qquad \sigma, \sigma' \in\{ +, -\},
\end{equation}
for these four components. Furthermore, we define the  $4$--dimensional \nc column vector
\begin{equation}
\mathcal{Y}_t:=\left(\begin{matrix}
Y_{+,+,t} \\
Y_{+,-,t} \\
Y_{-,+,t} \\
Y_{-,-,t}
\end{matrix}\right) \in \C^4,
\end{equation}
and the $4$ by $4$ matrices
\begin{equation}
\begin{split}
\widehat{\mathcal{M}}_t:&=\left(\begin{matrix}
2a_{+,t} & \ii b_t & -\ii b_t & 0\\
\ii b_t & a_{+,t}+a_{-,t} & 0 & -\ii b_t \\
-\ii b_t & 0 & a_{+,t}+a_{-,t} & \ii b_t \\
0 & -\ii b_t & \ii b_t & 2a_{-,t}
\end{matrix}\right), \\
\mathcal{D}_t:&=  \left(\begin{matrix}
2(a_t-a_{+,t}) & 0 & 0 & 0\\
0 & 2a_t-a_{+,t}-a_{-,t} & 0 & 0 \\
0 & 0 & 2a_t-a_{+,t}-a_{-,t} & 0 \\
0 & 0 & 0 & 2(a_t-a_{-,t}) 
\end{matrix}\right) 
\end{split}
\end{equation}
with
$$
a_{+,t}=a_+(\eta_{1,t}, z_{1,t}, \eta_{2,t}, z_{2,t}):=\langle M_{12,t}^{E_+} E_+\rangle,
\qquad a_{-,t}=a_-(\eta_{1,t}, z_{1,t}, \eta_{2,t}, z_{2,t}):=-\langle M_{12,t}^{E_-} E_-\rangle
$$ 
and\footnote{In this proof we assumed that $\eta_{i,t}>0$ for simplicity, 
but  here we give the definition of $a_t$ for all cases.}
\begin{equation}
\label{eq:defab}
a_t:=\begin{cases}
a_{+,t} &\mathrm{if}\quad \eta_{1,t}\eta_{2,t}<0, \\
a_{-,t} &\mathrm{if}\quad \eta_{1,t}\eta_{2,t}>0,
\end{cases}
\nc
\qquad\qquad\quad
b_t:=b(\eta_{1,t}, z_{1,t}, \eta_{2,t}, z_{2,t}).
\end{equation}
  Notice that by \eqref{eq:neednow} $\mathcal{D}_t$ is a positive matrix.
 Additionally, define
\begin{equation}
\mathcal{M}_t:=\left(\begin{matrix}
2a_t & \ii b_t & -\ii b_t & 0\\
\ii b_t & 2a_t & 0 & -\ii b_t \\
-\ii b_t & 0 & 2a_t & \ii b_t \\
0 & -\ii b_t & \ii b_t & 2a_t
\end{matrix}\right),
\end{equation}
so that $\widehat{\mathcal{M}}_t=\mathcal{M}_t-\mathcal{D}_t$. 

Then, using these notations, by \eqref{eq:critical}, it follows that
\begin{equation}
\label{eq:4by4ode}
\dd \mathcal{Y}_t=(\mathcal{M}_t-\mathcal{D}_t)\mathcal{Y}_t \dd t + \mathcal{F}_t\dd t\nc+\dd \mathcal{E}_t,
\end{equation}
with  $\mathcal{F}_t$ being a forcing term, and $\mathcal{E}_t$ being a  $4\times 4$ matrix
valued martingale  such that
\begin{equation}
\label{eq:bounderrterm}
\int_0^{t\wedge\tau}\lVert  \mathcal{F}_s\rVert  \dd s +\left(\left\lVert\left[\int_0^{\cdot} \dd \mathcal{E}_s\right]_{t\wedge \tau} \right\rVert\right)^{1/2}\lesssim n^\xi\alpha_{t\wedge\tau}.
\end{equation}
In order to estimate $\mathcal{Y}_t$, the solution of \eqref{eq:4by4ode}, 
 we will rely on the following simple lemma which is  a stochastic matrix--valued Gronwall inequality:
\begin{lemma}
\label{lem:matrixgronwall}
Fix $d\in \N$, and let $\mathcal{X}_t\in \C^d$ be the solution of the following matrix valued SDE:          
\[
\dd \mathcal{X}_t=\mathcal{P}_t\mathcal{X}_t \dd t+\mathcal{F}_t \dd t+\dd \mathcal{E}_t,
\]
with given initial condition $\mathcal{X}_0$, 
where the forcing term $\mathcal{F}_t\in \C^d$ and the coefficient matrix $\mathcal{P}_t\in \C^{d \times d}$ 
are adapted to continuous family of $\sigma$-algebras and 
 $\mathcal{E}_t\in \C^d$ is a martingale.  We also assume that $\lVert \mathcal{F}_t\rVert+\lVert \mathcal{E}_t\rVert\le n^C$ for some large parameter $n$ and for some constant $C>0$.
  Furthermore, assume that $\mathrm{Spec}(\Re \nc\mathcal{P}_t)\subset (-\infty, f_t]$, for some non--negative real valued function $f_t\ge 0$. Let $\mathcal{C}_t\dd t\in\C^{d\times d}$ be the covariation process of $\dd \mathcal{E}_t$, 
   let $\tau$ be a stopping time, and assume that
\begin{equation}
\label{eq:desbneed}
\left(\int_0^{t\wedge \tau} \lVert\mathcal{F}_s\rVert\,\dd s\right)^2+\int_0^{t\wedge \tau} \lVert \mathcal{C}_s\rVert\,\dd s\lesssim \alpha_{t\wedge \tau}^2,
\end{equation}
for some positive deterministic function $\alpha_t$.
 Then, with very high probability, for any arbitrary small $\zeta>0$
and for any $t\ge 0$, we have
\begin{equation}
\label{eq:desiredb}
\sup_{0\le s\le t\wedge\tau}\lVert \mathcal{X}_s\rVert^2\lesssim \lVert \mathcal{X}_0\rVert^2
+n^{3\zeta}\alpha_t^2+\int_0^{ t\wedge\tau}\left(\lVert \mathcal{X}_0\rVert^2+n^{3\zeta}\alpha_s^2\right) f_s\exp\left(2(1+n^{-\zeta})\int_s^{t\wedge\tau} f_r\,\dd r\right)\,\dd s.
\end{equation}
\end{lemma}
\begin{proof}
We consider the evolution of $\lVert \mathcal{X}_t\rVert^2$:
\[
\dd \lVert \mathcal{X}_t\rVert^2= 2 \langle \mathcal{X}_t,(\Re \mathcal{P}_t)\mathcal{X}_t\rangle\dd t+2\Re \langle\mathcal{F}_t,\mathcal{X}_t\rangle\dd t+2\Re\langle \dd \mathcal{E}_t,\mathcal{X}_t\rangle+\mathrm{Tr}[\mathcal{C}_t]\dif t .
\]
Integrating this in time, we get
\begin{equation}
\label{eq:step1gron}
\lVert \mathcal{X}_t\rVert^2\le \lVert \mathcal{X}_0\rVert^2+2\int_0^t f_s\lVert \mathcal{X}_s\rVert^2 \dif s+C\int_0^t \lVert \mathcal{F}_s\rVert\cdot \lVert\mathcal{X}_s\rVert\dif s+2\int_0^t\Re\langle \dd \mathcal{E}_s,\mathcal{X}_s\rangle+d\int_0^t \lVert \mathcal{C}_s\rVert\,\dd s
\end{equation}
for some $C=C(d)>0$. Define $\mathcal{Z}_t:=\sup_{0\le s\le t}\lVert \mathcal{X}_s\rVert^2$, then, to estimate the stochastic term in \eqref{eq:step1gron} we use the martingale inequality (see \cite[Appendix B.6, Eq. (18)]{SW09} with $c=0$ for continuous martingales)
\begin{equation}
\label{eq:goddmartin}
\P\left(\sup_{0\le s\le t}\left|\int_0^{s\wedge \tau}\Re\langle \dd \mathcal{E}_r,\mathcal{X}_r\rangle\right|\ge \lambda,\quad \left[\int_0^{\cdot}\Re\langle \dd \mathcal{E}_s,\mathcal{X}_s\rangle\right]_{t\wedge \tau} \le \mu\right)\le 2e^{-\lambda^2/2\mu},
\end{equation}
for any fixed $\lambda,\mu>0$. Using the boundedness of $\mathcal{E}_t$, $\mathcal{X}_t$ and a dyadic argument, \eqref{eq:goddmartin} implies
\begin{equation}
\label{eq:martineq}
\sup_{0\le s\le t}\left|\int_0^{s\wedge \tau}\Re\langle \dd \mathcal{E}_r,\mathcal{X}_r\rangle\right|\le \log n \left[\int_0^{\cdot}\Re\langle \dd \mathcal{E}_s,\mathcal{X}_s\rangle\right]_{t\wedge \tau}^{1/2}+n^{-100},
\end{equation}
with very high probability.
Then, using that by \eqref{eq:desbneed} we have
\[
\left[\int_0^{\cdot}\Re\langle \dd \mathcal{E}_s,\mathcal{X}_s\rangle\right]_{t\wedge \tau}\lesssim\int_0^{t\wedge \tau} \langle \mathcal{X}_s,\mathcal{C}_s\mathcal{X}_s\rangle\, \dd s\lesssim \int_0^{t\wedge\tau} \lVert \mathcal{C}_s\rVert\cdot\lVert\mathcal{X}_s\rVert^2\,\dd s\lesssim \mathcal{Z}_{t\wedge\tau}\alpha_{t\wedge\tau}^2,
\]
the estimate \eqref{eq:martineq}, plugged into~\eqref{eq:step1gron},  implies (here we also use a standard discretization argument in $t$, see e.g. \cite[Eqs. (3.34)--(3.35)]{meso30}, to ensure that \eqref{eq:martineq} holds simultaneously for all $t$'s)
\begin{equation}
\mathcal{Z}_{t\wedge \tau}\le \lVert \mathcal{X}_0\rVert^2+2\int_0^{t\wedge\tau} f_s \mathcal{Z}_s\,\dd s+\alpha_{t\wedge\tau}\mathcal{Z}_{t\wedge\tau}^{1/2}\log n+C\alpha_{t\wedge\tau}^2,
\end{equation}
with very high probability. Using $\alpha_{t\wedge\tau}\mathcal{Z}_{t\wedge\tau}^{1/2}\le n^{2\zeta}\alpha_{t\wedge\tau}^2+n^{-\zeta}\mathcal{Z}_{t\wedge\tau}$, we thus conclude (neglecting $\log n$--factors) 
\[
\mathcal{Z}_{t\wedge\tau}\le (1+n^{-\zeta}) \lVert \mathcal{X}_0\rVert^2+2(1+n^{-\zeta})\int_0^{t\wedge\tau} f_s \mathcal{Z}_s\,\dd s+
n^{3\zeta}\alpha_{t\wedge\tau}^2 .
\]
Then by a standard Gronwall inequality we conclude the proof.
\end{proof}
In order to apply this lemma, we now \nc rewrite $a_t$ from \eqref{eq:defab} as a time derivative of a certain quantity plus a lower order term (its proof is 
given in Appendix~\ref{sec:det}):
\begin{lemma}
\label{lem:nclfp}
Let $a_t$ from \eqref{eq:defab} , and define\footnote{Here $\beta_{\pm,t}$ are
 the eigenvalues of the time--dependent two--body stability operator defined around \eqref{eq:M12} (see also \eqref{eq:dedfevalues} below). We use the standard convention that the complex square root $\sqrt{\cdot}$ is defined using the branch cut $\C\setminus (-\infty,0)$.}
\begin{equation}
\label{eq:defevmore}
\beta_{\pm,t}:=1-\Re[z_{1,t}\overline{z_{2,t}}]u_{1,t}u_{2,t}
\pm\sqrt{m_{1,t}^2m_{2,t}^2-(\Im[z_{1,t}\overline{z_{2,t}}])^2u_{1,t}^2u_{2,t}^2}.
\end{equation}
Then, we have
\begin{equation}
\label{eq:impeqpm}
a_t=-\frac{1}{2}\partial_t\log\big|\beta_{+,t}\beta_{-,t}\big|+\frac{|m_{1,t}m_{2,t}|}{|\beta_{+,t}\beta_{-,t}|}.
\end{equation}
Additionally, we have
\begin{equation}
\label{eq:asympev}
|\beta_{\pm,t}|\sim \gamma_t=|z_{1,t}-z_{2,t}|+\frac{\eta_{1,t}}{\rho_{1,t}}+\frac{\eta_{2,t}}{\rho_{2,t}}.
\end{equation}
\end{lemma}

We are now ready to conclude the proof of  Proposition~\ref{pro:mainpro}. Using that the largest eigenvalue of $\Re\mathcal{M}_t$ is $2a_t$ (in fact, $\Re\mathcal{M}_t$ is diagonal), as a consequence of $a_t,b_t$, which are defined in \eqref{eq:defab}, being real from the sentence above \eqref{eq:rel12}  together with \eqref{eq:neednow}, and that $\mathcal{D}_t=\Re \mathcal{D}_t$ is positive, the propagator of \eqref{eq:4by4ode} has operator norm bounded by
\begin{equation}
\label{eq:propest}
\exp\left(\int_s^t 2a_r\,\dd r\right)
=\exp\left[-\int_s^t \left(\partial_r\log|\beta_{+,r}\beta_{-,r}|+\frac{|m_{1,r}m_{2,r}|}{|\beta_{+,r}\beta_{-,r}|}\right)\,\dd r\right]\lesssim \frac{|\beta_{+,s}\beta_{-,s}|}{|\beta_{+,t}\beta_{-,t}|}\sim \frac{\gamma_s^2}{\gamma_t^2}.
\end{equation}
 Here, by $|\beta_{\pm,r}|\sim \gamma_r\gtrsim \eta_{i,r}/\rho_{i,r}$ from \eqref{eq:asympev}, we used that
\[
\int_s^t\frac{|m_{1,r}m_{2,r}|}{|\beta_{+,r}\beta_{-,r}|}\,\dd r\lesssim \int_s^t\frac{\rho_{1,r}^2\rho_{2,r}^2}{\eta_{1,r}\eta_{2,r}}\,\dd r\le \int_s^t \left(\frac{\rho_{1,r}^4}{\eta_{1,r}^2}+\frac{\rho_{2,r}^4}{\eta_{2,r}^2}\right) \,\dif r\lesssim \frac{\rho_{1,t}^3}{\eta_{1,t}}+\frac{\rho_{2,t}^3}{\eta_{2,t}}\lesssim 1,
\]
where in the last inequality we used that $\rho_{i,t}\lesssim \eta_{i,t}^{1/3}$ by \eqref{rho}. The leading term $\partial_r\log|\beta_{+,r}\beta_{-,r}|$ has been explicitly  integrated out in~\eqref{eq:propest}.

We now  conclude the proof of  Proposition~\ref{pro:mainpro} applying Lemma~\ref{lem:matrixgronwall} 
 for the equation \eqref{eq:4by4ode}. The fact that $\mathcal{F}_t$, $\dif \mathcal{E}_t$ in \eqref{eq:4by4ode} satisfy \eqref{eq:desbneed}, with $\alpha_t$ defined in~\eqref{alphadef},
follows from \eqref{eq:bounderrterm}. Using  \eqref{eq:propest} in the form $\exp(\int_s^t 4a_r\,\dd r)\lesssim (\gamma_s/\gamma_t)^4$, we can thus apply Lemma~\ref{lem:matrixgronwall} (with $3\zeta=\xi$), for  $f_t=2a_t\lesssim 1/\gamma_t$, to obtain \nc
\begin{equation}
\label{eq:h}
\lVert \mathcal{Y}_t\rVert^2\lesssim n^{3\xi}\alpha_t^2 +n^{3\xi}\int_0^t\frac{\alpha_s^2}{\gamma_s}\exp\left(\int_s^t 4a_r\,\dd r\right)\,\dd s \lesssim  n^{3\xi}\alpha_t^2 +n^{3\xi}\int_0^t\frac{\alpha_s^2\gamma_s^3}{\gamma_t^4}\,\dd s \lesssim   n^{3\xi}\alpha_t^2,
\end{equation}
with very high probability,
 where we used that $\lVert \mathcal{Y}_0\rVert\lesssim \alpha_0\le \alpha_s$, by the monotonicity of $t\mapsto \alpha_t$.  The estimate of the integral in the last step will be explained later, first we conclude
the main proof.
 This proves the desired bound for the initial condition $\eta_{i,0}$. 
 We point out that a completely analogous proof holds if we consider any other initial condition 
 $\widetilde{\eta}_{i,0}\in [\eta_{i,0},\omega_1]$, obtaining exactly the same bound as in the last line of \eqref{eq:h}
but with a tilde on $\eta$'s and on all quantities depending on them (see the argument 
in footnote~\ref{disc} 
below \eqref{eq:deftildetau} to ensure this bound first on a fine mesh $\mathcal{P}$ of $[\eta_{i,0},\omega_1]$ and then for all  $\widetilde{\eta}_{i,0}$'s) \nc. This together, with the definition of $\tau$ \eqref{eq:deftildetau} 
  (note  the $n^{2\xi}$ threshold!)
   shows that $\tau=T_1^\epsilon\wedge T_2^\epsilon$, with very high probability.  
  

\underline{\textbf{Proof of \eqref{eq:h}:}} We close this proof proving the last inequality of \eqref{eq:h}. Here we used 
(recall the definition of $\alpha_t$ from \eqref{alphadef})
\begin{equation}
\label{eq:impbound}
\int_0^t\frac{\alpha_s^2\gamma_s^3}{\gamma_t^4}\, \dd s=\nc\int_0^t \left(\frac{1}{\sqrt{n\ell_s}\gamma_s}\wedge\frac{1}{n\eta_{1,s}\eta_{2,s}}\right)^2\frac{\gamma_s^3}{\gamma_t^4}\,\dd s\le \frac{1}{n\ell_t\gamma_t^2}\wedge \frac{1}{n^2\eta_{1,t}^2\eta_{2,t}^2}=\alpha_t^2.\nc
\end{equation}
Finally, we prove the last inequality  \eqref{eq:impbound}.
 Let $s_*$ be such that for any $s\ge s_*$ we have $\gamma_s\sim |z_{1,s}-z_{2,s}|$
and $\gamma_s\sim (\eta_{1,s}/\rho_{1,s})+(\eta_{2,s}/\rho_{2,s})$ for $s<s_*$  (by the monotonicity
 of $t\to\eta_{i,t}/\rho_{i,t}$ it follows that at least one such $s_*$  exists, maybe $s_*=0$ or $s_*=t$). We thus have
\[
\int_{s_*}^t \left(\frac{1}{\sqrt{n\ell_s}\gamma_s}\wedge\frac{1}{n\eta_{1,s}\eta_{2,s}}\right)^2\frac{\gamma_s^3}{\gamma_t^4}\,\dd s\lesssim \frac{1}{n^2\gamma_t}\int_0^t \frac{1}{\eta_{1,s}^2\eta_{2,s}^2}\,\dd s\le \frac{1}{n^2\gamma_t\eta_{2,t}^2}\int_0^t\frac{1}{\eta_{1,s}^2}\,\dd s\lesssim \frac{1}{n^2\eta_{1,t}^2\eta_{2,t}^2},
\]
where in the first inequality we used that $\gamma_{s_*}\sim |z_{1,s}-z_{2,s}|\sim |z_{1,0}-z_{2,0}|$ implies $\gamma_s\sim |z_{1,0}-z_{2,0}|$ for any $s_*\le s\le t$, and so that $\gamma_s/\gamma_t\sim 1$, and in the last inequality we used the first relation in \eqref{eq:relineta} together with the bound $\gamma_t\ge \eta_{1,t}/\rho_{1,t}$. Additionally, we 
also have the estimate
\[
\int_{s_*}^t \left(\frac{1}{\sqrt{n\ell_s}\gamma_s}\wedge\frac{1}{n\eta_{1,s}\eta_{2,s}}\right)^2\frac{\gamma_s^3}{\gamma_t^4}\,\dd s\lesssim \frac{1}{n\gamma_t^3}\int_0^t\frac{1}{\ell_s}\,\dd s\le \frac{1}{n\ell_t\gamma_t^2},
\]
where in the last inequality we used \eqref{eq:usefintused}. For the complementary regime, we have
\[
\begin{split}
\int_0^{s_*}\left(\frac{1}{\sqrt{n\ell_s}\gamma_s}\wedge\frac{1}{n\eta_{1,s}\eta_{2,s}}\right)^2\frac{\gamma_s^3}{\gamma_t^4}\,\dd s&
\lesssim\frac{1}{n^2\gamma_t^4} \int_0^{s_*} \frac{\gamma_s^3}{\eta_{1,s}^2\eta_{2,s}^2}\,\dd s \\
&\sim\frac{1}{n^2\gamma_t^4}  \int_0^t \Big[ \frac{\eta_{2,s}}{\eta_{1,s}^2\rho_{2,s}^3}+ 
\frac{\eta_{1,s}}{\eta_{2,s}^2\rho_{1,s}^3}\Big]\,\dd s \\
&\lesssim \frac{1}{n^2\gamma_t^4}  \int_0^t \Big[ \frac{\eta_{2,t}}{\eta_{1,s}\eta_{1,t}\rho_{2,t}^3}
+ \frac{\eta_{1,t}}{\eta_{2,s}\eta_{2,t}\rho_{1,t}^3}\Big]\,\dd s \\
&\quad+\frac{1}{n^2\gamma_t^4\rho_{1,t}\rho_{2,t}}\int_0^t\left[\frac{1}{\eta_{1,s}\rho_{2,t}}+\frac{1}{\eta_{2,s}\rho_{1,t}}\right]\,\dd s \\
&\lesssim \frac{1}{(n\gamma_t)^2\rho_{1,t}\rho_{2,t}\eta_{1,t}\eta_{2,t}} \\
&\lesssim \frac{1}{(n\ell_t\gamma_t)^2}\wedge \frac{1}{(n\eta_{1,t}\eta_{2,t})^2},
\end{split}
\]
where in the first $\sim$ relation we used 
  that $\gamma_s\sim (\eta_{1,s}/\rho_{1,s})+(\eta_{2,s}/\rho_{2,s})$ in this regime.
  In the second inequality we used
$\rho_{i,s}\sim \rho_{i,t}$ and we 
 estimated $\eta_{i,s}/\rho_{i,s}$ by using
 \[
\frac{\eta_{i,s}/\rho_{i,s}}{\eta_{j,s}/\rho_{j,s}}\lesssim 1+\frac{\eta_{i,t}/\rho_{i,t}}{\eta_{j,t}/\rho_{j,t}},
\]
that follows from the third relation in \eqref{eq:relchar}. Finally, 
 in the last two inequalities we used \eqref{eq:relineta}  (ignoring $\log n$--factors)
   and $\rho_{1,t}\rho_{2,t}\gamma_t^2\ge \eta_{1,t}\eta_{2,t}$.
\end{proof}

\subsection{Improved local law for $G^zFG^zF^*$ (Theorem~\ref{theorem_F})}  
\label{sec:gfgf}
As we mentioned after Proposition~\ref{pro:mainpro},  its Part 1 would hold for any 
observables $A_1, A_2$, giving a bound of order $1/(n\eta^2)$, considering $\eta_1=\eta_2=\eta$ 
for simplicity. Now we aim to  improve this bound 
 for the specific observables $F, F^*$ to gain an extra $\rho^2$, i.e. we gain a factor $\rho$ per $F$ or $F^*$ matrix. \nc
 Moreover, along the way we also need an improved single-resolvent 
 local law with $F$, and a two-resolvent local law when one observable is $F$ and the other is the 
 identity as these quantities naturally emerge. These three local laws will be proven simultaneously. This plan justifies the following short-hand notations.

 Recall the definition of the characteristics $\Lambda_t$ in \eqref{eq:matchar}, define the resolvent $G_t=G^{z_t}(\ii \eta_t):=(W_t-\Lambda_t)^{-1}$, and let 
 \begin{equation}
 \label{eq:singdetappr}
 M_t^A:=M_{12}^A(\eta_t,z_t,\eta_t,z_t) 
 \end{equation}
be the deterministic approximation of $G_tAG_t$ as defined in \eqref{eq:M12}. Note that unlike in the previous section here we only consider resolvents evaluated at the same spectral parameters since this is simpler and enough for the current application, in particular  in this section we drop the 12 indices from 
\begin{equation}
\label{eq:singdetappr1}
M^A:=\mathcal{B}^{-1} [MAM] := \big(1-M\mathcal{S}[\cdot] M)^{-1} [MAM].
\end{equation}
We define
\begin{equation}
\begin{split}
\label{eq:defphiss}
\Phi_1(t):&=\frac{n|\eta_t|}{\rho_t}\big|\langle (G_t-M_t)F\rangle\big|, \\
\Phi_F(t):&=\frac{n|\eta_t|^2}{\rho_t}\big|\langle G_tFG_t-M_t^F\rangle\big|, \\
\Phi_2(t):&=\frac{\sqrt{n}|\eta_t|^{3/2}}{\rho_t^{5/2}}\big|\langle (G_tFG_t-M_t^F)F^*\rangle\big|,
\end{split}
\end{equation}
where we recall that $\rho_t=\pi^{-1}|\Im m^{z_t}(\ii\eta_t)|$. We will show that
the bound $\Phi_1+\Phi_F + \Phi_2\prec 1$ propagates along the characteristic.
We point that the bound $\Phi_1(t)+\Phi_F(t) \prec 1$ 
is optimal, but the bound $\Phi_2(t)\prec 1 $ is not, i.e. the optimal bound for $\langle(G_tFG_t-M_t^F)F^*\rangle$ 
would be better by a factor $\sqrt{n|\eta|\rho}$. 
We do not pursue this additional gain here as it is not needed for the purpose of this paper. 
Furthermore, we can treat $F$ and $F^*$ on the same footing, in particular
\nc it is not necessary to define $\Phi_F(t)$ for $F^*$ instead of $F$
 as $\langle ( G_tFG_t)^*\rangle=-\langle G_tF^*G_t\rangle$. We also point out that by the first two equalities in \eqref{eq:identities} it follows that it is sufficient to \nc define $\Phi_F$ only when $G_tFG_t$ is tested against the identity and not against $E_-$. We are now ready to state the main result of this section.
 

\begin{proposition}
\label{pro:mainprogfgf}
Fix $n$--independent constants $\epsilon,\omega_1>0$, and define the spectral parameter
\begin{equation}
\Lambda_0:=\left(\begin{matrix}
\ii\eta_0 & z_0 \\
z_0 & \ii\eta_0
\end{matrix}\right),
\end{equation}
with $ |z_0|\ge 1$, and $\eta_{i,0}\ne 0$. For $0\le t<T^*(\eta_0,z_0)$, let $\Lambda_t$ be the solution of \eqref{eq:matchar} with initial condition $\Lambda_0$. Let $\Phi_1(t),\Phi_2(t),\Phi_F(t)$ be defined as in \eqref{eq:defphiss}, choose $|\eta_0|\le \omega_1$ and assume that for a small $0<\xi\le \epsilon/10$ it holds
\begin{equation}
\label{eq:inasspart1gfgf}
\Phi_1(0)+\Phi_F(0)+\Phi_2(0)\lesssim n^\xi,
\end{equation}
with very high probability. Then
\begin{equation}
\label{eq:goalgfgf}
 \Phi_1(t)+\Phi_F(t)+\Phi_2(t)\lesssim n^{2\xi},
 \end{equation}\
with very high probability uniformly in $t\le T^*(z_0,\eta_0)$
 and spectral parameters 
satisfying $n\eta_t\rho_t\ge n^\epsilon$.
\end{proposition}

We are now ready to conclude the proof of Theorem~\ref{theorem_F}, and then we present the proof of Proposition~\ref{pro:mainprogfgf}.

\begin{proof}[Proof of Theorem~\ref{theorem_F} for Gaussian component]

Using the global law \eqref{eq:global} as an input, the proof of this theorem for matrices with an order one Gaussian component is very similar (in fact simpler) to the proof of Theorem~\ref{thm:2G}. Here we also used that $n^{-1}\le \rho_0^{5/2}/(\sqrt{n}\eta_0^{3/2})$ as a consequence of
\[
\frac{\rho_0^{5/2}}{\sqrt{n}\eta_0^{3/2}}= \frac{\sqrt{n\rho_0\eta_0}}{n} \left(\frac{\rho_0}{\eta_0}\right)^2
\]
and $n\rho_0\eta_0\gg 1$, $\rho_0/\eta_0\gtrsim 1$. The fact that the Gaussian component can be removed is proven in Section~\ref{sec:proof}.
\end{proof}

\begin{proof}[Proof of Proposition~\ref{pro:mainprogfgf}]

Some estimates in this proof are similar to the ones presented in the proof of Proposition~\ref{pro:mainpro}. We thus only focus on the main differences. For notation simplicity, from now on we assume that $\eta_t>0$.

We now state some bounds on the deterministic approximations appearing in \eqref{eq:defphiss} (the proof is postponed to Appendix~\ref{sec:det}):
\begin{lemma}
\label{lem:Mbounds}
Let $H^z$ be defined as in \eqref{def_G}, let $G^z=(H^z-\ii\eta)^{-1}$, and let $M^z$, $M_{12}^F$ be defined as in  \eqref{Mmatrix} and \eqref{eq:defM12} with $z=z_1=z_2=z$, $\eta_1=\eta_2=\eta$, and $A=F$, respectively. Then there exists $C>0$ such that 
\begin{equation}
\label{eq:addbneedef}
\big|\langle M_{12}^FF^*\rangle\big|\le C\frac{\rho^3}{\eta}, \qquad\quad \big|\langle M_{12}^F\rangle\big|\le C \frac{\rho^2}{\eta}, \qquad\quad 0\le \langle M_{12}^{E_+}\rangle\le \frac{\Im m}{\eta}. \nc
\end{equation}
\end{lemma}
\noindent
Note that
 the last bound in~\eqref{eq:addbneedef} needs to be sharp; no multiplicative constant $C$ could be
afforded.

We will consider the time evolution of the quantities $\Phi_1(t),\Phi_2(t), \Phi_F(t)$ for times $t\le \widetilde{\tau}$, with $\widetilde{\tau}$ the stopping time defined by
\begin{equation}
\label{eq:stoptimwnew}
\widetilde{\tau}:=\inf\left\{t\ge 0\, : \, 
 \max\{ \Phi_1(t), \Phi_2(t), \Phi_F(t)\} =n^{2\xi} \right\}\wedge T^\epsilon,
\end{equation}
for some $0<\xi\le \epsilon/10$, and $T^\epsilon:=\min\{t>0: n\ell_t =  n^\epsilon\}$ is defined 
analogously to \eqref{Teps}. Note that $\widetilde{\tau}>0$ with very high probability as a consequence of our assumption \eqref{eq:inasspart1gfgf}.

We start considering $\Phi_1(t)$. By It\^{o}'s formula we immediately see that  (see Appendix~\ref{sec:rand} for its derivation)
\begin{equation}
\label{eq:k=1}
\dd \langle (G_t-M_t)F\rangle=\frac{1}{\sqrt{n}}\sum_{a,b=1}^{2n}\partial_{ab}\langle G_tF\rangle \dd \mathfrak{B}_{ab,t}+\frac{1}{2}\langle (G_t-M_t)F\rangle+\frac{1}{2}\langle G_t-M_t\rangle\langle G_t^2F\rangle,
\end{equation}
where we used \eqref{eq:defF} and that $\langle G_t E_-\rangle=0$ by the third equality in \eqref{eq:identities}.
 The second term in the rhs. of \eqref{eq:k=1} can be incorporated into the lhs. by considering the evolution of $e^{-t/2}\langle (G_t-M_t)F\rangle$ and thus neglected (recall that $t\lesssim 1$).

The quadratic variation of the stochastic term in \eqref{eq:k=1} is equal to
\[
\frac{1}{n^2\eta_t^2}\langle \Im G_t F\Im G_t F^*\rangle\, \dd t.
\]
Integrating~\eqref{eq:k=1} in time, for $t\le \widetilde{\tau}$,  and
using the BDG inequality, we thus estimate the stochastic term by (recall $\ell_t=\eta_t\rho_t$)
\begin{equation}
\frac{n\eta_t}{\rho_t}\left(\int_0^t \frac{1}{n^2\eta_s^2}\langle \Im G_s F\Im G_s F^*\rangle\,\dd s\right)^{1/2}\lesssim \frac{n\eta_t}{\rho_t}\left(\int_0^t \left(\frac{\rho_s^3}{n^2\eta_s^3}+\frac{n^{2\xi}\rho_s^{5/2}}{n^{5/2}\eta_s^{7/2}}\right)\, \dd s\right)^{1/2}\lesssim 1
\end{equation}
where in the first inequality we used Lemma~\ref{lem:Mbounds}, together with the definition of the stopping time \eqref{eq:stoptimwnew} to control $\Phi_2(s)\le n^{2\xi}$, while in the
second step we used~\eqref{eq:relineta} as before,
 and that $n^{10\xi}\le n\ell_s$, for any $0\le s\le \widetilde{\tau}$.

Similarly, for the last term in the rhs. of \eqref{eq:k=1} we estimate
\begin{equation}
\frac{n\eta_t}{\rho_t}\int_0^t \big|\langle G_s-M_s\rangle\big|\big|\langle G_s^2F\rangle\big|\, \dd s\lesssim \frac{n\eta_t}{\rho_t}\int_0^t \frac{n^\xi}{n\eta_s}\left(\frac{\rho_s^2}{\eta_s}+\frac{n^{2\xi}\rho_s}{n\eta_s^2}\right)\, \dd s\lesssim 1,
\end{equation}
with very high probability. \nc

Using \eqref{eq:inasspart1gfgf} to estimate $\Phi_1(0)\lesssim n^\xi$, and integrating
\eqref{eq:k=1}, we thus obtain
\begin{equation}
\label{eq:inphi1}
\Phi_1(t)\lesssim n^\xi, \qquad \mbox{ with very high probability, for any $t\le \widetilde{\tau}$}.
\end{equation}

Next, we consider $\Phi_2(t)$. The evolution of $\langle (G_{1,t}FG_{2,t}-M_{12,t}^F)F^*\rangle$ is given as in \eqref{eq:fulleqaasimp} choosing $A_1=F$ and $A_2=F^*$. In the following bounds we will repeatedly use Lemma~\ref{lem:Mbounds} to estimate the deterministic terms,
and  the definition of the stopping time $\widetilde{\tau}$ for the fluctuation 
even if not stated explicitly.  All the following estimates are performed for $t\le \widetilde{\tau}$, even if not stated explicitly. We estimate the quadratic variation from \eqref{eq:quadvarnew} by (we only write the bound for a representative term, the rest is analogous)
\begin{equation}
\label{4G}
\frac{\sqrt{n}\eta_t^{3/2}}{\rho_t^{5/2}}\left(\int_0^t \frac{1}{n^2\eta_s^2}\langle \Im G_s F G_s F^*\Im G_s F G_s^* F^*\rangle\,\dd s\right)^{1/2}  \lesssim\frac{\sqrt{n}\eta_t^{3/2}}{\rho_t^{5/2}}\left(\int_0^t \frac{1}{n\eta_s^2}\left(\frac{\rho_s^6}{\eta_s^2}+\frac{n^{2\xi} \rho_s^5}{n\eta_s^3}\right)\,\dd s\right)^{1/2} \lesssim 1.
\end{equation}
We point out that in the first inequality we used that $FG_sF^*=\ii F\Im G_s F^*$ (see last identity in \eqref{eq:identities}) to obtain
\begin{equation}
\label{exp}
\langle \Im G_s F G_s F^*\Im G_s F G_s^* F^*\rangle= \langle \Im G_s F \Im G_s F^*\Im G_s F \Im G_s^* F^*\rangle \le n  \langle \Im G_s F \Im G_s F^*\rangle^2
\end{equation}
to reduce the trace of the product of four resolvents in terms of traces of two resolvents. The inequality in \eqref{exp} follows analogously to \eqref{eq:redinnew} and we thus omit the proof.

 For the two terms in the last line of \eqref{eq:fulleqaasimp} we use the estimate (we only write the bound for one term)
\begin{equation}
\frac{\sqrt{n}\eta_t^{3/2}}{\rho_t^{5/2}}\int_0^t \big|\langle G_s-M_s\rangle\big|\big|\langle G_s^2 FG_sF^*\rangle\big|\, \dd s \lesssim\frac{\sqrt{n}\eta_t^{3/2}}{\rho_t^{5/2}}\int_0^t \frac{n^\xi}{n\eta_s^2}\left(\frac{\rho_s^3}{\eta_s}+\frac{n^{2\xi} \rho_s^{5/2}}{\sqrt{n}\eta_s^{3/2}}\right)\, \dd s \lesssim \frac{n^\xi}{\sqrt{n\ell_t}},
\end{equation}
where in the first inequality we used that 
\[
\big|\langle G_s^2 FG_sF^*\rangle\big|=\big| \langle G_s^2 F\Im G_sF^*\rangle\big|\le \frac{1}{\eta_t} \langle \Im G_s F\Im G_sF^*\rangle.
\]

Next, by (recall the definition of $M_s^F$ from \eqref{eq:singdetappr1}, which should not be confused with $M_s$)
\[
\langle (G_sFG_s-M_s^F)E_1\rangle=\frac{1}{2} \langle G_sFG_s-M_s^F\rangle+\frac{1}{2} \langle (G_sFG_s-M_s^F)E_-\rangle=\frac{1}{2} \langle G_sFG_s-M_s^F\rangle,
\]
where in the last equality we used \eqref{eq:impcanc}, we estimate
\begin{equation}
\frac{\sqrt{n}\eta_t^{3/2}}{\rho_t^{5/2}}\int_0^t \big|\langle (G_sFG_s-M_s^F)E_1\rangle\big| \big|\langle M_s^{F^*}E_2\rangle\big|\, \dd s\lesssim\frac{\sqrt{n}\eta_t^{3/2}}{\rho_t^{5/2}}\int_0^t \frac{n^{2\xi}\rho_s}{n\eta_s^2}\frac{\rho_s^2}{\eta_s}\, \dd s\lesssim \frac{n^{2\xi}}{\sqrt{n\ell_t}}
\end{equation}
and
\begin{equation}
\frac{\sqrt{n}\eta_t^{3/2}}{\rho_t^{5/2}}\int_0^t \big|\langle (G_sFG_s-M_s^F)E_1\rangle\big|\big|\langle (G_sF^*G_s-M_s^{F^*})E_2\rangle\big|\, \dd s \lesssim\frac{\sqrt{n}\eta_t^{3/2}}{\rho_t^{5/2}}\int_0^t \frac{n^{4\xi}\rho_s^2}{(n\eta_s^2)^2}\, \dd s \lesssim \frac{n^{4\xi}}{(n\ell_t)^{3/2}}.
\end{equation}

Using \eqref{eq:inasspart1gfgf} to estimate
\begin{equation}
\label{eq:bincond}
\frac{\sqrt{n}\eta_t^{3/2}\rho_0^{5/2}}{\sqrt{n}\eta_0^{3/2}\rho_t^{5/2}}\, \Phi_2(0)\lesssim n^\xi \sqrt{n\ell_t},
\end{equation}
where we used that $\rho_0\sim \rho_t$, we thus conclude 
\begin{equation}
\label{eq:inphi2}
\Phi_2(t)\lesssim n^\xi, \quad \mbox{with very high probability, for any $t\le \widetilde{\tau}$}.
\end{equation}

We are now only left with the derivation of the bound for $\Phi_F$. The estimates
 of the various terms in the evolution of $\langle G_t F G_t-M_t^F\rangle$ are very similar to those
  for $\langle (G_t F G_t-M_t^F) F^*\rangle$, so we only highlight the main differences.

Using the BDG inequality, we estimate of the stochastic term in~\eqref{eq:fulleqaasimp}:
\begin{equation}
\begin{split}
\label{eq:quadvargf}
\frac{n\eta_t^2}{\rho_t}\left(\int_0^t \frac{1}{n^2\eta_s^2}\langle \Im G_s F G_s \Im G_s G_s^* F^*\rangle\,\dd s\right)^{1/2}&\lesssim\frac{n\eta_t^2}{\rho_t}\left(\int_0^t \frac{1}{n^2\eta_s^4}\langle \Im G_sF\Im G_s F^*\rangle\,\dd s\right)^{1/2} \\
&\lesssim\frac{n\eta_t^2}{\rho_t}\left(\int_0^t \frac{1}{n^2\eta_s^4}\left(\frac{\rho_s^3}{\eta_s}+\frac{n^{2\xi}\rho_s^{5/2}}{\sqrt{n}\eta_s^{3/2}}\right)\,\dd s\right)^{1/2} \\
&\lesssim 1.
\end{split}
\end{equation}
Note that in \eqref{eq:quadvargf}, compared to \eqref{4G}--\eqref{exp}, we do not lose an $n$--factor to reduce a trace with four $G$'s to traces
 with two $G$'s. This is because in the current case the quadratic variation consists of a trace containing only two $F$'s that we need to preserve rather then four as in \eqref{4G}. 
 This makes the bound \eqref{eq:quadvargf} easier and stronger as we could estimated two of the four $G$'s just by $G_s\Im G_sG_s^*\le \Im G_s/\eta_s^2$ in the first inequality. Here we also tacitly used the cyclicity of the trace
 and the positivity of $\Im G$ to write $\langle\Im G F G\Im G G^* F^*\rangle=
 \langle (\Im G)^{1/2} F (G \Im G G^*) F^*(\Im G)^{1/2}\rangle$ before the operator inequality was applied. 

Next, we estimate
\begin{equation}
\begin{split}
\frac{n\eta_t^2}{\rho_t}\int_0^t \big|\langle G_s-M_s\rangle\big|\big|\langle G_s^3 F\rangle\big|\, \dd s &\lesssim\frac{n\eta_t^2}{\rho_t}\int_0^t \frac{1}{\eta_s^{3/2}} \big|\langle G_s-M_s\rangle\big| \langle \Im G_s F \Im G_s F^*\rangle^{1/2}\langle \Im G_s\rangle^{1/2}\, \dd s   \\
&\lesssim \frac{n\eta_t^2}{\rho_t}\int_0^t \frac{n^\xi \rho_s^{1/2}}{n\eta_s^{5/2}}\left(\frac{\rho_s^{3/2}}{\eta_s^{1/2}}+\frac{n^\xi\rho_s^{5/4}}{(n\eta_s^{3/2})^{1/4}}\right)\, \dd s \lesssim n^\xi.
\end{split}
\end{equation}
Using the first equality in \eqref{eq:identities} to turn 
 $\langle G_sFG_ sE_1 \rangle$ into $\frac{1}{2}\langle G_sFG_s \rangle$
in order to be able to use the bound on this quantity from the stopping time in \eqref{eq:stoptimwnew},
we estimate the term in the fifth line of \eqref{eq:fulleqaasimp} by
\begin{equation}
\label{eq:bast}
\frac{n\eta_t^2}{\rho_t}\int_0^t \big|\langle (G_sFG_s-M_s^F)E_1\rangle\big|\big|\langle (G_sE_1G_s-M_s^{E_1})E_2\rangle\big|\, \dd s\lesssim \frac{n\eta_t^2}{\rho_t}\int_0^t \frac{n^{4\xi}\rho_s}{(n\eta_s^2)^2}\, \dd s\lesssim \frac{n^{4\xi}}{n\ell_t},
\end{equation}
where we used 
\[
\big|\langle G_sAG_s-M_s^A)B\rangle\big|\lesssim \frac{n^\xi}{n\eta_s^2},
\]
for $A,B\in \{E_1,E_2\}$ to estimate the second trace in the lhs. of \eqref{eq:bast}. This local law follows by \eqref{local_2g}. Actually, in this case the proof would be much easier as we only consider the case $z_1=z_2$ and $\eta_1=\eta_2$. Furthermore, we estimate the terms in the third line of  \eqref{eq:fulleqaasimp} by
\begin{equation}
\label{eq:semilastb}
\frac{n\eta_t^2}{\rho_t}\int_0^t \big|\langle M_s^F E_1\rangle\big|\big|\langle (G_sE_1G_s-M_s^{E_1})E_2\rangle\big|\, \dd s\lesssim \frac{n\eta_t^2}{\rho_t}\int_0^t \frac{n^\xi\rho_s^2}{n\eta_s^3}\, \dd s\lesssim n^\xi.
\end{equation}

For the terms in the second line of  \eqref{eq:fulleqaasimp} we write
\[
2\langle (G_tFG_t-M_t^F)E_1\rangle\langle M_t^{E_+}E_2\rangle+2\langle (G_tFG_t-M_t^F)E_2\rangle\langle M_t^{E_+}E_1\rangle=\langle G_tFG_t-M_t^F\rangle\langle M_t^{E_+}\rangle,
\]
where we used the  definition \eqref{eq:defF} and that $\langle M_t^{E_+}E_1\rangle=\langle M_t^{E_+}E_2\rangle=\langle M_t^{E_+}\rangle/2$.
 Then, combining, \eqref{eq:quadvargf}--\eqref{eq:semilastb}, we are left with
\begin{equation}
\label{eq:remainsde}
\dd \langle G_tFG_t-M_t^F\rangle=\langle M_t^{E_+}\rangle\langle G_tFG_t-M_t^F\rangle\dd t+\dd\widehat{e}_t,
\end{equation}
with $\dd \widehat{e}_t$ an error term such that
\begin{equation}\label{eq:st}
\sup_{0\le t\le T}\frac{n\eta_t^2}{\rho_t}\left|\int_0^t \dd \widehat{e}_s\right|\lesssim n^\xi,
\end{equation}
with very high probability. Using the last inequality of \eqref{eq:addbneedef}  and \eqref{eq:relineta}, \nc  the 
propagator of \eqref{eq:remainsde} is bounded by
\begin{equation}
\label{eq:propag}
\exp\left(\int_0^t\langle M_s^{E_+}\rangle\,\dd s\right)\le \exp\left(\int_0^t\frac{\Im m_s}{\eta_s}\,\dd s\right)\sim \frac{\eta_s}{\eta_t}.
\end{equation}

Finally, using a Gronwall inequality and \eqref{eq:st},
we conclude that
 \begin{equation}
 \label{eq:inphiF}
 \Phi_F(t)\lesssim n^\xi \quad \mbox{with very high probability, for any $t\le \widetilde{\tau}$},
 \end{equation}
 where we proceeded similarly to \eqref{eq:bincond} to estimate the initial condition. Finally, combining \eqref{eq:inphi1}, \eqref{eq:inphi2}, and \eqref{eq:inphiF}, 
 we have shown that $\widetilde{\tau}=T^\epsilon$ with very high probability,
 and so that \eqref{eq:goalgfgf} holds. 
\end{proof}

\section{Proof of the GFT: Theorem~\ref{main_thm}}
\label{sec:proof_gft}

This section is divided into two parts: in Section~\ref{sec:strategy_gft} we present the main technical inputs to prove Theorem~\ref{main_thm} (Proposition~\ref{lemma_L} and Theorem~\ref{GFT} below), informally sketch the main steps of their proofs, and conclude the proof of Theorem~\ref{main_thm}. Then, in Section~\ref{sec:actualproof} we present the proof of Proposition~\ref{lemma_L}; the proof of Theorem~\ref{GFT} is postponed to Section~\ref{sec:L_0}.

To simplify the presentation of the proof, we may assume, additionally to Assumption \ref{ass:mainass}, that there exist $\alpha,\beta>0$ such that the probability density of $\chi$, denoted by $\varphi$, satisfies 
\begin{align}\label{assumption_b}
	\varphi \in L^{1+\alpha}(\mathbb{C}), \qquad \|\varphi\|_{1+\alpha} \leq n^{\beta}.
\end{align}
	This condition is used  
	to control the rare event 
	that there is a tiny singular value of $X-z$ in a simple way (see~\eqref{density_bound} below). However, this restriction can easily be removed using the argument in~\cite[Section 6.1]{TV15} (see also \cite[Section 2.2]{Kopel15}). More precisely, if we consider a matrix $X$ which does not satisfy \eqref{assumption_b}, then we can instead study $X+n^{-\gamma} X^{\rm Gin}$, with $X^{\rm Gin}$ a Ginibre matrix independent of $X$ and $\gamma>0$ large, which does satisfy  \eqref{assumption_b}. The tiny Gaussian component $n^{-\gamma} X^{\rm Gin}$ can then be easily removed following the proof of \cite[Theorem 23]{TV15} (or its refinement \cite[Lemma 4]{Kopel15}) which combines a sampling idea with a standard (sufficiently high order) moment matching technique (see also \cite[Remark 2.2]{maxRe}).

\subsection{Proof strategy of Theorem \ref{main_thm}}
\label{sec:strategy_gft}

We will sketch only the proof for $k=1$, as the proof for a general $k \geq 1$  is analogous. 
By Girko's formula, for any test function $f\in C^{2}_c(\C)$ and any $T>0$, we have
\begin{align}\label{girko0}
	\LL_f:=\sum_{i=1}^n f(\sigma_i)=&-\frac{1}{4 \pi}  \int_{\C} \Delta_z f(z) \int_0^{T}  \Im \Tr G^z(\ii \eta) \dd \eta \dd^2 z +\frac{1}{4 \pi} \int_{\C} \Delta_z f(z) \log |\det (H^{z}-\ii T)| \dd^2 z.
\end{align}
We now summarize the main steps to prove Theorem \ref{main_thm} (for $k=1$): 

\medskip
 
\begin{enumerate}	
	\item[\bf Step 1:] We pick a very large $T$, \eg $T:=n^{100}$ to ensure the last term in (\ref{girko0}) is very small with very high probability, as stated in (\ref{girko_formula}) below.

	\item[\bf Step 2:]  Under the density condition in (\ref{assumption_b}), a very tiny $\eta$-integral over $[0,\eta_c)$, with $\eta_c:=n^{-L}$, for a large $L>100$, can be neglected in the first absolute moment, \ie
	$L^1(\dd \P)$-sense; see \eqref{tiny_eta} below.

	\item[\bf Step 3:] In order to use the local law in Theorem \ref{local_thm} effectively, we subtract a deterministic term $M^{z}$ from $G^z$ on the right side of (\ref{girko0}) using that for $|z|>1$ this term vanishes, see \eqref{M_zero1} below.

	\item[\bf Step 4:] We truncate the smallest eigenvalue at the level $E_0 := n^{-3/4-\epsilon}$ for a fixed small $\epsilon>0$. Using the precise tail bound of the smallest eigenvalue in Proposition \ref{prop1}, the part with $\one_{\lambda_1^z \leq E_0}$ is negligible in expectation~(see \eqref{small_eigen}). Thus we obtain, in the first absolute moment sense, that
	\begin{align}\label{eq0}
		\LL_f \approx  -\frac{1}{4\pi} \int \Delta_z f(z) \Big(\int_{\eta_c}^{T}  \Im \Tr \big( G^z(\ii \eta)-M^z(\ii \eta) \big) \dd \eta \Big)  \one_{\lambda^z_1 > E_0} \dd^2 z.
	\end{align}

	\item[\bf Step 5:]  We regularize the truncating function $\one_{\lambda_1^z > E_0}$ in terms of  resolvents using Lemma \ref{lemma_approx} below:
	\begin{align}\label{eq2}
	q_z:=q\left( \int_{-E_0 }^{E_0}  \Im \Tr G^z(y+\ii \eta_0) \dd y \right), \qquad E_0= n^{-3/4-\epsilon}, \quad \eta_0:=n^{-3\zeta}E_0, 
    \end{align}
for some small $\zeta>0$, where 
 $q: \R_+ \rightarrow [0,1]$ is a smooth and non-increasing cut-off function with 
	\begin{equation}\label{q_function}
		q(y)=1, \quad \mbox{if} \quad 0 \leq y \leq 1/9; \qquad q(y)=0, \quad \mbox{if} \quad y \geq 2/9.
	\end{equation}
Thus,
using Proposition~\ref{prop1}, we have $\E|\LL_f-\wh\LL_f|=o(1)$, where 
\begin{align}\label{L_big0}
	\wh \LL_{f} :=-\frac{1}{4\pi}  \int \Delta_z f(z) \Big(\int_{\eta_c}^{T}  \Im \Tr \big( G^z(\ii \eta)-M^z(\ii \eta) \big) \dd \eta \Big)  q_z \dd^2 z,
\end{align} 
 is obtained from (\ref{eq0}) with $\one_{\lambda_1^z > E_0}$ replaced by $q_z$. 
  Hence a simple Taylor expansion of $\F$ yields
 \begin{equation}\label{appr1}
 \big| \E[\F(\LL_f)]-\E[\F(\wh \LL_f)]\big| \leq \|\F'\|_{\infty} \E\big| \LL_f-\wh\LL_f\big|=o(1).
 \end{equation}
The formal statement is presented in Proposition \ref{lemma_L} below even for a general $k\geq 1$.
	
		\item[\bf Step 6:] Finally it suffices to study $\E[\F(\wh \LL_f)]$. Using a GFT argument along a dynamical interpolating matrix flow we show that
\begin{equation}\label{appr2}
\Big| \E[\F(\wh \LL_f)]-\E^{\mathrm{Gin}}[\F(\wh \LL_f)]\Big| =o(1),
 \end{equation}
see the formal statement in Theorem \ref{GFT} for a general $k\geq 1$. The detailed proof of this step, which fundamentally relies on the independence estimates in Proposition~\ref{prop_zz} and the improved local laws in Theorem \ref{theorem_F}, will be presented in Section~\ref{sec:L_0}.

\end{enumerate}

We point out that the different approximation steps hold in different senses.  
The large $T$ cutoff (Step 1) holds with very high probability,  the small $\eta$ cutoff (Step 2) and the
small $\lambda_1^z$ cutoff (Step 4) work in first moment sense, while  the regularisation 
of the cutoff (Step 5) and the actual comparison with the Ginibre ensemble (Step 6)
are only  done by comparing expectations.

\bigskip

The proof of Theorem~\ref{main_thm} then follows directly from the following two results. The proof of the first result, Proposition \ref{lemma_L} below, will be presented in the next subsection~(as summarized in Step 1--5 for $k=1$). 
\begin{proposition}\label{lemma_L}
	Under the same assumption of Theorem~\ref{main_thm}, we consider linear statistics $\LL_{f_p}:=\sum^n_{i=1} f_p(\sigma_i), ~p \in [k]$, for some test functions $f_p \in C_c^2(\C)$ satisfying (\ref{f_norm}) and (\ref{f_cond})~(or alternatively (\ref{f_cond_2})).  Then there exists a sufficiently large $L>100$ (depending on $\alpha,\beta$ in (\ref{assumption_b})) such that
	\begin{align}\label{F_L_0}
		\Big|\E[\F(\LL_{f_1},\cdots, \LL_{f_k})]-\E[\F(\wh \LL_{f_1}, \cdots, \wh \LL_{f_k})]\Big| =O(n^{-\epsilon}),
	\end{align}
	where for any $p \in [k]$ we defined
		\begin{align}\label{L_0}
	\wh \LL_{f_p} := -\frac{1}{4\pi}
	\int \Delta_z f_p(z) \Big(\int_{\eta_c}^{T}  \Im \Tr \big( G^z(\ii \eta)-M^z(\ii \eta) \big) \dd \eta \Big)  q_z \dd^2 z,
	\end{align}
with $\eta_c=n^{-L}$, $T=n^{100}$, and $q_z$ given in~\eqref{eq2} with the smooth cut-off $q$ from~\eqref{q_function}.
\end{proposition}

Once we translated $\LL_f$ to $\wh \LL_f$, we use the following Green function comparison theorem, whose proof will be presented in Section \ref{sec:L_0}~(as sketched in Step 6 for $k=1$). 
\begin{theorem}\label{GFT}
	Under the same conditions of Theorem~\ref{main_thm}, we have
	$$\Big| \E[\F(\wh \LL_{f_1}, \cdots, \wh \LL_{f_k})] -\E^{\mathrm{Gin}}[\F(\wh \LL_{f_1}, \cdots, \wh \LL_{f_k})]\Big|=O(n^{-1/4+C\epsilon}),
	$$
	for some constant $C>0$ independent of $\epsilon>0$.
\end{theorem}
	Finally, combining Proposition \ref{lemma_L} with Theorem \ref{GFT}, we conclude the proof
Theorem~\ref{main_thm}.

\subsection{Proof of Proposition \ref{lemma_L}}
\label{sec:actualproof} We follow Steps 1--5 in Section \ref{sec:strategy_gft}.
Recall the Girko's formula 
\begin{align}\label{girko}
	\LL_f=-\frac{1}{4 \pi}  \int_{\C} \Delta_z f(z) \int_0^{T}  \Im \Tr G^z(\ii \eta) \dd \eta \dd^2 z +\frac{1}{4 \pi} \int_{\C} \Delta_z f(z) \log |\det (H^{z}-\ii T)| \dd^2 z,
\end{align}
where we dropped the subscript index $p\in [k]$ of $f_p$ and $\LL_{f_p}$ for notational simplicity. 

 \medskip

{\bf Step 1:} We start with showing that, choosing a large $T=n^{100}$, the last term in (\ref{girko}) is very small 
with very high probability. This follows directly from 
	\begin{align}\label{log_T}
		\log |\det (H^{z}-\ii T)|=&2n \log T +\sum_{j} \log \left( 1+\Big(\frac{\lambda_j^{z}}{T}\Big)^2\right) =2n \log T+O_\prec\left( \frac{n^2}{T^2}\right),
	\end{align}
	together with the $L^1$-bound
\begin{equation}\label{L1}
   \int_\C |\Delta_z f(z)| \dd^2 z \lesssim n^{1/2+2\nu+\tau} \le n^{1/2+\epsilon/20},
\end{equation}
which directly follows from the
 $L^\infty$-norm bound of $\Delta f$ in (\ref{f_norm}) and the area of its support~\eqref{f_cond}.  
  We thus have
	\begin{align}\label{girko_formula}
		\LL_{f}=&-\frac{1}{4 \pi}  \int_{\C} \Delta_z f(z) \int_0^{T}  \Im \Tr G^z(\ii \eta) \dd \eta \dd^2 z +O_\prec( n^{-100}), \qquad T=n^{100}.
	\end{align}
	
	{\bf Step 2:} We next show that the $\eta$-integral over $[0, \eta_c]$, with $\eta_c=n^{-L}$ for some very large $L > 100$, is negligible in the first absolute moment sense. By a direct computation we have
	\begin{align}\label{eta_c_sp}
		\int_0^{\eta_c} \Im \Tr \gz(\ii \eta) \dd \eta 
		=\frac{1}{2} \Big( \sum_{|\lambda^z_i| \leq n^{-L}} +\sum_{n^{-L} < |\lambda^z_i|\leq E_0 }+\sum_{|\lambda^z_i|> E_0  } \Big) \log \Big( 1+\frac{n^{-2L}}{(\lambda^z_i)^2}\Big),
	\end{align}
where $E_0$ is slightly below the typical eigenvalue spacing near zero, \ie
	$$E_0=n^{-3/4-\epsilon}, \qquad \qquad  \mbox{for a fixed small}~\epsilon>0.$$
Note that the last sum in (\ref{eta_c_sp}) can be trivially bounded by $n^{1-2L}/E_0^2\leq n^{-100}$. 	To estimate the first 
sum in (\ref{eta_c_sp}), we recall \cite[Proposition 5.7]{AEK18}, \ie that under the density condition~(\ref{assumption_b}), there exists $C_\alpha>0$ such that
	\begin{align}\label{density_bound}
		\mathbb{P}\left(  |\lambda^z_1| \leq \frac{u}{n} \right)\le C_\alpha u^{\frac{2\alpha}{1+\alpha}} n^{\beta+1},\qquad z \in \C, u>0,
	\end{align}
	with $\alpha, \beta$ given in (\ref{assumption_b}). 
	Then, following \cite[Eq. (5.34)-(5.35)]{AEK18}, the first sum in (\ref{eta_c_sp}) can be bounded by
	$$\E\Big[\sum_{|\lambda^z_i| \leq n^{-L}} \log \Big( 1+\frac{n^{-2L}}{(\lambda^z_i)^2}\Big) \Big]  \lesssim n \E \big[ |\log \lambda^z_1| \one_{\lambda^z_1 \lesssim n^{-L}} \big] \lesssim n^{-100},$$
	with $L>100$ large enough depending on $\alpha,\beta$. The second sum in (\ref{eta_c_sp}) can be bounded by
	\begin{align}\label{step_two}
		\E\Big[\sum_{n^{-L} < |\lambda^z_i|\leq E_0 } \log \Big( 1+\frac{n^{-2L}}{(\lambda^z_i)^2}\Big) \Big] \lesssim n^{\xi} \P\big( \lambda^z_1 \leq E_0\big),
	\end{align}
	for any small $\xi>0$, using the rigidity estimate of eigenvalues in (\ref{rigidity3})~(\ie $\#\{|\lambda^z_i| \leq E_0\} \leq n^{\xi}$).	
	Using  the support condition (\ref{f_cond}) or (\ref{f_cond_2}) and
Proposition~\ref{prop1}, we have
	\begin{align}\label{tail_precise}
		\P\big( \lambda^z_1 \leq E_0\big) \lesssim  n^{3/2} E_0^2 \ee^{-n \delta^2/2} \lesssim \begin{cases}
			n^{-1/2-2\epsilon}, \quad  \mbox{~if~$z$~satisfies~}(\ref{f_cond}), \\
			n^{-1/4-2\epsilon},   \quad  \mbox{~if~$z$~satisfies~}(\ref{f_cond_2}).
		\end{cases}
	\end{align}
  Combining with the $L^1$-norm of $\Delta f$ in~\eqref{L1},
   we know that the very tiny $\eta$-integral is negligible, \ie  there exists a large $L>100$ such that
	\begin{align}\label{tiny_eta}
		\E\Big|\int_{\C} \Delta_z f(z) \int_0^{\eta_c}  \Im \Tr G^z(\ii \eta) \dd \eta \dd^2 z\Big|\lesssim n^{-\epsilon},  \qquad \eta_c=n^{-L}.
	\end{align}

	{\bf Step 3:} By Steps 1 and 2  it suffices to focus on the integral over $\eta\in [\eta_c, T]$
	in~\eqref{girko}.  In order to use the local law in (\ref{average}), we 
	 subtract the deterministic term $M^{z}(\ii \eta)$ from $G^{z}(\ii \eta)$ on the right side 
	 of~\eqref{girko_formula}. By a direct computation using (\ref{Mmatrix}) and
	  (\ref{m_function}), we have (see also~\cite[Eq. (3.6)]{CES19})  
	$$\partial_\eta m^z(\ii \eta)=\frac{1-\beta}{\beta}, \qquad \partial_\eta u^z(\ii \eta)=\frac{2u^z m^z }{\beta}, \qquad \beta:=1-(m^z)^2-(u^z)^2|z|^2,$$ 
	and thus $m^z=m^z(\ii \eta)=\partial_\eta \big(\log u^z-(m^z)^2\big)/2$. For any $|z|>1$, we thus have
	$$ \int_{\eta_c}^{T} \<M^z(\ii\eta)\> \dd \eta =\int_{\eta_c}^{T} m^z(\ii\eta) \dd \eta=\frac{\log u^z(\ii\eta)-\big(m^z(\ii\eta)\big)^2}{2}\Big|^{\eta=T}_{\eta=\eta_c}=\log |z|^2+O(n^{-100}),$$
	where we used \eqref{rho} and that $m^z(0)=0, u^z(0)=\frac{1}{|z|^2}$ for $|z|>1$. Using that $\log |z|^2$ is harmonic away from zero, i.e. on the support of $f$, after integration by parts we have
	\begin{align}\label{M_zero}
		\int_{\C} \Delta_z f(z) \int_{\eta_c}^{T}  \Tr M^z(\ii \eta) \dd \eta \dd^2 z=O(n^{-98}).
	\end{align}
Therefore, we conclude that
	\begin{align}\label{M_zero1}
	\E\left|\LL_f+\frac{1}{4 \pi}  \int_{\C} \Delta_z f(z)  \Big(\int_{\eta_c}^{T}  \Im \Tr \big( G^z(\ii \eta)-M^z(\ii \eta) \big)\dd \eta\Big) \dd^2 z\right|=O(n^{-\epsilon}).
\end{align}

{\bf Step 4:} We further truncate the smallest eigenvalue at the level $E_0=n^{-3/4-\epsilon}$ as in (\ref{eta_c_sp}), \ie
	\begin{align}\label{girko_split}
	\LL_{\leq E_0}+\LL_{>E_0}:=&-\frac{1}{4 \pi}  \int_{\C} \Delta_z f(z)  \Big(\int_{\eta_c}^{T}  \Im \Tr \big( G^z(\ii \eta)-M^z(\ii \eta) \big)\dd \eta\Big)  \Big( \one_{\lambda^z_1 \leq  E_0}+\one_{\lambda^z_1 >  E_0}\Big)\dd^2 z.
	\end{align}
	Using the $L^1$-norm of $\Delta_z f$ in~\eqref{L1},
	the local law in (\ref{average}) and the tail bounds in (\ref{tail_precise}), we have the bound
	\begin{align}\label{small_eigen}
		&\E \big|\LL_{\leq E_0}\big| \prec (\log n)  \int  \big|\Delta_z f(z)\big| \P \big( \lambda^{z}_1 \leq E_0 \big) \dd^2 z \lesssim n^{-\epsilon}.
	\end{align}

	{\bf Step 5:} For $\LL_{>E_0}$ in (\ref{girko_split}), we regularize the truncating function $\one_{\lambda_1^z > E_0}$ by showing that
	\begin{align}\label{L_big}
		\E\big|\LL_{>E_0} -\wh\LL_f\big|=O(n^{-\epsilon}), 
	\end{align}
where  $\wh\LL_f$ is given in \eqref{L_big0}.
The proof of (\ref{L_big}) relies on the following lemma which links the eigenvalue truncating function $\one_{\lambda^z_1 > E_0}$ with $q_z$ in terms of the resolvent. We also remark the lemma below is in the same spirit of \cite[Lemma 6.1-6.2]{EYY12} designed originally to prove the edge universality for Wigner matrices; we sketch the proof in Appendix~\ref{app:lemma} for the convenience of the reader.	
	\begin{lemma}\label{lemma_approx}
		Let $0\leq |z|-1 \leq c$ for a small constant $c>0$. Fix any small $\epsilon'>0$ and $0<\zeta<\epsilon'/10$. Choose any spectral parameters $E\gg l\gg l' \gg \eta \gg n^{-1}$ such that
		\begin{align}\label{eta}
		n^{-1+\epsilon'}   \leq n^{6\zeta} \eta \leq   n^{3\zeta} l' \leq n^{\zeta} l \leq  E  \leq C n^{-3/4}.
		\end{align}
		Then we have, with very high probability,
		\begin{align}\label{approx}
			\Big|\#\{|\lambda^z_i| \leq E\}-\int_{-E}^{E}  \Im \Tr G^z(y+\ii  \eta) \dd y\Big| \lesssim \#\{ |\lambda^z_i| \in [E-l',E+l']\} +O(n^{-\zeta}),
		\end{align}
		and thus with very high probability,
		\begin{align}\label{approx_1}
			\int_{|y|\leq E-l}  \Im \Tr G^z (y+\ii  \eta)\dd y -O(n^{-\zeta})&\leq 	\#\{|\lambda^z_i| \leq E\}\nonumber\\
			&  \leq \int_{|y|\leq E+l}  \Im \Tr G^z(y+\ii  \eta) \dd y +O(n^{-\zeta}).
		\end{align}
	\end{lemma}

Recall the definitions of $q$ and $q_z$ in (\ref{q_function}) and (\ref{eq2}). Using that $\one_{\lambda^z_1> E_0}=q(\#\{|\lambda^z_i| \leq E_0\})$ and (\ref{approx}) for $E=E_0$, we have, with very high probability,
\begin{align}
	\big|\one_{\lambda^z_1> E_0}-q_z\big| \lesssim \one_{\#\big\{|\lambda^z_i| \in \big[(1-n^{-3\zeta})E_0,(1+n^{-3\zeta})E_0\big]\big\} \geq 1} \lesssim \one_{\lambda^z_1\leq 3E_0/2}, 
\end{align}
for some small $\zeta>0$. Using (\ref{tail_precise}) with $3E_0/2$ instead of $E_0$
 and the $L^1$-norm of $\Delta f$ in~\eqref{L1},   we have
\begin{align}\label{big_eigen}
	\E\big|\LL_{>E_0} -\wh\LL_f\big| \lesssim  (\log n) \int_\C
	 \big| \Delta_z f(z)\big| \P\big(\lambda^z_1\leq 3E_0/2\big) \dd^2 z  =O(n^{-\epsilon}).
\end{align}
Hence putting the above estimates in (\ref{M_zero1}), (\ref{small_eigen}) and (\ref{big_eigen}) together, we obtain
\begin{align}
	\E\big|\LL_f-\wh \LL_f\big|=O(n^{-\epsilon}).
\end{align}
This holds true for any test function $f=f_p\in C_c^2(\C)~(p \in [k])$ satisfying (\ref{f_norm}) and (\ref{f_cond})~(or alternatively (\ref{f_cond_2})). Then, since $\F(w_1,\cdots,w_k)$ is a smooth function satisfying (\ref{F_cond}), we obtain
\begin{align}\label{taylor1}
	\Big|\E[\F(\LL_{f_1},\cdots, \LL_{f_k})]-\E[\F(\wh \LL_{f_1}, \cdots, \wh \LL_{f_k})]\Big| \lesssim \max_{|\mathbf{\alpha}|=1}\{\|\partial^{\mathbf{\alpha}} \F\|_\infty\} \sum_{p=1}^k \E\big|\LL_{f_p}-\wh \LL_{f_p}\big|=O(n^{-\epsilon}).
\end{align}
This concludes the proof of Proposition \ref{lemma_L}. \qed

\section{Proof of Theorem \ref{GFT}}\label{sec:L_0}

In this section, we prove Theorem \ref{GFT} via a continuous interpolating flow. 
Given the initial ensemble $H^{z}$ in (\ref{def_G}) with $X$ satisfying Assumption \ref{ass:mainass}, we consider the following Ornstein-Uhlenbeck matrix flow
\begin{align}\label{flow}
	\dd H^z_t=-\frac{1}{2} (H^z_t+Z)\dd t+\frac{1}{\sqrt{n}} \dd \mathfrak{B}_t, \quad Z:=\begin{pmatrix}
		0  &  zI  \\
		\overline{z}I   & 0
	\end{pmatrix}, \quad 
	\mathfrak{B}_t:=\begin{pmatrix}
		0  &  B_t  \\
		B^*_t   & 0
	\end{pmatrix},
\end{align}
with initial condition $H^z_{t=0}= H^z$, where $B_t$ is an $n \times n$ matrix with i.i.d. standard complex valued Brownian motion entries. Set $W_t=H_t^z+Z$, then $\dd H^z_t=-\frac{1}{2} W_t \dd t +\frac{1}{\sqrt{n}} \dd \mathfrak{B}_t$ and 
\begin{align}\label{W}
     W_t=\begin{pmatrix}
		0  &  X_t  \\
		X^*_t   & 0
	\end{pmatrix}, \quad \qquad X_t \stackrel{{\rm d}}{=} \ee^{-\frac{t}{2}} X+\sqrt{1-\ee^{-t}} X^{\mathrm{Gin}},
\end{align}
where $X_{\infty} \stackrel{{\rm d}}{=}  X^{\mathrm{Gin}}$ is a complex 
Ginibre matrix which is independent of $X$.
The matrix flow $H_t^z$ hence interpolates between the initial matrix $H^{z}$
and the same  matrix 
 with $X$ replaced with an independent complex Ginibre ensemble.
The Green function of $H_t^{z}$ is denoted by $G^{z}_t(\ii\eta):=(H_t^z-\ii\eta)^{-1}$.
Now we define  $\wh \LL_f(t)$ as
in (\ref{L_0}) but with the time dependent $G_t^{z}$. 
For notational brevity, we define, for any non-negative integer $b\in \N$,
\begin{align}\label{qq}
	\X^z=\X_t^z:=\int_{\eta_c}^{T} \Im \Tr \big( G_t^z(\ii \eta)-M^z(\ii \eta) \big) \dd \eta, \qquad q_{z}^{(b)}=q_{z,t}^{(b)}:=q^{(b)}\Big( \int_{-E_0}^{E_0} \Im \Tr G_t^z(y+\ii \eta_0) \dd y \Big),
\end{align}
 with $\eta_c$ and $T$ chosen as in Proposition \ref{lemma_L}, where $q^{(b)}$ is the $b$-th derivative of the smooth and uniformly bounded cut-off function $q$ in (\ref{q_function}) with $E_0$ and $\eta_0$ chosen as in (\ref{eq2}). We also use the convention that $q^{(0)}=q$, $q'=q^{(1)}$ and $q''=q^{(2)}$. 
Then $\wh \LL_f(t)$ from~\eqref{L_0} with $f=f_p$ is given by 
\begin{align}\label{LX}
	\wh \LL_f=\wh \LL_f(t):=-\frac{1}{4\pi} \int \Delta_z f(z) \X^z_t  q_{z,t} \dd^2 z.
\end{align}
Since the flow in (\ref{flow}) is stochastically H\"older continuous in time, the local law for the Green function in Theorem \ref{local_thm} also holds true for the time dependent $G^z_t(\ii \eta)$ simultaneously for all  $t\in [0,n^{100}]$\nc, \ie
\begin{align}\label{local_t}
	\sup_{t\in [0,n^{100}]} \nc \sup_{z\in \mathrm{supp}(f)}
	 \sup_{\eta_c\le \eta\le T} (n\eta) \big|\< \gz_t(\ii \eta)-M^{z}(\ii \eta)\>\big| \prec 1. 
\end{align}
The proof of~\eqref{local_t}
uses a standard grid argument  together with the H\"older regularity of $G^{z}_t(\ii \eta)$
in all three variables
 for any  $0 \leq t \leq n^{100}$\nc, $z\in \mathrm{supp}(f)$ and $\eta_c\le\eta\le T$;
see similar arguments at the beginning of the proof of Proposition~\ref{pro:mainpro}.
Thus the local law in~(\ref{local_t}) implies that $|\X_t^z| \prec \log n$ for any $t\in [0,n^{100}]$.
 In the following we often drop the $t$-dependence in the above notations, and all the estimates throughout this section hold true uniformly in  $t\in [0,n^{100}]$ \nc even if not mentioned explicitly.

We next compute how the time-dependent expectation $\E\big[\F \big(\wh \LL_f(t)\big)\big]$ evolves along the flow in (\ref{flow}).   Due to the $2\times 2$ block structure of the matrix $W=W_t$ in (\ref{W}), in the following we use lower case letters (\eg $a, b$) to denote the indices taking values in $[1,n]$ and upper 
case letters (\eg $A,B$) to denote the indices taking values in $[n+1,2n]$. 
We also use Fraktur letters, \eg $\mathfrak{u,v}$ to denote the indices ranging fully from $1$ to $2n$.

We apply  It\^{o}'s formula to $\F(\wh \LL_f)$ using that $\partial/\partial w_{aB}=\partial/\partial h_{aB}$ and $\overline{h_{aB}}={h_{Ba}}$ and perform the complex version of cumulant expansion formula on the expectation~(see \cite[Lemma 7.1]{HK17}). 
Since the variances of the matrix entries of $X_t$ are preserved along the flow in (\ref{W}), the second order cumulant terms are exactly cancelled and the following sum starts from the third order terms~($p+q+1=3$), \ie
   \begin{align}\label{cumulant_exp0}
	\frac{\dd}{\dd t}\E[\F(\wh \LL_f)]=&\sum_{p+q+1 \geq 3}^{10} \frac{1}{n^\frac{p+q+1}{2}} \sum_{a=1}^{n} \sum_{B=n+1}^{2n} \left(\frac{\kappa^{(p+1,q)}}{p!q!} \E\Big[ \frac{\partial^{p+q+1} \F(\wh \LL_f)}{\partial h^{p+1}_{aB}\partial  h^{q}_{Ba}}\Big]+\frac{\kappa^{(q,p+1)}}{p!q!} \E \Big[ \frac{\partial^{p+q+1} \F(\wh \LL_f)}{\partial h^{q}_{aB}\partial  h^{p+1}_{Ba}}\Big]\right)\nonumber\\
	&+\mathcal{E}_{11},
\end{align}
  where
 $\kappa^{(p,q)}=\kappa^{(p,q)}_t$ are the  $(p,q)$-cumulants of  $\chi_t$, where
 $\chi_t\stackrel{{\rm d}}{=}(\sqrt{n}X_t)_{ij}$ is the common distribution of
  the normalized time-dependent entries.
 In particular $|\kappa_t^{(p,q)}|\lesssim 1$  uniformly in $t\in \R^+$ by~\eqref{eq:hmb} and~\eqref{W}.
 Notice that we stopped the cumulant expansion at level $p+q+1=10$ with a truncation error denoted by
  $\mathcal{E}_{11}$. We claim that this error is bounded by 
 \begin{align}\label{tr_error}
 	\mathcal{E}_{11}=O(n^{-3/4+c\epsilon}),
 \end{align}
 for some small constant $c>0$. Such truncation argument is rather standard and has been used many times (see e.g. \cite{CES21,maxRe,SpecRadius,SX22}). Since controlling  $\mathcal{E}_{11}$ is much easier than estimating the cumulant expansion terms in (\ref{cumulant_exp0}) with $3 \leq p+q+1\leq 10$, so we will first focus on these terms and postpone the proof of (\ref{tr_error}) till the end of Section \ref{subsec:proof_lemma_key}.

  \nc

\medskip

 To study the terms on the right side of (\ref{cumulant_exp0}), we establish
  the following differentiation rules (these apply in the complex case). The proof is postponed till Section \ref{subsec:proof_some_lemma}.
 \begin{lemma}\label{lemma_deri_rule}
 	For any $\mathfrak{u,v} \in [1,2n]$, $a\in [1,n]$ and $B\in [n+1,2n]$, we have
 	\begin{align}\label{rule}
 		\frac{\partial \gz_\mathfrak{uv}}{\partial h_{aB}}=-\gz_{\mathfrak{u}a}\gz_{B\mathfrak{v}},\qquad \frac{\partial \gz_\mathfrak{uv}}{\partial h_{Ba}}=-\gz_{\mathfrak{u} B}\gz_{a \mathfrak{v}}.
 	\end{align}
 	Then for $\X^z$ and $q^{(b)}_z~(b \in \N)$ defined in (\ref{qq}), we have the following differentiation rules
 	\begin{align}\label{rule_q}
 		\frac{\partial\X^z}{\partial h_{Ba}}=- \gz_{aB}(\ii\eta_c)+O_\prec(n^{-100}), \quad\qquad \frac{\partial q^{(b)}_z}{\partial h_{Ba}}=&- q^{(b+1)}_z \Dim  (\gz_{aB}),
 	\end{align}
 	where we used the short-hand notations for  any function $g,h: \C\setminus \R \rightarrow \C$, 
 	\begin{align}\label{dim}
 		\wt \Im g (w):=&\frac{1}{2\ii } \big(g (w)-g (\bar w) \big), \qquad \mathfrak{D}_{E_0,\eta_0}h:= h(E_0+\ii \eta_0) -h(-E_0+\ii \eta_0),
 	\end{align}
  with fixed $E_0$ and $\eta_0$ chosen in (\ref{eq2}). The same rules above also hold when we interchange $a$ 
  and $B$. Moreover, both  operations
  $\wt \Im$ and  $\mathfrak{D}_{E_0,\eta_0}$ in (\ref{dim}) commute with the derivatives
   $\partial/\partial h_{Ba}$ or $\partial/\partial h_{aB}$. 
 \end{lemma}

By direct computations using Lemma \ref{lemma_deri_rule}, we have the following estimates on the partial derivatives of $\wh \LL_f$ given by (\ref{LX}), whose proof is presented in Section~\ref{subsec:proof_lemma_key} below.
\begin{lemma}\label{lemma_key}
	For any $a\in [1,n]$ and $B \in [n+1,2n]$, we have
	\begin{align}\label{key_1}
	\Big|\frac{\partial \wh \LL_f}{\partial h_{Ba}} \Big| \prec   \big( n^{\frac{3}{4}+c\epsilon}+n^{1+c\epsilon}\one_{a=B-n}\big) \int_{\mathrm{supp}(f)} \one_{\lambda_1^z \leq 3E_0/2}\dd^2z + n^{-99},
	\end{align}
for some constant $c>0$. In general for any $p,q\in \N$, we have, 
\begin{align}\label{key_pq}
	&\Big|\frac{\partial^{p+q+1} \wh \LL_f}{\partial h^{p+1}_{Ba} \partial h^{q}_{aB}} \Big| \prec  \Big(  n^{\frac{3}{4}+(-\frac{1}{4}+\epsilon)(p+q)+c\epsilon}  +n^{1+c\epsilon}\one_{a= B-n}\Big)\int_{\mathrm{supp}(f)}
	 \one_{\lambda_1^z \leq 3E_0/2}\dd^2z+ n^{-99},
\end{align}
and the same results holds true if we interchange $a$ with $B$.
\end{lemma}

Using Lemma \ref{lemma_key}, we are now ready to estimate the terms on the right side of (\ref{cumulant_exp0}) starting from the third order $p+q+1=3$.  As an example, we first consider one of the most critical third order terms when 
$\partial/\partial h_{Ba}$ acts on the function $\F$ for three times, \ie 
\begin{align}\label{triple}
	\left|\frac{1}{n^{3/2}} \sum_{a,B} \E\Big[ \F^{(3)}(\wh\LL_f) \Big(\frac{\partial \wh \LL_f}{\partial h_{Ba}}\Big)^3 \Big]\right| \lesssim & \frac{1}{n^{3/2}} \sum_{a \neq B-n} \E \Big|\frac{\partial \wh \LL_f}{\partial h_{Ba}}\Big|^3 +\frac{1}{n^{3/2}} \sum_{a} \E \Big|\frac{\partial \wh \LL_f}{\partial h_{a+n,a}}\Big|^3\nonumber\\
	\lesssim & \; n^{\frac{11}{4}+3c\epsilon} 
	 \E\Big[\Big(\int_{\mathrm{supp}(f)} \one_{\lambda_1^z \leq 3E_0/2}\dd^2z\Big)^3\Big]+n^{-50},
\end{align}
where we used that $\F$ has uniformly bounded derivatives and the estimate in (\ref{key_1})
  in the last step.  Then we split the triple $z$-integrals into three regimes
 $$
 \E\Big[\Big(\int_{\mathrm{supp}(f)} \one_{\lambda_1^z \leq 3E_0/2}\dd^2 z\Big)^3\Big] 
 = \Big( \iiint_{(1)}+\iiint_{(2)}+\iiint_{(3)}\Big)
  \P\Big( \lambda_1^{z_j} \leq 3E_0/2, j\in [3]\Big)\dd^2z_1\dd^2z_2\dd^2z_3,
 $$
 where the regimes are given as follows:
 \begin{itemize} 
 	\item[(1)] all $z_1,z_2,z_3 \in \mathrm{supp}(f)$ are sufficiently far from each other, \ie $\min\{|z_1-z_2|,|z_1-z_3|,|z_2-z_3|\} \geq n^{-1/2+\omega}$ for some $\omega>0$ to be chosen; 
 	\item[(2)] two of them are close 
  while the remaining one is separated from them, \eg $|z_1-z_2| \leq n^{-1/2+\omega}$ and $\min\{|z_1-z_3|,|z_2-z_3|\} \geq n^{-1/2+\omega}$ (and similarly for other two index combinations); \item[(3)] all $z_1,z_2,z_3 \in \mathrm{supp}(f)$ are close
   to each other, \ie $\max\{|z_1-z_2|,|z_1-z_3|,|z_2-z_3|\} \leq n^{-1/2+\omega}$. 
 \end{itemize}
  For any $z_j \in \mathrm{supp}(f)$ satisfying (\ref{f_cond}) with $\min_{j \neq j'}\{|z_j-z_{j'}|\} \geq n^{-1/2+\omega}$, choosing $\omega=2\epsilon/\gamma$  with the small constant $\gamma>0$ from Proposition \ref{prop_zz}, 
  the independence estimate in (\ref{lambdatail}) implies that
 \begin{align}\label{ginibre_tail_prod}
 	\P\Big( \lambda_1^{z_1} \leq 3E_0/2,\cdots,  \lambda_1^{z_k} \leq 3E_0/2\Big) \lesssim \prod_{j=1}^k \P\big( \lambda_1^{z_j} \leq 2 E_0 \big) +n^{-D},  \qquad E_0=n^{-3/4-\epsilon},
 \end{align}  
where we also used that $0<\tau<\epsilon/100$.
Combining this with (\ref{tail_precise}), we obtain that
\begin{align}\label{ginibre_tail}
	\P\Big( \lambda_1^{z_1} \leq 3E_0/2,\cdots,  \lambda_1^{z_k} \leq  3E_0/2\Big) \lesssim \big(n^{-1/2-2\epsilon}\big)^k +n^{-D}.
\end{align}
\nc
 Thus, using $|\mbox{supp}(f)|\lesssim n^{-1/2+\tau}$, the integrals over the above three regimes are bounded by
\begin{align}
	\E\Big[\Big(\int_{\mathrm{supp}(f)} \one_{\lambda_1^z \leq 3E_0/2}\dd^2z\Big)^3\Big] \lesssim &\big( n^{-1/2+\tau}\big)^3 \big( n^{-1/2-2\epsilon} \big)^3 + n^{-1+2\omega} \big( n^{-1/2+\tau}\big)^2\big( n^{-1/2-2\epsilon} \big)^2\nonumber\\
	&+(n^{-1+2\omega})^2 \big( n^{-1/2+\tau}\big) \big( n^{-1/2-2\epsilon} \big) =O\big(n^{-3+C\epsilon}\big), \label{n3}
\end{align}
for some constant $C>0$ depending on $\gamma$ from Proposition \ref{prop_zz}. We mention  
that if $f$ satisfies  the alternative support condition in (\ref{f_cond_2}),
then the tail bound~\eqref{ginibre_tail} becomes weaker $(n^{-1/4-2\epsilon})^k$, which is exactly
compensated by the smaller volume of the support, $|\mbox{supp}(f)|\lesssim n^{-3/4+\tau}$,
 hence we obtain the same upper bound  as in~\eqref{n3}. Plugging the above estimates into (\ref{triple}), we obtain
$$\left| \frac{1}{n^{3/2}} \sum_{a,B} \E\Big[ \F^{(3)}(\wh \LL_f) \Big(\frac{\partial \wh \LL_f}{\partial h_{Ba}}\Big)^3 \Big] \right|=O_\prec(n^{-1/4+C\epsilon}).
$$ 
All the other third order terms in (\ref{cumulant_exp0}) can be estimated similarly using Lemma \ref{lemma_key} and (\ref{ginibre_tail}), \eg
\begin{align}
	&\left|\frac{1}{n^{3/2}} \sum_{a,B} \E\Big[ \F''(\wh \LL_f) \Big(\frac{\partial \wh \LL_f}{\partial h_{aB}}\Big) \Big(\frac{\partial^2 \wh \LL_f}{\partial h^2_{Ba}}\Big) \Big] \right|=O_\prec(n^{-1/4+C\epsilon}), \nonumber\\ 
	&\left| \frac{1}{n^{3/2}} 
	\sum_{a,B} \E\Big[ \F'(\wh \LL_f)  \Big(\frac{\partial^3 \wh \LL_f}{\partial h^3_{Ba}}\Big) \Big]\right|=O_\prec(n^{-1/4+C\epsilon}).
\end{align}

Due to the moment condition in~(\ref{eq:hmb}), the higher order terms~($p+q+1\geq 4$) can be estimated similarly with much less efforts. More precisely, using Lemma \ref{lemma_key} and similar arguments as in (\ref{n3}), for any $3 \leq p+q+1\leq 10$, we have
\begin{align}\label{higher_term}
	  \left| \frac{1}{n^\frac{p+q+1}{2}} \sum_{a,B}   \E\Big[ \frac{\partial^{p+q+1} \F(\wh \LL_f)}{\partial h^{p+1}_{aB}\partial  h^{q}_{Ba}}\Big] \right| 
	  \lesssim & \frac{1}{n^\frac{p+q+1}{2}} \sum_{a,B}  \E \Bigg[  \sum_{l_1=1}^{p+1} \sum_{l_2=1}^{q} \Big|\F^{(l_1+l_2)}(\wh \LL_f)\Big| \prod_{\substack{(p_i, q_i)\in \N^2\\p_1+\cdots +p_{l_1}=p+1\\q_1+\cdots +q_{l_2}=q}} \Big|\frac{\partial^{p_i+q_i} \wh \LL_f}{\partial h^{p_i}_{aB} \partial h^{q_i}_{Ba}}\Big|\Bigg]  \nonumber\\
	   \prec& \frac{n^{2+C\epsilon}}{ n^{(\frac{3}{4}-c\epsilon) (p+q+1)}},
\end{align}
for some constant $C>0$ depending on $\gamma$ from Proposition \ref{prop_zz}. Note that the same upper bound also applies when we switch the index $a$ with $B$.

Therefore, we conclude from (\ref{cumulant_exp0}) that
\begin{align}\label{time_deri}
	\Big|\frac{\dd}{\dd t}\E\big[\F\big(\wh \LL_f(t)\big)\big]\Big|=O_{\prec}(n^{-1/4+C\epsilon}).
\end{align}
Integrating both sides of (\ref{time_deri}) over $t\in [0,t_0]$ 
 with a sufficiently large $t_0=O(\log n)$ to be fixed later \nc, we then obtain that
 \begin{align}\label{t1}
 	\Big|\E\big[\F\big(\wh \LL_f(0)\big)\big]-\E\big[\F\big(\wh \LL_f(t_0)\big)\big]\Big|=O_{\prec}(n^{-1/4+C\epsilon}).
 \end{align}
 Recall from (\ref{W}) that, for any fixed $t_0\geq 0$, $X_{t_0} \stackrel{{\rm d}}{=} \wt X_{t_0}:=\ee^{-\frac{t_0}{2}} X+\sqrt{1-\ee^{-t_0}} X^{\mathrm{Gin}}$, where $X^{\mathrm{Gin}}$ is the Ginibre ensemble independent of the initial matrix $X$. We then use $\wt H_{t_0}^{z}$ to denote the Hermitization of $\wt X_{t_0}-z$ as in (\ref{def_G}) and use $\wt G^{z}_{t_0}$ to denote its resolvent. With a slight abuse of notation, we  also write $\wt X_\infty=X^{\mathrm{Gin}}$ and use $\wt G^{z}_{\infty}$ to denote the corresponding Hermitized resolvent. Owing to the exponential relaxation of $\wt X_{t_0}$ to the Ginibre ensemble, we have $\| \wt X_{t_0} - \wt X_\infty\| \prec \ee^{-t_0/2}\big[ \|\wt X_\infty\|
 + \|X\|\big] \prec  \ee^{-t_0/2}$, using that $|\wt X_{ij}| \prec n^{-1/2}$ from the moment assumption in \eqref{eq:hmb}.
Thus, fixing $t_0:=10 L \log n$ with a large $L>100$ from $\eta_c=n^{-L}$ chosen in (\ref{tiny_eta}), we have 
\begin{align}\label{approxxxx}
	\|\wt G^{z}_{t_0}(\ii \eta)-\wt G^{z}_{\infty}(\ii \eta)\| \leq \|\wt G^{z}_{t_0}(\ii \eta)\| \|\wt G^{z}_{\infty}(\ii \eta)\| \|\wt X_{t_0}-\wt X_{\infty}\| 
	\prec n^{-2L}, \qquad \eta \geq \eta_{c}=n^{-L},
\end{align}
where we used that $\|\wt G^z(\ii \eta)\| \leq \eta^{-1}$. Using (\ref{approxxxx}), the $L^1$-norm bound of $\Delta f$ in (\ref{L1}), and that $X_{t_0} \stackrel{{\rm d}}{=} \wt X_{t_0}$, we have 
\begin{align}\label{t2}
	\Big|\E[\F\big(\wh \LL_f(t_0)\big)-\E^{\mathrm{Gin}}[\F\big(\wh \LL_f\big)]\Big|=O_{\prec}(n^{-10}).
\end{align}	\nc
Combining (\ref{t1}) with (\ref{t2}), we hence finish the proof of Theorem~\ref{GFT}.

 \subsection{Proof of Lemma \ref{lemma_key} and Eq. (\ref{tr_error})\nc}\label{subsec:proof_lemma_key}
We will first prove Lemma \ref{lemma_key}. The proof relies on the following useful lemmas, whose proofs are postponed to Section \ref{subsec:proof_some_lemma} below.

\bigskip
 Due to the $2 \times 2$ block structure of $H^z$ in (\ref{def_G}), we observe the following (almost) precise algebraic cancellations between the eigenvalues and eigenvectors. Their origin also lies in the chiral symmetry of $H^z$,
 similarly to Lemma~\ref{lem:chirid}.
\begin{lemma}[Chiral symmetry] \label{lemma_algebraic_cancel} 
	For any $a\in [1,n]$, $B\in [n+1,2n]$, we have
	\begin{align}\label{cancel}
		\Big|\big(\gz F \gz(\ii \eta_c) \big)_{aB}\Big|\one_{\lambda^z_1 \gtrsim E_0}=O_{\prec}(n^{-100}), \qquad \Big|\big(\gz F^* \gz(\ii \eta_c) \big)_{Ba}\Big|\one_{\lambda^z_1 \gtrsim E_0}=O_{\prec}(n^{-100}),
	\end{align}
	with $E_0=n^{-3/4-\epsilon}$ and $\eta_c=n^{-L}$ for a sufficiently large $L>100$.
	In particular, 
		\begin{align}\label{cancel_full}
		\Big|\big\langle \gz F\gz F^*(\ii \eta_c) \big\rangle \Big| \one_{\lambda^z_1 \gtrsim E_0}=O_{\prec}(n^{-100}).
		\end{align}
	Moreover, for any bounded deterministic vector $ \mathbf{x}\in \mathrm{Ker}(E_1) \cup \mathrm{Ker}(E_2)$, we have 
	\begin{align}\label{G_tiny}
	\Big|	\big\<\mathbf{x},\gz(\ii\eta_c)\mathbf{x}\big\>\Big| \one_{\lambda_1^z \gtrsim E_0}=O_\prec(n^{-100}).
	\end{align}
Here $\lambda_1^z \gtrsim E_0$ means that $\lambda_1^z \geq c E_0$ for some constant $c>0$.
\end{lemma}

Moreover, we state the following local law estimates on the Green function entries. 
\begin{lemma}[Local law estimates]\label{lemmaG}
Fix a small $\tau>0$. For any $0 \leq |z|^2-1 \lesssim n^{-1/2+\tau}$ 
and $w= E+\ii\eta \in \C^+$ with $|E|+\eta \lesssim n^{-3/4}$, the following hold
\begin{align} 
	&\big\<\gz (w) F^{(*)}\big\>=\big\<\gz(\ii \wt \eta) F^{(*)}\big\>+O_{\prec}\Big(\frac{n^{-1/4+\xi+\tau}}{n\eta}\Big), \qquad \wt \eta:=n^{-3/4+\xi}, \label{F_0estimate_1}\\
	& \big|\<\gz (w) F \gz (w) F^*\>\big|\prec \frac{n^{-1/2+2\xi+2\tau}}{n\eta^2}, \label{FF_0estimate_1}
\end{align}
 for a sufficiently small $\xi>0$. Moreover, for any $0 \leq |z|^2-1 \lesssim n^{-1/2+\tau}$ and any bounded deterministic vectors $\mathbf{x}, \mathbf{y}\in\C^{2n}$ and deterministic matrices $F,F^* \in \C^{2n \times 2n}$ from \eqref{eq:defF}, we have
\begin{align}
	&\Big|\<\mathbf{x}, \big(G^{z}(\ii\eta_c)-M^{z}(\ii\eta_c)\big) \mathbf{y}\>\Big| \one_{\lambda_1^{z} \gtrsim E_0} \prec \frac{1}{nE_0}, \label{E_0estimate}\\
	& \Big| \big\< \big(G^{z}(\ii\eta_c)-M^{z}(\ii\eta_c)\big) F^{(*)} \big\>\Big| \one_{\lambda_1^{z} \gtrsim E_0} \prec \frac{n^{-1/4+\xi+\tau}}{nE_0}, \label{F0_estimate}		
\end{align}
for a sufficiently small $\xi>0$, with $E_0=n^{-3/4-\epsilon}$ and $\eta_c=n^{-L}$, for a large $L>100$.
\end{lemma}	
These bounds are modifications of the local laws in Theorem \ref{local_thm} and Theorem \ref{theorem_F}. First, the local law in (\ref{local_F}) was stated only for $w=\ii \eta$ with $\eta$ bigger than the typical eigenvalue spacing $n^{-3/4}$. Now in (\ref{F_0estimate_1}) we have 
extended the estimate to a small neighbourhood of the imaginary axis below $n^{-3/4}$. Second, note that while  the error terms in the local laws stated in (\ref{entrywise}) and (\ref{local_F}) blow up with a very tiny $\eta_c$, we have obtained a much better estimate in (\ref{E_0estimate})-(\ref{F0_estimate}) using the event $\lambda^z_1 \gtrsim E_0$.

\medskip
In Section~\ref{subsec:proof_some_lemma} we will also prove
the following lemma 
on $q$-related functions. 
\begin{lemma}[Derivative estimates of $q_z$]\label{lemma_q}	
	Recall $q_z$ given in (\ref{eq2}) and $q^{(b)}_z~(b\geq 1)$ in (\ref{qq}). 
	Fix a small $\tau>0$. For any $0 \leq |z|^2-1 \lesssim n^{-1/2+\tau}$,  the following estimates hold
	\begin{align}\label{q_prime}
		\big|q_{z}\big| \lesssim \one_{\lambda_1^z\geq E_0/2}, \qquad \qquad	\big|q_{z}^{(b)}\big|	\lesssim \one_{E_0/2\leq \lambda_1^z \leq 3E_0/2}\qquad  (b \geq 1),
	\end{align}	
	with very high probability. Moreover we have
	\begin{align}\label{deri_qz}
		\big|\partial_z q_z\big| =& \left| q'_z \Dim\big( \Tr \gz F\big) \right|=O_\prec\big(n^{1/2+\epsilon+\tau \nc}\big) \one_{E_0/2\leq \lambda_1^z \leq 3E_0/2},
	\end{align}
	with $\wt \Im$ and $\mathfrak{D}_{E_0,\eta_0}$ defined in (\ref{dim}). The same result also holds for $\partial_{\bar z} q_z$ if we replace $F$ with $F^*$.	 Finally, 
	\begin{align}\label{lap_qz}
		\big|\Delta_z q_z\big|=O_\prec\big(n^{1+2\epsilon +2\tau \nc}\big) \one_{E_0/2\leq \lambda_1^z \leq 3E_0/2}.
	\end{align}
\end{lemma}

Armed with the above lemmas,  we are now ready to prove Lemma \ref{lemma_key}. 
\begin{proof}[Proof of Lemma \ref{lemma_key}]
We start with the simplest case in (\ref{key_1})~(\ie a special case of (\ref{key_pq}) with $p=q=0$). 
Recalling the definition of $\wh \LL_f$ in (\ref{LX}) and  using the differentiation rules in (\ref{rule_q}), we have
\begin{align}\label{L_0aB}
	\frac{\partial \wh \LL_f}{\partial h_{Ba}} = \frac{1}{4\pi} \int_\C \Delta_z f(z) G^{z}_{aB}(\ii \eta_c) 
	 q_z \dd^2 z +\frac{1}{4\pi}
	\int_\C \Delta_z f(z) \X^z q'_z \Dim (\gz_{aB})  \dd^2 z +O_\prec(n^{-99}). 
\end{align}
We start with the second term in (\ref{L_0aB}).	
From the definition of $\wt \Im$ and $\mathfrak{D}_{E_0,\eta_0}$, given in (\ref{dim}), using the local law in Theorem \ref{local_thm}, the properties of $M^{z}$ in (\ref{Mmatrix}) and (\ref{rho_E}), we have for any $\mathfrak{u,v} \in [1,2n]$
\begin{align}\label{G_E0}
	\big|\Dim (\gz_{\mathfrak{uv}})\big| \prec n^{-1/4+c\epsilon}, \qquad |\X^z|=\Big| \int_{\eta_c}^{T}  \Im \Tr \gz(\ii \eta) \dd \eta  \Big| \prec \log n.
\end{align}
Combining these with (\ref{q_prime}) and the $L^1$-norm bound of $\Delta f$ from~\eqref{L1}, 
the second term in (\ref{L_0aB}) is bounded by
\begin{align}\label{trivia_1}
	\Big| \int_\C \Delta_z f(z) \X^z  q'_z   \Dim (\gz_{aB})  \dd^2 z\Big| \prec n^{3/4+c\epsilon}\int_{\mathrm{supp}(f)} \one_{\lambda^z_1 \leq 3E_0/2} \dd^2 z.
\end{align}
We next look at the first term in (\ref{L_0aB}). Applying integration by parts
 with respect to $\partial_z$, using that
\begin{align}\label{rule_z}
	\partial_{z} \gz =G^zFG^z, \qquad \partial_{\bar z} \gz=G^zF^*G^z,
\end{align}
and $\Delta_z =4 \partial_{z} \partial_{\bar z}$, we obtain 
\begin{align}\label{int_by_part}
	\int_\C \Delta_z f(z)  G^{z}_{aB}(\ii \eta_c) q_z \dd^2 z
	=&4\int _\C\partial_{\bar z} f(z)  \big( \gz F \gz (\ii \eta_c)\big)_{aB} q_z \dd^2 z+4\int_\C \partial_{\bar z} f(z)   G^{z}_{aB}(\ii \eta_c) \partial_z q_z \dd^2 z.
\end{align}
Note that the first term above almost vanishes
 using \eqref{q_prime} and (\ref{cancel}), \ie 
 \begin{align}\label{vanish}
 \big|\mbox{first term of \eqref{int_by_part}}\big|=	4\Big|\int _\C\partial_{\bar z} f(z)  \big( \gz F \gz (\ii \eta_c)\big)_{aB} q_z \dd^2 z \Big|=O_\prec(n^{-99}).
 \end{align}
Moreover, using (\ref{deri_qz}) and the $L^\infty$-bound of $\partial_{\bar z}f$  in (\ref{f_norm}), the last term in (\ref{int_by_part}) is bounded by
$$\big|\mbox{second term of \eqref{int_by_part}}\big| \prec n^{1+c\epsilon \nc} \int_{\mathrm{supp}(f)} \Big| G^{z}_{aB}(\ii \eta_c)\Big| \one_{E_0/2\leq \lambda_1^z \leq 3E_0/2} \dd^2 z.$$
We recall the local law estimate in (\ref{E_0estimate}), which yields to
\begin{align}\label{bound_G}
	G_{\mathfrak{uv}}^{z}(\ii\eta_c) \one_{\lambda_1^{z} \gtrsim E_0} = &  M_{\mathfrak{uv}}^{z}(\ii\eta_c) \one_{\lambda_1^{z} \gtrsim E_0}+ O\left( \big| \big(G^{z}(\ii\eta_c)-M^{z}(\ii\eta_c)\big)_{\mathfrak{uv}} \big| \one_{\lambda_1^{z} \gtrsim E_0} \right)\nonumber\\
	= & -z \cdot \nc \one_{|\mathfrak{u}-\mathfrak{v}|=n} \cdot \one_{\lambda_1^{z} \gtrsim E_0} +O_\prec(n^{-1/4+\epsilon}),
\end{align}
with $E_0=n^{-3/4-\epsilon}$ and $\eta_c=n^{-L}$ for a large $L>100$, where we also used the following estimates of the deterministic matrix $M^{z}$ given by (\ref{Mmatrix}): for any $a\in [n]$, $B\in [n+1,2n]$ with $a\neq B-n$,
\begin{align}
	&M^{z}_{aB}(\ii \eta_c)=M^{z}_{Ba}(\ii \eta_c)=0, \qquad M^{z}_{aa}(\ii \eta_c)=M^{z}_{BB}(\ii \eta_c)=\ii \Im m^{z}(\ii \eta_c)=O(n^{-99}),\label{M_1}\\
	&M^{z}_{a,a+n}(\ii \eta_c)=\overline{M^{z}_{B,B-n}(\ii \eta_c)}=-z\frac{\Im m^{z}(\ii \eta_c)}{\eta_c+\Im m^{z}(\ii \eta_c)}=-z\nc +O(n^{-1/2}),\label{M_2}
\end{align}
 which follows from the properties of $\Im m^{z}$ in (\ref{rho}) and that $z\in \mathrm{supp}(f)$ from (\ref{f_cond}). \nc 
Thus  we have
  \begin{equation*}
  	\big|\mbox{second term of \eqref{int_by_part}}\big| \prec \big( n^{3/4+c\epsilon}+ n^{1+c\epsilon} \one_{a = B-n} \big)\int_{\mathrm{supp}(f)} \one_{\lambda_1^z \leq 3E_0/2} \dd^2 z.
  \end{equation*}
  Hence combining with (\ref{vanish}), we conclude from (\ref{int_by_part}) that
\begin{align}\label{trivia_2}
\Big|\int_\C \Delta_z f(z)  G^{z}_{aB}(\ii \eta_c) q_z \dd^2 z\Big|	
\prec &~\big( n^{3/4+c\epsilon}+ n^{1+c\epsilon} \one_{a = B-n} \big)\int_{\mathrm{supp}(f)} \one_{\lambda_1^z \leq 3E_0/2} \dd^2 z + n^{-99}. \nc
\end{align}
Hence  we have obtained (\ref{key_1}) from (\ref{L_0aB}) using  (\ref{trivia_1}) and (\ref{trivia_2}).  

\medskip

Next, we prove (\ref{key_pq}) for a general $p+q+1\geq 2$. Recalling $\wh \LL_f$ in (\ref{LX}) and using Lemma~\ref{lemma_deri_rule}, the partial derivatives of $\wh \LL_f$ consist of  a sum of two types of terms: 
\begin{enumerate}
	
		\item[1)] Terms from acting by all the partial derivatives 
		$\partial/\partial h_{aB}, \partial/\partial h_{Ba}$ on $\X^z$ and thus leaving 
		$q_z$ untouched.  By direct computations using (\ref{rule})-(\ref{rule_q}), \nc these terms are given by
			 \begin{align}
			 	\int_\C \Delta_z f(z) \frac{\partial^{p+q} \big(G^z_{aB}(\ii \eta_c)\big)}{\partial h^{p}_{aB} \partial h^{q}_{Ba}} q_z \dd^2 z=\int_\C \Delta_z f(z)  \big(G^{z}_{aB}(\ii \eta_c)\big)^{p+q+1} q_z \dd^2 z+O_\prec(n^{-99}),
			 \end{align}
		where the terms with at least one $\gz_{aa}$ or $\gz_{BB}$ factor are bounded by $O_\prec(n^{-100})$ using (\ref{G_tiny}) and (\ref{q_prime}). Using integration by parts as in (\ref{int_by_part}), we have
		 \begin{align}\label{case1}
		 	\int_\C \Delta_z f(z)  &\big(G^{z}_{aB}(\ii \eta_c)\big)^{p+q+1}  q_z \dd^2 z=\int_\C \partial_{\bar z} f(z)  \big(G^{z}_{aB}(\ii \eta_c)\big)^{p+q+1}  \partial_{z} q_z \dd^2 z+O_{\prec}(n^{-99})\\
			 &= O_\prec\Big(n^{\frac{3}{4}+(-\frac{1}{4}+\epsilon)(p+q)+c\epsilon} 
			 +n^{1+c\epsilon} \one_{a= B-n}\Big) \int_{\mathrm{supp}(f)} \one_{\lambda_1^z \leq 3E_0/2} \dd^2 z+ O_{\prec}(n^{-99}) \nc,
			 \nonumber
		\end{align}
	where we used the chiral symmetry (\ref{cancel}) in the first line,  and in the last line we used the derivative estimate of $q_z$ in (\ref{deri_qz}) and the  estimate (\ref{bound_G}) from the
	 local law. \nc Note that we gained additional $n^{(-\frac{1}{4}+\epsilon)(p+q)}$ for $a\neq B-n$ using (\ref{bound_G}) compared to (\ref{trivia_2}).

	\item[2)]
	Terms from acting by at least one partial derivative~(either $\partial/\partial h_{aB}$ or $\partial/\partial h_{Ba}$) on $q_z$.
	Such terms are denoted by $\int \Delta_z f(z) \big(\cdots \big) q^{(b)}_z \dd^2 z$ with $b\geq 1$, where $(\cdots)$ contains $p+q+1$ Green function entries whose row and column indices are assigned either $a$ or $B$ according to (\ref{rule})-(\ref{rule_q}). Hence using (\ref{f_norm}),  (\ref{q_prime}), (\ref{G_E0}), and (\ref{bound_G}), these terms can be bounded by
	\begin{align}\label{case2}
		\Big|\int_\C \Delta_z f(z) \big(\cdots \big) q^{(b)}_z\dd^2 z \Big| \prec \Big(	n^{\frac{3}{4}+(-\frac{1}{4}+\epsilon)(p+q)+c\epsilon}+n^{1+c\epsilon} \one_{a=B-n} \Big)\int_{\mathrm{supp}(f)} \one_{\lambda_1^z \leq 3E_0/2}\dd^2 z.
	\end{align}
Compared to (\ref{trivia_1}) for $p+q+1=1$, we have gained additional $n^{(-\frac{1}{4}+\epsilon)(p+q)}$ for $a\neq B-n$ since we get $p+q$ more off-diagonal Green function entries. 
\end{enumerate}

Summing up the above estimates in (\ref{case1}) and (\ref{case2}), we have proved (\ref{key_pq}) 
for any $p+q+1\geq 2$. The same proof applies if we switch $a$ with $B$, hence we finished the proof of Lemma \ref{lemma_key}. 
\end{proof}

We next bound the truncating error $\mathcal{E}_{11}$ in the cumulant expansion (\ref{cumulant_exp0}), i.e. 
we prove (\ref{tr_error}).

\begin{proof}[Proof of Eq. (\ref{tr_error})]
Recall from \cite[Eq. (7.1)]{HK17} that the truncating error $\mathcal{E}_{11}$ in (\ref{cumulant_exp0}) is given by
\begin{align}\label{error_bound}
	\mathcal{E}_{11}=\frac{1}{n^\frac{11}{2}} \sum_{a=1}^{n} \sum_{B=n+1}^{2n} \Big( R^{(aB)}_{11}+R^{(Ba)}_{11} \Big),
\end{align}
where $R^{(aB)}_{11}$ is bounded by
\begin{align}\label{R_5}
	|R^{(aB)}_{11}| &\lesssim   \E[ |\chi_t|^{11}] \E \Big[ \max_{p+q+1=11} \Big\{ \sup_{|w| \leq n^{-1/2+\xi}} \Big| \frac{\partial^{p+q+1}}{\partial h^{p+1}_{aB} \partial h^q_{Ba} }\F \Big(\wh \LL_{f}^{(aB)}(w, \overline{w})\Big)\Big| \Big\} \Big]\nonumber\\
	&+  \E \Big[ |\chi_t|^{11} 1_{|\chi_t|> n^{\xi}}\Big] \E \Big[\max_{p+q+1=11} \Big\{ \sup_{w \in \C} \Big| \frac{\partial^{p+q+1}}{\partial h^{p+1}_{aB} \partial h^q_{Ba} }\F \Big(\wh \LL_{f}^{(aB)}(w, \overline{w})\Big)\Big| \Big\} \Big],
\end{align}
with any fixed $\xi>0$, and  $R^{(Ba)}_{11}$ is bounded similarly switching the index $a$ with $B$.  Here the quantity $\wh \LL_{f}^{(aB)}(w, \overline{w})$ is defined as  $\wh \LL_{f}$ in (\ref{LX}) but the $(aB)$ matrix element of $W$ in (\ref{W})
is replaced with the deterministic value $w$, i.e.  $H^{z}$ replaced by the matrix $H^{(aB)}:=H^z+(w-w_{aB}) E^{(aB)}+(\overline{w}-\overline{w_{aB}}) E^{(Ba)}$, using the notation $E^{(\mathfrak{uv})}:=(\delta_{\mathfrak{uv}})_{\mathfrak{u,v}=1}^{2n}$. The partial derivatives in (\ref{R_5}) can be computed explicitly for the modified matrix $H^{(aB)}$, using the differentiation rules in Lemma~\ref{lemma_deri_rule}. With a slight abuse of notation, we use the superscript $(aB)$ to indicate that the original matrix $H$ is replaced with $H^{(aB)}$, \eg the resolvent of $H^{(aB)}$ is then denoted by $G^{{(aB)}}$. 

Since any moment of $\chi_t$ is finite from \eqref{eq:hmb} and \eqref{W}, using H\"{o}lder and Markov inequalities, one has, for any large constant $D>0$,
$$\E \Big[ |\chi_t|^{11} 1_{|\chi_t|> n^{\xi}}\Big] \leq n^{-D}.$$
Combining this with the deterministic bound $\max_{\mathfrak{uv}}\{ |G^{{(aB)}}_{\mathfrak{uv}}(\ii \eta)| \} \leq \|G^{{(aB)}}(\ii \eta)\| \leq \eta^{-1}$ for $\eta\geq \eta_c= n^{-L}$, and $L^1$-norm bound of $\Delta_z f$ in~\eqref{L1}, the last line of (\ref{R_5}) can also be bounded by $n^{-D+11L} \leq n^{-C}$, for any large constant $C>0$, choosing $D$ sufficiently large depending on $L$ and $C$.

\medskip

We next bound the first line of (\ref{R_5}). We will proceed using similar arguments as in the proof of Lemma~\ref{lemma_key} to bound partial derivatives of $\wh \LL_{f}^{(aB)}(w,\overline{w})$. We claim that, for any $p,q\in \N$
\begin{align}\label{bound_L0}
	\left|\frac{\partial^{p+q+1} \wh \LL_{f}^{(aB)}(w,\overline{w})}{\partial h^{p+1}_{Ba} \partial h^{q}_{aB}} \right| \prec n^{1/4+c\epsilon}+n^{1+c\epsilon} \Big( \int_{\mathrm{supp}(f)} \one_{\lambda_1^{{(aB)}} \leq 3E_0/2} \dd^2 z \Big) \one_{a = B-n},
\end{align}
for any $|w|\leq n^{-1/2+\xi}$ with $0<\xi \leq \epsilon/100$. 
Note that the above estimate is much weaker than (\ref{key_pq}), but will be enough to bound the higher order terms with $p+q+1\ge11$.

\newcommand{\e}{\epsilon}
\newcommand{\bw}{\mathbf{w}}

Next we prove that $\lambda_1^{(aB)} \le 3E_0/2$ implies $\lambda_1^z \le 5E_0/2$ with very high probability. This follows by a simple perturbation argument.  For any\footnote{This local $\e$ is not to be confused with 
the $\varepsilon$ exponent used in other parts of the paper.} $\e\in [0,1]$, set
$$  
    H^{(\e)}: =  H^z + \e D, \qquad D:= (w-w_{aB}) E^{(aB)}+(\overline{w}-\overline{w_{aB}}) E^{(Ba)},
$$
then $H^{(\e=0)} = H^z$ and $H^{(\e=1)}= H^{(aB)}$. Let $\lambda_1^{(\e)}$ be the 
 smallest non-negative eigenvalue of $H^{(\e)}$ and
let $\bw^{(\e)}$ be the corresponding normalized eigenvector.  We have the following delocalization result
 simultaneously for all $\e\in [0,1]$:
\begin{equation}\label{deloc}
  \sup_{\e\in [0,1]} \max_{i}  |\bw^{(\e)}(i)|^2 \prec n^{-1}.
\end{equation}
Indeed, for any fixed $\e\in [0,1]$ 
delocalization follows 
from  $|\bw^{(\e)}(i)|^2 \le \eta \big[\Im G^{(\e)}(\lambda_1^{(\e)}+\ii\eta)\big]_{ii}$ at $\eta\sim n^{-1+\xi}$
in the standard way. The bound on the diagonal element $\Im G^{(\e)}_{ii}\prec 1$, as well as rigidity,
$|\lambda_1^{(\e)}|\prec 1/n$, follow
 from  the local law in Theorem \ref{local_thm}.
Though Theorem \ref{local_thm} is stated for $G^{z}$ only, the same estimates can be easily extended to the 
perturbed resolvent $G^{(\e)}$, using the following resolvent expansion
\begin{align}\label{resolvent_expansion}
	G^{(\e)}_{ij}=G_{ij}+ \e\Big(G^{(\e)} \big( (w_{aB}-w)E^{(aB)}+(\overline{w_{aB}}-\overline{w}) E^{(Ba)} \big) G \Big)_{ij},
\end{align}
with $|w_{aB}| \prec n^{-1/2}$ from (\ref{eq:hmb}) and $|w|\leq n^{-1/2+\xi}$ for $0<\xi \leq \epsilon/100$. 
Note that $\Im G^{(\e)}(E+\ii\eta)$ is Lipschitz continuous in $\e$ with a Lipschitz constant $1/\eta^2$
thus high probability order one bounds on the resolvent for any fixed $\e$ imply 
the same bounds simultaneously for any $\e\in [0,1]$ by a standard grid argument.

By standard perturbation theory we have for any $\e$  that 
\begin{equation}\label{dlambda}
    \frac{\rm d}{\rm d\e} \lambda_1^{(\e)} = \langle \bw^{(\e)}, D \bw^{(\e)}\rangle.
\end{equation}
Using the definition of $D$, we have
$$
   | \lambda_1^{(\e=1)}- \lambda_1^{(\e=0)}|\le \int_0^1 |\langle \bw^{(\e)}, D \bw^{(\e)}\rangle| {\rm d}\e
   \le  \max( |w|+ |w_{aB}|)  \sup_{\e\in [0,1]}
  \max_{i} |\bw^{(\e)}(i)|^2 \prec n^{-3/2}.
 $$
 Thus $\lambda_1^z\le \lambda_1^{(aB)} + O_\prec (n^{-3/2})$, so indeed 
 $\lambda_1^{(aB)} \le 3E_0/2$ implies $\lambda_1^z \le 5E_0/2$ with very high probability. \nc

Therefore, we conclude from (\ref{bound_L0}) that
\begin{equation*}
	\left|\frac{\partial^{p+q+1} \wh \LL_{f}^{(aB)}(w,\overline{w})}{\partial h^{p+1}_{Ba} \partial h^{q}_{aB}} \right| \prec n^{1/4+c\epsilon}+n^{1+c\epsilon} \Big( \int_{\mathrm{supp}(f)} \one_{\lambda_1^{z} \leq 5E_0/2} \dd^2 z \Big) \one_{a = B-n}.
\end{equation*}
Then we apply the tail bound of $\lambda^{z}_1$ in (\ref{tail_bound}) and joint tail bound of $\lambda^{z_j}_1$ with different $z_j$ in (\ref{lambdatail}). Using similar arguments as in (\ref{n3}), we can bound 
\begin{align}
	\mbox{first line of (\ref{R_5})} \leq & n^{11/4+c\epsilon} + n^{11+c\epsilon} \E\Big[\Big( \int_{\mathrm{supp}(f)} \one_{\lambda_1^{{(aB)}} \leq 5E_0/2} \dd^2 z \big)^{11} \Big] \one_{a = B-n}\nonumber\\
	\leq & n^{11/4+c\epsilon} + n^{C\epsilon} \one_{a = B-n} \nonumber
\end{align}
for some constant $C>0$ depending on $\gamma$ from Proposition \ref{prop_zz}. Hence the truncating error $\mathcal{E}_{11}$ in (\ref{error_bound}) is bounded by $O_\prec(n^{-3/4+c\epsilon})$. 
This proves (\ref{tr_error}), modulo  (\ref{bound_L0}).

\medskip

It remains to prove (\ref{bound_L0}). Using Lemma~\ref{lemma_deri_rule},  one has (\cf (\ref{L_0aB}))
\begin{align}\label{dL}
	\frac{\partial \wh \LL_{f}^{(aB)}(w,\overline{w})}{\partial h_{Ba}} =& \frac{1}{4\pi} \int_\C \Delta_z f(z) G^{{(aB)}}_{aB}(\ii \eta_c) 
	q^{{(aB)}}_z \dd^2 z \nonumber\\
	&+\frac{1}{4\pi}
	\int_\C \Delta_z f(z) \X^{{(aB)}} {q^{{(aB)}}_z}' \Dim (G^{{(aB)}}_{aB})  \dd^2 z +O_\prec(n^{-99}),
\end{align}
with $\wt \Im$ and $\mathfrak{D}_{E_0,\eta_0}$ defined in (\ref{dim}), where $\X^{{(aB)}}$ and $q^{{(aB)}}_z$ are defined as in (\ref{qq}) after replacing $G^{z}$ with $G^{{(aB)}}$. We recall that the local law for $G^{z}$ in Theorem \ref{local_thm} can be easily extended to the modified resolvent $G^{(aB)}$, using the resolvent expansion in (\ref{resolvent_expansion}).
Thus one obtains the same upper bound as in (\ref{G_E0}) for the modified matrix $H^{(aB)}$, \ie
\begin{align}\label{X_bound}
	&\big|\Dim (G^{{(aB)}}_{\mathfrak{uv}})\big| \prec n^{-1/4+c\epsilon}, \qquad |\X^{{(aB)}}|\prec \log n.
\end{align}
By the definition of $q^{{(aB)}}_z$, we have $|{q^{{(aB)}}_z}'| \lesssim 1$. Hence using the $L^1$-norm bound of $\Delta_z f$ in~\eqref{L1}, the second line of (\ref{dL}) can be bounded by $O_\prec(n^{1/4+c\epsilon})$.

We next bound the first line of (\ref{dL}). By the definition of $q^{{(aB)}}_z$ in (\ref{dim}), it is straightforward to check that (\cf (\ref{q_prime})), 
\begin{align}\label{qqqq}
	|q^{{(aB)}}_z|\lesssim  \one_{\lambda^{(aB)}_1\geq E_0/2}, \qquad |1-q^{{(aB)}}_z|\lesssim  \one_{\lambda^{(aB)}_1 \leq  3E_0/2},
\end{align}
with $E_0=n^{-3/4-\epsilon}$. Recall the local law estimates in Lemma \ref{lemmaG}. Though these estimates, \eg (\ref{E_0estimate}) were shown for the original resolvent $G^z$, by inspecting the proof  and using the resolvent expansion in (\ref{resolvent_expansion}), the same upper bounds also apply to $G^{{(aB)}}$. In particular, (\ref{E_0estimate}) implies that (\cf (\ref{bound_G}))
\begin{align}\label{G_bound}
	&G^{{(aB)}}_{\mathfrak{uv}}(\ii\eta_c) q^{{(aB)}}_z=   -z \cdot \nc \one_{|\mathfrak{u}-\mathfrak{v}|=n} \cdot q^{{(aB)}}_z +O_\prec\big(n^{-1/4+\epsilon}\big),
\end{align}
where we also used the deterministic estimates from (\ref{M_1})-(\ref{M_2}). Using the $L^1$-norm bound of $\Delta_z f$ in~\eqref{L1}, the first line of (\ref{dL}) can be bounded by
\begin{align}\label{step_middle}
	\Big|\int_\C \Delta_z f(z) G^{{(aB)}}_{aB}(\ii \eta_c) 
	q^{{(aB)}}_z \dd^2 z\Big|
	\lesssim& \Big| \int_\C \Delta_z f(z)  
		\cdot z \cdot \nc q^{{(aB)}}_z \dd^2 z\Big|\one_{a=B-n}+O_\prec(n^{1/4+c\epsilon})\nonumber\\
		=& \Big| \int_\C \Delta_z f(z)  
		 \cdot z \cdot \nc \big(1-q^{{(aB)}}_z\big) \dd^2 z\Big|\one_{a=B-n}+O_\prec(n^{1/4+c\epsilon})\nonumber\\
		\lesssim& n^{1+c\epsilon} \one_{a=B-n}  \int_{\mathrm{supp}(f)}  \one_{\lambda_1^{(aB)} \leq 3E_0/2} \dd^2 z  +O_\prec(n^{1/4+c\epsilon}),
\end{align}
where we also used (\ref{qqqq}) and the $L^{\infty}$-norm of $\Delta f$ in (\ref{f_norm}). This proves (\ref{bound_L0}) for $p=q=0$. 

The proof of (\ref{bound_L0}) for any general $p, q\in \N$ is similar and requires less effort, as we get more Green function entries that can be bounded using (\ref{X_bound}) and (\ref{G_bound}) after taking higher derivatives. The worst term would be given by 
$$\int_\C \Delta_z f(z) \frac{\partial^{p+q} \big(G^{{(aB)}}_{aB}(\ii \eta_c) \big)}{\partial h^{p}_{aB} \partial h^{q}_{Ba}}  
q^{{(aB)}}_z \dd^2 z,$$
which can be bounded similarly as in (\ref{step_middle}) using the local law estimate (\ref{G_bound}). 
This proves (\ref{bound_L0}) and hence finishes the proof of (\ref{tr_error}).
\end{proof}

\nc

\subsection{Proof of the auxiliary lemmas}\label{subsec:proof_some_lemma}

\begin{proof}[Proof of Lemma \ref{lemma_deri_rule}]
The differentiation rule in (\ref{rule}) follows directly from the definition of $\gz$ in (\ref{def_G}). Using (\ref{rule}) and that $\Im \Tr G(\ii \eta)=-\ii \Tr G(\ii \eta)$, we have
	\begin{align}
		\frac{\partial \X^z}{\partial h_{Ba}}=\ii\int_{\eta_c}^{T}  (\gz)^2_{Ba}(\ii \eta) \dd \eta=- \gz_{aB}(\ii\eta_c)+O_\prec(n^{-100}),
	\end{align}
where we also used the identity
$G^2(\ii \eta)= -\ii \partial_\eta G(\ii \eta)$. 
This proves the second rule in (\ref{rule_q}).  Moreover, using that $\Im \Tr G(w)=\frac{1}{2\ii} \big( \Tr  G(w)- \Tr  G(\bar w)\big)$ and $G^2(x+\ii \eta)=  \partial_x G(x+\ii \eta)$, we have 
	\begin{align}
		\frac{\partial q^{(b)}_z}{\partial h_{Ba}}=& -q^{(b+1)}_z \int_{-E_0}^{E_0}  \wt\Im  (G^2)_{aB}(x+\ii \eta_0) \dd x=- q^{(b+1)}_z \Dim (\gz_{aB}),
	\end{align}
	with $\wt \Im$ and $\mathfrak{D}_{E_0,\eta_0}$ defined in (\ref{dim}). This proves the first rule in (\ref{rule_q}), and hence finishes the proof.
\end{proof}

\begin{proof}[Proof of Lemma \ref{lemma_algebraic_cancel}]
	We first prove (\ref{cancel}). Using the spectral decomposition of $H^z$ and denoting by $\bm{e}_a$ the unit standard vector $\bm{e}_a(b)=\delta_{ab}$, we have	
	\begin{align}
		\big(\gz F \gz(\ii \eta_c) \big)_{aB}=\sum_{j,k=-n}^{n} \frac{\<\mathbf{e}_{a},\ww^z_{j}\>\<\ww^z_{j},F\ww^z_{k}\>\<\ww^z_{k} , \mathbf{e}_{B}\>}{(\lambda^z_j-\ii \eta_c)(\lambda^z_k-\ii \eta_c)}.
	\end{align}
	Recalling the property of chiral symmetry that $\lambda^{z}_{-j}=-\lambda^{z}_{j}$ and the corresponding eigenvectors 
	are related by $\ww^z_{\pm j}=(\uu^z_j,\pm \vv^z_j)$, we thus have
	\begin{align}
		\big(\gz F \gz(\ii \eta_c) \big)_{aB}=-4\sum_{j,k=1}^{n} \frac{\eta_c^2 \<\mathbf{e}_{a},\uu^z_{j}\>\<\uu^z_{j}, \vv^z_{k}\>\<\vv^z_{k} ,\mathbf{e}_{B}\>}{\big((\lambda^z_j)^2+ \eta^2_c\big)\big((\lambda^z_k)^2+ \eta^2_c\big)}.
	\end{align}
 Working on the event $\lambda^{z}_1 \gtrsim E_0$,  we have
	\begin{align}
		\Big|\big(\gz F \gz(\ii \eta_c) \big)_{aB}\Big|\one_{\lambda_1^z \gtrsim E_0}\lesssim \frac{n^2 \eta^2_c}{E_0^4} \lesssim n^{-100},
	\end{align}
	for $\eta_c=n^{-L}$ with a large $L>100$. The same bound also applies to $\big(\gz F^* \gz(\ii \eta_c) \big)_{Ba}$ and the averaged trace $\<G(\ii \eta_c)FG(\ii \eta_c)F^*\>$.
	Similarly, for any bounded deterministic vector $ \mathbf{x}\in \mathrm{Ker}(E_1) \cup \mathrm{Ker}(E_2)$, using the property of chiral symmetry and $I=E_1+E_2$, we have 
	\begin{align}
	\Big|\big\<\mathbf{x},\gz(\ii\eta_c)\mathbf{x}\big\>\Big| \one_{\lambda_1^z \gtrsim E_0}=\left|\sum_{j=-n}^n \frac{ |\<\ww_j,\mathbf{x}\>|^2}{\lambda_j-\ii \eta_c}\right| \one_{\lambda_1^z \gtrsim E_0}  = \left|\sum_{j=1}^n \frac{2\ii\eta_c |\<\ww_j,\mathbf{x}\>|^2}{\lambda^2_j+\eta_c^2}\right| \one_{\lambda_1^z \gtrsim E_0} \lesssim n^{-100},
	\end{align} 
	for $\eta_c=n^{-L}$ with a large $L>100$. We hence finished the proof of Lemma \ref{lemma_algebraic_cancel}.
\end{proof}

\begin{proof}[Proof of Lemma \ref{lemmaG}] 
	We start with the proof of (\ref{F_0estimate_1}). By the spectral decomposition of $H^z$, using the chiral symmetry that $\lambda^z_{-j}=-\lambda^z_{j}$, $\ww_j^{z}=(\uu_j^z,\vv_j^z)$, $\uu^z_{-j}=\uu^z_{j}$ and $\vv^z_{-j}=-\vv^z_{j}$, we have
	\begin{align}\label{symmetry_G}
		\<\gz(w) F\>=\frac{1}{n}  \sum_{j=1}^n \frac{\lambda^z_j\<\uu^{z}_j,\vv^{z}_j\>}{(\lambda^z_j-E-\ii \eta)(\lambda^z_j+E+\ii \eta)}, \qquad w= E+\ii\eta \in \C^+.
	\end{align}
	Given the fixed constant $\omega_c>0$ from Corollary~\ref{coro_overlap}, set $\wt \eta:=n^{-3/4+\xi}$ for a sufficiently small $\xi>0$ such that,
	if $0\leq \lambda^z_j \leq \wt \eta $, then $1\leq j\leq n^{\omega_c}$. Thus using the overlap estimate of the 
	singular vectors in (\ref{uv_bound}) corresponding to small indices, hence small eigenvalues $\lambda^z_j \leq \wt \eta$, we have
		\begin{align}\label{small_part}
		\<\gz(w) F\>
		=&\frac{1}{n}  \Big( \sum_{0\leq \lambda^z_j \leq \wt \eta}+ \sum_{\lambda^z_j \geq \wt \eta}   \Big) \frac{\lambda^z_j\<\uu^{z}_j,\vv^{z}_j\>}{(\lambda^z_j-E-\ii \eta)(\lambda^z_j+E+\ii \eta)}\nonumber\\
		=&\frac{1}{n}  \sum_{\lambda^z_j \geq \wt \eta} \frac{\lambda^z_j\<\uu^{z}_j,\vv^{z}_j\>}{(\lambda^z_j-E-\ii \eta)(\lambda^z_j+E+\ii \eta)}+O_{\prec}\Big(\frac{n^{-1/4+\xi+\tau}}{n\eta}\Big),
	\end{align}
	where we also used that $\#\{ 0\leq \lambda^z_j \leq \wt \eta\} \prec n^{\xi}$ from the rigidity estimate of the eigenvalues in (\ref{rigidity3}). Note that to get the last error bound for the sum over  $0\leq \lambda^z_j \leq \wt \eta$, we also used the fact that $\big|\lambda^z_j+|E|-\ii \eta\big| \geq \lambda^z_j \geq 0$ and $\big|\lambda^z_j-|E|-\ii \eta\big| \geq \eta$.

	Next we estimate the error to replace  $w=E+\ii \eta$ in (\ref{small_part}), for any $|E|+\eta \lesssim n^{-3/4}$, with $w=\ii \wt \eta=\ii n^{-3/4+\xi}$, above the local eigenvalue spacing where the local laws are effective. By a direct computation, we have
	\begin{align}\label{diff}
		&\left| \frac{1}{n}  \sum_{\lambda^z_j \geq \wt \eta} \frac{\lambda^z_j\<\uu^{z}_j,\vv^{z}_j\>}{(\lambda^z_j-E-\ii \eta)(\lambda^z_j+E+\ii \eta)}-\frac{1}{n} \sum_{\lambda^z_j \geq \wt \eta} \frac{\lambda^z_j\<\uu^{z}_j,\vv^{z}_j\>}{(\lambda^z_j-\ii \wt \eta)(\lambda^z_j+\ii \wt \eta)} \right|
		\lesssim \frac{1}{n}  \sum_{j=1}^n \frac{ \wt \eta^2  \lambda_j^z   \big|\<\uu^{z}_j,\vv^{z}_j\>\big|}{\big((\lambda^z_j)^2+\wt \eta^2\big)^2}\nonumber\\
		&\qquad \qquad\qquad\qquad\qquad\qquad\qquad \lesssim  \sqrt{\frac{1}{n}\sum_{j=1}^n \frac{\wt \eta^2}{(\lambda^z_j)^2+\wt \eta^2}} \sqrt{\frac{1}{n}\sum_{j=1}^n \frac{\wt \eta^2\big|\<\uu^{z}_j,\vv^{z}_j\>\big|^2}{\big((\lambda^z_j)^2+\wt \eta^2\big)^2}}\nonumber\\
		&\qquad \qquad\qquad\qquad\qquad\qquad\qquad\lesssim  \sqrt{\wt\eta\<\Im \gz (\ii \wt \eta)\>} \sqrt{\<\gz F\gz F^*(\ii \wt \eta)\>} \prec n^{-1/2+\xi},
	\end{align}
	where we also used the local laws in (\ref{average}) with  (\ref{rho}) and (\ref{local_2F}) with (\ref{MFMF}).
	Combining (\ref{diff}) with (\ref{small_part}), we obtain that
	\begin{align}
		\<\gz(w) F\>=&
		\frac{1}{n} \sum_{\lambda^z_j \geq \wt \eta} \frac{\lambda^z_j\<\uu^{z}_j,\vv^{z}_j\>}{(\lambda^z_j-\ii \wt \eta)(\lambda^z_j+\ii \wt \eta)}+O_{\prec}\Big(\frac{n^{-1/4+\xi+\tau}}{n\eta}\Big)=\<\gz(\ii \wt\eta) F\>+O_{\prec}\Big(\frac{n^{-1/4+\xi+\tau}}{n\eta}\Big),
	\end{align}
	where we also used (\ref{small_part}) with $w=\ii \wt \eta$.
	Thus we have proved (\ref{F_0estimate_1}).

	 The proof of (\ref{FF_0estimate_1}) is similar. Using the chiral symmetry among eigenvalues and eigenvectors, for any $w= E+\ii\eta \in \C^+$, we have
	\begin{align}\label{symmetry_GF}
		\<\gz(w) F \gz(w) F^*\>=\frac{2}{n}  \sum_{j,k=1}^n \frac{(E+\ii \eta)^2 |\<\uu^{z}_j,\vv^{z}_k\>|^2}{\big((\lambda^z_j)^2 -(E+\ii \eta)^2 \big)\big((\lambda^z_k)^2 -(E+\ii \eta)^2\big)}.
	\end{align}
	Similarly to (\ref{small_part}), we use the rigidity of eigenvalues (\ref{rigidity3}) and the singular vector overlap (\ref{uv_bound}) to estimate the sum over small eigenvalues $0\leq \lambda^z_j, \lambda^z_k  \leq \wt \eta $ with small indices,  \ie
	\begin{align}
		&\left|\frac{1}{n} \sum_{0\leq \lambda^z_j, \lambda^z_k  \leq \wt \eta } \frac{(E+\ii \eta)^2 |\<\uu^{z}_j,\vv^{z}_k\>|^2}{\big((\lambda^z_j)^2 -(E+\ii \eta)^2 \big)\big((\lambda^z_k)^2 -(E+\ii \eta)^2\big)} \right|  \prec  \frac{n^{2\xi}}{n}\frac{n^{2\tau}}{\sqrt{n}} \sup_{0\leq\lambda \leq \wt \eta} \left\{\frac{E^2+\eta^2}{\big|\lambda^2-E^2+\eta^2-2\ii E \eta\big|^2} \right\}\nonumber\\
		&\qquad \qquad \qquad \qquad \qquad \qquad\prec \frac{n^{2\xi}}{n}\frac{n^{2\tau}}{\sqrt{n}} \left( \frac{E^2+\eta^2}{(\eta^2-E^2)^2} \one_{E< \eta/100}+\frac{E^2+\eta^2}{E^2\eta^2} \one_{E\geq \eta/100} \right)\nonumber\\
		&\qquad \qquad \qquad \qquad \qquad \qquad=O_{\prec}\Big(\frac{n^{-1/2+2\xi+2\tau}}{n\eta^2}\Big).\label{term1}
	\end{align}
For the sum over large eigenvalues $\lambda^z_j, \lambda^z_k  \geq \wt \eta $, we have
\begin{align}
	\left|\frac{1}{n}  \sum_{\lambda^z_j, \lambda^z_k \geq \wt \eta } \frac{(E+\ii \eta)^2 |\<\uu^{z}_j,\vv^{z}_k\>|^2}{\big((\lambda^z_j)^2 -(E+\ii \eta)^2 \big)\big((\lambda^z_k)^2 -(E+\ii \eta)^2\big)}\right| \lesssim \frac{1}{n} 
	\sum_{\lambda^z_j,\lambda^z_k  \geq \wt \eta }
	\frac{\wt \eta^2 |\<\uu^{z}_j,\vv^{z}_k\>|^2}{\big((\lambda^z_j)^2 +\wt \eta^2\big)\big((\lambda^z_k)^2 +\wt \eta^2\big)},\label{term2}
\end{align}
where we used that $|E|+\eta \lesssim n^{-3/4} \ll \wt \eta=n^{-3/4+\xi}$. For the remaining sum over crossing indices, \ie for $0\leq \lambda^z_j  \leq \wt \eta$ and $\lambda^z_k\geq \wt \eta$, we similarly have 
\begin{align}
	&\left|\frac{1}{n} \sum_{0\leq \lambda^z_j \leq \wt \eta, \lambda^z_k  \geq \wt \eta }\frac{(E+\ii \eta)^2 |\<\uu^{z}_j,\vv^{z}_k\>|^2}{\big((\lambda^z_j)^2 -(E+\ii \eta)^2 \big)\big((\lambda^z_k)^2 -(E+\ii \eta)^2\big)} \right|\nonumber\\
	& \qquad\qquad\qquad \lesssim  \frac{1}{n} 
	\sum_{0\leq \lambda^z_j \leq \wt \eta, \lambda^z_k  \geq \wt \eta }
	\sup_{0\leq\lambda \leq \wt \eta} \left\{\frac{E^2+\eta^2}{\big|\lambda^2-E^2+\eta^2-2\ii E \eta\big|} \right\}
	\frac{ |\<\uu^{z}_j,\vv^{z}_k\>|^2}{\big((\lambda^z_k)^2 +\wt\eta^2\big)} \nonumber\\
		&\qquad\qquad\qquad \lesssim \Big(1+\frac{|E|}{\eta}\Big) \left(\frac{1}{n} 
		\sum_{0\leq \lambda^z_j \leq \wt \eta, \lambda^z_k  \geq \wt \eta }
		\frac{\wt \eta^2 |\<\uu^{z}_j,\vv^{z}_k\>|^2}{\big((\lambda^z_j)^2 +\wt \eta^2\big)\big((\lambda^z_k)^2 +\wt \eta^2\big)}\right),\label{term3}
\end{align}
where in the last step we also used that $\sup_{0\leq\lambda \leq \wt \eta}\{\cdots\} \lesssim 1+|E|/\eta$ and $\wt \eta^2 \leq (\lambda^z_j)^2 +\wt \eta^2$. 
Thus for any $w=E+\ii \eta$ with $|E|+\eta \lesssim n^{-3/4}$, we obtain from (\ref{symmetry_GF})-(\ref{term3}) that
	\begin{align}
		|\<\gz(w) F \gz(w) F^*\>| \lesssim& \Big(1+\frac{|E|}{\eta}\Big) \left( \frac{1}{n} \sum_{j,k=1}^n \frac{\wt \eta^2 |\<\uu^{z}_j,\vv^{z}_k\>|^2}{\big((\lambda^z_j)^2 +\wt \eta^2 \big)\big((\lambda^z_k)^2 +\wt  \eta^2\big)}\right)+O_{\prec}\Big(\frac{n^{-1/2+2\xi+2\tau}}{n\eta^2}\Big)\nonumber\\
		\lesssim& \Big(1+\frac{|E|}{\eta}\Big) \<\gz(\ii \wt \eta) F \gz (\ii \wt \eta)F^*\>+O_{\prec}\Big(\frac{n^{-1/2+2\xi+2\tau}}{n\eta^2}\Big).
	\end{align}
\nc
Since $\wt \eta=n^{-3/4+\xi}$, we are above the local eigenvalue spacing in $\<\gz(\ii \wt \eta) F \gz (\ii \wt \eta)F^*\>$, thus using the local law in (\ref{local_2F}) and the second deterministic upper bound in (\ref{MFMF}) with $|E|+\eta \lesssim n^{-3/4} \ll \wt \eta=n^{-3/4+\xi}$, we have proved (\ref{FF_0estimate_1}).

	Next we prove (\ref{E_0estimate}). Using the spectral decomposition of $H^z$ and the Cauchy-Schwarz inequality, for any bounded deterministic vectors $\mathbf{x},\mathbf{y} \in \C^{2n}$, we have
	\begin{align}\label{G_diff}
		\Big|\big\<\mathbf{x}, \big(G^{z}(\ii\eta_c)-G^{z}(\ii E_0)\big) \mathbf{y}\big\>\one_{\lambda_1^z \gtrsim E_0} \Big| \lesssim& \Big| \sum_{j} \frac{E_0 \<\mathbf{x},\ww^z_{j}\> \<\ww^z_{j}, \mathbf{y}\> }{(\lambda^z_j-\ii \eta_c) (\lambda^z_j-\ii E_0)} \Big|\one_{\lambda_1^z \gtrsim E_0}\nonumber\\
		\lesssim & \sqrt{ \<\mathbf{x},\Im G(\ii E_0) \mathbf{x}\> \<\mathbf{y},\Im G(\ii E_0) \mathbf{y}\> }  \prec \frac{1}{nE_0},
	\end{align}
	where we also used the local law in (\ref{entrywise}) for $w=\ii E_0$ and the property of $M^{z}$ in (\ref{rho}). Moreover, from  (\ref{Mmatrix}) and (\ref{rho}), we have 
	\begin{align}\label{M_diff}
		\Big| \< M^{z}(\ii \eta_c)-M^{z}(\ii E_0)\> \big| \lesssim  (E_0)^{1/3}, \qquad \Big| \< \big(M^{z}(\ii \eta_c)-M^{z}(\ii E_0)\big)F^{(*)}\>  \big| \lesssim  (E_0)^{2/3}.
	\end{align}
	This, together with (\ref{G_diff}), proves (\ref{E_0estimate}). 
	
	 Finally, we prove (\ref{F0_estimate}) similarly. Using (\ref{symmetry_G}) and a similar argument as in (\ref{diff}), we have
	\begin{align}
		\Big| \big\< \big( G^{z}(\ii\eta_c)-G^{z}(\ii E_0)  \big) F \big\>\Big| \one_{\lambda_1^{z} \gtrsim E_0} \lesssim & \frac{1}{n}  \sum_{j=1}^n \frac{(E_0)^2  \lambda^z_j \big|\<\uu^{z}_j,\vv^{z}_j\>\big|}{\big((\lambda^z_j)^2+E_0^2\big)^2}\nonumber\\
		\lesssim & \sqrt{E_0\<\Im \gz (\ii E_0)\>} \sqrt{\<\gz F\gz F^*(\ii E_0)\>} \prec \frac{n^{-1/4+\xi+\tau}}{nE_0},
	\end{align}
	where we also used the local laws in (\ref{average}) and (\ref{FF_0estimate_1}) for $w=\ii E_0$. Together with (\ref{M_diff}), we have proved (\ref{F0_estimate}) and hence finished the proof of Lemma \ref{lemmaG}. \nc
\end{proof}

\begin{proof}[Proof of Lemma \ref{lemma_q}]
	We first prove the estimates in (\ref{q_prime}).  The former one  follows directly from (\ref{approx}) with $E=E_0$, hence we focus on proving the latter one. Recall the definition of $q$ given in (\ref{q_function}) with uniformly bounded derivatives and $q^{(b)}_z$ defined in (\ref{qq}). Recall $E_0=n^{-3/4-\epsilon}$ and choose the parameters $l_0 \gg l_0' \gg \eta_0$ as in Lemma \ref{lemma_approx}, \ie
	\begin{align}\label{eta_0}
		l_0=n^{-\zeta} E_0, \qquad  l_0'=n^{-3\zeta} E_0, \qquad  \eta_0= n^{-6\zeta} E_0, 
	\end{align} 
	for some small $\zeta>0$. Note that if $q^{(b)}_z \neq 0~(b\geq 1)$, then it implies that
	\begin{align}\label{integer1}
		\int_{-E_0 }^{E_0}  \Im \Tr G^z(y+\ii \eta_0) \dd y \in [1/9,2/9].
	\end{align}
	Using the first inequality in (\ref{approx_1}) with $E=E_0+l_0$ and the second inequality with $E=E_0-l_0$, we have
	\begin{align}
		\#\{|\lambda^z_i| \leq E_0-l_0\}-O(n^{-\zeta}) \leq \int_{-E_0}^{E_0}  \Im \Tr G^z (y+\ii  \eta_0)\dd y\leq 	\#\{|\lambda^z_i| \leq E_0+l_0\}+O(n^{-\zeta}).
	\end{align}
	Combining this with (\ref{integer1}) and 
using that both $\#\{|\lambda^z_i| \leq E_0\pm l_0\}$ are integer valued, we obtain
	\begin{align}\label{integer2}
		\#\{|\lambda^z_i| \leq E_0-l_0\}=0, \qquad 	\#\{|\lambda^z_i| \leq E_0+l_0\}\geq 1.
	\end{align}
	Hence we conclude that
	$$\big|q_{z}^{(b)}\big|=\Big| q^{(b)}\Big( \int_{-E_0}^{E_0} \Im \Tr G^z(y+\ii \eta_0) \dd y \Big) \Big| \lesssim \one_{E_0/2\leq \lambda_1^z \leq 3E_0/2}.$$

	Next, we prove (\ref{deri_qz}). By a direct computation using (\ref{rule_z}) , we have
	\begin{align}\label{some}
		\partial_z q_z =& q'_z \int_{-E_0}^{E_0} \wt \Im \big(\Tr  \gz F\gz (x+\ii \eta_0)\big) \dd x = q'_z \Dim\big( \Tr \gz F\big),
	\end{align}
	with $\wt \Im$ and $\mathfrak{D}_{E_0,\eta_0}$ defined in (\ref{dim}), and $\partial_{\bar z} q_z$ can be computed similarly. Using (\ref{F_0estimate_1}) for any small $\xi>0$ and that $E_0$, $\eta_0$ are chosen as in (\ref{eta_0}) for any small $\zeta>0$, we obtain 
	\begin{align}\label{F_0estimate_2}
		\Big|\Dim \big(\<\gz F^{(*)}\>\big)\Big|=O_\prec(n^{-1/2+\epsilon+\tau}).
	\end{align}
	Then (\ref{deri_qz}) follows directly from (\ref{some}) using (\ref{q_prime}) and (\ref{F_0estimate_2}). 
	The proof of (\ref{lap_qz}) is similar using additionally (\ref{FF_0estimate_1}). \nc This concludes the proof of Lemma~\ref{lemma_q}.
\end{proof}

\subsection{Extension of Theorem \ref{main_thm} to general $\mathcal{F}$ }\label{sec:polyF}

We  sketch how  to extend Theorem \ref{main_thm} to any function $\mathcal{F}$ with derivatives of polynomial growth as in (\ref{F_gen}).
 We follow the same scheme as in Steps~1--6 of Section~\ref{sec:strategy_gft}. More precisely, 
 we repeat Steps 1--5 with modifications to generalise Proposition~\ref{lemma_L} to any function $\F$ with polynomial growth as in (\ref{F_gen}), \ie for any fixed $p\in \N$, there exists a large $L>100$, depending on $p$ and $\alpha,\beta$ in (\ref{assumption_b}), such that 
\begin{align}\label{M_zerop}
	\E\big|\LL_f-\wh\LL_f\big|^p =O(n^{- c(p) \epsilon}),
\end{align}
for some constant $c(p)>0$, where $\wh \LL_f$ is defined in (\ref{L_0}) with $\eta_c=n^{-L}$. After these, in Step 6, we extend Theorem \ref{GFT} to a general function $\F$ in (\ref{F_gen}) using similar GFT arguments as in Section \ref{sec:L_0}.

\smallskip

We first note that, the estimates in Step 1 and Step 3 of Section~\ref{sec:proof_gft} 
also hold true in any finite moment sense. So we only focus on the modifications needed in Step 2 and Steps 4--6.

\smallskip

{\bf Modifications in Steps 2, 4, 5: } We mainly focus on Step 4, \ie to show that the size of $\LL_{\leq E_0}$ in (\ref{small_eigen}) is negligible in any finite moment sense, using additionally
the independence estimates on the smallest eigenvalues in (\ref{lambdatail}). The proof of (\ref{tiny_eta})
 in Step 2 and (\ref{big_eigen}) in Step 5 in any finite moment sense is exactly the same, so we omit the details.

Recall that $E_0=n^{-3/4-\epsilon}$, for some small fixed $\epsilon>0$, and that 
$\lVert\Delta f\rVert_\infty\lesssim n^{1+2\nu}$, $\lVert \Delta f\rVert_1\lesssim n^{1/2+\tau+2\nu}$, for some arbitrary small $\tau,\nu>0$,
in particular $\tau,\nu$ will be chosen  much smaller than $\epsilon$. 
  We now  show that for any $p\in\N$ there exist constants $c(p)>0$ such that
 \begin{equation}
\label{eq:hopeb}
  \E \big|\LL_{\leq E_0}\big|^p \lesssim (\log n)^p \int \ldots \int \P \big( \lambda^{z_1}_1 \leq E_0, \dots, \lambda^{z_p}_1 \leq E_0  \big) 
    \prod_{i=1}^p  \big|\Delta f(z_i)\big| \dd^2 z_i\le n^{-c(p)\epsilon}.
\end{equation}
To keep the notation simpler,  we set $\tau=\nu=0$, as any term of the form $n^{C_p(\tau+\nu)}$ (with some large $C_p>0$)
accumulated along the proof  can be made smaller than $n^{-c(p)\epsilon}$ from \eqref{eq:hopeb}, 
by choosing $\tau=\tau(\epsilon)$, $\nu=\nu(\epsilon)$ sufficiently small.
We will prove \eqref{eq:hopeb}  by induction on $p$; in particular, we will see that  the 
choice $c(p)=\gamma^p$ works, with $\gamma$ being the small constant from Proposition~\ref{prop_zz}.

The fact that \eqref{eq:hopeb} holds for $p=1$, even with $c(1)=1$,
 immediately follows from Proposition~\ref{prop1} for $E=E_0$~(\ie (\ref{tail_precise})). We now assume that \eqref{eq:hopeb} holds for any $l\le p-1$, for some $p\ge 2$, and we show that \eqref{eq:hopeb} holds for $l=p$ as well. For this purpose we define
 the set
\begin{equation}
Z:=\left\{|z_i-z_j|\ge n^{-1/2+\omega}, \,\, \forall\, i\ne j\right\}\subset \C^p,
\end{equation}
for some $\omega=\omega(p,\epsilon)>0$ which we will choose later in the proof, and 
split the integral in~\eqref{eq:hopeb} as
\begin{equation}
  I_Z+ I_{Z^c}:=  \left(\int \ldots \int_Z+\int \ldots \int_{Z^c}\right)\P \big( \lambda^{z_1}_1 \leq E_0, \cdots, \lambda^{z_p}_1 \leq E_0  \big) \prod_{i=1}^p  \big|\Delta f(z_i)\big| \dd^2 z_i .
\end{equation}
On $Z$ we can use the upper bound in
Proposition~\ref{prop_zz}, followed by~\eqref{tail_bound}, to show that the probability essentially 
factorizes: 
\begin{equation}
\label{eq:news1}
I_Z \lesssim \int\ldots\int_Z \P\big(\lambda_1^{z_i}\le E_0+n^{-3/4-\gamma\omega}\big) \prod_{i=1}^p  \big|\Delta f(z_i)\big| \dd^2 z_i \lesssim  n^{-2p\gamma\omega}+n^{-2p\epsilon}.
\end{equation}
Here we used that on the support of $f$, by \eqref{f_cond}, we have
\begin{equation}\label{gooddecay}
\P(\lambda_i\le E_0+n^{-3/4-\gamma\omega})\lesssim n^{3/2}(E_0+n^{-3/4-\gamma\omega})^2e^{-\gamma_n/2(1+O(C_n/\gamma_n))}\lesssim n^{-1/2}\big( n^{-2\gamma\omega}+n^{-2\epsilon}\big),
\end{equation}
where we used the explicit form of $\gamma_n$ from~\eqref{gamma} and that $C_n/\sqrt{\log n}\to 0$ as $n\to \infty$. 
On $Z^c$ we may assume, 
without loss of generality that 
$|z_p-z_{p-1}|\le n^{-1/2+\omega}$, then we have 
\begin{equation}
\label{eq:news2}
I_{Z^c}\lesssim n^{2\omega}\int\ldots\int\P \big( \lambda^{z_1}_1 \leq E_0, \cdots, \lambda^{z_{p-1}}_1 \leq E_0  \big) \prod_{i=1}^{p-1}  \big|\Delta f(z_i)\big| \dd^2 z_i\lesssim n^{2\omega-c(p-1)\epsilon},
\end{equation}
where in the first inequality we integrated out $z_p$ by using that the volume of the region
 $|z_p-z_{p-1}|\le n^{-1/2+\omega}$ is of order $n^{-1+2\omega}$ and that $\lVert\Delta f\rVert_\infty\lesssim n$.
In the last inequality we used the induction hypothesis. Putting \eqref{eq:news1}--\eqref{eq:news2} together, we thus obtain
\begin{equation}
\label{eq:bsum3t}
\int\ldots\int \P \big( \lambda^{z_1}_1 \leq E_0, \cdots, \lambda^{z_p}_1 \leq E_0  \big)  \prod_{i=1}^p
 \big|\Delta f(z_i)\big|\dd^2 z_i\lesssim n^{-2p\gamma \omega}+n^{-2p\epsilon}+n^{2\omega-c(p-1)\epsilon}.
\end{equation}
We now show that we can choose $c(p)$ and $\omega=\omega(p,\epsilon)$ such that
\begin{equation}
\label{eq:impin}
\begin{cases}
2p\gamma\omega>c(p)\epsilon, \\
c(p-1)\epsilon-2\omega> c(p)\epsilon,
\end{cases}
\end{equation}
are both satisfied; this would complete the induction step of proving \eqref{eq:hopeb}. 
 Choosing $c(p)=\gamma^p$, we need to choose 
  $\omega=\omega(p,\epsilon)$  so that
  \begin{equation}
\frac{\gamma^{p-1}}{2p}<\frac{\omega}{\epsilon}<\frac{\gamma^{p-1}-\gamma^p}{2},
\end{equation} 
which is possible since $\frac{1}{2p} \gamma^{p-1} < \frac{1}{2}(\gamma^{p-1}-\gamma^p)$, recalling $p\ge 2$.
This concludes the proof of \eqref{eq:hopeb}.

\smallskip

{\bf Modifications in Step 6:}  Even though the derivatives of $\F$ have polynomial growth, the following finite moment bound
 \begin{align}\label{L_bound}
	\E|\wh \LL_f|^{p} \lesssim n^{c \epsilon p},
\end{align}
for some fixed constant $c>0$ independent of $p$, still allows us to perform a similar GFT argument as in Section~\ref{sec:L_0}. To prove (\ref{L_bound}), we apply integration by parts as in (\ref{int_by_part})
but  in both $\partial_z$ and $\partial_{\bar z}$, \ie
 \begin{align}\label{two_int_by_part}
 \wh \LL_{f} =&  -\frac{1}{4\pi}
 \int  f (z) \Delta_z \left[ \Big(\int_{\eta_c}^{T}  \Im \Tr \big( G^z(\ii \eta)-M^z(\ii \eta) \big) \dd \eta \Big)  q_z \right] \dd^2 z \nonumber\\
  =& \frac{\ii}{4\pi}
 \int f (z) \int_{\eta_c}^{T} \Delta_z\big[\Tr G^z(\ii \eta)\big] q_z \dd^2 z+O_\prec(n^{1+c\epsilon}) \int_{\mathrm{supp}(f)}\one_{E_0/2\leq \lambda_1^z \leq 3E_0/2},
 \end{align}
for some constant $c>0$, where we used $\Im G^{z}(\ii \eta)=-\ii G^{z}(\ii \eta)$, (\ref{M_zero}), the $z$-derivative bounds of $q_z$ in (\ref{deri_qz})-(\ref{lap_qz}) in combination with the local law estimate for $GF$ 
in (\ref{F0_estimate}).    Note that $\Delta_z\big[\Tr G^z(\ii \eta)\big]=4\Tr\big( G^z F G^z F^* G^z(\ii \eta)+G^z F^* G^z F G^z(\ii \eta)\big)= -4\ii \frac{\dd}{ \dd \eta}  \Tr \big(\gz F \gz F^*(\ii \eta)\big)$.
Thus the first term in (\ref{two_int_by_part}) almost vanishes, \ie
 \begin{align}\label{term_1}
 	\frac{\ii}{4\pi}
\int f (z) \int_{\eta_c}^{T} \Delta_z\big[\Tr G^z(\ii \eta)\big] q_z \dd^2 z=-\frac{1}{\pi}
\int f (z)\Big(\Tr G^z F G^z F^*(\ii\eta)\big|_{\eta=\eta_c}^{\eta=T}\Big) q_z \dd^2 z=O_\prec(n^{-100}),
\end{align}
with $\eta_c=n^{-L}$ for some large $L>100$, using the precise cancellation in (\ref{cancel_full}) due to the  chiral symmetry of $H^z$. Moreover,  the second term in (\ref{two_int_by_part}) can be controlled effectively  by a similar argument as in (\ref{eq:hopeb}), using the precise tails bound of $\lambda^z_1$ in (\ref{tail_bound}) together with the independence estimates in (\ref{lambdatail}). That is, we obtain from (\ref{two_int_by_part}) and (\ref{term_1}) that
\begin{align}\label{term_2}
	\E \big|\wh \LL_{f}\big|^{p}  \lesssim & \big(n^{1+c\epsilon}\big)^p  \int_{\mathrm{supp}(f) } \ldots
	 \int_{ \mathrm{supp}(f) } \P \big( \lambda^{z_1}_1 \leq 2E_0, \dots, \lambda^{z_p}_1 \leq 2E_0  \big) 
	\prod_{i=1}^p  \dd^2 z_i+n^{-100p}\lesssim  n^{c \epsilon p},
\end{align}
where we also used (\ref{eq:hopeb}) with the choice $c(p)=\gamma^p$ with small $\gamma>0$ from (\ref{lambdatail}). 
This, together with (\ref{two_int_by_part}), proves the finite moment bound in (\ref{L_bound}). 
Given this size bound, 
a similar GFT argument as in Theorem \ref{GFT} applies
 to any function $\mathcal{F}$ in (\ref{F_gen}) with derivatives of polynomial growth,
 concluding the proof of Theorem \ref{main_thm} for such test functions.

\bigskip

\appendix

\section{Technical results}
\label{app:addtechlem}

In this appendix we present the proof of several additional technical results used within the proofs of Section~\ref{sec:G1G2llaw}. Additionally, this appendix is divided into two main parts. In Section~\ref{sec:det} we prove several deterministic bounds for the deterministic approximation of the products of two resolvents, while in Section~\ref{sec:rand} we present the derivation of stochastic equations appearing in Section~\ref{sec:G1G2llaw} as well as a certain bound for the product of two resolvents.
 
 \subsection{Properties of stability operators}
 \label{sec:det}
 
 Recall the definition of the two--body stability operator from \eqref{eq:defstabop}.  The stability operator $\mathcal{B}_{12}$ has two non--trivial small eigenvalues\footnote{We use the standard convention that the complex square root $\sqrt{\cdot}$ is defined using the branch cut $\C\setminus (-\infty,0)$.} (see \cite[Appendix B]{CES22} for more details about the eigendecomposition of $\mathcal{B}_{12}$):
\begin{equation}
\label{eq:dedfevalues}
\beta_\pm=\beta_\pm(\eta_1,z_1,\eta_2,z_2):=1-u_1u_2\Re[z_1\overline{z_2}]\pm\sqrt{s}, \qquad\quad s:=m_1^2m_2^2-u_1^2u_2^2(\Im[z_1\overline{z_2}])^2;
\end{equation}
the other eigenvalues are all ones.  Note that the eigenvalues in \eqref{eq:dedfevalues} coincide with $\beta_{\pm,t}$ from \eqref{eq:defevmore} when evaluated along time--dependent spectral parameters, i.e. $\beta_{\pm,t}=\beta_\pm(\eta_{1,t},z_{1,t},\eta_{2,t},z_{2,t})$. 
The main input for most of the proofs in this section is the following size relation for $\beta_\pm$:  
\begin{lemma}
\label{lem:strongerb}
For any sufficiently small \nc $\epsilon_1,\epsilon_2>0$ and a large $C>0$, and fix $z_1,z_2\in \C$, $\eta_1,\eta_2\ne 0$ such that $1\le |z_i|\le C$, $|z_1-z_2|\le \epsilon_1$, and $|\eta_i|\le \epsilon_2$. Let $\beta_\pm$ be defined as in \eqref{eq:dedfevalues}, then we have
\begin{equation}
\big|\beta_\pm\big|\sim \gamma = \nc |z_1-z_2|+\frac{|\eta_1|}{\rho_1}+\frac{|\eta_2|}{\rho_2}.
\end{equation}
\end{lemma}

\begin{proof}[Proof of Lemma~\ref{lem:strongerb}]

Note that by $-u_i=m_i^2-|z|^2u_i^2$ (which is equivalent to \eqref{m_function}) it readily follows (recall $\rho_i:=\pi^{-1}|\Im m_i|$)
\begin{equation}
\label{eq:relu}
u_i=\frac{1}{|z_i|^2}-\pi^2\rho_i^2+O\left(\rho_i^4\right).
\end{equation}
Additionally, we note that
\begin{equation}
\label{eq:gain2}
\Re[z_1\overline{z_2}]=\frac{|z_1|^2+|z_2|^2-|z_1-z_2|^2}{2}.
\end{equation}
Using \eqref{eq:relu} in the first equality and \eqref{eq:gain2}, with $||z_1|-|z_2||\le |z_1-z_2|$, in the second one, we then compute
\begin{equation}
\begin{split}
\label{eq:impexp}
1-u_1u_2\Re[z_1\overline{z_2}]&=1-\frac{\Re[z_1\overline{z_2}]}{|z_1z_2|^2}+\frac{\pi^2\rho_1^2}{|z_2|^2}\Re[z_1\overline{z_2}]+\frac{\pi^2\rho_2^2}{|z_1|^2}\Re[z_1\overline{z_2}]+O(\rho_1^4+\rho_2^4)\\
&=1-\frac{|z_1|^2+|z_2|^2}{2|z_1z_2|^2}+\big[\pi^2\rho_1^2+\pi^2\rho_2^2\big](1+O(|z_1-z_2|))+\frac{|z_1-z_2|^2}{2|z_1z_2|^2}+O(\rho_1^4+\rho_2^4).
\end{split}
\end{equation}


Then, using that $1\le |z_i|\le C$, $u_i\sim 1$, and that
\begin{equation}
\label{eq:trrel}
1-\frac{|z_1|^2+|z_2|^2}{2|z_1z_2|^2}=\frac{|z_1|^2(|z_2|^2-1)+|z_2|^2(|z_1|^2-1)}{2|z_1z_2|^2}\sim |z_1|^2+|z_2|^2-2,
\end{equation}
this implies 
\begin{equation}
\begin{split}
\label{eq:rel}
\beta_+\beta_-&=\big[1-u_1u_2\Re[z_1\overline{z_2}]\big]^2-m_1^2m_2^2+u_1^2u_2^2(\Im[z_1\overline{z_2}])^2 \\
&=\left[1-u_1u_2\Re[z_1\overline{z_2}]-\pi^2\rho_1\rho_2\right]\left[1-u_1u_2\Re[z_1\overline{z_2}]+\pi^2\rho_1\rho_2\right]+u_1^2u_2^2(\Im[z_1\overline{z_2}])^2 \\
& \gtrsim \left(|z_1|^2+|z_2|^2-2+\rho_1^2+\rho_2^2\right)^2+(\Im[z_1\overline{z_2}])^2 \\
&\gtrsim \left(|z_1|^2+|z_2|^2-2+\rho_1^2+\rho_2^2\right)^2+|z_1-z_2|^2.
\end{split}
\end{equation}
In the last inequality we used that 
\begin{equation}
\label{eq:missstep}
\left(|z_1|^2+|z_2|^2-2+\rho_1^2+\rho_2^2\right)^2+(\Im[z_1\overline{z_2}])^2\gtrsim |z_1-z_2|^2.
\end{equation}
whose proof is postponed to the end.

On the other hand, using again \eqref{eq:impexp}--\eqref{eq:trrel} and that $|\Im [z_1\overline{z_2}]|\lesssim |z_1-z_2|$, we readily obtain
\begin{equation}
\label{eq:opprel}
|\beta_\pm|\lesssim |z_1|^2+|z_2|^2-2+\rho_1^2+\rho_2^2+|z_1-z_2|.
\end{equation}

Combining \eqref{eq:rel}--\eqref{eq:opprel} we finally obtain
\begin{equation}
\label{eq:boundbpm}
\big|\beta_\pm\big|\gtrsim |z_1|^2+|z_2|^2-2+|z_1-z_2|+\rho_1^2+\rho_2^2.
\end{equation}
Noticing that by \eqref{m_function} it readily follows
\begin{equation}
\label{eq:imprelder}
\frac{\eta_i}{\Im m_i}=|z_i|^2-1+\pi^2\rho_i^2|z_i|^4+O(\rho_i^4),
\end{equation}
and combining \eqref{eq:opprel}--\eqref{eq:boundbpm}, we conclude the proof of Lemma~\ref{lem:strongerb}
modulo the proof of~\eqref{eq:missstep} that we present now.

Assume that $1\le |z_1|\le  |z_2|$. By rotational symmetry we can assume that $z_1=r\in\R_+$ and write $z_2=r+\varepsilon e^{\ii\theta}$, with $\varepsilon\ge 0$ and $\theta\in [0,2\pi)$. Then we have
\begin{equation}
\label{eq:bneedlast}
\left(|z_1|^2+|z_2|^2-2+\rho_1^2+\rho_2^2\right)^2+(\Im[z_1\overline{z_2}])^2\ge \big(2r^2-2+\varepsilon^2+2r\varepsilon \cos\theta\big)^2+\big(\varepsilon r\sin \theta\big)^2.
\end{equation}

We now distinguish two cases: i) $|\sin \theta|\ge \delta$, ii) $|\sin \theta|< \delta$, for some small $\delta>0$ we will choose shortly. In case i) we readily conclude
\[
\big(\varepsilon r\sin \theta\big)^2\gtrsim \varepsilon^2.
\]

For case ii), we notice that $|\sin \theta|< \delta$ implies $|\cos \theta|\ge 1-\delta^2$. Additionally, since we assume that $|z_1|\le |z_2|$ and $|z_1-z_2|\le \epsilon_1$, if $\delta$ is chosen so that $10\epsilon_1\le\delta\le 1/10$ then this also implies that $\cos\theta>0$. We thus have
\[
\big(2r^2-2+\varepsilon^2+2r\varepsilon \cos\theta\big)^2\ge \varepsilon^2 (\cos\theta)^2\gtrsim \varepsilon^2.
\]
Recalling that $\varepsilon^2=|z_1-z_2|^2$, this concludes the proof of~\eqref{eq:missstep}, hence
Lemma~\ref{lem:strongerb}.

\end{proof}

\begin{proof}[Proof of Lemma~\ref{lem:propertms}]
By explicit computations we find that
\begin{equation}
\big((\mathcal{B}_{12}^{-1})^*[E_\pm]\big)^*=\frac{1}{1+|z_1z_2|^2u_1^2u_2^2-m_1^2m_2^2-2u_1u_2\Re[z_1\overline{z_2}]}\left(\begin{matrix}
1-\overline{z_1}z_2 u_1u_2\pm m_1m_2 & 0 \\
0 & m_1m_2\pm(1-z_1\overline{z_2} u_1u_2)
\end{matrix}\right).
\end{equation}

Using that (here $*$ denotes some explicit term that we do not need to compute)
\[
M_1E_\pm M_2=\left(\begin{matrix}
m_1m_2\pm z_1\overline{z_2}u_1u_2 & * \\
* & \pm m_1m_2+\overline{z_1}z_2u_1u_2
\end{matrix}\right),
\]
we thus obtain
\begin{equation}
-\langle M_{12}^{E_-} E_+\rangle=\langle M_{12}^{E_+} E_-\rangle=\frac{\ii \Im[z_1\overline{z_2}]u_1u_2}{1+|z_1z_2|^2u_1^2u_2^2-m_1^2m_2^2-2u_1u_2\Re[z_1\overline{z_2}]}=:\ii  b(\eta_1,z_1,\eta_2,z_2).
\end{equation}
Notice that $b(\eta_1,z_1,\eta_2,z_2)$ is  real  by definition. The  denominator above  can also be written as
\begin{equation}
\label{eq:verusebetapm}
\beta_+\beta_-=1+|z_1z_2|^2u_1^2u_2^2-m_1^2m_2^2-2u_1u_2\Re[z_1\overline{z_2}]>0,
\end{equation}
where the last inequality follows by \eqref{eq:rel}. In particular, \eqref{eq:verusebetapm} shows that even if $\beta_+,\beta_-$ may not be real their product is always real and positive. This proves the first relation in the second line of \eqref{eq:rel12}. 

By explicit calculation and using  \eqref{eq:verusebetapm},
we also obtain 
\begin{equation}
\label{eq:superusinfo}
\begin{split}
\pm \langle M_{12}^{E_\pm} E_\pm\rangle&=-1+\frac{\pm m_1m_2+1-\Re[\overline{z_1}z_2]u_1u_2}{1+|z_1z_2|^2u_1^2u_2^2-m_1^2m_2^2-2u_1u_2\Re[z_1\overline{z_2}]} 
=-1+\frac{\pm m_1m_2+(\beta_++\beta_-)/2}{\beta_+\beta_-} 
\end{split}
\end{equation}
hence, recalling that $|\beta_\pm|\sim \gamma$ by Lemma~\ref{lem:strongerb}, we have
\begin{equation}
|\langle M_{12}^{E_\pm} E_\pm\rangle|
\lesssim \frac{\rho_1\rho_2}{\beta_+\beta_-}+\frac{1}{\gamma} \lesssim\frac{1}{\gamma}.
\end{equation}
 In the last inequality we used
\[
\frac{\rho_1\rho_2}{\beta_+\beta_-}\sim \frac{\rho_1\rho_2}{\gamma^2} \le \frac{1}{\gamma}\cdot \frac{(\rho_1\rho_2)^{3/2}}{|\eta_1\eta_2|^{1/2}} \lesssim \frac{1}{\gamma},
\]
where we also used  the lower bound $\gamma\ge \sqrt{|\eta_1\eta_2|/(\rho_1\rho_2)}$ in the second inequality,  and that $\rho_i\lesssim |\eta_i|^{1/3}$ for $|z_i|\ge 1$ in the last inequality. This proves the bounds in the first line of \eqref{eq:rel12}.  Note that \eqref{eq:superusinfo} also shows \eqref{eq:neednow}
since $-\sgn (m_1 m_2) = \sgn( \Im m_1 \Im m_2)= \sgn (\eta_1\eta_2)  $.

 In order to conclude the remaining bound in the second line of \eqref{eq:rel12} we estimate
\[
 |b(\eta_1,z_1,\eta_2,z_2)|=\left|\frac{\Im[z_1\overline{z_2}]u_1u_2}{\beta_+\beta_-}\right|
 \lesssim \frac{|z_1-z_2|}{\gamma^2} \le \frac{1}{\gamma},
\]
 using  $ |\Im[z_1\overline{z_2}]| \lesssim |z_1-z_2|$ and then 
 $\gamma\ge |z_1-z_2|$ in the last step.
\end{proof}

\begin{proof}[Proof of Lemma~\ref{lem:nclfp}]
From the definition of $a_t$ in \eqref{eq:defab}, it follows 
(for both cases $\eta_{1,t}\eta_{2,t}<0$ and $\eta_{1,t}\eta_{2,t}>0$)\nc
\begin{equation}
\begin{split}
a_t&=\frac{|m_{1,t}m_{2,t}|+m_{1,t}^2m_{2,t}^2+\Re[z_{1,t}\overline{z_{2,t}}]u_{1,t}u_{2,t}-|z_{1,t}z_{2,t}|^2u_{1,t}^2u_{2,t}^2}{1+|z_{1,t}z_{2,t}|^2u_{1,t}^2u_{2,t}^2-m_{1,t}^2m_{2,t}^2-2u_{1,t}u_{2,t}\Re[z_{1,t}\overline{z_{2,t}}]} \\
&=\frac{|m_{1,t}m_{2,t}|}{|\beta_{+,t}\beta_{-,t}|}-\frac{1}{2}\partial_t\log|\beta_{+,t}\beta_{-,t}|,
\end{split}
\end{equation}
 where we used \eqref{eq:verusebetapm} and that by Lemma~\ref{lem:usinf} we have \nc
\begin{equation}\label{mu}
m_{i,t}=e^{t/2}m_{i,0}, \qquad\quad u_{i,t}=e^tu_{i,0},\qquad\quad z_{i,t}=e^{-t/2}z_{i,0}.
\end{equation}
This concludes the proof of \eqref{eq:impeqpm}.  The asymptotics in \eqref{eq:asympev} immediately follows from 
$\beta_{\pm,t}=\beta_\pm(\eta_{1,t},z_{1,t},\eta_{2,t},z_{2,t})$ and Lemma~\ref{lem:strongerb}.
\end{proof}

 \begin{proof}[Proof of \eqref{local_2g_M}] 
In Lemma~\ref{lem:propertms} we already prove bounds 
$|\langle M_{12}^{A} E_\pm\rangle|\lesssim 1/\gamma$ for $A\in\{E_+,E_-\}$. Completely
analogous calculations give the same bound for $|\langle M_{12}^{A} F^{(*)}\rangle|$.
Since $M_{12}^{E_\pm}$  is block-constant matrix and $\{ E_+, E_-, F, F^*\}$ 
is a basis in the $2\times 2$ matrices, this proves
\begin{equation}
\label{eq:boundspecA}
\left\lVert M_{12}^A\right\rVert \lesssim \frac{1}{|\beta_+|}+\frac{1}{|\beta_-|}\sim \frac{1}{\gamma}
\end{equation}
for $A\in\{E_+,E_-\}$.
Note that so far we proved \eqref{eq:boundspecA} only when $A\in\{E_+,E_-\}$
 which is the case needed in this paper.  
For completeness we mention that to prove the bound for general deterministic matrices $A\in\C^{2n\times 2n}$ we need to rely on the spectral decomposition of the non--Hermitian operator $\mathcal{B}_{12}$ similarly to what it was done in 
\cite[Eq. (5.20)]{CES22} for the bulk regime; the only difference is that now we would need to project the general matrix $A$ into the two--dimensional subspace corresponding to the eigenvalues $\beta_\pm$, whilst in the bulk regime only $\beta_-$ matters (i.e. we had projection into a one--dimensional subspace). We do not write the details of this 
elementary proof as it is not needed for our analysis.
\end{proof}

\begin{proof}[Proof of Lemma~\ref{lem:chirid}]
We present the proof of the first two equalities in \eqref{eq:identities}, the proof of the remaining ones is analogous and so omitted.  In  this proof we will omit the $z,\eta$--dependence of the various quantities, as in Lemma~\ref{lem:chirid} all the quantities are evaluated at $\eta_1=\eta_2=\eta$ and $z_1=z_t=z$; in particular, we use $\mathcal{B}_{12}=\mathcal{B}$. \nc By the chiral symmetry of $W-\Lambda$ we have $GE_-=-E_-G^*$. This follows by spectral decomposition
\[
E_-GE_-=\sum_i\frac{E_-\bm{w}_i\bm{w}_i^* E_-}{\lambda_i-\ii\eta}=\sum_i\frac{\bm{w}_{-i}\bm{w}_{-i}^*}{\lambda_i-\ii\eta}=\sum_i\frac{\bm{w}_i\bm{w}_i^*}{\lambda_{-i}-\ii\eta}=-\sum_i\frac{E_-\bm{w}_i\bm{w}_i^* E_-}{\lambda_i+\ii\eta}=-G^*,
\]
where we used that $E_-\bm{w}_i=\bm{w}_{-i}$ and that $\lambda_{-i}=-\lambda_i$. Noticing that $E_-^2=E_+$ this shows $GE_-=-E_-G^*$. Then we compute
 \begin{equation}
\label{eq:impcanc}
\langle GFGE_-\rangle=-\langle G FE_- G^*\rangle=\frac{1}{\eta}\langle \Im GF\rangle=0,
\end{equation}
where we used $GE_-=-E_-G^*$ in the first equality and that $\Im G_t$ is a diagonal matrix, which again follows by the chiral symmetry, in the last equality.

Next, we compute
\begin{equation}
\label{eq:MFM}
MFM=\left(\begin{matrix}
-\overline{z}um & m^2 \\
-\overline{z}^2u^2 & -\overline{z}um
\end{matrix}\right)
\end{equation}
using \eqref{Mmatrix}. Since the diagonal blocks are the same, and the same remains true once we compute 
$\mathcal{B}^{-1}[MFM]$
this shows that $\langle M^FE_-\rangle=0$.

\end{proof}

\nc

\begin{proof}[Proof of Lemma~\ref{lem:Mbounds}]

We start noticing that
\begin{equation}
(\mathcal{B}^{-1})^*[F]=\left(\begin{matrix}
\frac{m(1-|z|^2u^2)\overline{z}u+m^3\overline{z}u}{\beta_+\beta_-} & 1\\
0 & \frac{m(1-|z|^2u^2)\overline{z}+m^3\overline{z}u}{\beta_+\beta_-}
\end{matrix}\right).
\end{equation}

Recall that $m=\ii\pi \rho$, and from \eqref{rho} that
\[
\rho\sim \frac{\eta}{|z|^2-1+\eta^{2/3}}, \qquad 1-u= \frac{\eta}{\eta+\Im m}.
\]
Then, this implies
\[
1-|z|^2u^2+m^2=1-u =\frac{\eta}{\Im m}\left(1+O\left(\frac{\eta}{\rho}\right)\right),
\]
where we used $-u=m^2-|z|^2u^2$ (this follows by \eqref{m_function}) in the first equality. Using the definition of $\beta_\pm$ in \eqref{eq:dedfevalues}, by similar computations, this implies that
\[
\beta_\pm\ge \frac{\eta}{\Im m}\left(1-C\frac{\eta}{\rho}\right),
\]
for some $C>0$. We thus conclude that taking the absolute value entry--wise we have
\begin{equation}
(\mathcal{B}^{-1})^*[F]\dot{\leq} \left(\begin{matrix}
\frac{\Im m^2}{\eta}+C_1\rho & C_2\\
0 &\frac{\Im m^2}{\eta}+C_1\rho
\end{matrix}\right),
\end{equation}
for some constants $C_1,C_2>0$, where we used that $|zu|\le 1$. Here the notation $\dot{\leq}$ denotes that the inequality holds entry--wise. This concludes the proof  of the first two bounds in \eqref{eq:addbneedef} \nc using \eqref{eq:MFM} and that
\[
M^2=\left(\begin{matrix}
m^2+|z|^2u^2 & -2zum \\
-2\overline{z}um & m^2+|z|^2u^2
\end{matrix}\right),
\]
together with $\rho^3/\eta\gtrsim\rho^2$. The third (and last) bound in \eqref{eq:addbneedef} is a 
special case of \eqref{local_2g_M} as $z_1=z_2$ and $\eta_1=\eta_2$, which also imply $\Im m_1=\Im m_2$.
\end{proof}

 \subsection{Derivation of equation and bounds for products of resolvents}
 \label{sec:rand}

\begin{proof}[Proof of \eqref{eq:absvalbound}]

Recall that $G_{i,t}=G_t^{z_{i,t}}(\eta_{i,t})$ and that $t\le \tau$, with $\tau$ being the stopping time defined in \eqref{eq:deftildetau}.  The main input to prove \eqref{eq:absvalbound} is the following integral representation which expresses $|G^z|$ in terms of its imaginary parts (see e.g. \cite[Eq. (5.4)]{Optimal_local}):
\begin{equation}
\label{eq:absv}
\big|G^z(\ii\eta)\big|=\frac{2}{\pi}\int_0^\infty \frac{\Im G^z\big(\ii\sqrt{v^2+\eta^2}\big)}{\sqrt{v^2+\eta^2}}\,\dd v
= \frac{2}{\pi}\int_0^{n^{100}} \frac{\Im G^z\big(\ii\sqrt{v^2+\eta^2}\big)}{\sqrt{v^2+\eta^2}}\,\dd v+ O(n^{-10}).
\end{equation}
We point out that the same integral representation holds true for the deterministic approximation of $\big|G^z(\ii\eta)\big|$.
The integration regime $v\ge n^{100}$ can be easily estimated by an error 
$n^{-10}$ in operator norm; we omit this cutoff for brevity, but in the application below it is important that 
the $s$-integration is always on a finite regime $v\le n^{100}$.

In order to estimate all the $v$--regimes in our application of \eqref{eq:absv} below, we need to ensure a bound for $G_t^{z_{1,t}}(\ii\eta_1)G_t^{z_{2,t}}(\ii\eta_2)$ for all $\eta_i\in [\eta_{i,t},n^{100}]$ and not just $\eta_i =  \eta_{i,t}$.
The more delicate regime, when both $\eta_i$'s are such that $\eta_i\le \omega_1/2$ for some small fixed 
constant $\omega_1$,  has been covered in \eqref{eq:ufhig}. If at least for one $\eta_i$ we have $\eta_i\gtrsim 1$ then it is easy to see (similarly to \cite[Appendix B]{Optimal_local}) that we get the bound $1/(n\eta_*)$, with $\eta_*:=\eta_1\wedge \eta_2$,
 which is even stronger than our target $1/(\sqrt{n\ell_t}\gamma_t)$\nc. We thus obtain
\begin{equation}
\label{eq:ufhig2}
\sup_{\eta_{i,t}\le\widehat{\eta}_i\le n^{100}}(n\widehat{\ell}_t)^{1/2}\widehat{\gamma}_t
\big|\langle \big(G_t^{z_{1,t}}(\ii\widehat{\eta}_1)A_1G_t^{z_{2,t}}(\ii\widehat{\eta}_2)-
M_{12}^{A_1}(\widehat{\eta}_1,z_{1,t},\widehat{\eta}_2,z_{2,t})\big)A_2\rangle\big|\le n^{2\xi}.
\end{equation}

Define $\eta_{2,t,v}:=\sqrt{\eta_{2,t}^2+v^2}$, then we estimate
\begin{equation}
\label{eq:absintrep}
\begin{split}
\langle \Im G_{1,t}A_1|G_{2,t}|A_1^*\rangle&=\frac{2}{\pi}\int_0^{n^{100}} \langle \Im G_{1,t}A_1 \Im  G_t^{z_{2,t}}(\ii\eta_{2,t,v})A_1^*\rangle\, \frac{\dd v}{\eta_{2,t,v}} \\
&\lesssim\int_0^{n^{100}}\left(\frac{1}{\gamma_t}+\frac{n^{2\xi}}{\sqrt{n\ell_t}\gamma_t}\right)\, \frac{\dd v}{\eta_{2,t,v}} \\
&\lesssim \frac{\log n}{\gamma_t}+\frac{n^{2\xi}\log n}{\sqrt{n\ell_t}\gamma_t},
\end{split}
\end{equation}
where in the first inequality we used that $(\eta_{1,t}\rho_{1,t})\wedge (\eta_{2,t,v}/\rho^{z_{2,t}}(\ii\eta_{2,t,v}))\ge \ell_t$ and that
\[
\frac{1}{|z_{1,t}-z_{2,t}|+(\eta_{1,t}/\rho_{1,t})+(\eta_{2,t,v}/\rho^{z_{2,t}}(\ii\eta_{2,t,v}))}\lesssim \frac{1}{|z_{1,t}-z_{2,t}|+(\eta_{1,t}/\rho_{1,t})+(\eta_{2,t}/\rho_{2,t})}\sim \frac{1}{\gamma_t}
\]
as a consequence of $\eta_{2,t,v}/\rho^{z_{2,t}}(\ii\eta_{2,t,v})\ge\eta_{2,t}/\rho_{2,t}$.

\end{proof}

\begin{proof}[Derivation of \eqref{eq:flowbefchar} and \eqref{eq:k=1}]

We present only the proof of \eqref{eq:flowbefchar}. The proof of \eqref{eq:k=1} is completely analogous (in fact simple since it involves only one resolvent) and so omitted.

Denote $R_t:=\langle G_{1,t}A_1G_{2,t}A_2\rangle$. We now see $R_t$ as a function of the matrix valued stochastic process $W_t$ from \eqref{eq:bigstoch}. Then, by It\^{o}'s formula we have\footnote{For two Brownian motions $b_{1,t}, b_{2,t}$ by $\dd\langle b_{1,t},b_{2,t}\rangle$ we denote its covariation process.}
\begin{equation}
\label{eq:itonew}
\begin{split}
\dd R_t&=\sum_{a,b=1}^{2n} (\partial_{ab} R_t)(\dd W_t)_{ab} +\frac{1}{2}\sum_{a,b,c,d,=1}^{2n} \partial_{ab}\partial_{cd} R_t \dd \langle \mathfrak{B}_{ab,t}, \mathfrak{B}_{cd,t}\rangle \\
&\quad+\sum_{i=1}^2\big[(\partial_tz_{i,t})\partial_{z_{i,t}}R_t+(\partial_t\overline{z_{i,t}})\partial_{\overline{z_{i,t}}}R_t\big]\dd t+\sum_{i=1}^2(\partial_t\eta_{i,t})\partial_{\eta_{i,t}}R_t\dd t.
\end{split}
\end{equation}
Next, we notice
\begin{equation}
\label{eq:qvbms}
\dd \langle \mathfrak{B}_{ab,t}, \mathfrak{B}_{cd,t}\rangle=\frac{1}{n}\bm1(1\le a\le n, n+1\le b\le 2n)\bm1(n+1\le a\le 2n,  1\le b\le 2n)\delta_{ad}\delta_{bc}\dd t.
\end{equation}
As a consequence of the block structure of $W_t$, we will consider $(a,b)$-summations for indices running only for indices $1\le a\le n, n+1\le b\le 2n$ and $n+1\le a\le 2n,  1\le b\le 2n$. For these sums we introduce the short-hand notation $\sum_{ab}^*$. Using \eqref{eq:bigflow} and \eqref{eq:qvbms} in \eqref{eq:itonew} we obtain
\begin{equation}
\label{eq:impeqnew}
\begin{split}
\dd R_t&=-\frac{1}{2}\sum_{a,b=1}^{2n} (\partial_{ab} R_t)(W_t)_{ab} \dd t+\sum_{a,b=1}^{2n} (\partial_{ab} R_t)\frac{(\dd \mathfrak{B}_t)_{ab}}{\sqrt{n}}+\frac{1}{2n}\sum_{ab}^* \partial_{ab}\partial_{ba} R_t \dd t\\
&\quad+\sum_{i=1}^2\big[(\partial_tz_{i,t})\partial_{z_{i,t}}R_t+(\partial_t\overline{z_{i,t}})\partial_{\overline{z_{i,t}}}R_t\big]\dd t+\sum_{i=1}^2(\partial_t\eta_{i,t})\partial_{\eta_{i,t}}R_t\dd t.
\end{split}
\end{equation}
The second term in the RHS of \eqref{eq:impeqnew} gives the stochastic term in \eqref{eq:flowbefchar}. We are thus left with computing the remaining ones. We now repeatedly use $\partial_{ab}G_{i,t}=-G_{i,t}\Delta^{ab}G_{i,t}$. Here $\Delta^{ab}$ denotes a $2n\times 2n$ matrix whose only non-zero entry is the $(a,b)$-entry. We start computing the first term in the RHS of \eqref{eq:impeqnew} (recall the definition of $\Lambda_{i,t}$ from \eqref{def:Lambda} and $G_{i,t}=(W_{t}-\Lambda_{i,t})^{-1}$):
\begin{equation}
\label{eq:computation1}
\begin{split}
-\frac{1}{4n}\sum_{a,b=1}^{2n} (\partial_{ab} R_t)(W_t)_{ab}&=\frac{1}{4n}\sum_{a,b=1}^{2n} \big[(G_{1,t}A_1G_{2,t}A_2G_{1,t})_{ba}+(G_{2,t}A_2G_{1,t}A_1G_{2,t})_{ba}\big] (W_t)_{ab}\\
&=\frac{1}{2}\big[\langle G_{1,t}A_1G_{2,t}A_2G_{1,t} W_t\rangle+\langle G_{2,t}A_2G_{1,t}A_1G_{2,t} W_t\rangle \big] \\
&=\langle G_{1,t}A_1G_{2,t}A_2\rangle\dd t +\frac{1}{2}\big[\langle G_{1,t}A_1G_{2,t}A_2G_{1,t} \Lambda_{1,t}\rangle+\langle G_{2,t}A_2G_{1,t}A_1G_{2,t} \Lambda_{2,t}\rangle \big].
\end{split}
\end{equation}
For the third term in the RHS of \eqref{eq:impeqnew} we have (recall that the $(a,b)$-summation is symmetric in $a$ and $b$):
\begin{equation}
\label{eq:computation2}
\begin{split}
\frac{1}{2n}\sum_{ab}^* \partial_{ab}\partial_{ba} R_t&=-\frac{1}{4n^2}\sum_{ab}^* \partial_{ba} \big[(G_{1,t}A_1G_{2,t}A_2G_{1,t})_{ba}+(G_{2,t}A_2G_{1,t}A_1G_{2,t})_{ba}\big] \\
&=\frac{1}{4n^2}\sum_{ab}^* \big[2(G_{1,t})_{aa}(G_{1,t}A_1G_{2,t}A_2G_{1,t})_{bb}+(G_{1,t}A_1G_{2,t})_{aa}(G_{2,t}A_2G_{1,t})_{bb}\big] \\
&\quad+ \frac{1}{4n^2}\sum_{ab}^* \big[2(G_{2,t})_{aa}(G_{2,t}A_2G_{1,t}A_1G_{2,t})_{bb}+(G_{2,t}A_2G_{1,t})_{aa}(G_{1,t}A_1G_{2,t})_{bb}\big]\\
&=\langle G_{1,t}\rangle\langle G_{1,t}^2A_1G_{2,t}A_2\rangle+2\sum_{i\ne j} \langle G_{1,t} A_1G_{2,t}E_i\rangle\langle G_{2,t}A_2G_{1,t}E_j\rangle\\
&\quad+\langle G_{2,t}\rangle\langle G_{2,t}^2A_2G_{1,t}A_1\rangle,
\end{split}
\end{equation}
where in the last equality we used $\langle G_{i,t}E_1\rangle=\langle G_{i,t}E_2\rangle=\langle G_{i,t}\rangle/2$. Next, we compute the first term in the second line of \eqref{eq:impeqnew}, using the second relation in (\ref{eq:scalchar}) (we compute only the term for $i=1$):
\begin{equation}
\label{eq:computation3}
\begin{split}
(\partial_tz_{1,t})\partial_{z_{1,t}}R_t+(\partial_t\overline{z_{1,t}})\partial_{\overline{z_{1,t}}}R_t&=-\frac{z_{1,t}}{2}\langle G_{1,t}FG_{1,t}A_1G_{2,t}A_2\rangle-\frac{\overline{z_{1,t}}}{2}\langle G_{1,t}F^*G_{1,t}A_1G_{2,t}A_2\rangle \\
&=-\frac{1}{2}\langle G_{1,t}Z_{1,t}G_{1,t}A_1G_{2,t}A_2\rangle,
\end{split}
\end{equation}
with the matrix $Z_{1,t}$ defined as in (\ref{def:Lambda}) with $z=z_1$. For the second term in the second line of \eqref{eq:impeqnew}, using the first relation in (\ref{eq:scalchar}), we have (we compute only the term for $i=1$):
\begin{equation}
\label{eq:computation4}
(\partial_t\eta_{1,t})\partial_{\eta_{1,t}}R_t=-\left( m_{1,t}+\frac{\ii\eta_{1,t}}{2}\right) \langle G_{1,t}^2A_1G_{2,t}A_2\rangle,
\end{equation}
where we used that $\<M_{i,t}\>=m_{i,t}=\ii\Im m_{i,t}$. Finally, combining \eqref{eq:computation1}--\eqref{eq:computation4} and using that $\Lambda_{i,t}=Z_{i,t}+\ii \eta_{i,t}$ from (\ref{def:Lambda}) we obtain \eqref{eq:flowbefchar}.

\end{proof}

\begin{proof}[Derivation of \eqref{eq:quadvarnew}] The martingale term in \eqref{eq:fulleqaasimp} is given by
\begin{equation}
\label{eq:mart}
\sum_{a,b=1}^{2n}\partial_{ab}\langle G_{1,t}A_1G_{2,t}A_2\rangle \frac{\dd \mathfrak{B}_{ab,t}}{\sqrt{n}}=-\frac{1}{2n}\sum_{a,b=1}^{2n} \big[(G_{1,t}A_1G_{2,t}A_2G_{1,t})_{ab}+(G_{2,t}A_2G_{1,t}A_1G_{2,t})_{ab}\big] \frac{\dd \mathfrak{B}_{ab,t}}{\sqrt{n}},
\end{equation}
where we used that $\partial_{ab}G_{i,t}=-G_{i,t}\Delta^{ab}G_{i,t}$. Using \eqref{eq:qvbms}, we obtain that the quadratic variation process of \eqref{eq:mart} is given by
\[
\frac{1}{2n^3}\sum_{ab}^* \big|(G_{1,t}A_1G_{2,t}A_2G_{1,t})_{ba}+(G_{2,t}A_2G_{1,t}A_1)_{ba}\big|^2\dd t.
\]
Performing the $a,b$-summation we immediately obtain \eqref{eq:quadvarnew}.

\end{proof}

\nc

\section{Proof of Lemma \ref{lemma_approx}} \label{app:lemma}
The proof of Lemma \ref{lemma_approx} follows the one of \cite[Lemma 6.1-6.2]{EYY12} or \cite[Lemma 2.4-2.5]{SX22}, and relies on the rigidity bound of eigenvalues in Corollary \ref{cor:rigidity}.
\begin{proof}[Proof of Lemma \ref{lemma_approx}]
	For any $E>0$, we define $\chi_{E}:=\one_{[-E,E]}$, then $\#\{j:  |\lambda^z_j| \leq E\}=\Tr \chi_E(H^z)$. For $\eta>0$, we define the mollifier 
	\begin{equation}\label{le mollifier}
		\theta_{\eta}(x):=\frac{\eta}{\pi(x^2+\eta^2)}=\frac{1}{\pi} \Im \frac{1}{x-\ii \eta}.
	\end{equation}
	We can thus relate $\Tr \chi_{E} \star \theta_{\eta}(H)$ to the Green function by the following identity:
	\begin{equation}\label{link0}
		\Tr \chi_{E} \star \theta_{\eta}(H^z)=\frac{1}{\pi} \int \chi_E(y) \Im \Tr \gz (y+\ii \eta) \dd y=\frac{1}{\pi} \int_{-E}^{E} \Im \Tr \gz(y+\ii \eta) \dd y\,.
	\end{equation}
Choose the spectral parameters $E \gg l \gg l' \gg \eta$ as in (\ref{eta}). Using \cite[Eq.(A.2)-(A.4))]{SX22}, we have
	\begin{align}
		\Big| \Tr \big(\chi_{E}-\chi_{E} \star \theta_{\eta}\big)(H^z) \Big| \lesssim \#\{ |\lambda^z_j| \in [E-l',E+l']\}+\frac{\eta}{l'}\#\{ |\lambda^z_j|\leq E-l'\}+\Tr g(H^{z}),
	\end{align}
	with
	$g(x):=\Tr \chi_{E} \star \theta_{\eta}(x) \one_{|x|\geq E+l'}$. Using the rigidity estimate of eigenvalues in (\ref{rigidity3}) and the properties of $\rho^z$ in (\ref{rho_E}), 
		we have
	\begin{align}\label{eq1}
		\Big| \Tr \big(\chi_{E}-\chi_{E} \star \theta_{\eta}\big)(H^z) \Big| \lesssim \#\{ |\lambda^z_j| \in [E-l',E+l']\}+\frac{n^\xi\eta}{l'}+\Tr g(H^{z}),
	\end{align}
	for any small $\xi>0$ with very high probability. It is then sufficient to estimate $\Tr g(H^z)$.  Similarly to \cite[Eq.(A.6)-(A.7))]{SX22}, if $x\geq E+l'$ with $E$ and $l'$ chosen in (\ref{eta}), then $x \pm E \geq l' \gg \eta$, thus we have
	\begin{align}\label{bound_g}
		g(x) =&\int_{-E-x}^{E-x} \theta_{\eta}(y) \dd y
		\lesssim 
		\min\Big\{\frac{\eta E}{(x-E)^2}, \frac{ \eta}{x-E}\Big\}.
	\end{align}
Choose a sufficiently small $c>0$ and we write
	\begin{align}\label{trace_f}
		\Tr g(H^z) =2\sum_{|\lambda^z_i| \geq E +l'}g(\lambda^z_i) =2 \Big(\sum_{k=0}^{K_0} \sum_{\lambda^z_i \in \mathcal{I}_k}+\sum_{\lambda^z_i\geq c}\Big) g(\lambda^z_i)  \,, \qquad \mathcal{I}_k:=[E+3^{k}l', E+3^{k+1}l')\,,
	\end{align}
where $[E+l',c]=\bigcup_{k=0}^{K_0} \mathcal{I}_k$ with $K_0=O(\log n)$, using that $g$ is an even function and the spectrum $\{\lambda^z_{\pm i}\}_{i\in [n]}$ is symmetric. Combining with the rigidity estimate of eigenvalues in (\ref{rigidity3}) and (\ref{bound_g}), we have
	\begin{align}\label{upper_f}
		\Tr g(H^z) \lesssim \sum_{k=0}^{K_0} \min\Big\{\frac{\eta E}{(3^k l')^2}, \frac{ \eta}{3^k l'}\Big\} \Big(\#\{i: \lambda^z_i \in \mathcal{I}_k\} \Big)+n\eta E \lesssim \frac{n^\xi\eta}{l'},
	\end{align}
for any small $\xi>0$ with very high probability. Combining this with (\ref{link0}) and (\ref{eq1}), we have proved
\begin{align}\label{eq_upper_f}
	\Big| \Tr \big(\chi_{E}-\chi_{E} \star \theta_{\eta}\big)(H^z)  \Big| \lesssim \#\{ |\lambda^z_j| \in [E-l',E+l']\}+\frac{n^\xi\eta}{l'},
	\end{align}
for any small $\xi>0$ with very high probability and hence (\ref{approx}).
 Moreover, using that $E\mapsto\chi_E$ is increasing, together with (\ref{eq_upper_f}), we obtain that
	\begin{align}
		\Tr \chi_{E}(H^z) \leq& \frac{1}{l}\int_{E}^{E+l}\Tr \chi_{E'}(H^z) \dd E' \nonumber\\
		\leq & \frac{1}{l}\int_{E}^{E+l}\Tr \chi_{E} \star \theta_{\eta} (H^z)  \dd E' + \frac{C}{l}\int_{E}^{E+l} \Big(\#\{ |\lambda^z_i| \in [E'-l',E'+l']\}+\frac{n^\xi \eta}{l'} \Big) \dd E'\nonumber\\
		\leq & \Tr \chi_{E+l} \star \theta_{\eta}(H^z) +Cn^\xi \Big(\frac{\eta}{l'} +\frac{l'}{l}\Big),
	\end{align}
	for any small $\xi>0$ with very high probability. One can obtain a lower bound similarly.  With the choices of parameters in (\ref{eta}) and (\ref{link0}), we hence finished the proof of (\ref{approx_1}).
\end{proof}

\section{Proof of Proposition~\ref{prop_zz}}\label{app:tail_gft}

The proof of this proposition consists of two parts: 1) we first prove Proposition~\ref{prop_zz} for the matrix with a small ($\sim n^{-1/2+\omega_1}$) Gaussian component, 2) we remove the Gaussian component using a GFT argument. In particular, Part 1) already proves Proposition~\ref{prop_zz} 
for Ginibre matrices. 

We will only prove Proposition~\ref{prop_zz} for any $\{ E_j\}_{j\in [k]}$ with $ n^{-3/4-\gamma \omega+\tau} < E_j \lesssim n^{-3/4}$ and $\tau\leq (\gamma\omega)/10$, while the same upper bound applies to the complementary regime $0<E_j \leq n^{-3/4-\gamma \omega+\tau}$ using a simple monotonicity argument~(the lower bound is a completely trivial
in this case). \nc

\begin{proof}[Proof of Part 1) of Proposition~\ref{prop_zz}]

This proof is similar to the proof of \cite[Proposition 4.3]{SpecRadius}, so we will be brief.
For simplicity of notation we only present the proof for $k=2$. Fix a time  $t_1:=n^{-1/2+\omega_1}$, for some small $\omega_1>0$. Consider the matrix valued flow\footnote{We warn the reader that within Appendix~\ref{app:tail_gft},
  to keep the same notation of \cite[Section 7]{SpecRadius}, by $X_t$ we denote the solution of the flow \eqref{flow0}, and not the solution of the Ornstein--Uhlenbeck flow from \eqref{W}.}
\begin{align}\label{flow0}
\dd X_t=\frac{\dd B_t}{\sqrt{n}}, \qquad\quad X_0= \ee^{-t_1/2} X \nc, \qquad t\geq 0,
\end{align}
where $X$ is an i.i.d. complex matrix satisfying Assumption~\ref{ass:mainass},  
and $B_t$ is an $n\times n$ matrix whose entries are i.i.d. standard complex Brownian motions. 
 Hence, $X_{\widehat ct_1}$, with $\widehat c=\widehat c(t_1):=\sqrt{1-e^{-t_1}}/t_1$,
  is an i.i.d. matrix which satisfies Assumption~\ref{ass:mainass}, and also has a Gaussian component of size $\sim\sqrt{t_1}$. \nc By simple second order perturbation theory (see e.g. \cite[Appendix B]{CES19}) the eigenvalues $\lambda_i^z(t)$ of the Hermitization of $X_t-z$ are the unique strong solution of the following Dyson Brownian motion (DBM)
\begin{equation}
\label{eq:DBM}
\dd \lambda_i^z(t)=\frac{\dd b_i^z(t)}{\sqrt{2n}}+\frac{1}{2n}\sum_{j\ne i}\frac{1}{\lambda_i^z(t)-\lambda_j^z(t)}\dd t.
\end{equation}
Here $b_i^z(t)$ are i.i.d. standard real Brownian motions for $i\in [n]$, and $b_{-i}^z=-b_i^z$ for $i\in [n]$.
 Note that the collections $\{ b_i^z(t)\; : \;  i\in [n]\}$ are not independent for different $z$'s. \nc

To prove the (asymptotic) independence of $\lambda_1^{z_j}(t)$ for different $z_j$'s, we couple the evolution \eqref{eq:DBM} with the evolution of fully independent points $\{\mu_i\}$ 
and show that after short time $\sim n^{-1/2+\omega_1}$ these two processes are very close. More precisely, for $l=1,2$, define
$\{ \mu_i^{(l)}(t)\}$ as the solution to
\begin{equation}
\label{eq:DBMmu}
\dd \mu_i^{(l)}(t)=\frac{\dd \beta_i^{(l)}(t)}{\sqrt{2n}}+\frac{1}{2n}\sum_{j\ne i}\frac{1}{\mu_i^{(l)}(t)-\mu_j^{(l)}(t)}\dd t,
\end{equation}
with the singular values of two independent complex Ginibre matrices as initial conditions.
Here the family $\{ \beta_i^{(l)}:\, i\in [n],\, l=1,2\}$ is a $2n$-dimensional  standard Brownian motion 
and $\beta_{-i}^{(l)}=-\beta_i^{(l)}$. 
 For any $z_l \in\C$ such that $0\leq |z_l|-1\leq n^{-1/2+\tau}$ and $|z_1-z_2| \gtrsim n^{-1/2+\omega}$
 for any sufficiently small $\tau,\omega$  with $\tau<\gamma \omega/10$,  
 Corollary~\ref{eigenvector_overlap} implies, by a standard grid argument, 
that 
\begin{align}
	|\langle \mathbf{u}_i^{z_1}(t),\mathbf{u}_j^{z_2}(t)\rangle|^2+|\langle \mathbf{v}_i^{z_1}(t),\mathbf{v}_j^{z_2}(t)\rangle|^2 \prec n^{-\omega+2\tau}=: n^{-\omega_E}, \qquad 
	\mbox{ $|i|,|j|\le n^{\omega_c}$},
\end{align}
holds
simultaneously in $0\leq t\leq \widehat c t_1$, for any small constants $\omega_c,\omega_E$ such that $10 \omega_1\leq 10 \omega_c \leq \omega_E \leq \omega/2$. 
Thus the assumptions of \cite[Theorem 7.2]{SpecRadius} are verified, there hence exists a specific
coupling between  the 
pair of correlated Brownian motions $\{ b_i^{z_1},b_i^{z_2}\, : \, i\in [n]\}$
and the independent ones $\{\beta_i^{(1)},\beta_i^{(2)}\, :\, i\in [n] \}$ such that, for $l=1,2$, we have
\begin{equation}
\label{eq:independencebound}
\big|\lambda_i^{z_l}(\widehat ct_1)-\mu_i^{(l)}(\widehat ct_1)\big|\prec n^{-3/4-\gamma\omega+\tau}, \qquad\quad |i|\le n^{\wh \omega},
\end{equation}
  for any sufficiently \nc small constants $\gamma,\omega,\widehat{\omega},\tau$ 
  satisfying $\widehat{\omega}\le(\gamma\omega)/10\le \omega_1/100$  and $\tau\le (\gamma\omega)/10$\nc. 
  The bound \eqref{eq:independencebound} holds with very high probability in the joint probability space of the $\lambda_i^{z_l}(\widehat ct_1)$'s and $\mu_i^{(l)}(\widehat ct_1)$'s. We thus obtain
\begin{align}\label{P_tt}
\P\Big( \lambda_1^{z_l}(\widehat ct_1) \le E_l,~l\in [2]\Big)&\le \P\Big(\mu_1^{(l)}(\widehat ct_1) \le E_l+n^{-3/4-\gamma\omega+\tau},~l \in [2]\Big)+O(n^{-D})\nonumber \\
&= \prod_{l=1}^2 \P\Big(\mu_1^{(l)}(\widehat ct_1) \le E_l+n^{-3/4-\gamma\omega+\tau}\Big)+O(n^{-D}) \nonumber\\
&\le \prod_{l=1}^2 \P\Big( \lambda_1^{z_l}(\widehat ct_1) \le E_l+2n^{-3/4-\gamma\omega+\tau}\Big)+O(n^{-D})
\end{align}
for any large $D>0$. In the second line of (\ref{P_tt}), we used the fact that $\mu_1^{(1)}(t)$ and $\mu_1^{(2)}(t)$ are independent processes by construction. Similarly we obtain the desired lower bound as well. This proves Proposition~\ref{prop_zz} with $k=2$ for $X_{\widehat ct_1}$, i.e. for a matrix with a small ($\sim n^{-1/2+\omega_1}$) Gaussian component, and the case of general $k$ is analogous.
\end{proof}

\bigskip

The proof of Part 2) follows the strategy used in \cite[Proposition 5.2]{SpecRadius}. In particular, we will remove the small $\sim n^{-1/2+\omega_1}$ Gaussian component of $X_{\widehat ct_1}$ from (\ref{P_tt}) that has been obtained in Part 1). 
 Again we  give the proof only 
 for any $\{ E_j\}_{j\in [k]}$ with $ n^{-3/4-\gamma \omega+\tau} < E_j \lesssim n^{-3/4}$ and $\tau\leq (\gamma\omega)/10$.

\begin{proof}[Proof of Part 2) of Proposition \ref{prop_zz}]

Consider $\widehat{X}_t$ to be the solution of the Ornstein--Uhlenbeck flow
\begin{equation}\label{hat_X}
\dd \widehat{X}_t=-\frac{1}{2}\widehat{X}_t \dd t+\frac{\dd \widehat{B}_t}{\sqrt{n}}, \qquad\quad \widehat{X}_0= X,
\end{equation}
where $X$ is an i.i.d. complex matrix satisfying Assumption~\ref{ass:mainass}, and the entries of $\widehat{B}_t$ are standard complex i.i.d. Brownian motions. The corresponding probability and expectation for the matrix flow $\widehat{X}_t~(t\geq 0)$ are then denoted by $\widehat{\P}_{t}$ and $\widehat{\E}_{t}$, respectively.  We will now show that
\[
\widehat{\P}_0 \Big( \lambda^{z_j}_1 \leq E_j,~ j\in [k]\Big)\lesssim \prod_{j=1}^k \widehat{\P}_{t_1} \Big( \lambda^{z_j}_1 \leq  E_j+n^{-3/4-\gamma\omega+\tau}\Big)+n^{-D}
\]
and the similar lower bound for $t_1=n^{-1/2+\omega_1}$. This will conclude the proof of Proposition~\ref{prop_zz} noticing the following relation
\begin{equation}
\label{eq:reldbmou}
\widehat{X}_{t_1}\stackrel{\dd}{=}e^{-t_1/2}X+\sqrt{1-e^{-t_1}}X^{\mathrm{Gin}}\stackrel{\dd}{=}X_{\widehat ct_1},
\end{equation}
where $X^{\mathrm{Gin}}$ is a Ginibre matrix
 independent of $X$ and $X_{\widehat ct_1}$ is the solution of \eqref{flow0} from Part 1) and for such matrix we already proved the asymptotic factorisation in~(\ref{P_tt}).

Let $F: \R^k_+\rightarrow [0,1]$ be a smooth cut-off function given by $F (w_1,\cdots,w_k):=\Phi(w_1) \cdots \Phi(w_k)$ with the non-decreasing function $\Phi$ defined in (\ref{F_function}). Then for any fixed $E_j \in \R^+$ and any $t\geq 0$,
$$\widehat{\P}_t \Big( \lambda^{z_j}_1 \leq E_j,~ j\in [k]\Big)=\widehat{\E}_t \Big[F\Big(\#\{|\lambda^{z_1}_i| \leq E_1\}, \cdots, \#\{|\lambda^{z_k}_i| \leq E_k\} \Big)\Big].$$
For notational brevity in the proof below, we define the following short hand 
\begin{align}\label{short}
	\II^{z,\eta}_{E}:=\int_{|y|\leq E}  \Im \Tr \wh{G}_t^{z}(y+\ii \eta) \dd y,\qquad\qquad \forall E,\eta \in \R^+,~z\in \C,
\end{align}
where $\wh{G}_t^{z}$ is the corresponding resolvent of the Hermitised matrix defined as in (\ref{flow}) with $\wh{X}_t$ given by (\ref{hat_X}). Clearly $\II^{z,\eta}_{E}$ is an increasing function in $E \in \R^+$. 
 In this section we set $\epsilon_*:=\gamma \omega -\tau>0$ for notational simplicity, hence $n^{-3/4-\epsilon_*}\leq E_j\lesssim n^{-3/4}$. \nc
Then using the inequalities in (\ref{approx_1}) in Lemma \ref{lemma_approx} for $E=E_j,l=l_j=
n^{-\epsilon_*}E_j,\eta=\eta_j=n^{-3\epsilon_*}E_j$ chosen as in (\ref{eta}) with $\zeta=\epsilon_*$, we have
	\begin{align}\label{approx2}
	\widehat{\E}_t \Big[ F\big(\II_{E_1-l_1}^{z_1,\eta_1}, \cdots, \II_{E_k-l_k}^{z_k,\eta_k} \big)\Big]-n^{-D}\leq \widehat{\P}_t \Big( \lambda^{z_j}_1 \leq E_j,~ j\in [k]\Big) \leq \widehat{\E}_t \Big[ F\big(\II_{E_1+l_1}^{z_1,\eta_1}, \cdots, \II_{E_k+l_k}^{z_k,\eta_k} \big)\Big]+n^{-D},
\end{align}
for any $t\geq 0$ and any large constant $D>0$.
We remark that the inequalities in (\ref{approx2}) hold true if we replace $E_j$ with $E_j+ml_j$~(for any fixed $m\in \Z$).
It then suffices to show that, there exists some constant $c_0>0$ such that 
\begin{align}\label{goal}
\left| \big(\widehat{\E}_0- \widehat{\E}_{t_1}\big)\Big[ F\big(\II_{E_1 \pm l_1}^{z_1,\eta_1}, \cdots, \II_{E_k \pm l_k}^{z_k,\eta_k} \big)\Big]
\right| \lesssim  n^{-c_0} \prod_{j=1}^k \widehat{\P}_{t_1} \Big( \lambda^{z_j}_1 \leq  E_j+n^{-3/4-\epsilon_*} \Big)+n^{-D}.
\end{align}
Once we prove (\ref{goal}), using with the inequalities in (\ref{approx2}) for $t=0$ and $t=t_1$, we obtain that
\begin{align}\label{kk}
	\widehat{\P}_0 \Big( \lambda^{z_j}_1 \leq E_j,~ j\in [k]\Big) \leq& \widehat{\E}_0 \Big[ F\big(\II_{E_1+l_1}^{z_1,\eta_1}, \cdots, \II_{E_k+l_k}^{z_k,\eta_k} \big)\Big]+n^{-D}\nonumber\\
	\lesssim &\widehat{\E}_{t_1} \Big[ F\big(\II_{E_1 + l_1}^{z_1,\eta_1}, \cdots, \II_{E_k + l_k}^{z_k,\eta_k} \big)\Big]+ n^{-c_0} \prod_{j=1}^k \widehat{\P}_{t_1} \Big( \lambda^{z_j}_1 \leq E_j+n^{-3/4-\epsilon_*}\Big)+n^{-D}\nonumber\\
	\lesssim & \widehat{\P}_{t_1} \Big( \lambda^{z_j}_1 \leq E_j+2l_j,~ j\in [k]\Big)+ n^{-c_0} \prod_{j=1}^k \widehat{\P}_{t_1} \Big( \lambda^{z_j}_1 \leq  E_j+n^{-3/4-\epsilon_*}\Big)+n^{-D} \nonumber\\
	\lesssim &\prod_{j=1}^k \widehat{\P}_{t_1} \Big( \lambda^{z_j}_1 \leq  E_j+2 n^{-3/4-\epsilon_*}\Big)+n^{-D},
\end{align}
where in the last line we applied (\ref{P_tt}) from from Part 1) for $t=\widehat ct_1$, noticing the relation in \eqref{eq:reldbmou}. 
One obtains a similar lower bound with $n^{-3/4-\epsilon_*}$ replaced with $-n^{-3/4-\epsilon_*}$. In particular for $k=1$ 
\begin{align}\label{kk1}
	\widehat{\P}_{t_1} \Big( \lambda^{z}_1 \leq  E-2 n^{-3/4-\epsilon_*}\Big)-n^{-D} \lesssim \widehat{\P}_0 \big( \lambda^{z}_1 \leq E \big) \lesssim \widehat{\P}_{t_1} \Big( \lambda^{z}_1 \leq  E+ 2 n^{-3/4-\epsilon_*}\Big)+n^{-D}.
\end{align}
Combining (\ref{kk}) for a general $k\geq 2$ and (\ref{kk1}) for $E=E_j+2n^{-3/4+2\epsilon_*}$, we prove  Proposition \ref{prop_zz} for the initial matrix $\wh{X}_0=X$, hence finish the proof of Proposition \ref{prop_zz}  for any i.i.d. matrix.

\medskip

In the following, we will focus on proving (\ref{goal}). More precisely, we claim that, for any fixed $m\in \Z$, 
\begin{align}\label{iterate_m}
	\frac{\dd}{\dd t}\widehat{\E}_t\left[F\Big(\II_{E_1 +m l_1}^{z_1,\eta_1}, \cdots, \II_{E_k +m l_k}^{z_k,\eta_k} \Big)\right] \lesssim &~ n^{-1/4+c\epsilon_*}\widehat{\E}_t \Big[ F\big(\II_{E_1+(m+2)l_1}^{z_1,\eta_1}, \cdots, \II_{E_k+(m+2)l_k}^{z_k,\eta_k} \big)\Big]+n^{-D},
\end{align}
uniformly in $0\leq t \leq t_1$.
Assuming we have proved (\ref{iterate_m}), integrating (\ref{iterate_m}) up to time $t= t_1$, we have
\begin{align}\label{step0}
	\left| \big( \widehat{\E}_t-\widehat{\E}_{t_1} \big)\Big[F\big(\II_{E_1 +m l_1}^{z_1,\eta_1}, \cdots, \II_{E_k +m l_k}^{z_k,\eta_k} \big)\Big] \right|\lesssim &~ n^{-1/4+c\epsilon_*} \int_{t}^{t_1} \left|\widehat{\E}_t \Big[ F\big(\II_{E_1+(m+2)l_1}^{z_1,\eta_1}, \cdots, \II_{E_k+(m+2)l_k}^{z_k,\eta_k} \big)\Big] \right|+n^{-D},
\end{align}
uniformly for any $0\leq t \leq t_1$. Using (\ref{approx2}), we know that, for any fixed $m\in \Z$,
\begin{align}\label{Gaussian}
	\widehat{\E}_{t_1} \Big[ F\big(\II_{E_1+ml_1}^{z_1,\eta_1}, \cdots, \II_{E_k+ml_k}^{z_k,\eta_k} \big)\Big] \lesssim& \widehat{\P}_{t_1}\Big( \lambda^{z_j}_1 \leq E_j+(m+1)l_j, ~  j \in [k]\Big)+n^{-D}\nonumber\\
	\lesssim &\prod_{j=1}^k \widehat{\P}_{t_1}\Big( \lambda_1^{z_j} \leq  E_j+n^{-3/4-\epsilon_*} \Big)+n^{-D},
\end{align}
where we also used (\ref{P_tt}) from Part 1) and the relation in \eqref{eq:reldbmou}. From (\ref{step0}) with (\ref{Gaussian}) and using that $\F$ is uniformly bounded, we have an initial estimate
\begin{align}\label{step1}
 \widehat{\E}_t\Big[F\big(\II_{E_1 +m l_1}^{z_1,\eta_1}, \cdots, \II_{E_k +m l_k}^{z_k,\eta_k} \big)\Big] \lesssim &~ n^{-1/4+c\epsilon_*}+\prod_{j=1}^k \widehat{\P}_{t_1} \Big( \lambda_1^{z_j} \leq E_j+n^{-3/4-\epsilon_*} \Big)+n^{-D},
\end{align}
for any fixed $m\in\Z$ and $0\leq t\leq t_1$ uniformly. 
 Plugging (\ref{step1})~(replacing $m$ with $m+2$) into the right side of (\ref{step0}), we then obtain
\begin{align}\label{step2}
	\left| \big( \widehat{\E}_t-\widehat{\E}_{t_1} \big)\Big[F\big(\II_{E_1 +m l_1}^{z_1,\eta_1}, \cdots, \II_{E_k +m l_k}^{z_k,\eta_k} \big)\Big] \right| \lesssim &~\big(n^{-1/4+c\epsilon_*}\big)^2+ n^{-1/4+c\epsilon_*} \prod_{j=1}^k \widehat{\P}_{t_1}\Big( \lambda_1^{z_j} \leq  E_j + n^{-3/4-\epsilon_*}\Big) +n^{-D},
\end{align}
uniformly in $0\leq t \leq t_1$.
Combining again with (\ref{Gaussian}), we obtain that, for any fixed $m\in \Z$ and $0\leq t \leq t_1$,
\begin{align}\label{step3}
	\left| \widehat{\E}_t\Big[F\big(\II_{E_1 +m l_1}^{z_1,\eta_1}, \cdots, \II_{E_k +m l_k}^{z_k,\eta_k} \big)\Big] \right|\lesssim &~ \big(n^{-1/4+c\epsilon_*}\big)^2+\prod_{j=1}^k \widehat{\P}_{t_1}\Big( \lambda_1^{z_j} \leq  E_j+n^{-3/4-\epsilon_*} \Big)+n^{-D}.
\end{align}
In this way we have improved the first error term in (\ref{step1}) by an additional factor $n^{-1/4+c\epsilon_*}$.  We further plug the above estimate into (\ref{iterate_m}) and iterate the above process for sufficiently many times. In general for any $s$-th iteration step, the following holds true for any $m\in \Z$ and $0\leq t \leq t_1$,
\begin{align}\label{steps}
	\left| \big( \widehat{\E}_t-\widehat{\E}_{t_1} \big)\Big[F\big(\II_{E_1 +m l_1}^{z_1,\eta_1}, \cdots, \II_{E_k +m l_k}^{z_k,\eta_k} \big)\Big] \right| \lesssim &~ \big(n^{-1/4+c\epsilon_*}\big)^s +n^{-1/4+c\epsilon_*} \prod_{j=1}^k \widehat{\P}_{t_1}\Big( \lambda_1^{z_j} \leq  E_j+n^{-3/4-\epsilon_*} \Big)+n^{-D}.
\end{align}
We hence stop the iteration at $s=O(D)$ to improve the first error in (\ref{steps}) to $n^{-D}$ for any large $D>0$. This finishes the proof of (\ref{goal}) with $t=0$.

It still remains to prove (\ref{iterate_m}). Recall that the Hermitized matrix flow $\wh{H}_t^{z}=(\wh h_{ab})_{a,b\in[2n]}$ is defined as in (\ref{flow}) with $\wh{X}_t$ given by (\ref{hat_X}).  Applying the It\^{o}'s formula on the corresponding $F\big(\II_{E_1 +m l_1}^{z_1,\eta_1}, \cdots, \II_{E_k +m l_k}^{z_k,\eta_k} \big)$ and using the cumulant expansion formula as in (\ref{cumulant_exp0}), we have
\begin{align}\label{cumulant_app}
	&\frac{\dd}{\dd t}\wh{\E}_t \left[F\Big(\II_{E_1 +m l_1}^{z_1,\eta_1}, \cdots, \II_{E_k +m l_k}^{z_k,\eta_k} \Big)\right]\nonumber\\
	=&\sum_{p+q+1 \geq 3}^{M_D} \frac{1}{n^\frac{p+q+1}{2}} \sum_{a=1}^{n} \sum_{B=n+1}^{2n} \left(\frac{\kappa^{(p+1,q)}}{p!q!} \wh{\E}_t\Big[ \frac{\partial^{p+q+1} F(\cdots)}{\partial \wh{h}^{p+1}_{aB}\partial  \wh{h}^{q}_{Ba}}\Big]+\frac{\kappa^{(q,p+1)}}{p!q!} \wh{\E}_t\Big[ \frac{\partial^{p+q+1} F(\cdots)}{\partial \wh{h}^{q}_{aB}\partial  \wh{h}^{p+1}_{Ba}}\Big]\right)+O_{\prec}(n^{-D}),
\end{align}
where we truncated the cumulant expansion at a sufficiently large order $M_D$ such that the error is bounded by $O_\prec(n^{-D})$, using the moment condition in (\ref{eq:hmb}) and that $\|G\|\leq |\eta|^{-1}$. Using the differentiation rule as in (\ref{rule}), for any $1\leq j\leq k$, we have
\begin{align}
	\frac{\partial \II_{E_j +m l_j}^{z_j,\eta_j}}{\partial \wh{h}_{aB}}=\wt \Im \wh{G}^{z_p}_{Ba}(y+\ii \eta_j)\Big|^{y=E_j +m l_j}_{y=-(E_j +m l_j)}=O_\prec(n^{-1/4+c\epsilon_*}),
\end{align}
with $\wt \Im$ defined in (\ref{dim}), and the last estimate follows from the local law in (\ref{entrywise}) and the choices of $l_j,\eta_j$ above (\ref{approx2}). In general, for any integers $p,q\geq 1$, we have
\begin{align}
	\Big|\frac{\partial^{p+q} \II_{E_j +m l_j}^{z_j,\eta_j}}{\partial \wh{h}^p_{aB}\partial \wh{h}^q_{Ba}}\Big|=O_\prec\Big(\big(n^{-1/4+c\epsilon_*}\big)^{p+q}\Big).
\end{align}
Therefore we obtain from (\ref{cumulant_app}) and that $F(w_1,\cdots,w_k)=\prod_{i=1}^k \Phi(w_i)$
\begin{align}
	\left| \frac{\dd}{\dd t}\wh{\E}_t\left[F\Big(\II_{E_1 +m l_1}^{z_1,\eta_1}, \cdots, \II_{E_k +m l_k}^{z_k,\eta_k} \Big)\right] \right| \prec n^{-1/4+c\epsilon_*} \sum_{|\mathbf{\alpha}|=1}^{M_D} \sum_{\alpha_1,\cdots, \alpha_k}\wh{\E}_t \Big|  \prod_{j=1}^k \Phi^{(\alpha_j)}\big(\II_{E_j +m l_j}^{z_j,\eta_j}\big)  \Big|,
\end{align}
where $\mathbf{\alpha}=(\alpha_1,\cdots, \alpha_k)$ is a multi-index.  Since $\Phi$ given in (\ref{F_function}) has uniformly bounded derivatives, then we have
$$\Big|\Phi^{(\alpha_j)}\big(\II_{E_j +m l_j}^{z_j,\eta_j}\big)\Big|\lesssim \one_{\II_{E_j +m l_j}^{z_j,\eta_j} \in [1/9,2/9]} \lesssim \one_{\lambda^{z_j}_1 \leq E_j +(m+1) l_j}, \qquad \forall \alpha_j\geq 1$$
and a similar upper bound also applies to $\alpha_j=0$. Thus we have
\begin{align}\label{deri_F}
	\left| \frac{\dd}{\dd t}\wh{\E}_t \left[F\Big(\II_{E_1 +m l_1}^{z_1,\eta_1}, \cdots, \II_{E_k +m l_k}^{z_k,\eta_k} \Big)\right] \right| \prec& n^{-1/4+c\epsilon_*} \wh{\P}_t \Big( \lambda^{z_j}_1 \leq E_j+(m+1)l_j,~j \in [k]\Big)\nonumber\\
	\lesssim &~ n^{-1/4+c\epsilon_*}\wh{\E}_t \Big[ F\big(\II_{E_1+(m+2)l_1}^{z_1,\eta_1}, \cdots, \II_{E_k+(m+2)l_k}^{z_k,\eta_k} \big)\Big]+n^{-D},
\end{align}
where in the last step we used the second inequality (\ref{approx2}) with $E_j$ replaced by $E_j+(m+1)l_j$. Therefore we have proved (\ref{iterate_m}) and hence finished the proof of Proposition \ref{prop_zz}. 
\end{proof}

 \section{GFT: remove Gaussian component in Proposition \ref{pro:mainpro} and \ref{pro:mainprogfgf}}
\label{app:lindGFT}

In this section, we will  remove the Gaussian component in Part A of Theorem \ref{thm:2G} using a standard GFT argument and hence prove Part B of Theorem \ref{thm:2G} for any i.i.d. matrix. The same proof also applies to Theorem \ref{theorem_F} using a similar GFT argument combining with Proposition \ref{pro:mainprogfgf}. 

\smallskip

We organize this section as follows. In the first two sections, we split the GFT argument used to prove Theorem \ref{thm:2G}~(Part B) for any i.i.d. matrix into two steps. 
\begin{enumerate}
	\item[\bf Step 1:]
	
	In Section \ref{sec:Part 1}, we first establish a standard GFT for $\<\ga A_1 \gb A_2\>$~(see Propositon \ref{prop_zag} below) to remove any Gaussian component of size of order
	 one with an error term bounded by $(n|\eta_1\eta_2|)^{-1}$. 
	This will be used to prove the local law with canonical error bound $(n|\eta_1\eta_2|)^{-1}$.
	
	\item[\bf Step 2:]
	
	In Section~\ref{sec:GFT_2G}, we  
	derive an improved GFT for quantities of the special form $\<\Im \ga A \Im \gb A^*\>$ for a large $|z_1-z_2| \gtrsim \eta_i/\rho_i$ (see Propositon \ref{prop_zag_im} below). This will be used to prove the local law in (\ref{local_2g_im}) with the $|z_1-z_2|$ decorrelation effect.		
\end{enumerate}
Moreover, in Section \ref{sec:GFT_FF} we show a refined GFT with $\rho$-improvement for $\<GFGF^*\>$ in order to prove the local laws in Theorem \ref{theorem_F}. Note that $\<GFGF^*\>=\<\Im G F \Im G F^*\>$ from Lemma~\ref{lem:chirid} due to the chiral symmetry of $H^z$, hence the GFT argument is completely analogous to Section~\ref{sec:GFT_2G}; for this reason we only sketch its proof. Lastly, in Section~\ref{sec:proof}, we combine Propositon~\ref{pro:mainpro} and Proposition \ref{pro:mainprogfgf} with the above GFT arguments, and present the proof of Theorem~\ref{thm:2G}~(Part B) and Theorem \ref{theorem_F}, respectively, for any i.i.d. matrix.

\subsection{GFT for $\<G_1A_1G_2A_2\>$ with error $(n|\eta_1\eta_2|)^{-1}$}\label{sec:Part 1}

Recall the Ornstein-Uhlenbeck matrix flow $H^z_t$ in (\ref{flow}) and its time-dependent resolvent $G^z_t$ defined as in (\ref{def_G}). For notational simplicity, we define 
\begin{align}\label{R_t}
	R_t:=\<\big(G_1 A_1 G_2 -M^{A_1}_{12}\big) A_2 \>, \qquad G_i:=G^{z_i}_t(\ii \eta_i),\quad i=1,2,
\end{align}
with the deterministic matrix $M^{A_1}_{12}:=M_{12}^{A_1}(\eta_1,z_1,\eta_2,z_2)$ given by (\ref{eq:defM12}), 
where both $z_i\in \C$ and $\eta_i~\in~\R~(i=1,2)$ are time independent parameters. Then we establish the following GFT to remove the Gaussian component in Proposition \ref{pro:mainpro} with the error bound $(n|\eta_1\eta_2|)^{-1}$.
\begin{proposition}\label{prop_zag}
	Fix any $m\geq 2$. 
	Then for any spectral parameters $|\eta_1|\sim |\eta_2| \sim \eta$, $\rho_1\sim \rho_2 \sim \rho$, the following holds
	\begin{align}\label{eq_GFT}
		\Big|\frac{\dd \E | R_t|^{2m}}{\dd t}\Big| \prec  \big( 1+\eta^{1/3}\big) \Big(\E | R_t|^{2m}+\Big(\frac{1}{n|\eta_1\eta_2|}\Big)^{2m}\Big).
	\end{align}
uniformly for any $t\geq 0$, $\|A_i\|\lesssim 1$ and $n|\eta_i| \rho_i \gtrsim 1$.
\end{proposition}

\begin{proof}[Proof of Proposition \ref{prop_zag}]
	Applying It\^{o}'s formula to $|R_t|^{2m}=(R_t)^{m} (R^*_t)^m$ and performing the cumulant expansions as in (\ref{cumulant_exp0}), we obtain
	\begin{align}\label{cumulant_app0}
		\frac{\dd \E |R_t|^{2m}}{\dd t}=&\sum_{p+q+1 \geq 3}^{M_0} \frac{1}{n^\frac{p+q+1}{2}} \sum_{a=1}^{n} \sum_{B=n+1}^{2n} \left(\frac{\kappa^{(p+1,q)}}{p!q!} \E\Big[ \frac{\partial^{p+q+1} (R_t)^{m} (R^*_t)^m}{\partial h^{p+1}_{aB}\partial  h^{q}_{Ba}}\Big]+\frac{\kappa^{(q,p+1)}}{p!q!} \E\Big[ \frac{\partial^{p+q+1}  (R_t)^{m} (R^*_t)^m}{\partial h^{q}_{aB}\partial  h^{p+1}_{Ba}}\Big]\right)
		\nonumber\\
		&+O_\prec\big( n^{-100m}\big),
	\end{align}
	where we stopped the cumulant expansions at a sufficiently high power $M_0=O(m)$ such that the error is bounded by $O_{\prec}\big(n^{-100m}\big)$ using the moment condition in (\ref{eq:hmb}) and that $\|G\|\leq |\eta|^{-1}$. 
	To estimate the terms on the right side of (\ref{cumulant_app0}) we introduce the following lemma, whose proof is postponed till the end of this section.
	
	\begin{lemma}\label{lemma_zag_deff}
		For any $\mathfrak{u},\mathfrak{v} \in [1,2n]$, we have
		\begin{align}\label{ineqqq}
			\big|(G_1A_1 G_2)_{\mathfrak{u} \mathfrak{v}} \big|  \prec \sqrt\frac{\rho_1\rho_2}{|\eta_1\eta_2|}, \qquad \big|(G_1A_1G_2A_2G_1)_{\mathfrak{u} \mathfrak{v}} \big| \prec \frac{\rho_1}{|\eta_1\eta_2|},
		\end{align}
		for any $\|A_1\|+\|A_2\|\lesssim 1$ and $n|\eta_i|\rho_i\gtrsim 1$. Thus for any $a \in [1,n]$ and $B \in [n+1,2n]$, we have
		\begin{align}\label{rule_R}
			\frac{\partial  \<G_1 A_1 G_2 A_2\> }{\partial h_{aB}}=-\frac{1}{n} \Big((G_1A_1G_2A_2G_1)_{Ba}+(G_2A_2G_1A_1G_2)_{Ba}\Big)=O_\prec\Big( \frac{\rho_1+\rho_2}{n|\eta_1\eta_2|}\Big),
		\end{align}
	and the same holds if we switch the index $a$ with $B$. In general for any $p,q\in \N$ with $p+q\geq 1$, we have
		\begin{align}\label{zag_deff_rule}
			\Big|\frac{\partial^{p+q}\<G_1 A_1 G_2 A_2\>}{\partial h^{p}_{aB}\partial  h^{q}_{Ba}} \Big| \prec \frac{\rho_1+\rho_2}{n|\eta_1\eta_2|}\big( \rho_1+\rho_2+\one_{a=B-n}\big)^{p+q-1},
		\end{align}
		for any $\|A_1\|+\|A_2\|\lesssim 1$ and $n|\eta_i|\rho_i\gtrsim 1$.  
	\end{lemma}

	Thus by a naive power counting using (\ref{zag_deff_rule}), the terms on the right side of (\ref{cumulant_app0}) are bounded by
	\begin{align}\label{higher}
		\sum_{p+q+1 \geq 3}^{K_0} \frac{n^2}{n^\frac{p+q+1}{2}} \sum_{s=1}^{2m} \Big(\frac{\rho_1+\rho_2}{n|\eta_1\eta_2|}\Big)^s \E|R_t|^{2m-s}.
	\end{align}
Note that the above naive estimate is already enough to show (\ref{eq_GFT}) as long as $p+q+1\geq 4$.  \nc
	It then suffices to estimate the most critical third order terms with $p+q+1=3$ in (\ref{cumulant_app0}), \ie
	\begin{align}\label{partial_three}
		\frac{1}{n^{3/2}} \sum_{a,B} \E\Big[  \big(\partial R_t\big)^3 R_t^{2m-3} + 
		\partial R_t\big(\partial^2 R_t\big)R_t^{2m-2} +(\partial^3 R_t)R_t^{2m-1}\Big],
	\end{align}
	where each $\partial$ represents either $\partial/\partial h_{aB}$ or $\partial/\partial h_{Ba}$ and each $R_t$ represents $R_t$ itself or its conjugate $R_t^*$. By a direct computation, using (\ref{rule_R}), the first part in (\ref{partial_three}) is bounded by
	\begin{align}\label{33}
		\frac{1}{n^{3/2}} \sum_{a,B} &\E\Big[ \big(\partial R_t\big)^3 R_t^{2m-3}\Big] \lesssim \frac{1}{n^{9/2}}\sum_{a,B} \E \Big[\Big| (G_1A_1G_2A_2G_1)_{Ba}+(G_2A_2G_1A_1G_2)_{Ba} \Big|^{3}|R_t|^{2m-3}\Big]\nonumber\\
		&\prec \frac{1}{n^{7/2}} \E\Big[\Big(\frac{\rho_1\<\Im G_1 A_1 G_2A_2 \Im G_1 A_2^* G_2^* A_1^*\>}{|\eta^3_1\eta_2|}+\frac{\rho_2\<\Im G_2 A_2 G_1A_1 \Im G_2 A_1^* G_1^* A_2^*\>}{|\eta_1\eta^3_2|}\Big) |R_t|^{2m-3} \Big]\nonumber\\
		&\prec \frac{(\rho^*)^{2}}{n^{7/2} |\eta_1 \eta_2|^{5/2} (\eta_*)^2}  \E|R_t|^{2m-3},
	\end{align}
with $\rho^*:=\rho_1\vee\rho_2$ and $\eta_*:=|\eta_1|\wedge|\eta_2|$, where we used the Ward identity in the second line, and in the last step we used the following inequality 
	\begin{align}\label{ineq_3}
		\<G_1 A_1 G_2 A_2 G^*_1 A^*_2 G^*_2 A^*_1 \> \leq \|A_1\| \|A_2\|  \sqrt{\frac{ \< \Im G_1 A_1 \Im G_2 A_1^*\> \< \Im G_2 A^*_2 \Im G_1A_2\>}{\eta_1^2\eta_2^2}} \prec \sqrt{\frac{\rho_1\rho_2}{|\eta_1\eta_2|^3}}.
	\end{align}
	Note that in (\ref{33}) we have gained an additional $1/(n\eta_*\rho^*)$ compared to the naive size in (\ref{higher}) by power counting, since we are summing up the products of two off-diagonal resolvent entries.

	We next compute the second term in (\ref{partial_three}). Using (\ref{rule_R}) and the Cauchy-Schwarz inequality, 
	we have 
	\begin{align}\label{3333}
		\frac{1}{n^{3/2}} \sum_{a,B} \E\Big[ \partial R_t\big(\partial^2 R_t\big)R_t^{2m-2}\Big] \prec& \frac{(\rho_1+\rho_2)^2}{n^{7/2}|\eta_1\eta_2|}  \sum_{a,B}  \E\Big[ \Big(\big|(G_1 A_1 G_2 A_2 G_1)_{Ba}\big|+\big|(G_2 A_2  G_1 A_1 G_2)_{Ba}\big| \Big)\big| R_t\big|^{2m-2}\Big]\nonumber\\
		\prec& \frac{1}{n^{3/2}} \frac{(\rho_1+\rho_2)^2}{|\eta_1\eta_2|}  \E\left[ \sqrt{\frac{1}{n^2}\sum_{a,B}\big|(G_1A_1G_2A_2G_1)_{Ba}\big|^2} |R_t|^{2m-2}\right]\nonumber\\
		\prec &\frac{(\rho^*)^{5/2}}{n^{2}|\eta_1\eta_2|^{9/4}}\E|R_t|^{2m-2},
	\end{align}
	where we gained an additonal $1/\sqrt{n\eta_*\rho^*}$ compared to the naive size in (\ref{higher}), since there exists one off-diagonal resolvent factor in the summation. Similarly, using (\ref{rule_R}), we see that the last part in (\ref{partial_three}) also contains at least one off-diagonal resolvent factor and so we have an additional smallness:
	\begin{align}\label{last_part}
		\frac{1}{n^{3/2}} \sum_{a,B} \E\big[ (\partial^3 R_t) R_t^{2m-1}\big] \lesssim \frac{(\rho^*)^{5/2}}{n|\eta_1\eta_2| \sqrt{\eta_*}}\E|R_t|^{2m-1}.
	\end{align}
	Combining (\ref{33}), (\ref{3333})--(\ref{last_part}) with (\ref{higher}), we obtain that
		\begin{align}
		\Big|\frac{\dd \E |R_t|^{2m}}{\dd t}\Big| \prec&\frac{(\rho^*)^{5/2}}{n |\eta_1\eta_2|\sqrt{\eta_*}}\E|R_t|^{2m-1}+\frac{(\rho^*)^{5/2}}{n^{2}|\eta_1\eta_2|^{9/4}}\E|R_t|^{2m-2}+\frac{(\rho^*)^2}{n^{7/2} |\eta_1 \eta_2|^{5/2}\eta_*^2}  \E|R_t|^{2m-3} \nonumber\\
		&+\sum_{k\geq 4}^{2m} \frac{n^2}{n^{k/2}}  \Big(\frac{\rho^*}{n|\eta_1\eta_2|}\Big)^k \E |R_t|^{2m-k}.
	\end{align}
Further assuming that $|\eta_1|\sim |\eta_2| \sim \eta$, $\rho_1\sim \rho_2 \sim \rho$ and using Young's inequality, we obtain that 
		\begin{align}\label{eq_zag}
		\Big|\frac{\dd \E | R_t|^{2m}}{\dd t}\Big| \prec  \Big( 1+\frac{\rho^{5/2}}{\eta^{1/2}}+\frac{\rho^{2}}{\sqrt{n}\eta}\Big) \Big(\E | R_t|^{2m}+\Big(\frac{1}{n|\eta_1\eta_2|}\Big)^{2m}\Big).
	\end{align}
Thus (\ref{eq_GFT}) follows directly from (\ref{eq_zag}) using that $\rho\lesssim \eta^{1/3}$ and $n\eta \rho \gtrsim 1$. We hence finish the proof of Proposition \ref{prop_zag}. 
\end{proof}

We conclude this section with the proof of Lemma \ref{lemma_zag_deff}.
\begin{proof}[Proof of Lemma \ref{lemma_zag_deff}] 
	Using the Cauchy-Schwarz inequality and the Ward identity, we have
	\begin{align}\label{ineq_2}
		\big|(G_1A G_2)_{\mathfrak{u} \mathfrak{v}} \big| 
				\leq \frac{\|A\|\sqrt{(\Im G_1)_{\mathfrak{u}\mathfrak{u}}(\Im G_2)_{\mathfrak{v}\mathfrak{v}} }}{\sqrt{|\eta_1\eta_2|}}\prec   \frac{1}{\sqrt{|\eta_1\eta_2|}} \big(\sqrt{\rho_1\rho_2}+\frac{1}{n\sqrt{|\eta_1\eta_2|}}\big),
	\end{align}
	where in the last step we used the local law in (\ref{entrywise}). This proves the first estimate in (\ref{ineqqq}). Similarly to (\ref{ineq_2}), we further have
	\begin{align}\label{ineq_1}
		\big|(G_1A_1G_2A_2G_1)_{\mathfrak{u} \mathfrak{v}} \big| \leq& \sqrt{(G_1A_1G_2A_2A_2^*G_2^*A_1^*G_1^*)_{\mathfrak{u}\mathfrak{u}} (G_1G_1^*)_{\mathfrak{v}\mathfrak{v}}} 
		\leq \frac{\|A_1\|\|A_2\|\sqrt{(\Im G_1)_{\mathfrak{v}\mathfrak{v}}(\Im G_1)_{\mathfrak{u}\mathfrak{u}} }}{|\eta_1\eta_2|}\nonumber\\
		\prec  & \frac{1}{|\eta_1\eta_2|} \big(\rho_1+\frac{1}{n|\eta_1|}\big),
	\end{align}
	where we also used (\ref{average}) and that $\|G(\eta)\|\leq |\eta|^{-1}$. This proves the second estimate in (\ref{ineqqq}). 
	
	Note that the differentiation rule in (\ref{rule_R}) follows directly from (\ref{rule}). Combining this with (\ref{ineq_1}) we have proved the last estimate in (\ref{rule_R}). This also proves (\ref{zag_deff_rule}) for $p+q=1$. By further computations using (\ref{rule}), the corresponding partial derivatives of $\<G_1 A_1 G_2 A_2\>$ with $p+q=2$ are given by linear combinations of the following terms
	\begin{align}\label{terms}
		&\frac{1}{n} (G_1A_1G_2A_2G_1)_{Ba} (G_1)_{Ba},\quad \frac{1}{n} (G_1A_1G_2A_2G_1)_{aa} (G_1)_{BB},\quad \frac{1}{n} (G_1A_1G_2A_2G_1)_{BB} (G_1)_{aa},\nonumber\\
		&\frac{1}{n}(G_1A_1G_2)_{Ba} (G_2A_2G_1)_{Ba}, \quad \frac{1}{n}(G_1A_1G_2)_{aa} (G_2A_2G_1)_{BB}, \quad \frac{1}{n}(G_1A_1G_2)_{BB} (G_2A_2G_1)_{aa},
	\end{align}
	and the terms obtained by switching the subscript 1 with 2. Thus for $p+q=2$, (\ref{zag_deff_rule}) follows from (\ref{ineq_2}), (\ref{ineq_1}) and the local law in (\ref{entrywise}). For a general $p+q\geq 2$, one proves (\ref{zag_deff_rule}) similarly by induction using (\ref{ineq_2}) and (\ref{ineq_1}). We hence finish the proof of Lemma \ref{lemma_zag_deff}.
\end{proof}

\bigskip

\subsection{GFT for $\Im G_1 A \Im G_2 A^*$ with $|z_1-z_2|$-decorrelation}\label{sec:GFT_2G}

In this section, we consider a special form in terms of (\ref{R_t}) and establish an improved GFT to remove the Gaussian component in Part A of Theorem \ref{thm:2G} for a large $|z_1-z_2| \gtrsim |\eta_i|/\rho_i$.
Again for simplicity, we do it only for the regime $|\eta_1|\sim |\eta_2|$ and $\rho_1 \sim \rho_2$. 
\begin{proposition}\label{prop_zag_im}
	Define
	\begin{align}\label{spectial}
		\wh{R_t}:=\<\big(\Im G_1 A \Im G_2 -\wh{M}^{A}_{12}\big) A^* \>, \qquad G_i:=G^{z_i}_t(\ii \eta_i),\quad i=1,2,
	\end{align}
	where $\wh{M}^{A}_{12}=\wh{M}^{A}_{12}(\eta_1,z_1,\eta_2,z_2)$ is the corresponding deterministic limit defined as in (\ref{eq:defM12}) using that $\Im G=(G-G^*)/(2\ii)$. 
	Fix any $m> 1$. Then for any spectral parameters $ |\eta_1| \sim |\eta_2| \sim \eta$ and $ \rho_1 \sim \rho_2 \sim \rho$, the following holds
	\begin{align}\label{eq_zag_im}
		\Big|\frac{\dd \E |\wh{R_t}|^{2m}}{\dd t}\Big| \prec \Big(1+\frac{1}{\sqrt{n}\eta}\Big) \Big( \E |\wh{R_t}|^{2m} +\Big( 
				\frac{1}{\sqrt{n\eta \rho} \big(|z_1-z_2|+\frac{\eta}{\rho}\big)}\Big)^{2m}\Big),
	\end{align}
	uniformly for any $t\geq 0$, $\|A\|\lesssim 1$ and $n\eta \rho \gtrsim 1$.	
\end{proposition}

\begin{proof}[Proof of Proposition \ref{prop_zag_im}] 	
	
	Applying It\^{o}'s formula to $|\wh{R_t}|^{2m}=(\wh{R_t})^{2m} $ and performing the cumulant expansions, we obtain the analogue of (\ref{cumulant_app0}) for $\wh{R_t}$ using that $\Im G=(G-G^*)/(2\ii)$. The most critical third order terms with $p+q+1=3$ are then given by (\ref{partial_three}) in terms of $\wh{R_t}$. Note that (\ref{rule_R}) implies
	\begin{align}\label{rule_im}
		\frac{\partial \wh{R_t}}{\partial_{aB}}=-\frac{1}{2 n \ii  } &\Big((G_1A \Im G_2A^* G_1)_{Ba}-(G^*_1A \Im G_2A^*G^*_1)_{Ba}\nonumber\\
		&\qquad\qquad\qquad +(G_2A^*G_1AG_2)_{Ba}-(G^*_2A^*\Im G_1AG^*_2)_{Ba}\Big).
	\end{align}
	By a direct computation, using (\ref{rule_im}), the first part of (\ref{partial_three}) is bounded by (\cf (\ref{33}))	
	\begin{align}\label{33_im}
		\frac{1}{n^{3/2}} \sum_{a,B} &\E\Big[ \big(\partial \wh{R_t}\big)^3 (\wh{R_t})^{2m-3}\Big] \lesssim \frac{1}{n^{9/2}}\sum_{a,B} \E \Big[\Big| (G_1 A \Im G_2 A^*G_1)_{Ba}+(G_2 A^*\Im G_1 A G_2)_{Ba} \Big|^{3}|\wh{R_t}|^{2m-3}\Big]\nonumber\\
		&\lesssim \frac{\rho}{n^{7/2}\eta^2} \E\Big[\Big( \frac{\<\Im G_1 A \Im G_2 A^* \Im G_1 A \Im G_2 A^*\>}{\eta_1^2} +\frac{\<\Im G_2 A^* \Im G_1 A \Im G_2 A^* \Im G_1 A\>}{\eta_2^2}  \Big)|\wh{R_t}|^{2m-3} \Big]\nonumber\\
		&\lesssim \frac{\rho}{n^{5/2}\eta^4} \E\Big[\big(\<\Im G_1 A \Im G_2 A^*\>\big)^2 |\wh{R_t}|^{2m-3} \Big]\nonumber\\
		&\lesssim	  \frac{\rho}{n^{5/2}\eta^4}  \Big( \E|\wh{R_t}|^{2m-1} +\frac{1}{\big(|z_1-z_2|+\frac{\eta}{\rho}\big)^2}\E|\wh{R_t}|^{2m-3}\Big),
	\end{align}
	where in the penultimate line we used
	\begin{align}\label{traceab}
		\<\Im G_1 A \Im G_2 A^* \Im G_1 A \Im G_2 A^*\> \leq n\<\Im G_1 A \Im G_2 A^*\>^2,
	\end{align}
	and in the last line we used the upper bound in (\ref{local_2g_M}) with $ |\eta_1| \sim |\eta_2| \sim \eta$ and $ \rho_1 \sim \rho_2 \sim \rho$.
	A similar upper bound also applies to the second part of (\ref{partial_three})~(\cf (\ref{3333})), \ie
	\begin{align}\label{argu_im}
		\frac{1}{n^{3/2}} \sum_{a,B} \E\Big[ \partial \wh{R_t} \big(\partial^2 \wh{R_t} \big)(\wh{R_t})^{2m-2}\Big] \prec& \frac{\rho^2}{n^{7/2}\eta^2} \sum_{a,B}  \E\Big[ \Big(\big|(G_1 A \Im G_2 A^* G_1)_{Ba}\big|+\big|(G_2 A \Im G_1 A^* G_2)_{Ba}\big| \Big)\big|\wh{R_t}\big|^{2m-2}\Big]\nonumber\\
		\lesssim &\frac{\rho^2}{n^{3/2} \eta^2}  \E\Big[ \sqrt{\frac{\<\Im G_1 A \Im G_2 A^* \Im G_1 A \Im G_2 A^*\>}{n\eta^2}}  \big|\wh{R_t}\big|^{2m-2}\Big]\nonumber\\
		\lesssim & \frac{\rho^2}{n^{3/2} \eta^3} \Big( \E|\wh{R_t}|^{2m-1} +\frac{1}{|z_1-z_2|+\frac{\eta}{\rho}} \E|\wh{R_t}|^{2m-2}\Big).
	\end{align}
		We use the same naive bound as in (\ref{last_part}) to estimate the last part $ n^{-3/2} \sum_{a,B} \E\big[ (\partial^3 \wh{R_t}) (\wh{R_t})^{2m-1}\big]$.	
	Combining this with (\ref{33_im}) and (\ref{argu_im}), the third order terms with $p+q+1=3$ are bounded by
	\begin{align}\label{third_im}
		&\frac{\rho^{5/2}}{n\eta^{5/2}}\E|\wh{R_t}|^{2m-1}+\frac{\rho^2}{n^{3/2} \eta^3} \frac{1}{|z_1-z_2|+\frac{\eta}{\rho}} |\wh{R_t}|^{2m-2}+\frac{\rho}{n^{5/2}\eta^4} \frac{1}{\big(|z_1-z_2|+\frac{\eta}{\rho}\big)^2}\E|\wh{R_t}|^{2m-3}\nonumber\\
		\lesssim& \frac{1}{\sqrt{n}\eta} \Big(\E |\wh{R_t}|^{2m}+\Big(\frac{\rho^2}{n\eta^2}\sqrt{n\eta \rho} \Big)^{2m} +\Big(\frac{1}{\sqrt{n\eta \rho}\big(|z_1-z_2|+\frac{\eta}{\rho}\big)}\Big)^{2m}\Big).
	\end{align}
	All the higher order terms with $p+q+1 \geq 4$ can be bounded similarly with less efforts due to the moment condition in (\ref{eq:hmb}). More precisely, for general $p+q+1 \geq 4$, they are given by
	\begin{align}\label{split_pi}
		\frac{1}{n^{\frac{p+q+1}{2}}} \sum_{a,B} \Big( \sum_{1\leq k\leq \lfloor \frac{p+q+1}{2} \rfloor}+\sum_{\lfloor \frac{p+q+1}{2} \rfloor<k\leq p+q+1}\Big) \sum_{d_1+\cdots+ d_k=p+q+1,\atop d_i\ge 1}\E\big[ \prod_{i=1}^k \big(\partial^{d_i} \wh{R_t}\big) (\wh{R_t})^{2m-k}\big]=:\Pi_1+\Pi_2,
	\end{align}
	where we split the above sum using the integer $k$, which is the number of times that $\partial_{aB}$ or $\partial_{Ba}$ acts on $\wh{R_t}$. For any $\lfloor \frac{p+q+1}{2} \rfloor < k\leq p+q+1$, since $\sum_{i=1}^k d_i=p+q+1$, then there exists at least one $d_i=1$, which implies that there exists at least one off-diagonal entry of $G_1 A \Im G_2 A^*G_1$ from (\ref{rule_im}). Using  similar arguments as in (\ref{argu_im}) and the upper bound in (\ref{local_2g_M}), the second part in (\ref{split_pi}) is bounded by
	\begin{align}\label{pi_1_im}
		|\Pi_2| \prec& \frac{\rho^{p+q}}{n^{\frac{p+q+3}{2}}} \sum_{a,B} \sum_{\lfloor \frac{p+q+1}{2} \rfloor <k\leq p+q+1}   \Big(\frac{1}{n\eta^2}\Big)^{k-1}  \E\Big[\Big(\big|(G_1 A \Im G_2 A^* G_1)_{Ba}\big|+\big|(G_2 A \Im G_1 A^* G_2)_{Ba}\big| \Big) (\wh{R_t})^{2m-k}\Big]\nonumber\\
		\lesssim & \frac{\rho^{p+q}}{n^{\frac{p+q-1}{2}}}   \sum_{\lfloor \frac{p+q+1}{2} \rfloor <k\leq p+q+1}\frac{1}{\rho^{2(k-1)}\eta}  \Big(\frac{\rho^2}{n\eta^2}\Big)^{k-1}  \E\left[ \<\Im G_1 A\Im G_2 A^*\>(\wh{R_t})^{2m-k}\right]\nonumber\\
		\lesssim & \frac{1}{\sqrt{n}\eta} \left( \E|\wh{R_t}|^{2m}+ \Big(\frac{\rho^2}{n\eta^2} \sqrt{n\eta \rho} +\frac{1}{\sqrt{n\eta \rho}\big(|z_1-z_2|+\frac{\eta}{\rho}\big)}\Big)^{2m}\right),
	\end{align}
	where we also used that $p+q+1\geq 4$ and $\rho \gtrsim n^{-1/4}$ for any $n\eta \rho \gtrsim 1$ from (\ref{rho}) with $|z|>1$. 
	For any $1\leq k\leq \lfloor \frac{p+q+1}{2} \rfloor$ with $p+q+1\geq 4$, we use a naive power counting by (\ref{zag_deff_rule}) to obtain that
	\begin{align}\label{pi_2}
		|\Pi_1| \prec& \frac{n^2}{n^{\frac{p+q+1}{2}}} \sum_{1\leq k\leq \lfloor \frac{p+q+1}{2} \rfloor} \rho^{p+q+1}\Big(\frac{1}{n\eta^2}\Big)^{k} \E \Big[\big| \wh{R_t}\big|^{2m-k}\Big]\lesssim \E |\wh{R_t}|^{2m}+\Big(\frac{\rho^2}{n\eta^2} \Big)^{2m}.
	\end{align}
	Combining with (\ref{third_im}) and (\ref{pi_1_im}), using that $\rho \lesssim \eta^{1/3}$ and $|z_1-z_2|+\eta/\rho \lesssim 1$, we hence finish the proof of Proposition \ref{prop_zag_im}.
\end{proof}

\subsection{GFT for $\<GFGF^*\>$ with $\rho$-improvement}\label{sec:GFT_FF}

To prove Theorem \ref{theorem_F} with the $\rho$-improvements, we establish the following refined GFTs to remove the Gaussian component in Proposition \ref{pro:mainprogfgf}.

\begin{proposition}\label{prop_zag_F}
Let 
\begin{align}\label{special_F}
	R^{F}_t:=\big\<\big(\gz_t(\ii \eta) F \gz_t(\ii \eta) -M_{12}^{F}\big) F^* \big\>, \qquad \eta \in \R_+, \quad z\in \C,
\end{align} 
where $M_{12}^{F}=M_{12}^F(\eta, z, \eta, z)$ is the corresponding deterministic limits given by (\ref{eq:defM12})
	and we also define 
	$$S^{F}_{t}:=\<(\gz_t(\ii \eta) -M^z)F\>, \qquad\qquad S^{2,F}_{t}:=\<\gz_t(\ii \eta) F \gz_t(\ii \eta) -M_{12}^F\>,$$
	with $M^{z}=M^z(\ii\eta)$ given by (\ref{Mmatrix}). 
	 Fix any $m> 1$. Then for any $t\geq 0$ and $n\eta \rho \gtrsim 1$ we have
	\begin{align}\label{GFT_GFGF}
		\Big|\frac{\dd \E |R^{F}_t|^{2m}}{\dd t}\Big| \prec& \Big(1+\frac{1}{\sqrt{n}\eta}\Big)\Big(\E |R^F_t|^{2m}+\Big(\frac{\rho^2}{n\eta^2} \sqrt{n\eta \rho}\Big)^{2m}\Big).
	\end{align}
	Assuming that the following
	\begin{align}\label{assump_app}
			|R^F_t| \prec \frac{\rho^2\sqrt{n\eta\rho}}{n\eta^2},
	\end{align}
holds uniformly for any $t\geq 0$ and $n\eta \rho \gtrsim 1$, then we have
	\begin{align}
		\Big|\frac{\dd \E |S^{F}_t|^{2m}}{\dd t}\Big| \lesssim & \Big(1+\frac{\rho}{\sqrt{n}\eta}\Big)\Big(\E |S^F_t|^{2m}+\Big(\frac{\rho}{n\eta} \Big)^{2m}\Big),\label{GFT_GF}\\
		\Big|\frac{\dd \E |S^{2,F}_t|^{2m}}{\dd t}\Big| \lesssim & \Big(1+\frac{\rho}{\sqrt{n}\eta}\Big)\Big(\E |S^{2,F}_t|^{2m}+\Big(\frac{\rho}{n\eta^2} \Big)^{2m}\Big).\label{GFT_GFG}
	\end{align}
\end{proposition}

\begin{proof}[Proof of Proposition \ref{prop_zag_F}]
	
Note that that $\<GFGF^*\>=\<\Im GF \Im GF^*\>$ from Lemma \ref{lem:chirid}.  The proof of the first GFT in (\ref{GFT_GFGF}) is completely analogous to the proof of Proposition \ref{prop_zag_im} with $A=F$, $\eta_1=\eta_2=\eta$ and $z_1=z_2=z$. The only difference is to replace the upper bound in (\ref{local_2g_M}) with the first upper bound in (\ref{eq:addbneedef}). Repeating the same arguments as in (\ref{33_im})-(\ref{third_im}), the most critical third order terms are thus bounded by 
	\begin{align}\label{third_F}
	\frac{\rho^{5/2}}{n\eta^{5/2}}\E|R^F_t|^{2m-1}+ \frac{\rho^{5}}{n^{3/2}\eta^{4}}\E|R^F_t|^{2m-2} +\frac{\rho^{7}}{n^{5/2}\eta^{6}}\E|R^F_t|^{2m-3}
	\lesssim \frac{1}{\sqrt{n}\eta} \Big(\E |R^F_t|^{2m}+\Big(\frac{\rho^2}{n\eta^2} \sqrt{n\eta \rho}\Big)^{2m}\Big),
\end{align}
where we also used that $n\eta \rho \gtrsim 1$. Repeating the same arguments as in (\ref{split_pi})-(\ref{pi_2}), we obtain similar upper bounds for the higher order terms with $p+q+1 \geq 4$, \ie 
\begin{align}\label{high_F}
\frac{1}{\sqrt{n}\eta} \Big( \E|\wh{R_t}|^{2m}+ \big(\frac{\rho^2}{n\eta^2} \sqrt{n\eta \rho} \big)^{2m}\Big),
\end{align}
where we used the first upper bound in (\ref{eq:addbneedef}) to replace the upper bound in (\ref{local_2g_M}).  This, together with (\ref{third_F}), prove the first GFT in (\ref{GFT_GFGF}) of Proposition \ref{prop_zag}

The other two GFTs in (\ref{GFT_GF})-(\ref{GFT_GFG}) with less $\rho$-improvements can be proved with much less effort, so we only sketch their proofs. We will only consider $S^{F}_{t}$ for simplicity and the other one can be proved exactly in the same way. Applying Ito's formula to $|S^{F}_t|^{2m}$ and performing cumulant expansions, we obtain the analogue of (\ref{cumulant_app0}). We next estimate the most critical third order terms with $p+q+1=3$, which are splitted into three parts as in (\ref{partial_three}).	Using (\ref{rule}) and the first upper bound in (\ref{ineqqq}), the first part as in (\ref{partial_three}) is bounded by
	\begin{align}
		\frac{1}{n^{3/2}} \sum_{a,B} \E\Big[ \big(\partial S^F_t\big)^3 &(S^F_t)^{2m-3}\Big] \lesssim \frac{1}{n^{9/2}}\sum_{a,B} \E \Big[\Big| (GFG)_{Ba}+(GF^*G)_{Ba} \Big|^{3}|S^F_t|^{2m-3}\Big]\nonumber\\
		&\prec \frac{\rho}{n^{7/2}\eta} \E\Big[\frac{\<\Im G F \Im G F^* \>}{\eta^2} |S^F_t|^{2m-3} \Big] \lesssim \frac{\rho^4}{n^{7/2}\eta^4} \E|S^F_t|^{2m-3},
	\end{align}
	where we also used the assumption (\ref{assump_app}) and the first upper bound in (\ref{eq:addbneedef}). A similar upper bound also applies to the second part, \ie
	$$\frac{1}{n^{3/2}} \sum_{a,B} \E\Big[\partial S^F_t \big(\partial^2 S^F_t\big) (S^F_t)^{2m-2}\Big] \prec \frac{\rho^{7/2}}{n^2 \eta^{5/2}}\E|S^F_t|^{2m-2}.$$
	We use the following naive bound from (\ref{ineqqq}) to estimate the last part of third order terms: 
	$$\frac{1}{n^{3/2}} \sum_{a,B} \E\Big[\partial^3 S^F_t  (S^F_t)^{2m-1}\Big] \prec \frac{\rho^{5/2}}{n\eta^{3/2}}\E|S^F_t|^{2m-1}.$$
	All the higher order terms with $p+q+1\geq 4$ can be bounded similarly 
	with much less efforts, so we omit the details. Therefore, using that $\rho \lesssim \eta^{1/3}$ from (\ref{rho}), we have proved (\ref{GFT_GF}). The last GFT in (\ref{GFT_GFG}) can be proved similarly. We hence complete the proof of Proposition \ref{prop_zag_F}.	
\end{proof}

\subsection{Proof of the local laws}\label{sec:proof}
We start with proving Part B of Theorem \ref{thm:2G} for any i.i.d. matrix, while Part A of Theorem \ref{thm:2G} with a Gaussian component was already proved in Section \ref{sec:g1g2}. And the proof of Theorem \ref{theorem_F} is similar, so we only sketch its proof at the end of this section with some minor modifications.

\begin{proof}[Proof of Theorem \ref{thm:2G} (Part B) for i.i.d. matrix]	
	The proof of  Theorem \ref{thm:2G} (Part B) is split into two steps. In Step 1, we prove (\ref{local_2g_im}) without the $|z_1-z_2|$-gain. Indeed we prove a more general result, \ie
	\begin{align}\label{goal_wo}
		\big|\langle (\ga(\ii \eta_1)A_1\gb(\ii \eta_2)-M_{12}^{A_1})A_2\rangle\big|\prec \frac{1}{n|\eta_1\eta_2|},
	\end{align}
uniformly in any spectral parameters in $n \ell \geq n^{\epsilon}$ with $\ell=\eta \rho$, and $\eta \sim |\eta_1|\sim |\eta_2|$, $\rho \sim \rho_1\sim \rho_2$. In Step 2, based on the result from Step 1, we further prove (\ref{local_2g_im}) with the $|z_1-z_2|$-decorrelation, \ie 
	\begin{align}\label{goal_z1z_2_im}
	\big|\langle (\Im \ga(\ii \eta_1)A \Im \gb(\ii \eta_2)-\wh M_{12}^{A})A^*\rangle\big|\prec \frac{1}{\sqrt{n \ell} \big(|z_1-z_2|+\frac{\eta}{\rho}\big)},
\end{align}	
uniformly in $n \ell \geq n^{\epsilon}$, with $\ell=\eta \rho$, and $\eta \sim |\eta_1|\sim |\eta_2|$, $\rho \sim \rho_1\sim \rho_2$. This, together with (\ref{goal_wo}) from Step 1, complete the proof of Theorem \ref{thm:2G} (Part B).

\medskip

	{\bf Step 1: Proof of (\ref{goal_wo}).} 
Recall Part A of Theorem \ref{thm:2G} that has already been proved in Section \ref{sec:g1g2}. Fixing $\mathfrak{s}>0$, for any matrix of the form $X_{\mathfrak{s}}:=\sqrt{1-\mathfrak{s}^2}X+\mathfrak{s} X^{\mathrm{Gin}}$, we have
	\begin{align}\label{goal_T}
		\big|\langle (\ga_{\mathfrak{s}}(\ii \eta_1)A_1\gb_{\mathfrak{s}}(\ii \eta_2)-M_{12}^{A_1})A_2\rangle\big|\prec \left(\frac{1}{(n\ell)^{1/2}}\frac{1}{|z_1-z_2|+\frac{|\eta_1|}{\rho_1}+\frac{|\eta_2|}{\rho_2}}\right)  \wedge \left(\frac{1}{n\eta_1\eta_2}\right).
	\end{align}
	Fix $z_i$, $\eta_i$ in (\ref{goal_T}) and recall the definition of $R_t$ in (\ref{R_t}). Using Proposition \ref{prop_zag} together with a Gronwall arguemnt, we know that the Gaussian component of size $\mathfrak{s} \sim 1$ in (\ref{goal_T}) can be removed with an error $1/(n\eta_1\eta_2)$. We hence prove (\ref{goal_wo}).
\medskip	

	{\bf Step 2: Proof of (\ref{goal_z1z_2_im}).} 
	We first assume that $|\eta_i| \sim \eta \gtrsim n^{-1/2}$. Similarly to Step 1 above, Proposition \ref{prop_zag_im} implies that the Gaussian component added in Part A of Theorem \ref{thm:2G} can be removed by a Gronwall argument, with an error $\big(\sqrt{n\ell} (|z_1-z_2|+\eta/\rho)\big)^{-1}$. This, together with (\ref{goal_T}) from Step~1, prove (\ref{goal_z1z_2_im}) for any $|\eta_i| \sim \eta \gtrsim n^{-1/2}$.	

	We next extend this result to the complementary regime, \ie $|\eta_i| \sim \eta \lesssim n^{-1/2}$. Recall $\eta \sim |\eta_1|\sim |\eta_2|$ and  $\rho\sim \rho_1\sim \rho_2$ for simplicity, with $n\eta \rho \gtrsim 1$, and choose $T\sim \sqrt{\eta/\rho}  \lesssim 1$. Then from Lemma \ref{lem:ode}, there exists $\eta_{i,0}$ and $ z_{i,0}$ such that $\eta_{i,T}=\eta_i$ and $ z_{i,T}=z_i$ with 
	$ \rho_{i,0}\sim \rho$, $ \eta_{i,0} \gtrsim \rho  T\gtrsim n^{-1/2}$ and $ n \ell_0 \geq n\ell \geq n^{\epsilon}$ with $\ell_0$ defined in (\ref{eq:defellbetas}) at $t=0$. As shown in Step 1 and at the beginning of Step 2 for $\eta_{i,0} \gtrsim n^{-1/2}$, we know 	that the initial condition (\ref{eq:inasspart1}) of Proposition \ref{pro:mainpro} is satisfied. Thus Proposition \ref{pro:mainpro} implies that
	\begin{align}\label{goal_222}
		\big|\<\big(\Im \ga_{ T}(\ii \eta_1) A \Im \gb_{ T}(\ii \eta_2) -\wh{M}^{A}_{12}\big) A^* \>\big| \prec \frac{1}{\sqrt{n\ell} \big(|z_1-z_2|+\frac{\eta}{\rho}\big)} \wedge \frac{1}{n\eta^2}, \qquad  T\sim \sqrt{\eta/\rho}.
	\end{align}
	Combining this with Proposition \ref{prop_zag_im} and using that $n\ell \gtrsim n^{\epsilon}$, we remove the Gaussian component of size $T \sim \sqrt{\eta/\rho}$ in (\ref{goal_222}) and hence prove (\ref{goal_z1z_2_im}) for $\eta \lesssim n^{-1/2}$. This completes the proof of Part B of Theorem \ref{thm:2G}.		
		\end{proof}

\begin{proof}[Proof of Theorem \ref{theorem_F}]		
The proof of Theorem \ref{theorem_F}, in particular the local law in (\ref{local_2F}), is exactly the same as the proof of (\ref{goal_z1z_2_im}) presented above, using  Proposition~\ref{pro:mainprogfgf} and Proposition \ref{prop_zag_F} as inputs. Additionally, assuming the condition in (\ref{assump_app}) is satisfied (\ie the local law in (\ref{local_2F}) holds), one obtains the other two local laws in (\ref{local_F})-(\ref{local_GFG})~(with less $\rho$-improvement)  using the GFT arguments in (\ref{GFT_GF})-(\ref{GFT_GFG}) with much less efforts, so we omit the details. The fact that the assumptions on the initial conditions
\begin{align}
	&\big\<\big(G^{z_0} F G^{z_0}(\ii \eta_0) -M_{12}^{F}\big) F^* \big\> \prec \frac{\rho_0^2\sqrt{n \eta_0 \rho_0}}{n \eta_0^2},\label{goal_FF}\\
	&\<(G^{z_0} (\ii \eta_0)-M^{z_0})F\> \prec \frac{\rho_0}{n \eta_0}, \qquad \<(G^{z_0} FG^{z_0} (\ii \eta_0)-M_{12}^F\> \prec \frac{\rho_0}{n \eta_0^2}.\label{goal_F}
\end{align}
in Proposition~\ref{pro:mainprogfgf} are satisfied immediately follows from \eqref{eq:global} with $|z_0|-1 \gtrsim 1$.
\end{proof}

\section{Ginibre computations}\label{app:Ginibre} \nc

In this section we consider the complex Ginibre ensemble, which is a special case of $X$ with i.i.d. Gaussian entries. Recall that the joint probability density of the eigenvalues $\{\sigma_i\}_{i=1}^n$ of a complex Ginibre matrix is given by~(see \eg \cite[Lemma 3]{maxRe_Gin} and references therein)
\begin{align}\label{det_pp}
		\rho_n(z_1,\ldots,z_n): = &\frac{n^n}{\pi^n 1!\cdots n!} \exp\Bigl(-n\sum_{i}\abs{z_i}^2\Bigr) \prod_{i<j}\bigl(n\abs{z_i-z_j}^2\bigr)\nonumber\\
	=&\frac{n^n}{\pi^n n!} e^{-n\abs{\bm z}^2} \det\Bigl( K_n(z_i, z_j)\Bigr)_{i,j=1}^n,
\end{align}
which forms a {\it determinantal point process} with the kernel 
$$K_n(z,w):=\sum_{l=0}^{n-1} \frac{(n z \overline{w})^l}{l!}.$$
Recalling \cite[Eq. (15)-(16)]{maxRe_Gin}, then for any function $\mathfrak{g}: \C \rightarrow [0,1]$ we have 
\begin{align}\label{det_K}
	\E^{\mathrm{Gin}} \prod_{i=1}^n \big(1-\mathfrak{g}(\sigma_i)\big)=\det \big(1-\sqrt{\mathfrak{g}} \wt{K}_n \sqrt{\mathfrak{g}}\big),
\end{align}
where $\det(\cdot)$ is the Fredholm determinant~(see \eg \cite[Definition 4]{maxRe_Gin} for definition) and the rescaled kernel $\wt K_n$ is given by 
\begin{equation}\label{tilde_K}
	\wt K_n(z,w):= \frac{n}{\pi} e^{-n(\abs{z}^2+\abs{w}^2)/2} K_n(z,w) = \frac{n}{\pi} e^{-n(\abs{z}^2+\abs{w}^2-2 z \ov w)/2} \frac{\Gamma(n,n z\ov{w})}{\Gamma(n)}.
\end{equation} 
Here \(\Gamma(\cdot,\cdot)\) denotes the incomplete Gamma function:
\begin{equation}
	\Gamma(s,z):=\int_z^\infty t^{s-1} e^{-t}\dif t, \qquad s\in \N, \quad z\in\C,
\end{equation}
with the integral contour going from $z\in \C$ to the real infinity. Using \cite[IV.(7.8)-(7.11)]{book_fredholm}, or more precisely \cite[Eq.(20)-(21)]{maxRe_Gin}, we have
\begin{equation}\label{det tr}
	\abs{\det(1-K) - \exp(-\Tr K)}\le  \norm{K}_2 e^{(\norm{K}_2+1)^2/2-\Tr K},
\end{equation}
for any finite-rank kernel $K$, where 
\begin{equation}\label{K_2}
	\Tr K = \int_\Omega K(x,x)\dif\mu(x), \quad \norm{K}_2^2 = \int_{\Omega^2} \abs{K(x,y)}^2 \dif\mu(x)\dif \mu(y).
\end{equation}

Then we state the following proposition as the analogue of \cite[Proposition 5]{maxRe_Gin}~(where we considered the rightmost eigenvalue as in Theorem \ref{thm:maxRe}).
\begin{proposition}\label{prop_K}
	For any $|t|\leq \sqrt{\log n}/2$ and $0\leq a<b\leq 2\pi$, define
	\begin{align}\label{A_t}
		A(t,a,b):=\Big\{z\in \C: |z| \geq 1+
	\sqrt{\frac{\gamma_n}{4n}}+\frac{t}{\sqrt{4n\gamma_n}},\quad \arg(z) \in [a,b)\Big\}.
    \end{align}
	Then for any function $\mathfrak{g}: \C \rightarrow [0,1]$ supported on $A(t,a,b)$, we obtain that
	\begin{align}\label{gKg}
		\Tr \sqrt{\mathfrak{g}} \wt K_n \sqrt{\mathfrak{g}}=\int_{t}^\infty \int_{a}^{b} \mathfrak{g}(z)\frac{e^{-r}}{2\pi}\dif \theta \dif r+O\Big({e^{-t}\frac{(\log\log n)^2+\abs{t}^2}{\log n}}\Big),
	\end{align}
	with $z=\ee^{\ii \theta}\big( 1+
	\sqrt{\frac{\gamma_n}{4n}}+\frac{r}{\sqrt{4n\gamma_n}}\big)$, and that 
	\begin{align}\label{off_K}
		\|\sqrt{\mathfrak{g}} \wt K_n \sqrt{\mathfrak{g}}\|_2 \lesssim \ee^{-\sqrt{\log n}/32}.
	\end{align}
\end{proposition}
 The proof of Proposition \ref{prop_K} is completely analogous to the proof of \cite[Proposition 5]{maxRe_Gin} using the following lemma. The proof of Lemma \ref{lemma_K} is postponed to the end of this section. 
\begin{lemma}\label{lemma_K}
	Recall $\gamma_n$ in (\ref{gamma}) and rescale the kernel variable as
\begin{align}\label{rescale}
	z=\ee^{\ii \theta} \Big( 1+\sqrt{\frac{\gamma_n}{4n}} +\frac{r}{\sqrt{4 \gamma_n n}}\Big), \qquad \theta \in [0,2\pi),\quad r \in \R.
\end{align}
Then for any $|r| \leq \sqrt{\log n}/2$, we have
	\begin{align}\label{K_zz}
	\frac{\wt K_n(z,z)}{\sqrt{4 n \gamma_n}} =\frac{\ee^{-r}}{2\pi} \left( 1+O\Big( \frac{\log \log n+r^2}{\log n}\Big)\right).
\end{align}
Choose different $z_1,z_2\in \C$, and rescale them as in (\ref{rescale}). Then for any $|r_i| \leq \sqrt{\log n}/2$ and $|\theta_1-\theta_2| \lesssim n^{-1/4} (\gamma_n)^{-1/2}$, we have
	\begin{align}\label{off_K1}
		\frac{|\widetilde{K}_n(z_1,z_2)|}{\sqrt{n\gamma_n} }
		&\lesssim e^{-\frac{1}{2}(r_1+r_2)}  \left(1+O\left(\frac{\log \log n}{\log n}\right)\right),
	\end{align}
while for any $|r_i| \leq \sqrt{\log n}/2$ and $|\theta_1-\theta_2| \gtrsim n^{-1/2+w}$ with a small $w>0$, we have
	\begin{align} \label{off_K2}
		\frac{|\widetilde{K}_n(z_1,z_2)|}{\sqrt{n\gamma_n} } \lesssim &\frac{\sqrt{\gamma_n}}{\sqrt{n}|\theta_1-\theta_2|} e^{-\frac{1}{2}(r_1+r_2)} \biggl(1+O\Big(\frac{\log\log n}{\log n}\Big)\biggr).
	\end{align}
	In addition, for any $r_i \geq 0$, we have the following uniform bound
	\begin{align}\label{uniform_K_off}
		\frac{|\wt K_n(z_1,z_2)|}{\sqrt{n \gamma_n}} \lesssim |z_1z_2| \ee^{-\frac{1}{6}(r_1+r_2)}. 
	\end{align}
\end{lemma}

\begin{proof}[Proof of Proposition \ref{prop_K}]	Note that from (\ref{K_2}) and changing the variables as in (\ref{rescale}), we have
	    \begin{align}
			\Tr \sqrt{\mathfrak{g}}\wt K_n\sqrt \mathfrak{g}&= \int_{A(t,a,b)} \mathfrak{g}(z) \wt K_n(z,z)\dif^2 z = \int_{t}^\infty \int_{a}^{b} \mathfrak{g}(z) \frac{\wt K_n(z,z)}{\sqrt{4 n \gamma_n}} \big( 1+\sqrt{\frac{\gamma_n}{4n}} +\frac{r}{\sqrt{4 \gamma_n n}}\big)\dd \theta \dd r   \nonumber\\
			&=\Big( \int_{t}^{r_0}+\int_{r_0}^\infty\Big) \int_{a}^{b} \mathfrak{g}(z) \frac{\wt K_n(z,z)}{\sqrt{4 n \gamma_n}} \big( 1+\sqrt{\frac{\gamma_n}{4n}} +\frac{r}{\sqrt{4 \gamma_n n}}\big) \dd \theta \dd r \nonumber\\
			&= \int_{t}^\infty \int_{a}^{b} \mathfrak{g}(z)\frac{e^{-r}}{2\pi}\dif \theta \dif r +O\Big({e^{-t}\frac{(\log\log n)^2+\abs{t}^2}{\log n}}\Big),
	\end{align}
	where we chose $r_0=4(\log \log n+|t|)$ and we used (\ref{K_zz}) for the regime $t\leq r \leq x_0$ and used (\ref{uniform_K_off}) for the complementary regime $r>r_0$. This proves (\ref{gKg}).

	For any function $\mathfrak{g}: \C \rightarrow [0,1]$ supported on $A(t,a,b)$  given by (\ref{A_t}), then from (\ref{K_2}) we have 
	$$\|\sqrt{\mathfrak{g}} \wt K_n \sqrt{\mathfrak{g}}\|^2_2=\Tr (\sqrt{\mathfrak{g}} \wt K_n \sqrt{\mathfrak{g}})^2 \le \iint_{A(t,a,b)} \abs{\wt K_n(z_1,z_2)}^2\dif^2 z_1 \dif^2 z_2.$$
	Changing the variables as in (\ref{rescale}), we use (\ref{off_K1}) to estimate the regime $|r_i|\leq \sqrt{\log n}/2,~|\theta_1-\theta_2| \lesssim n^{-1/4} (\gamma_n)^{-1/2}$, and (\ref{off_K1}) to estimate the regime $|r_i|\leq \sqrt{\log n}/2,~|\theta_1-\theta_2| \gtrsim n^{-1/4} (\gamma_n)^{-1/2}$, together with (\ref{uniform_K_off}) to estimate the complementary regime $|r_i|\geq \sqrt{\log n}/2$. Thus we prove (\ref{off_K}) and hence finish the proof of Proof of Proposition \ref{prop_K}.
\end{proof}

As a direct application of Proposition \ref{prop_K}, we obtain the following corollary. 
\begin{corollary}\label{cor:Ginibre}
	Both Theorem \ref{thm:gumbel} and Theorem \ref{thm:poisson} hold true for the complex Ginibre ensemble.
\end{corollary}
\begin{proof}[Proof of Corollary \ref{cor:Ginibre}]
We start with proving Theorem~\ref{thm:gumbel} for the complex Ginibre ensemble. Note that
	\begin{align}\label{k=1}
		\P^{\mathrm{Gin}}\Bigl(\max_{\arg(\sigma_i) \in [a,b)}\{|\sigma_i|\}< 1 + \sqrt{\frac{\gamma_n}{4n}}+\frac{t}{\sqrt{4\gamma_n n}} \Bigr)&
		=\E^{\mathrm{Gin}} \prod_{i=1}^n \Big(1- \one_{A(t,a,b)}(\sigma_i)\Big)= \det(1-\one_{A(t,a,b)} \wt K_n \one_{A(t,a,b)})\nonumber\\
				&\xrightarrow{n\to\infty} \exp\Bigl(-\int_{a}^b \int_{t}^\infty \frac{e^{-r}}{2\pi}\dif r\dif \theta\Bigr)
		=e^{-\frac{b-a}{2\pi}e^{-t}},
	\end{align}
with $A(t,a,b)$ given by (\ref{A_t}),	where we also used (\ref{det_K}) with $\mathfrak{g}(z)=\one_{A(t,a,b)}(z)$, together with (\ref{det tr}) and Proposition \ref{prop_K}.  Choose $[a,b)=[0,2\pi)$, then this proves the Gumbel law for the 
spectral radius $|\sigma_1|$ in (\ref{spectral_radius_gin}). 

Since the kernel function $K_n$ in (\ref{det_pp}) is invariant under rotations, \ie $K_n(z,w)=K_n\big(e^{\ii \theta}z,e^{\ii \theta}w\big)$, clearly $\arg(\sigma_1)$ is uniformly distributed on $[0,2\pi)$. 
Moreover, fix any $t\in \R$ and any integer $k\geq 1$, then the following holds uniformly in any $1\leq p\leq k$:
\begin{align}\label{indep}
	\P^{\mathrm{Gin}}\Bigl(|\sigma_1|< 1 + \sqrt{\frac{\gamma_n}{4n}}&+\frac{t}{\sqrt{4\gamma_n n}},~ 
	 \frac{2\pi (p-1)}{k} \leq \arg(\sigma_1) < \frac{2\pi p}{k}\Bigr)\nonumber\\
	=&\; \frac{1}{k} \P^{\mathrm{Gin}}\Bigl(|\sigma_1|< 1 + \sqrt{\frac{\gamma_n}{4n}}+\frac{t}{\sqrt{4\gamma_n n}} \Bigr)\nonumber\\
	=&\; \P^{\mathrm{Gin}}\Bigl( \frac{2\pi (p-1)}{k} \leq \arg(\sigma_1) < \frac{2\pi p}{k}\Bigr) \P^{\mathrm{Gin}}\Bigl(|\sigma_1|< 1 + \sqrt{\frac{\gamma_n}{4n}}+\frac{t}{\sqrt{4\gamma_n n}}\Bigr),
\end{align} 
where the first equality follows from the fact that the LHS is independent of $p$ by rotation invariance.
In the second equality we also used the uniform distribution of $\arg(\sigma_1)$. 
We can now extend the relation \eqref{indep} involving events $\arg(\sigma_1) \in [a,b)$ for any  interval
$[a,b)$  by approximating it with intervals of rational endpoints.
This  implies that $\arg(\sigma_1)$ is independent of $|\sigma_1|$ and hence proves Theorem \ref{thm:gumbel} for the complex Ginibre ensemble. 
	
	We continue to prove Theorem \ref{thm:poisson} for the complex Ginibre ensemble. For any function $g: \C \rightarrow [0,\infty)$ supported on $A(t,a,b)$, using (\ref{det_K}) with $\mathfrak{g}(z)=1-e^{-g(z)}$, (\ref{det tr}), and Proposition \ref{prop_K}, we have
	\begin{align}
			\E^{\mathrm{Gin}} \exp\Big(-\sum_{i=1}^n  g(\sigma_i)\Big) &= \det\Bigl(1-\sqrt{1-e^{-g}}\wt K_n \sqrt{1-e^{-g}}\Big)\nonumber\\ 
			&\xrightarrow{n\to\infty} \exp\Big(-\int_{0}^{2\pi} \int_{t}^\infty (1-e^{-g(z)}) \frac{e^{-r}}{2\pi}\dif r\dif \theta\Big),
	\end{align}
where we changed the variable $z$ into $(r,\theta)$ as in (\ref{rescale}). \nc
		We hence finished the proof of Corollary \ref{cor:Ginibre}.
\end{proof}

\begin{remark}\label{rem:kostlan}
Alternatively, the explicit joint distribution of the largest $k$ Ginibre eigenvalues (in modulus) can be directly computed due to Kostlan's observation~\cite{Kostlan92}. 
 That is, the collection of moduli of the eigenvalues $\{|\sigma_{i}|\}_{i=1}^n$ of a complex Ginibre matrix has the same distribution as the collection of independent chi-distributed random variables $\{\chi_{2k}\}_{k=1}^n$, \ie
	\begin{align}\label{kostlan}
		 \{|\sigma_1|,|\sigma_2|,\cdots, |\sigma_n|\} \stackrel{\mathrm{d}}{=} \{\chi_{2},\chi_{4}, \cdots, \chi_{2n}\}.
	\end{align}
In other words the distribution of the
decreasingly labelled eigenvalues $|\sigma_1| \geq |\sigma_2| \geq \cdots \geq |\sigma_n|$ coincides with
the decreasingly ordered statistics of independent chi-distributed variables with even parameters. 
\end{remark}

\bigskip 
Finally it still remains to prove the key technical input Lemma \ref{lemma_K}.
\begin{proof}[Proof of Lemma \ref{lemma_K}]
Note that the first estimate in (\ref{K_zz}) has already been proven in \cite[Lemma E.1]{maxRe_Gin}, while the last estimate in (\ref{uniform_K_off}) follows directly from \cite[Eq. (E.5)]{maxRe_Gin} using the Cauchy-Schwarz inequality.
In the following we focus on proving (\ref{off_K1}) and (\ref{off_K2}). Recall that
	\begin{align}\label{KK}
		\wt K_n(z_1,z_2) = \frac{n}{\pi} e^{-n(\abs{z_1}^2+\abs{z_2}^2-2 z_1 \ov z_2)/2} \frac{\Gamma(n,n z_1\ov{z_2})}{\Gamma(n)}.
	\end{align}
	Note that 
	\begin{align}\label{z1z2}
		z_1 \ov{z_2 } =\ee^{\ii (\theta_1-\theta_2)} \left( 1+ \sqrt{\frac{\gamma_n}{n}}+\frac{r_1+r_2}{\sqrt{4 \gamma_n n}}+  \frac{r_1+r_2+\gamma_n}{4n} +\frac{r_1r_2}{4 \gamma_n n} \right),
	\end{align}
	and hence for any $|r_i| \leq \sqrt{\log n}/2$, 
	\begin{align}\label{z1z200}
		|1-z_1 \ov z_2| \gtrsim |\theta_1-\theta_2| +\sqrt{\gamma_n/n}, \qquad |z_i|^2=1+\frac{\gamma_n +r_i}{\sqrt{\gamma_n n}}+O\Big(\frac{\gamma_n}{n}\Big).
	\end{align}
	
	For $|\theta_1-\theta_2| \lesssim n^{-1/4} (\gamma_n)^{-1/2}$, we use \cite[Lemma 3.3]{RS14} (or \cite[Eq. (47)]{maxRe_Gin}) \ie 
	\begin{align}
		\frac{\Gamma(n,nz_1 \ov{z_2})}{\Gamma(n)} &=  \frac{e^{-n(z_1 \ov{z_2}-1)^2/2 }}{\sqrt{2\pi n}(z_1 \ov{z_2}-1)} \biggl(1+O\Big({\abs{z_1 \ov{z_2}-1}+n\abs{z_1 \ov{z_2}-1}^3+\frac{1}{n\abs{z_1 \ov{z_2}-1}^2}}\Big)\biggr).
	\end{align}
	Combining with (\ref{KK})-(\ref{z1z200}) and using that $|1-z_1 \ov z_2| \lesssim n^{-1/4} (\gamma_n)^{-1/2}$, we obtain that
	\begin{align}
		|\widetilde{K}_n(z_1,z_2)|&
		\lesssim \frac{\sqrt{n}}{|\theta_1-\theta_2| +\sqrt{\gamma_n/n}} e^{-\frac{1}{2}(r_1+r_2+\gamma_n)} \\
		&\lesssim \sqrt{n\gamma_n} e^{-\frac{1}{2}(r_1+r_2)}  \left(1+O\left(\frac{\log \log n}{\log n}\right)\right),
	\end{align}
	where we also used that 
	\begin{align}\label{K_gamma}
		e^{-\frac{1}{2}\gamma_n}=\exp\left(-\frac{1}{2}\log\frac{n}{(\log n)^2 2\pi}\right)=\frac{(2\pi)^{1/2}\gamma_n}{\sqrt{n}}\left(1+O\left(\frac{\log\log n}{\log n}\right)\right).
	\end{align}
	This proves (\ref{off_K1}). 
	
	For $|\theta_1-\theta_2| \gtrsim  n^{-1/2+w}$, we use \cite[Lemma 3.4]{RS14} (or \cite[Eq. (42)]{maxRe_Gin})
	\begin{equation}
		\frac{\Gamma(n,nz_1\ov{z_2})}{\Gamma(n)} = e^{n(1-z_1\ov{z_2})} \frac{(z_1\ov{z_2})^n }{\sqrt{2\pi n}(1-z_1\ov{z_2})} \biggl(1+O\Big(\frac{1}{n\abs{1-z_1\ov{z_2}}^2}\Big)\biggr).
	\end{equation}
	Combining with (\ref{KK}), we obtain that
	\begin{align}
		\big|\wt K_n(z_1,z_2)\big| =& \frac{n}{\pi}  \frac{\ee^{-\frac{n}{2}(\abs{z_1}^2+\abs{z_2}^2-2-2\log|z_1 \ov z_2|)} }{\sqrt{2\pi n}(1-z_1\ov{z_2})} \biggl(1+O\Big(\frac{1}{n\abs{1-z_1\ov{z_2}}^2}\Big)\biggr)\nonumber\\
		\lesssim & \frac{\sqrt{n}}{|\theta_1-\theta_2| +\sqrt{\gamma_n/n}} e^{-\frac{1}{2}(r_1+r_2+\gamma_n)}\biggl(1+O\Big(\frac{1}{\sqrt{n}|\theta_1-\theta_2|+\sqrt{\gamma_ n}}\Big)^2\biggr)\nonumber\\
		\lesssim &\frac{\gamma_n}{|\theta_1-\theta_2|} e^{-\frac{1}{2}(r_1+r_2)} \biggl(1+O\Big(\frac{\log\log n}{\log n}\Big)\biggr),
	\end{align}
	where we also used (\ref{z1z2}), (\ref{K_gamma}) and that for any $|r_i| \leq \sqrt{\log n}/2$,
	$$\abs{z_i}^2-1-2 \log |z_i|=2(|z_i|-1)^2+O\big((|z_i|-1)^3\big)=\frac{(\gamma_n+r_i)^2}{2\gamma_n n}+O\Big(\frac{\gamma^{3/2}_n}{n^{3/2}}\Big).$$
		We hence proved (\ref{off_K2}) and completed the proof of Lemma \ref{lemma_K}. 
\end{proof}

\end{document}